\documentclass{article}

\usepackage[all]{xy}
\usepackage{amssymb}
\usepackage{amsmath}
\usepackage{amsthm}
\newtheorem{thm}{Theorem}
\newtheorem{prop}{Proposition}
\newtheorem{lem}{Lemma}
\newtheorem{rem}{Remark}
\newtheorem{cor}{Corollary}
\newtheorem{defi}{Definition}
\newtheorem{defiprop}{Definition-Proposition}

\def\Top{\mathop{\rm Top}\nolimits}
\def\CH{\mathop{\rm CH}\nolimits}

\def\dim{\mathop{\rm dim}\nolimits}

\def\Tr{\mathop{\rm Tr}\nolimits}
\def\Cor{\mathop{\rm Cor}\nolimits}

\def\SmVar{\mathop{\rm SmVar}\nolimits}

\def\PSmVar{\mathop{\rm PSmVar}\nolimits}
\def\Var{\mathop{\rm Var}\nolimits}
\def\Hom{\mathop{\rm Hom}\nolimits}
\def\Spec{\mathop{\rm Spec}\nolimits}

\def\QPVar{\mathop{\rm QPVar}\nolimits}
\def\AnSp{\mathop{\rm AnSp}\nolimits}
\def\CW{\mathop{\rm CW}\nolimits}
\def\Cw{\mathop{\rm Cw}\nolimits}
\def\PVar{\mathop{\rm PVar}\nolimits}

\def\Tot{\mathop{\rm Tot}\nolimits}
\def\sing{\mathop{\rm sing}\nolimits}

\def\Im{\mathop{\rm Im}\nolimits}

\def\Cone{\mathop{\rm Cone}\nolimits}
\def\ad{\mathop{\rm ad}\nolimits}
\def\card{\mathop{\rm card}\nolimits}

\def\An{\mathop{\rm An}\nolimits}
\def\Ho{\mathop{\rm Ho}\nolimits}
\def\PSh{\mathop{\rm PSh}\nolimits}
\def\AnSm{\mathop{\rm AnSm}\nolimits}
\def\Tr{\mathop{\rm Tr}\nolimits}
\def\DM{\mathop{\rm DM}\nolimits}
\def\AnDM{\mathop{\rm AnDM}\nolimits}
\def\pt{\mathop{\rm pt}\nolimits}
\def\DA{\mathop{\rm DA}\nolimits}
\def\AnDA{\mathop{\rm AnDA}\nolimits}
\def\CwDA{\mathop{\rm CwDA}\nolimits}
\def\CwDM{\mathop{\rm CwDM}\nolimits}

\def\Fun{\mathop{\rm Fun}\nolimits}
\def\Sh{\mathop{\rm Sh}\nolimits}
\def\Ab{\mathop{\rm Ab}\nolimits}
\def\Bti{\mathop{\rm Bti}\nolimits}
\def\TM{\mathop{\rm TM}\nolimits}
\def\Ouv{\mathop{\rm Ouv}\nolimits}
\def\coker{\mathop{\rm coker}\nolimits}
\def\Ext{\mathop{\rm Ext}\nolimits}
\def\CS{\mathop{\rm CS}\nolimits}
\def\GVar{\mathop{\rm GVar}\nolimits}

\def\TriCat{\mathop{\rm TriCat}\nolimits}

\title{On the realization functor of the derived category of mixed motives}

\author{Johann Bouali}

\begin{document}

\maketitle

\begin{abstract}
We give an alternative construction of the Betti realization functor on the derived category of
motives of complex algebraic varieties via the category of CW complexes
instead of the category of complex analytic spaces. In particular, we show that
the functor we define via the category of CW complexes coincide with Ayoub's one.
We deduce from this construction that Ayoub's realization functor on geometric motives
factor through Nori motives and that the image of this functor on the morphisms 
between the motive of a point and a Tate twist of the motive with compact support of a complex algebraic variety  
coincide with the classical cycle class map on higher chow groups.
\end{abstract}

\textbf{Notations}: 

\begin{itemize}

\item Denote by $\Top$ the category of topological spaces.
Denote by $\Var(k)$ the category of algebraic varieties over a field $k$, i.e. schemes of finite type over $k$.
Let us call $\PVar(k)\subset\QPVar(k)\subset\Var(k)$ the full subcategories  quasi-projective varieties and projective varieties respectively. 
Let us call $\PSmVar(k)\subset\SmVar(k)\subset\Var(k)$ the full subcategories of smooth varieties and smooth projective varieties respectively.
Denote by $\CW\subset\Top$ the full subcategory of $CW$ complexes, by $\CS\subset\CW$ the full subcategory of $\Delta$ complexes,
and by $\TM\subset\CW$ the full subcategory of topological manifolds
which admits a CW structure (a topological manifold admits a CW structure if it admits a differential structure). 
Denote by $\AnSp(\mathbb C)$ the category of analytic spaces over $\mathbb C$.
and by $\AnSm(\mathbb C)\subset\AnSp(\mathbb C)$ the full subcategory of smooth analytic spaces (i.e. complex analytic varieties).

\item For $V\in\Var(\mathbb C)$, we denote by $V^{an}\in\AnSp(\mathbb C)$ 
the complex analytic space associated to $V$ with the usual topology induced by the usual topology of $\mathbb C^N$. 
For $W\in\AnSp(\mathbb C)$, we denote by $W^{cw}\in\AnSp(\mathbb C)$ the topological space given by $W$ which is a $CW$ complex.
For simplicity, for $V\in\Var(\mathbb C)$, we denote by $V^{cw}:=(V^{an})^{cw}\in\CW$. 
We have then  
\begin{itemize}
\item the analytical functor $\An:\Var(\mathbb C)\to\AnSp(\mathbb C)$, $\An(V)=V^{an}$,
\item the forgetful functor $\Cw:\AnSp(\mathbb C)\to\CW$, $\Cw(W)=W^{cw}$,
\item the composite of these two functors $\widetilde\Cw=\Cw\circ\An:\Var(\mathbb C)\to\CW$, $\widetilde\Cw(V)=V^{cw}$. 
\end{itemize}

\item We denote by $\square^n=(\mathbb P^1\backslash \left\{1\right\})^n\subset(\mathbb P^1)^n$. For $X\in\Var(\mathbb C)$ let
$\mathcal Z^p(X,n)\subset \mathcal Z^p(X\times\square^n)$ be the subgroup of $p$ codimentional cycle in $X\times\square^n$ meeting
all faces of $\square^n$ properly.
We denote by $\pi_X:X\times(\mathbb P^1)^n\to X$ and $\pi_{(\mathbb P^1)^n}:X\times(\mathbb P^1)^n\to(\mathbb P^1)^n$ the projections.

\item For $X\in\Top$ a topological space, we denote by 
$C_{\bullet}^{\sing}(X,\mathbb Z)=\mathbb Z\Hom_{\Top}(\Delta^{\bullet},X)$ the complex
of singular chains, $\Delta^{p}\subset\mathbb R^p$ being the standard simplex.
We denote by $\mathbb I^n=[0,1]^n$ and we will consider the closed embeddings of CW complexes 
$i_n:[0,1]^n\hookrightarrow\square^n:=(\mathbb P^1_{\mathbb C}\backslash\left\{1\right\})^n$ 
whose image is the product $[0,\infty]^n\subset\square^n$ of the segments
$\mathbb R^-=[0,\infty]\subset\square^1$ (c.f. the definition of $T_z$ in \cite{MKerr}). 
In particular $i_n$ send $0$ to $0$ and $1$ to $\infty$
and gives a morphism of complexes $i:\mathbb I^*\hookrightarrow\square^*$ in $\mathbb Z(CW)$.
We denote by $\bar{\mathbb D}^n\subset\mathbb C^n$ the closed ball of radius 1, 
and by $i'_{1n}:\mathbb I^n\hookrightarrow\bar{\mathbb D}^n$ the inclusion of pro complex analytic spaces and
and ${i'_{1n}}^{cw}:\mathbb I^n\hookrightarrow\bar{\mathbb D}^n$ is the corresponding inclusion of CW complexes.

\item For a (small) category $\mathcal S$, we denote by 
$\PSh(\mathcal S,M):=\Fun(\mathcal S^{op},M)$ the big category of presheaves 
on $\mathcal S$ with value in $M$.
If  $M$ is a model category with some extra assumptions (c.f.\cite{AyoubT}), 
the projective fibration (rep. the injective cofibration) and the termwise weak equivalence of $\PSh(\mathcal S,M)$
define a projective (resp. injective) model structure $\PSh(\mathcal S,M)$.
In this paper, we will consider $\mathcal{M}_P(\PSh(\mathcal S,M))$ the projective model structure on  
$\PSh(\mathcal S,M)$.
presheaves. 
\begin{itemize}
\item For $X\in\mathcal S$, we denote by $\mathbb Z(X)\in\PSh(\mathcal S,\Ab)$ the preshesf
given by Yoneda embedding, that is the presheaf given by for $Z\in\mathcal S$,
$\mathbb Z(X)(Z)=\mathbb Z\Hom_{\mathcal S}(Z,X)$ and for $f:Z'\to Z$ a morphism in $\mathcal S$,
$\mathbb Z(X)(f):g\in\Hom_{\mathcal S}(Z,X)\mapsto g\circ f\in\Hom_{\mathcal S}(Z',X)$
\item For $h:X\to Y$ a morphism in $\mathcal S$, we denote by $\mathbb Z(h):\mathbb Z(X)\to\mathbb Z(Y)$,
the morphism in $\PSh(\mathcal S,C(\mathbb Z)$ given by Yoneda embedding, 
that is the morphism given by for $Z\in\mathcal S$, 
$\mathbb Z(h)(Z):g\in\Hom_{\mathcal S}(Z,X)\mapsto h\circ g\in\Hom_{\mathcal S}(Z,Y)$
\end{itemize}

\item We denote by $C(\mathbb Z)=C(\Ab)$ the category of abelian complexes, $C^-(\mathbb Z)\subset C^(\mathbb Z)$
the full subcategory consisting of bounded above complexes, $D(\mathbb Z)$ the derived category of
$C(\mathbb Z)$ with respect to quasi-isomorphism, and $D^-(\mathbb C)\subset D(\mathbb Z)$ the image
of $C^-(\mathbb Z)$ under the localization functor $C^(\mathbb Z)\to D(\mathbb Z)$.
For $S\in\Top$, we denote by $C(S):=\PSh(S,C(\mathbb Z))$ the category of complexes of presheaves on $S$,
$C^-(\mathbb Z)\subset C(\mathbb Z)$ the full subcategory consisting of bounded above complexes, 
$D(S)$ the derived category of $C(S)$ with respect to the morphisms of complexes of presheaves
which are quasi-isomorphisms after sheaftification, 
and $D^-(S)\subset D(S)$ the image of $C^-(S)$ under the localization functor $C(S)\to D(S)$.

\end{itemize}

\section{Introduction}

Let $S\in\Var(\mathbb C)$.
In \cite{AyoubB},\cite{AyoubG1} and \cite{AyoubG2}, J.Ayoub has given and studied
a construction of the Betti realization functor on $\DA^-(S,\mathbb Z)$
the derived category of mixed motives of $\Var(\mathbb C)/S$ the complex algebraic varieties
together with a morphism over $S$.
He mentioned also the construction of the Betti realization functor on $\DM^-(S,\mathbb Z)$, 
the derived category of mixed motives of $\Var(\mathbb C)/S$ with transfers.
Let $k$ a field.
\begin{itemize}
\item The derived category of (effective) motives of $\Var(k)/S$ is 
the homotopy category of the category $P^-(S)=\PSh(\Var(k)^{sm}/S,C^-(\mathbb Z))$ of bounded above complexes
of presheaves on the category of algebraic varieties over $k$ together with a smooth morphism over $S$,
with respect to the projective $(\mathbb A^1,et)$ model structure (c.f. definition \ref{Amodstr}(i)).
\item Similarly, the derived category of (effective) motives of $\Var(k)/S$ with transfers is 
the homotopy category of the category $PC^-(S)=\PSh(\Cor^{fs}_{\mathbb Z}(\Var(k)^{sm}/S),C^-(\mathbb Z))$ 
of bounded above complexes
of presheaves on the category of finite and surjective correspondences between algebraic varieties over $k$ 
together with a smooth morphism over $S$ with respect to the projective $(\mathbb A^1,et)$ model structure 
(c.f. definition \ref{Amodstr}(ii)). 
\item For $X\in\Var(k)$ and $l:D\hookrightarrow X$ a (locally closed) subvariety, 
the motive of the pair $(X,D)$ is $M(X,D)=D(\mathbb A^1,et)(\mathbb Z_{tr}(X,D))\in\DM^-(k,\mathbb Z)$, where
$\mathbb Z_{tr}(X,D)=\coker(\mathbb Z_{tr}(l))$ is the cokernel of the injective morphism of presheaves
$\mathbb Z_{tr}(l):\mathbb Z_{tr}(D)\hookrightarrow\mathbb Z_{tr}(X)$ 
and $D(\mathbb A^1,et):PC^-\to\DM^-(k,\mathbb Z)$ is the localization functor.
\end{itemize}
The construction of J.Ayoub is defined on $\DA^-(S)$ via $\AnDA^-(S^{an})$, 
the derived category of motives of complex analytic space, which is 
the homotopy category of $P^-(\An,S^{an})=\PSh_{\mathbb Z}(\AnSp(\mathbb C)^{sm}/S^{an},C^-(\mathbb Z))$ 
of bounded above complexes of presheaves on the category of  complex analytic spaces together with a smooth 
morphism over $S^{an}$ with respect to the projective $(\mathbb D^1,usu)$ model structure (c.f. definition \ref{RDmodstr}(i)).  
Similarly on $\DM^-(S,\mathbb Z)$ it is defined via $\AnDM^-(S^{an},\mathbb Z)$, 
the derived category of motives of complex analytic space, which is 
the homotopy category of  
$PC^-(\An,S^{an})=\PSh_{\mathbb Z}(\Cor^{fs}_{\mathbb Z}(\AnSp(\mathbb C)^{sm}/S^{an}),C^-(\mathbb Z))$ 
the category of bounded above complexes of presheaves on the category of finite and surjective correspondences 
between complex analytic spaces together with a smooth morphism over $S$
with respect to the projective $(\mathbb D^1,usu)$ model structure (c.f. definition \ref{RDmodstr}(ii)).  
For $S\in\AnSp(\mathbb C)$, we consider the commutative diagram 
\begin{equation*}
\xymatrix{\Cor^{fs}_{\mathbb Z}(\AnSp(\mathbb C)^{sm}/S)\ar[rd]^{e^{tr}_{an}(S)}\ar[rr]^{\Tr(S)} 
& \, & \AnSp(\mathbb C)^{sm}/S\ar[ld]_{e_{an}(S)} \\
\, & \Ouv(S) & \,}
\end{equation*}
the morphism of sites given respectively by 
the inclusion functors $e_{an}(T):\Ouv(S)\hookrightarrow\AnSp(\mathbb C)^{sm}/S$,  
$e^{tr}_{an}(S):\Ouv(S)\hookrightarrow\Cor^{fs}_{\mathbb Z}(\AnSp(\mathbb C)^{sm}/S)$
and $\Tr(S):\Cor^{fs}_{\mathbb Z}(\AnSp(\mathbb C)^{sm}/S)\hookrightarrow\AnSp(\mathbb C)^{sm}/S$.
The definition of Betti realization functor by J.Ayoub is
\begin{defi}\cite{AyoubB}\cite{AyoubG2}
\begin{itemize}
\item[(i)] The Betti realisation functor (without transfers) is the composite :
\begin{equation}
\Bti_0(S)^*:\DA^{-}(S,\mathbb Z)\xrightarrow{\An(S)^*}\AnDA^{-}(S^{an},\mathbb Z)\xrightarrow{Re_{an}(S)_*}D^{-}(\mathbb Z)
\end{equation}
\item[(ii)] The Betti realisation functor with transfers is the composite :
\begin{equation}
\Bti(S)^*:\DM^{-}(S,\mathbb Z)\xrightarrow{\An(S)^*}\AnDM^{-}(\mathbb Z)\xrightarrow{Re^{tr}_{an}(S)_*}D^{-}(\mathbb Z)
\end{equation}
\end{itemize}
\end{defi}
Since $\An(S)^*$ derive trivially by proposition \ref{RAnCwtCw}(ii)
and and $L\Tr(S^{an})^*:\AnDA^-(S^{an},\mathbb Z)\to\AnDM^-(S^{an},\mathbb Z)$ is the inverse of $\Tr(S^{an})_*$, 
we have $\Bti_0(S)^*=\widetilde\Bti(S)^*\circ L\Tr(S)^*$.

In \cite{AyoubG1}, J.Ayoub has constructed an explicit object which gives the localization functor for the $(\mathbb D^1,usu)$
model strucure. We recall this in theorem \ref{AnS}(i) and give a relative version (with and without transfers) 
in theorem \ref{RAnS} :
\begin{thm}
Let $S\in\AnSp(\mathbb C)$,
\begin{itemize}
\item[(i)] For $F^{\bullet}\in\PSh(\AnSp(\mathbb C)^{sm}/S,C^-(\mathbb Z))$,
$\underline{\sing}_{\bar{\mathbb D}^*}F^{\bullet}\in\PSh(\AnSp(\mathbb C)^{sm},C^-(\mathbb Z))$ is $\mathbb D^1$ local and 
the inclusion morphism $S(F^{\bullet}):F^{\bullet}\to\underline{\sing}_{\bar{\mathbb D}^*}F^{\bullet}$ is an $(\mathbb D^1,usu)$ equivalence. 
\item[(ii)] For $F^{\bullet}\in\PSh_{\mathbb Z}(\Cor^{fs}_{\mathbb Z}(\AnSp(\mathbb C)^{sm}),C^-(\mathbb Z))$, 
$\underline{\sing}_{\bar{\mathbb D}^*}F^{\bullet}\in\PSh_{\mathbb Z}(\Cor^{fs}_{\mathbb Z}(\AnSp(\mathbb C)^{sm}),C^-(\mathbb Z))$ 
is $\mathbb D^1$ local and 
the inclusion morphism $S(F^{\bullet}):F^{\bullet}\to\underline{\sing}_{\bar{\mathbb D}^*}F^{\bullet}$ is an $(\mathbb D^1,usu)$ equivalence. 
\end{itemize}
\end{thm}
The categories $\AnDM^-(S)$ and $\AnDA^-(S)$, for $S\in\AnSp(\mathbb C)$ satisfy the following (see \cite{AyoubB} and \cite{AyoubG2})
\begin{thm}
\begin{itemize}
\item[(i)] The adjonction
\begin{equation*}
(\Tr(S)^*,\Tr(S)_*):\PSh(\AnSp(\mathbb C)^{sm}/S,C^-(\mathbb Z))
\leftrightarrows\PSh_{\mathbb Z}(\Cor^{fs}_{\mathbb Z}(\AnSp(\mathbb C)^{sm}/S),C^-(\mathbb Z))
\end{equation*}
is a Quillen equivalence for the $(\mathbb D^1,usu)$ model structures. That is, the derived functor
$\Tr(S)_*:\AnDM^-(S)\xrightarrow{\sim}\AnDA^-(S)$
is an isomorphism and $L\Tr(S)^*$ is it inverse. 
\item[(ii)] The adjonction
$(e_{an}(S)^*,e_{an}(S)):C^-(S)\leftrightarrows\PSh_{\mathbb Z}(\AnSp(\mathbb C)^{sm}/S,C^-(\mathbb Z))$
is a Quillen equivalence for the $(\mathbb I^1,usu)$ model structures. That is, the derived functor
$e_{an}^*:D^-(S)\xrightarrow{\sim}\AnDA^-(S,\mathbb Z)$
is an isomorphism and $Re_{an}(S)_*:\AnDA^-(S,\mathbb Z)\xrightarrow{\sim}D^-(S)$ is it inverse. 
\item[(iii)] The adjonction
$(e^{tr}_{an}(S)^*,e^{tr}_{an}(S)_*):C^-(S)\leftrightarrows
\PSh_{\mathbb Z}(\Cor^{fs}_{\mathbb Z}(\AnSp(\mathbb C)^{sm}/S),C^-(\mathbb Z))$
is a Quillen equivalence for the $(\mathbb D^1,usu)$ model structures. 
That is, the derived functor
$e^{tr}_{an}(S)^*:D^-(S)\xrightarrow{\sim}\AnDM^-(S,\mathbb Z)$
is an isomorphism and $Re^{tr}_{an*}:\AnDM^-(S,\mathbb Z)\xrightarrow{\sim}D^-(S)$ is it inverse. 
\end{itemize}
\end{thm}

In this paper we give a construction of the Betti realization functor via the category of CW complexes.
The reason we do this is that the cycle class map on complex analytic spaces 
and the action of correspondence on homology on smooth complex analytic spaces (i.e. complex analytic manifold)
are defined in a purely topological way so that we does not the need complex structure which gives in the smooth case
the Fr\"olicher filtration on then 
(the Fr\"olicher filtration gives the Hodge fitration on cohomoloy in the compact Kalher case. 

Let $X,Y,Z\in\Top$. 
Assume $Y$ is Hausdorf (equivalently the diagonal $\Delta_Y\subset Y\times Y$ is a closed subset).
There is a natural composition law (c.f.\ref{CorTop}) between closed subset of $X\times Y$ which are
finite and surjective over $X$ and closed subset of $Y\times Z$ which are finite and surjective over $Y$.
We denote $\mathbb I^1:=\left[0,1\right]$.
Consider now the full subcategory $\CW\subset\Top$ consisting of CW complexes.
By a CW subcomplex of $X\in\CW$, we mean a topological embbeding $Z\hookrightarrow X$ with $Z$ a CW complex. 
By a closed CW subcomplex of $X\in\CW$, we mean a topological closed embbeding $Z\hookrightarrow X$ with $Z$ a CW complex,
that is the image of the embedding is a closed subset of $X$. 
Let $X,S\in\CW$, $S$ connected, and $h:X\to S$ a finite and surjective morphism in $\CW$.
We say that $X/S=(X,h)\in\CW/S$ is reducible if $X=X_1\cup X_2$, with $X_1,X_2$ closed CW subcomplexes
finite and surjective over $S$ and $X_1,X_2\neq X$.
A pair $Y/S=(Y,h')\in\CW/S$ with $Y\in\CW$ and $h':Y\to S$ a finite and surjective morphism
is called irreducible if it is not reducible. In particular $Y$ is connected.
\begin{itemize}
\item For $X\in\CW$, $\Lambda$ a commutative ring and $p\in\mathbb N$, 
we denote by $\mathcal Z_p(X,\Lambda)$ the free $\Lambda$ module generated 
by the closed CW subcomplex of $X$ of dimension $p$ and by 
$\mathcal Z^p(X,\Lambda)=\mathcal Z_{d_X-p}(X,\Lambda)$ 
the free $\Lambda$ module generated by the closed CW subcomplex of $X$ of codimension $p$.
\item For $X,Y\in\CW$, $X$ connected and $\Lambda$ a commutative ring, we define :
$\mathcal{Z}^{fs/X}(X\times Y,\Lambda)\subset\mathcal{Z}_{d_X}(X\times Y,\Lambda)$ 
the free $\Lambda$ module generated by the closed CW subcomplexes of $X\times Y$ finite and surjective over $X$
which are irreducible. 
\item For $X,Y\in\CW$, and $\Lambda$ a commutative ring, we define :
$\mathcal{Z}^{fs/X}(X\times Y,\Lambda):=\oplus_i\mathcal{Z}^{fs/X_i}(X_i\times Y,\Lambda)$ 
where $X=\sqcup_i X_i$, with $X_i$ the connected components of $X$.
\end{itemize}
We define the category of finite surjective correspondences on CW complexes.
\begin{defi}
\begin{itemize}
\item We define $\Cor^{fs}_{\Lambda}(\CW)$ to be the category whose objects are CW complexes and whose space of morphisms between
$X,Y\in\CW$ is the free $\Lambda$ module $\mathcal{Z}^{fs/X}(X\times Y,\Lambda)$.
The composition law is the one given by (\ref{CorTopCW}).
\item Let $S\in\CW$. We define $\CW^{sm}/S$ to be the category whose objects are $X/S=(X,h)$ with
$X\in\CW$ and $h:X\to S$ a smooth morphism. Let $X/S=(X,h_1),Y/S=(Y,h_2)\in\CW^{sm}/S$.
A morphism $f:X/S\to Y/S$ is a morphism $f:X\to Y$ such that $f\circ h_1=h_2$. 
\item Let $S\in\CW$. We define $\Cor^{fs}_{\Lambda}(\CW^{sm}/S)$ to be the category whose objects 
are those of $\CW^{sm}/S$ and whose space of morphisms between
$X,Y\in\CW$ is the free $\Lambda$ module $\mathcal{Z}^{fs/X}(X\times_S Y,\Lambda)$.
The composition law is the one given by (\ref{CorTopCW}).
\end{itemize}
\end{defi}
We define 
\begin{itemize}
\item $\CwDA^-(\mathbb Z)$, 
the derived category of motives of CW complexes to be the homotopy category of 
$P^-(CW)=\PSh_{\mathbb Z}(\CW,C^-(\mathbb Z))$, 
the category of bounded above complex of presheaves on the category $\CW$ of CW complexes,
with respect to the projective $(\mathbb I^1,usu)$ model structure (\ref{CWmodstr}(i)).
\item Similarly, we define $\CwDM^-(\mathbb Z)$, the derived category of motives of CW complexes, as 
the homotopy category of $PC^-(CW)=\PSh_{\mathbb Z}(\Cor^{fs}_{\mathbb Z}(\CW),C^-(\mathbb Z))$, 
the category of bounded above complex of presheaves on the category $\Cor^{fs}_{\mathbb Z}(\CW)$,
with respect to the projective $(\mathbb I^1,usu)$ model structure (\ref{CWmodstr}(ii)). 
\item For $X\in\CW$ and $l:D\hookrightarrow X$ a CW subcomplex, the motive of the pair $(X,D)$ is
$M(X,D)=D(\mathbb I^1,usu)(\mathbb Z_{tr}(X,D))\in\CwDM^-(\mathbb Z)$, where
$\mathbb Z_{tr}(X,D)=\coker(\mathbb Z_{tr}(l))$ is the cokernel of the injective morphism of presheaves
$\mathbb Z_{tr}(l):\mathbb Z_{tr}(D)\hookrightarrow\mathbb Z_{tr}(X)$ and
$D(\mathbb I^1,usu):PC^-(CW)\to\CwDM^-(\mathbb Z)$ is the localization functor.
\end{itemize}
We consider the commutative diagram 
\begin{equation*}
\xymatrix{\Cor^{fs}_{\mathbb Z}(\CW)\ar[rd]_{e^{tr}_{cw}}\ar[rr]^{\Tr} & \, & \CW^{sm}\ar[ld]^{e_{cw}} \\
\, & \left\{\pt\right\} & \,}
\end{equation*}
the morphism of sites given respectively by the inclusion functors 
$e_{cw}:\left\{\pt\right\}\hookrightarrow\CW$,  
$e^{tr}_{cw}:\left\{\pt\right\}\hookrightarrow\Cor^{fs}_{\mathbb Z}(\CW)$ 
and $\Tr:\Cor^{fs}_{\mathbb Z}(\CW)\hookrightarrow\CW$.

For $S\in\CW$, we define
\begin{itemize}
\item $\CwDA^-(S,\mathbb Z)$, 
the derived category of motives of CW complexes, as 
the homotopy category of $P^-(CW,S)=\PSh_{\mathbb Z}(\CW^{sm}/S,C^-(\mathbb Z))$ 
the category of bounded above complex of presheaves on $\CW^{sm}/S$.
with respect to the projective $(\mathbb I^1,usu)$ model structure (\ref{RImodstr}(i)). 
\item Similarly, for $S\in\CW$, we define $\CwDM^-(S,\mathbb Z)$, 
the derived category of motives of CW complexes, as 
the homotopy category of $PC^-(CW,S)=\PSh_{\mathbb Z}(\Cor^{fs}_{\mathbb Z}(\CW^{sm}/S),C^-(\mathbb Z))$ 
the category of bounded above complex of presheaves on the category $\Cor^{fs}_{\mathbb Z}(\CW^{sm}/S)$
with respect to the projective $(\mathbb I^1,usu)$ model structure (\ref{RImodstr}(ii)). 
\end{itemize}
For $S\in\CW$, we consider the commutative diagram 
\begin{equation*}
\xymatrix{\Cor^{fs}_{\mathbb Z}(\CW^{sm}/S)\ar[rd]_{e^{tr}_{cw}(S)}\ar[rr]^{\Tr(S)} 
& \, & \CW^{sm}/S\ar[ld]^{e_{cw}(S)} \\
\, & \Ouv(S) & \,}
\end{equation*}
the morphism of sites given respectively by 
the inclusion functors $e_{cw}(S):\Ouv(S)\hookrightarrow\CW^{sm}/S$, 
$e^{tr}_{cw}(T):\Ouv(S)\hookrightarrow\Cor^{fs}_{\mathbb Z}(\CW^{sm}/S)$
and $\Tr(S):\Cor^{fs}_{\mathbb Z}(\CW^{sm}/S)\hookrightarrow\CW^{sm}/S$.

We consider first the absolute case. 
We give an explicit object which induces the $\mathbb I^1$ localization functor on the category of presheaves 
on $\CW$ and additive presheaves on $\Cor_{\mathbb Z}(\CW)$.
More precisely by, considering $\mathbb I^n:=\left[0,1\right]^n$, 
we prove in theorem \ref{CWsingI} the following :
\begin{thm}
\begin{itemize} 
\item For $F^{\bullet}\in\PSh_{\mathbb Z}(\CW,C^-(\mathbb Z))$, 
$\underline{\sing}_{\mathbb I^*}F^{\bullet}\in\PSh_{\mathbb Z}(\Cor^{fs}_{\mathbb Z}(\CW),C^-(\mathbb Z))$ 
is $\mathbb I^1$ local and the inclusion morphism 
$S(F^{\bullet}):F^{\bullet}\to\underline{\sing}_{\mathbb I^*}F^{\bullet}$ is an $(\mathbb I^1,usu)$ equivalence. 
\item For $F^{\bullet}\in\PSh_{\mathbb Z}(\Cor^{fs}_{\mathbb Z}(\CW),C^-(\mathbb Z))$, 
$\underline{\sing}_{\mathbb I^*}F^{\bullet}\in\PSh_{\mathbb Z}(\Cor^{fs}_{\mathbb Z}(\CW),C^-(\mathbb Z))$ 
is $\mathbb I^1$ local and the inclusion morphism 
$S(F^{\bullet}):F^{\bullet}\to\underline{\sing}_{\mathbb I^*}F^{\bullet}$ is an $(\mathbb I^1,usu)$ equivalence. 
\end{itemize}
\end{thm}
Then, we prove that the categories $\CwDM^-$ and $\CwDA^-$ satisfy the following (c.f.theorem \ref{CWMot}) :
\begin{thm}
\begin{itemize} 
\item[(i)]The adjonction
$(\Tr^*,\Tr_*):\PSh(\CW,C^-(\mathbb Z))\leftrightarrows\PSh_{\mathbb Z}(\Cor^{fs}_{\mathbb Z}(\CW),C^-(\mathbb Z))$
is a Quillen equivalence for the $(\mathbb I^1,usu)$ model structures. That is, the derived functor 
\begin{equation}
L\Tr^*:\CwDA^-(\mathbb Z)\xrightarrow{\sim}\CwDM^-(\mathbb Z) 
\end{equation}
is an isomorphism
and $\Tr_*:\CwDM^-(\mathbb Z)\xrightarrow{\sim}\CwDA^-(\mathbb Z)$ is it inverse. 
\item[(ii)] The adjonction
$(e_{cw}^*,e_{cw*}):C^-(\mathbb Z)\leftrightarrows\PSh_{\mathbb Z}(\CW,C^-(\mathbb Z))$
is a Quillen equivalence for the $(\mathbb I^1,usu)$ model structures. That is, the derived functor
$e_{cw}^*:D^-(\mathbb Z)\xrightarrow{\sim}\CwDA^-(\mathbb Z)$ is an isomorphism
and $Re_{cw*}:\CwDA^-(\mathbb Z)\xrightarrow{\sim}D^-(\mathbb Z)$ is it inverse. 
\item[(iii)] The adjonction
$(e^{tr*}_{cw},e^{tr}_{cw*}):C^-(\mathbb Z)\leftrightarrows\PSh_{\mathbb Z}(\Cor^{fs}_{\mathbb Z}(\CW),C^-(\mathbb Z))$
is a Quillen equivalence for the $(\mathbb I^1,usu)$ model structures. That is, the derived functor
$e^{tr*}_{cw}:D^-(\mathbb Z)\xrightarrow{\sim}\CwDM^-(\mathbb Z)$
is an isomorphism and $Re^{tr}_{cw*}:\CwDM^-(\mathbb Z)\xrightarrow{\sim}D^-(\mathbb Z)$ is it inverse.  
\end{itemize}
\end{thm}
For point (i), we use proposition \ref{CWadTr} to prove that $L\Tr^*$ is this inverse of $\Tr_*$.
In proposition \ref{CWadTr}, we prove a key result that for $X\in\CW$
\begin{equation*}
\ad(\Tr^*\Tr_*)(\mathbb Z_{tr}(X):\underline{\sing}_{\mathbb I^*}\mathbb Z(X)
\to\Tr_*\underline{\sing}_{\mathbb I^*}\mathbb Z_{tr}(X)
\end{equation*}
is an equivalence usu local. To see this we use the fact that a CW complex is $\mathbb I^1$ 
homotopy equivalent to a $\Delta$-complex, and that for a $\Delta$-complex there exist
a countable open covering such that the intersection of a finite number of members
of this covering is either empty or a contractible topological space.
We deduce from proposition \ref{CWadTr} and proposition \ref{SerreP} the point (iii) of proposition \ref{CWadTr}
which says in particular that for $X\in\CW$, the complex $\sing_{\mathbb I^*}\mathbb Z_{tr}(X)$,
where $\mathbb Z_{tr}(X)$ is the presheaf represented by $X$, is quasi-isomorphic to
the complex of singular chains $C^{\sing}_*(X,\mathbb Z)$. Indeed, considering
$\mathbb Z(Y,E)=\coker(\mathbb Z(l))$, the cokernel of the injective morphism of presheaves
$\mathbb Z(l):\mathbb Z(E)\hookrightarrow\mathbb Z(Y)$, we have 
\begin{prop}
For $Y\in\CW$ and $l:E\hookrightarrow Y$ a CW subcomplex, the followings embeddings are quasi-isomorphism :
\begin{equation*}
C^{\sing}_*(Y,E,\mathbb Z)\xrightarrow{\mathbb Z(Y,E)(L)}\sing_{\mathbb I^*}\mathbb Z(Y,E)
\xrightarrow{e_{cw*}\ad(\Tr^*,\Tr_*)(\underline{\sing}_{\mathbb I^*}\mathbb Z(Y,E))}
\sing_{\mathbb I^*}\mathbb Z_{tr}(Y,E)
\end{equation*}
where $C^{\sing}_*(Y,E,\mathbb Z)=\coker{l_*}$ is the relative cohomology, with
$l_*:\mathbb Z\Hom_{\CW}(\Delta^*,E)\hookrightarrow\mathbb Z\Hom_{\CW}(\Delta^*,Y)$.
\end{prop}
Let $X\in\TM(\mathbb R)$ be a differential manifold, $Y\in\CW$ and $E\subset Y$ a subcomplex.
Let $T=\sum_i n_iT_i\in\mathcal Z^{fs/\mathbb I^n\times X}(\mathbb I^n\times X\times Y)$ such that
$\partial_{\mathbb I^*}T:
=\sum_{i=1}^n(-1)^n(T_{|\mathbb I_{i,0}^n\times X\times Y}-T_{|\mathbb I_{i,1}^n\times X\times Y})=0$. 
Denote by $p_X:\mathbb I^n\times X\times Y\to X$, $p_Y:\mathbb I^n\times X\times Y\to Y$ and
$p_{X\times Y}:\mathbb I^n\times X\times Y\to X\times Y$ the projections,
$m_i:T_i\hookrightarrow\mathbb I^n\times X\times Y$ the closed CW embeddings for all $i$.
Denote by $p_{Xi}=p_X\circ m_i:T_i\to X$ and $p_{Yi}=p_X\circ m_i:T_i\to Y$.  
The action of $p_{X\times Y}(T)\in\mathcal Z_{d_X+n}(X\times Y,\mathbb Z)$ on homology is
\begin{equation*}
K_n(X,(Y,E))(p_{X\times Y}(T)):\sum_i n_i(c_{Y,E}[n])\circ (p_{Yi*}[n])\circ p_{Xi}^*
\in\Hom_{D^-(\mathbb Z)}(C_*(X,\mathbb Z),C_*(Y,E,\mathbb Z)[n]), \
\end{equation*}
where, for each $i$ :
\begin{itemize}
\item $p_{Xi}^*\in\Hom_{D^-(\mathbb Z)}(C_*(X,\mathbb Z),C_*(T_i,\mathbb Z)[n])$ 
is the Gynsin morphism ($p_{Xi}$ is proper and $X\in\TM(\mathbb R)$ is a topological manifold),
\item $p_{Yi*}=\mathbb Z(p_{Yi})(\Delta^*):C_*(T_i,\mathbb Z)\to C_*(Y,\mathbb Z)$ is the classical map on singular chain, 
\item $c_{Y,E}:C_*(Y,\mathbb Z)\to C_*(Y,E,\mathbb Z)$ is the quotient map.
\end{itemize}
We identify in proposition \ref{Kunneth},
for $X\in\TM(\mathbb R)$ a differential manifold, $Y\in\CW$ and $E\subset Y$ a subcomplex, the image
of a morphism 
$[T]\in\Hom_{PC^-(\CW)}(\mathbb Z_{tr}(X),
\underline{\sing}_{\mathbb I^*}\mathbb Z_{tr}(Y,E)[n])$ 
under the $(\mathbb I^1,usu)$ localization functor with the action of $p_{X\times Y}(T)$ on homology : 
\begin{prop}
Let $X\in\TM(\mathbb R)$ connected and $Y\in\CW$. Let $l:E\hookrightarrow Y$ a CW subcomplex. 
\begin{itemize}
\item[(i)] Let $[T]=[\sum_i n_iT_i]\in\Hom_{PC^-(CW)}(\mathbb Z_{tr}(X),\sing_{\mathbb I^*}(\mathbb Z_{tr}(Y,E))[n])$  
Then, 
\begin{eqnarray*}
Re^{tr}_{cw*}\circ D(\mathbb I^1,usu)([T])=K_n(X,(Y,E))(p_{X\times Y}(T)) \\
\in\Hom_{D^-(\mathbb Z)}(\sing_{\mathbb I^*}\mathbb Z_{tr}(X),\sing_{\mathbb I^*}\mathbb Z_{tr}(Y,E)[n])
=\Hom_{D^-(\mathbb Z)}(C_*(X,\mathbb Z),C_*(Y,E,\mathbb Z)[n]), 
\end{eqnarray*}
\item[(ii)] If $Y$ is compact and $E\subset Y$ is closed then the factorization 
\begin{equation*}
K_n(X,(Y,E))=\overline{K_n(X,(Y,E))}\circ [\cdot]:\mathcal Z_{d_X+n}(X\times Y,\mathbb Z)\to
\Hom_{D^-(\mathbb Z)}(C_*(X,\mathbb Z),C_*(Y,E,\mathbb Z)[n])
\end{equation*}
where 
$[\cdot]:\mathcal Z_{d_X+n}(X\times Y,X\times E,\mathbb Z)\to H^{BM}_{d_X+n}(X\times Y,X\times E,\mathbb Z)$
is the fundamental class gives the classical isomorphism comming from the six functor formalism (c.f.\cite{CH})
\begin{equation}
\overline{K_n(X,(Y,E))}
:H^{BM}_{d_X+n}(X\times Y,X\times E,\mathbb Z)
\xrightarrow{\sim}\Hom_{D^-(\mathbb Z)}(C_*(X,\mathbb Z),C_*(Y,E,\mathbb Z)[n]).
\end{equation} 
\end{itemize}
\end{prop}

By definition, we have the following commutative diagram of sites : 
\begin{equation*}
\xymatrix{
\widetilde\Cw: \Cor^{fs}_{\mathbb Z}(\CW)\ar[r]^{\Cw}\ar[d]^{\Tr} & 
\Cor^{fs}_{\mathbb Z}(\AnSm(\mathbb C))\ar[r]^{\An}\ar[d]^{\Tr(S)} 
& \Cor^{fs}_{\mathbb Z}(\SmVar(\mathbb C))\ar[d]^{\Tr} \\
\widetilde\Cw: \mathbb Z(\CW)\ar[r]^{\Cw}\ar[rd]^{\Cw} & 
\mathbb Z(\AnSm(\mathbb C))\ar[r]^{\An} & \mathbb Z(\SmVar(\mathbb C)) \\
\, & \mathbb Z(\AnSp(\mathbb C))\ar[r]^{\An}\ar[u]^{\iota_{an}} & 
\mathbb Z(\Var(\mathbb C))\ar[u]^{\iota_{var}} } 
\end{equation*}  
In proposition \ref{AnCwtCw}(ii), we prove that $\Cw^*$, and hence $\widetilde\Cw^*$, 
derive trivially for the $(\mathbb D^1,usu)$, resp. $(\mathbb A^1,et)$,
and $(\mathbb I^1,usu)$ model structure.
Our definition of the Betti realization functor is the following :
\begin{defi}
\begin{itemize}
\item[(i)] The CW-Betti realization functor (without transfers) is the composite :
\begin{equation*}
\widetilde{\Bti}_0^*:\DA^{-}(\mathbb C,\mathbb Z)\xrightarrow{\widetilde\Cw^*}\CwDA^{-}(\mathbb Z)
\xrightarrow{Re_{cw*}}D^{-}(\mathbb Z)
\end{equation*}
\item[(ii)] The CW-Betti realization functor with transfers is the composite :
\begin{equation*}
\widetilde\Bti^*:\DM^{-}(\mathbb C,\mathbb Z)\xrightarrow{\widetilde\Cw^*}\CwDM^{-}(\mathbb Z)
\xrightarrow{Re^{tr}_{cw*}}D^{-}(\mathbb Z)
\end{equation*}
\end{itemize}
\end{defi}
Since $\widetilde\Cw^*$ derive trivially by proposition \ref{AnCwtCw}(ii)
and $L\Tr^*:\CwDA^-(\mathbb Z)\to\CwDM^-(\mathbb Z)$ is the inverse of $\Tr_*$ by theorem \ref{CWMot}(i), we have 
$\widetilde\Bti_0^*=\widetilde\Bti^*\circ L\Tr^*$.
In section 3.1 we prove (c.f. theorem \ref{mainthm}) that the absolute version of this construction coincide with Ayoub'one.
\begin{thm}
\begin{itemize}
\item[(i)] For $Y\in\Var(\mathbb C)$, and $E\subset Y$ a subvariety, we have $\Bti^*M(Y,E)=\widetilde\Bti^*M(Y,E)$
\item[(ii)] For $X,Y\in\Var(\mathbb C)$, $D\subset X$, $E\subset Y$ subvarieties, 
and $n\in\mathbb Z$, $n\leq 0$, the following diagram is commutative
\begin{equation*}
\xymatrix{
\Hom_{\DM^{-}(\mathbb C,\mathbb Z)}(M(X,D),M(Y,E)[n])\ar^{\widetilde\Cw^*}[rr]\ar[d]_{\An^*} & \, & 
\Hom_{\CwDM^-(\mathbb Z)}(M(X,D),M(Y,E)[n])\ar[d]^{Re^{tr}_{cw*}} \\  
\Hom_{\AnDM^-(\mathbb Z)}(M(X,D),M(Y,E)[n])\ar[rr]^{Re^{tr}_{an*}} & \, & 
\Hom_{D^-(\mathbb Z)}(\sing_{\mathbb I^*}\mathbb Z_{tr}(X,D),\sing_{\mathbb I^*}\mathbb Z_{tr}(Y,E)[n])} 
\end{equation*}
where we denoted for simplicity $X$ for $X^{an}$ and $X^{cw}$, and similarly for $D$, $Y$ and $E$. 
\end{itemize}
\end{thm}
To prove this theorem, we give, for $X\in\AnSp(\mathbb C)$, an equivalence $(\mathbb D^1,usu)$ local
$B(\mathbb Z_{tr}(X)):\sing_{\mathbb D^*}\mathbb Z_{tr}(X)\to\Cw_*\sing_{\mathbb I^*}\mathbb Z_{tr}(X^{cw})$ by 
considering the two canonical morphism of functors (see section 3.1),
\begin{itemize}
\item the morphism $\psi^{\Cw}$, which,
for $G^{\bullet}\in PC^-(\An)$, associate the following canonical morphism 
$\psi^{\Cw}(G^{\bullet}):\Cw^*(\underline{\sing}_{\mathbb I_{an}^*}G^{\bullet})
\to\underline{\sing}_{\mathbb I^*}\Cw^*G^{\bullet}$ in $PC^-(CW)$, 
\item the morphism $\psi^{\widetilde{\Cw}}$, which,
for $F^{\bullet}\in PC^-$, associate the following canonical morphism
$\psi^{\widetilde{\Cw}}(F^{\bullet}):\widetilde{\Cw}^*(\underline{\sing}_{\mathbb I_{et}^*}F^{\bullet})
\to\underline{\sing}_{\mathbb I^*}\widetilde{\Cw}^*F^{\bullet}$ in $PC^-(CW)$,
\end{itemize}
we define the following two morphism of functors :
\begin{itemize}
\item[(i)] the morphism $W$, which,
for $G^{\bullet}\in PC^-(\An)$, associate the composition in $PC^-(CW)$
\begin{equation*}
W(G^{\bullet}):
\Cw^*(\underline{\sing}_{\bar{\mathbb D}^*}G^{\bullet})\xrightarrow{\Cw^*(\underline{G}^{\bullet}(i'_1))}
\Cw^*(\underline{\sing}_{\mathbb I_{an}^*}G^{\bullet})\xrightarrow{\psi^{\Cw}(G^{\bullet})}
\underline{\sing}_{\mathbb I^*}\Cw^*G^{\bullet},
\end{equation*}
\item[(ii)] the morphism $\widetilde W$, which,
for $F^{\bullet}\in PC^-$, associate the composition in $PC^-(CW)$
\begin{equation*}
\widetilde W(F^{\bullet}):
\widetilde\Cw^*(\underline{C}_*F^{\bullet})\xrightarrow{\widetilde\Cw^*(\underline{F}^{\bullet}(i))}
\widetilde\Cw^*(\underline{\sing}_{\mathbb I_{et}^*}F^{\bullet})\xrightarrow{\psi^{\widetilde\Cw}(F^{\bullet})}
\underline{\sing}_{\mathbb I^*}\widetilde\Cw^*F^{\bullet}.
\end{equation*}
\end{itemize}
The morphism of functor $B$ is defined, by associating
to $G^{\bullet}\in\PSh(\Cor^{fs}_{\mathbb Z}(\AnSm(\mathbb C)),C^-(\mathbb Z))$, the composite 
\begin{equation*}
B(G^{\bullet}):
\underline{\sing}_{\bar{\mathbb D}^*}G^{\bullet}\xrightarrow{\ad(\Cw^*,\Cw_*)(\underline{\sing}_{\bar{\mathbb D}^*}G^{\bullet})}
\Cw_*\Cw^*\underline{\sing}_{\bar{\mathbb D}^*}G^{\bullet}\xrightarrow{\Cw_*(W(G^{\bullet}))}
\Cw_*\underline{\sing}_{\mathbb I^*}\Cw^*G^{\bullet}
\end{equation*}
in $\PSh(\Cor^{fs}_{\mathbb Z}(\AnSm(\mathbb C)),C^-(\mathbb Z))$
We deduce this theorem from the proposition \ref{Bettiprop} :
\begin{prop}
\begin{itemize}
\item[(i)] For $Y\in\AnSp(\mathbb C)$ and $E\subset Y$ an analytic subset,
\begin{equation*}
e^{tr}_{an*}(B(\mathbb Z_{tr}(Y,E))):
\sing_{\bar{\mathbb D^*}}\mathbb Z_{tr}(Y,E)\to\sing_{\mathbb I^*}\mathbb Z_{tr}(Y^{cw},E^{cw})
\end{equation*}
is a quasi isomorphism in $C^{-}(\mathbb Z)$.
\item[(ii)] For $Y\in\AnSp(\mathbb C)$, and $E\subset Y$ an analytic subset, the morphism 
\begin{equation*}
B(\mathbb Z_{tr}(Y,E)):
\underline{\sing}_{\bar{\mathbb D}^*}\mathbb Z_{tr}(Y,E)\to\Cw_*\underline{\sing}_{\mathbb I^*}\mathbb Z_{tr}(Y^{cw},E^{cw})
\end{equation*}
is an equivalence $(\mathbb D^1,usu)$ local in $\PSh(\Cor^{fs}_{\mathbb Z}(\AnSm(\mathbb C)),C^-(\mathbb Z))$.
\end{itemize}
\end{prop}
We deduce the point (ii) of this proposition from point (i). To prove point (i) we reduce to the smooth case
using a desingularization by \cite{Hironaka}.
For $X\in\AnSp(\mathbb C)$ the morphism of complexes
\begin{equation*}
e_{an*}B(\mathbb Z_{tr}(X)):\sing_{\mathbb D^*}\mathbb Z_{tr}(X)\to\sing_{\mathbb I^*}\mathbb Z_{tr}(X^{cw}) 
\end{equation*}
is given by associating to $\alpha\in\mathcal Z^{fs/\bar{\mathbb D^n}}(\bar{\mathbb D}^n\times X)$, the
restriction $\mathbb Z_{tr}(X^{cw})(i_n)(\alpha^{cw})=\alpha_{|\mathbb I^n\times X^{cw}}$ 
of $\alpha^{cw}\in\mathcal Z^{fs/\bar{\mathbb D}^n}(\mathbb D^n\times X^{cw})$ by the closed embedding of CW complexes
$i_n\times I_{X^{cw}}:\mathbb I^n\times X^{cw}\hookrightarrow\mathbb D^n\times X^{cw}$ 
induced by the closed embedding $i_n:\mathbb I^n\hookrightarrow\bar{\mathbb D}^n$.
If $X$ is smooth connected, we see using a covering by geodesically convex open subset,
we see that there exist a countable open covering by open subsets isomorphic to open balls in $\mathbb C^{d_X}$ 
such that the intersection of a finite number of members are either empty or isomorphic to an open ball in $\mathbb C^{d_X}$.

In the section 3.2, we use our construction of the Betti realization functor via CW complexes
to give, for $X,Y\in\Var(\mathbb C)$ and $E\subset Y$ a closed subvariety, 
an explicit image of a morphism $\alpha\in\Hom_{PC^-}(\mathbb Z_{tr}(X),\mathbb Z_{tr}(Y,E)[n])$ 
by the Betti realization functor. 
\begin{defi}
\begin{itemize}
\item[(i)]For $V\in\Var(\mathbb C)$ and $p\in\mathbb N$, we consider
the Bloch cycle complex $\mathcal Z^p(V,*)$, and
$\mathcal Z^p(V^{cw},*)$ the cubical complex such that for $n\in\mathbb N$,
$\mathcal Z^p(V^{cw},n)\subset\mathcal Z^p(\mathbb I^n\times X^{cw})$ is the abelian subgroup
consisting of cycle meeting the face of $\mathbb I^n$ properly 
and whose differential is given by, for $\gamma\in\mathcal Z^p(V^{cw},n)$, 
$\partial_{\mathbb I}\gamma=
\sum_{i=1}^n(-1)^i(\gamma_{|\mathbb I_{i,0}^{n-1}\times V^{cw}}-\gamma_{|\mathbb I_{i,1}^{n-1}\times V^{cw}})$
where $\mathbb I_{i,0},\mathbb I_{i,1}\subset\mathbb I^n$ are the faces.
We then denote
\begin{equation*}
\hat{\mathcal T}^{p}_{V}:\mathcal Z^p(V,*)\to\mathcal Z^{2p}(V^{cw},*) \; , \;  
\alpha\in\mathcal Z^p(V,n)\mapsto\alpha^{cw}_{|\mathbb I^n\times V^{cw}} 
\end{equation*}
the map complexes given by restriction with respect to the closed embedding of CW complexes 
$i_n\times I_V:\mathbb I^n\times V^{cw}\hookrightarrow\square^n\times V^{cw}$ whose image
is $[0,\infty]^n\times V^{cw}\subset\square^n\times V^{cw}$ (c.f. notations).

\item[(ii)] Let $V\in\Var(\mathbb C)$ quasi-projective. 
Let $Y\in\PVar(\mathbb C)$ be a compactification of $V$ (e.g. the projectivisation of $V$) 
and $E=Y\backslash V$.
The higher cycle class map (\cite{MKerr}) is the morphism of complexes of abelian groups :
\begin{equation*}
\mathcal T^p:\mathcal Z^p(V,*)\to C^{sing}_{2d_Y-2p+*}(Y^{cw},E^{cw},\mathbb Z), \; 
\alpha\mapsto [p_Y(\hat{T}_{\alpha})]=[p_{Y}(\bar{\alpha}^{cw}_{|\mathbb I^*\times Y^{cw}})],
\end{equation*}
where $\bar{\alpha}\in\mathcal Z^p(Y\times\square^*)$ is the closure of $\alpha$ and
$p_Y:Y^{cw}\times\mathbb I^*\to Y^{cw}$ is the projection.
\end{itemize}
\end{defi}
Let $X\in\SmVar(\mathbb C)$ and $Y\in\Var(\mathbb C)$.
Let $E\subset Y$ be a closed subset and  $V=Y\backslash E$.
Denote by $j:V\hookrightarrow Y$ the open embeddings.
We have the open embedding $I_X\times j:X\times V\hookrightarrow X\times Y$. 
The map of complexes 
\begin{equation*}
\hat{\mathcal T}^{d_Y}_{X\times Y}:\mathcal Z^{d_Y}(X\times Y,*)\to\mathcal Z^{2d_Y}(X^{cw}\times Y^{cw},*) 
\end{equation*}
induces maps denoted by the same way on the subcomplexes indicated in the following diagram
in $C^-(\mathbb Z)$:
\begin{equation*}
\xymatrix{
\mathcal Z^{d_V}(X\times V,*)\ar[rr]^{\hat{\mathcal T}^{d_V}_{X\times V}} & \, & 
\mathcal Z^{2d_V}(X^{cw}\times V^{cw},*) \\
\underline{C}_*\mathbb Z_{tr}(Y,E)(X)\ar[rr]^{\hat{\mathcal T}^{d_V}_{X\times V}}\ar@{^{(}->}[u]^{(I_X\times j)^*} & \, &
\underline{\sing}_{\mathbb I^*}\mathbb Z_{tr}(Y^{cw},E^{cw})(X^{cw})\ar@{^{(}->}[u]^{(I_X\times j)^{cw*}} \\
\underline{C}_*\mathbb Z_{tr}(V)(X)
\ar[rr]^{\hat{\mathcal T}^{d_V}_{X\times V}}\ar@{^{(}->}[u]^{\underline{C}_*(c(Y,E)\circ\mathbb Z_{tr}(j))(X)} & \, &
\underline{\sing}_{\mathbb I^*}\mathbb Z_{tr}(V^{cw})(X^{cw})
\ar@{^{(}->}[u]_{\underline{\sing}_{\mathbb I^*}(c(Y^{cw},E^{cw})\circ\mathbb Z_{tr}(j^{cw}))(X^{cw})}}
\end{equation*}  
where,
\begin{itemize}
\item $I_X\times j^*:\mathcal Z^{d_Y}(X\times Y,\mathbb Z)
\hookrightarrow\mathcal Z^{d_Y}(X\times V,\mathbb Z) \; ; \; Z \mapsto Z\cap (X\times V)$ 
\item $I_X\times j^*:\mathcal Z^{2d_Y}(X^{cw}\times Y^{cw},\mathbb Z)
\hookrightarrow\mathcal Z^{2d_Y}(X^{cw}\times V^{cw},\mathbb Z) \; ; \; Z \mapsto Z\cap (X\times V)$ 
\end{itemize}
Recall that $D^{tr}(\mathbb A^1,et):PC^-\to\DM^-(\mathbb C,\mathbb Z))$ and 
$D^{tr}(\mathbb I^1,usu):PC^-(CW)\to\CwDM^-(\mathbb Z)$ are the canonical localization functors.
We prove in proposition \ref{TAnCW} the following
\begin{prop}
Let $X\in\SmVar(\mathbb C)$ and $Y\in\Var(\mathbb C)$.
Let $E\subset Y$ be a closed subset and  $V=Y\backslash E$.
For $n\in\mathbb Z$, $n\leq 0$,
the following diagram is commutative
\begin{equation*}
\xymatrix{
\Hom_{PC^-}(\mathbb Z_{tr}(X),\underline{C}_*\mathbb Z_{tr}(Y,E)[n])
\ar[d]_{D(\mathbb A^1,et)}\ar[rr]^{H^n(\hat{\mathcal T}_{X\times V}^{d_X})} & \, & 
\Hom_{PC^-(\CW)}(\mathbb Z_{tr}(X^{cw}),\underline{\sing}_{\mathbb I^*}\mathbb Z_{tr}(Y^{cw},E^{cw})[n])
\ar[d]^{D(\mathbb I^1,usu)} \\
\Hom_{\DM^{-}(\mathbb C,\mathbb Z)}(M(X),M(Y,E)[n])\ar^{\widetilde\Cw^*}[rr] & \, & 
\Hom_{\CwDM^-(\mathbb Z)}(M(X^{cw}),M(Y^{cw},E^{cw})[n]),} 
\end{equation*}
\end{prop}
On the other side, we identified in proposition \ref{Kunneth} in section 2.3 the image
of a morphism 
$T\in\Hom_{PC^-(\CW)}(\mathbb Z_{tr}(X^{cw}),
\underline{\sing}_{\mathbb I^*}\mathbb Z_{tr}(Y^{cw},E^{cw})[n])$ 
under the $(\mathbb I^1,usu)$ localization functor with the action on homology : 
Using proposition \ref{TAnCW} and proposition \ref{Kunneth} (i), 
we immediately deduce from theorem \ref{mainthm} the following (c.f.\ref{cormain}):
\begin{cor} 
Let $X\in\SmVar(\mathbb C)$, $Y\in\Var(\mathbb C)$, $E\subset Y$ a closed subvariety 
and $V=Y\backslash E$ the open complementary. 
Let $n\in\mathbb Z$, $n\leq 0$. Then,
\begin{itemize}
\item[(i)] the following diagram (\ref{cormaindia}) is commutative 
\begin{equation*}
\xymatrix{
\Hom_{PC^-}(\mathbb Z_{tr}(X),\underline{C}_*\mathbb Z(Y,E)[n])
\ar[d]_{\An^*\circ D(\mathbb A^1,et)}\ar[rr]^{H^n(\hat{\mathcal T}_{X\times V}^{d_X})} & \, & 
\Hom_{PC^-(\CW)}(\mathbb Z_{tr}(X),\underline{\sing}_{\mathbb I^*}\mathbb Z_{tr}(Y,E)[n])
\ar[d]^{Re^{tr}_{cw}\circ D(\mathbb I^1,usu)} \\  
\Hom_{\AnDM^-(\mathbb Z)}(M(X),M(Y,E)[n])\ar[rr]^{Re^{tr}_{an*}} & \, & 
\Hom_{D^-(\mathbb Z)}(\sing_{\mathbb I^*}\mathbb Z_{tr}(X),\sing_{\mathbb I^*}\mathbb Z_{tr}(Y,E)[n])} 
\end{equation*}
where we denoted for simplicity $X$ for $X^{an}$ and $X^{cw}$, and similarly for $Y$ and $E$. 
\item[(ii)] for $\alpha\in\Hom_{PC^-}(\mathbb Z_{tr}(X),\underline{C}_*\mathbb Z(Y,E)[n])$,
we have
\begin{equation}
\Bti^*\circ D(\mathbb A^1,et)(\alpha)=K_n(X,Y)(p_{X\times Y}(H^n\hat{\mathcal T}_{X\times V}(\alpha))),
\end{equation}
where $p_{X\times Y}:\mathbb I^n\times X^{cw}\times Y^{cw}\to X^{cw}\times Y^{cw}$ is the projection.
\end{itemize}
\end{cor}
We deduce from this corollary \ref{cormain}(ii), 
proposition \ref{Noriprop}, end lemma \ref{Dfactor}, the following main result 
which say that Ayoub's Betti realization functor factor through Nori motives. 
Consider the functor $\mathcal N:C^b\Cor_{\mathbb Z}(\SmVar(\mathbb C))\to D^b(\mathcal N)$
from the category of bounded complexes of correspondences of algebraic varieties 
to the derived category of Nori motives (c.f.\cite{Nori}).
This functor factor through the localization functor : 
\begin{equation*}
\xymatrix{C^b(\Cor_{\mathbb Z}(\SmVar(\mathbb C)))\ar[rd]^{D(\mathbb A^1,et)}\ar[rr]^{\mathcal N}
 & \, & D^b(\mathcal N) \\
\, & \DM^{gm}(\mathbb C,\mathbb Z)\ar[ru]^{\bar{\mathcal N}} & \, }
\end{equation*}
Denote by $o_N:D^b(\mathcal N)\to D^b(\mathbb Z)$ the forgetful functor.
In theorem \ref{corNorithm} we prove :
\begin{thm}
\begin{itemize}
\item[(i)] For $X\in\SmVar(\mathbb C)$, 
$\Bti^*\circ D(\mathbb A^1,et)(\mathbb Z(X))=o_N\circ\mathcal N(X)$
\item[(ii)] For $X,Y\in\SmVar(\mathbb C)$, the following diagram commutes
\begin{equation*} 
\xymatrix{
\Hom_{PC^-}(\mathbb Z(X),\mathbb Z(Y))\ar[r]^{\mathcal N}\ar[d]_{D(\mathbb A^1,et)} & 
\Hom_{D^b(\mathcal N)}(N(X),N(Y))\ar[d]^{o_N} \\
\Hom_{\DM^-(\mathbb C,\mathbb Z)}(M(X),M(Y))\ar[r]^{\Bti^*} & 
\Hom_{D^b(\mathbb Z)}(C_*(X,\mathbb Z), C_*(Y,\mathbb Z))}
\end{equation*}
\item[(iii)] The Betti realisation functor factor through  Nori motives. That is
$\Bti^*=o_N\circ\bar{\mathcal N}$
\end{itemize}
\end{thm}

Let $V\in\Var(\mathbb C)$ quasi-projective. 
Let $Y\in\PVar(\mathbb C)$ a compactification of $V$ and $E=Y\backslash V$.
Denote by $j:V\hookrightarrow Y$ the open embedding.
\begin{itemize}
\item For $p\leq d_V$, consider the following composition of isomorphisms of abelian groups
\begin{eqnarray*}
H^{p,n}(V):\CH^p(V,n)\xrightarrow{i_{V*}}\CH^{d_V}(\mathbb A^{d_V-p}\times V,n)
\xrightarrow{(H^n(I_{\mathbb A}\times j)^{*})^{-1}}
\Hom_{PC^-}(\mathbb Z_{tr}(\mathbb A^{d_V-p}),\underline{C}_*\mathbb Z_{tr}(Y,E)[n]) \\
\xrightarrow{D(\mathbb A^1,et)}\Hom_{\DM^-(\mathbb C,\mathbb Z)}(M(\mathbb A^r),M(Y,E)[n]) 
\end{eqnarray*}
\item For $p\geq d_V$ and $E'=(Y\times\mathbb P^{p-d_V})\backslash(V\times\mathbb A^{p-d_V})$,
consider the following composition of isomorphisms of abelian groups
\begin{eqnarray*}
H^{p,n}(V):\CH^p(V,n)\xrightarrow{p_V^*}\CH^p(\mathbb A^{p-d_V}\times V,n)
\xrightarrow{(H^n(a\times j)^{*})^{-1}}
\Hom_{PC^-}(\mathbb Z,\underline{C}_*\mathbb Z_{tr}(\mathbb P^{p-d_V}\times Y,E')[n]) \\
\xrightarrow{D(\mathbb A^1,et)}\Hom_{\DM^-(\mathbb C,\mathbb Z)}(\mathbb Z,M(Y\times\mathbb P^{p-d_V},E')[n]) 
\end{eqnarray*}.
\end{itemize}
We prove that under these identifications, the image of the Betti realization
functor on morphism coincide with the Bloch cycle class map (c.f. theorem \ref{corbloch}) : 
\begin{thm}
Let $V\in\Var(\mathbb C)$. Let $Y\in\PVar(\mathbb C)$ be a compactification of $V$ and $E=Y\backslash V$.
Then, 
\begin{itemize}
\item[(i)] for $p\leq d_V$, the following diagram commutes :
\begin{equation*}
\xymatrix{
CH^p(V,n)\ar[d]^{\sim}_{H^{p,n}(V)}\ar[rr]^{H^n\mathcal T_V} & \, &
 H_{n+p}(Y,E,\mathbb Z)\ar[d]_{\sim}^{\overline{K_n(\pt,(Y,E))}} \\
\Hom_{\DM^-(\mathbb C,\mathbb Z)}(M(\mathbb A^{d_V-p}),M(Y,E)[n])\ar[rr]^{\widetilde\Bti^*} & \, &
\Hom_{D^-(\mathbb Z)}(\mathbb Z,C_*(Y,E)[n])}
\end{equation*}
\item[(ii)] for $p\geq d_V$, the following diagram commutes :
\begin{equation*}
\xymatrix{
CH^p(V,n)\ar[d]^{\sim}_{H^{p,n}(V)}\ar[rr]^{H^n\mathcal T_V} & \, & 
H_{n+p}(Y,E,\mathbb Z)\ar[d]_{\sim}^{\overline{K_n(\pt,(Y,E))}} \\
\Hom_{\DM^-(\mathbb C,\mathbb Z)}(\mathbb Z,M(\mathbb A^{p-d_V}\times Y,E)[n])\ar[rr]^{\widetilde\Bti^*} & \, &
\Hom_{D^-(\mathbb Z)}(\mathbb Z,C_*(Y,E)[n])}
\end{equation*}
\end{itemize}
\end{thm}

In the last section we give a relative version of theorem \ref{mainthm}.
We first give in theorem \ref{RCWsingI}  a relative version of theorem \ref{CWsingI}. 
\begin{thm}
Let $S\in\CW$. 
\begin{itemize}
\item[(i)] For $F^{\bullet}\in\PSh_{\mathbb Z}(\Cor^{fs}_{\mathbb Z}(\CW^{sm}/S),C^-(\mathbb Z))$, 
$\underline{\sing}_{\mathbb I^*}F^{\bullet}\in\PSh_{\mathbb Z}(\Cor^{fs}_{\mathbb Z}(\CW^{sm}/S),C^-(\mathbb Z))$ 
is $\mathbb I^1$ local and the inclusion morphism 
$S(F^{\bullet}):F^{\bullet}\to\underline{\sing}_{\mathbb I^*}F^{\bullet}$ is an $(\mathbb I^1,usu)$ equivalence. 
\item[(ii)] For $F^{\bullet}\in\PSh_{\mathbb Z}(\Cor^{fs}_{\mathbb Z}(\CW^{sm}/S),C^-(\mathbb Z))$, 
$\underline{\sing}_{\mathbb I^*}F^{\bullet}\in\PSh_{\mathbb Z}(\Cor^{fs}_{\mathbb Z}(\CW^{sm}/S),C^-(\mathbb Z))$ 
is $\mathbb I^1$ local and the inclusion morphism 
$S(F^{\bullet}):F^{\bullet}\to\underline{\sing}_{\mathbb I^*}F^{\bullet}$ is an $(\mathbb I^1,usu)$ equivalence. 
\end{itemize}
\end{thm}
We then prove in theorem \ref{RCWMot},  a relative version of point (ii) and (iii) of \ref{CWMot} :
\begin{thm}
Let $S\in\CW$.
\begin{itemize}
\item[(i)] The adjonction
$(e_{cw}(S)^*,e_{cw}(S)_*):C^-(S)\leftrightarrows\PSh_{\mathbb Z}(\CW^{sm}/S,C^-(\mathbb Z))$
is a Quillen equivalence for the $(\mathbb I^1,usu)$ model structures. That is, the derived functor
$e_{cw}(S)^*:D^-(S)\xrightarrow{\sim}\CwDA^-(S,\mathbb Z)$ is an isomorphism
and $Re_{cw}(S)_*:\CwDA^-(S,\mathbb Z)\xrightarrow{\sim}D^-(S)$ is it inverse. 
\item[(ii)] The adjonction
$(e^{tr}_{cw}(S)^*,e^{tr}_{cw}(S)_*):C^-(S)\leftrightarrows
\PSh_{\mathbb Z}(\Cor^{fs}_{\mathbb Z}(\CW^{sm}/S),C^-(\mathbb Z))$
is a Quillen equivalence for the $(\mathbb I^1,usu)$ model structures. That is, the derived functor
$e^{tr}_{cw}(S)_*:D^-(S)\xrightarrow{\sim}\CwDM^-(S,\mathbb Z)$ is an isomorphism
and $Re^{tr}_{cw}(S)_*:\CwDM^-(S,\mathbb Z)\xrightarrow{\sim}D^-(S)$ is it inverse.  
\end{itemize}
\end{thm}

By definition, for each $S\in\Var(\mathbb C)$, we have (c.f.\ref{DCat(S)}) 
the following commutative diagram of sites $DCat(S)$
\begin{equation}
DCat(S):=\xymatrix{
\widetilde\Cw(S): \Cor^{fs}_{\mathbb Z}(\CW^{sm}/S^{cw})\ar[r]^{\Cw(S)}\ar[d]^{\Tr(S)} & 
\Cor^{fs}_{\mathbb Z}(\AnSp(\mathbb C)^{sm}/S^{an})\ar[r]^{\An(S)}\ar[d]^{\Tr(S)} 
& \Cor^{fs}_{\mathbb Z}(\Var(\mathbb C)^{sm}/S)\ar[d]^{\Tr(S)} \\
\widetilde\Cw: \mathbb Z(\CW^{sm}/S^{cw})\ar[r]^{\Cw(S)}\ar[rd]^{\Cw(S)} & 
\mathbb Z(\AnSp(\mathbb C)^{sm}/S^{an})\ar[r]^{\An(S)} & \mathbb Z(\Var(\mathbb C)^{sm}/S) \\
 \, & \mathbb Z(\AnSp(\mathbb C)/S^{an})\ar[r]^{\An(S)}\ar[u]^{\iota_{an}(S)} & \mathbb Z(\Var(\mathbb C)/S)\ar[u]^{\iota_{var}(S)} } 
\end{equation}
For $T,S\in\Var(\mathbb C)$ and $f:T\to S$ a morphism, the morphism of sites
\begin{itemize}
\item $P(f):\Var(\mathbb C)/T\to\Var(\mathbb C)/S$, $P(f^{an}):\AnSp(\mathbb C)/T^{an}\to\AnSp(\mathbb C)/S^{an}$,
and $P(f^{cw}):\CW/T^{cw}\to\CW/S^{cw}$ 
given by the pullback functor,
\item $P(f):\Cor^{fs}_{\mathbb Z}(\Var(\mathbb C)^{sm}/T)\to\Cor^{fs}_{\mathbb Z}(\Var(\mathbb C)^{sm}/S)$, 
$P(f^{an}):\Cor^{fs}_{\mathbb Z}(\AnSp(\mathbb C)/T^{an})\to\Cor^{fs}_{\mathbb Z}(\AnSp(\mathbb C)/S^{an})$, and
$P(f^{cw}):\Cor^{fs}_{\mathbb Z}(\CW/T^{cw})\to\Cor^{fs}_{\mathbb Z}(\CW^{sm}/S^{cw})$ 
given by the pullback functor,
\end{itemize}
gives a morphism of diagram of sites 
\begin{equation}
P(f):DCat(T)\to DCat(S).
\end{equation}  
Our definition of the Betti realization functor in the relative setting is :
\begin{defi}
Let $S\in\Var(\mathbb C)$.
\begin{itemize}
\item[(i)] The CW-Betti realization functor (without transfers) is the composite :
\begin{equation*}
\widetilde{\Bti}_0(S)^*:\DA^{-}(S,\mathbb Z)\xrightarrow{\widetilde\Cw(S)^*}\CwDA^{-}(S^{cw},\mathbb Z)
\xrightarrow{Re_{cw}(S)_*}D^{-}(S^{cw})
\end{equation*}
\item[(ii)] The CW-Betti realisation functor with transfers is the composite :
\begin{equation*}
\widetilde\Bti(S)^*:\DM^{-}(S,\mathbb Z)\xrightarrow{\widetilde\Cw(S)^*}\CwDM^{-}(S^{cw},\mathbb Z)
\xrightarrow{Re^{tr}_{cw}(S)_*}D^{-}(S^{cw})
\end{equation*}
\end{itemize}
\end{defi}
Similarly, in the relative case, since $\widetilde\Cw(S)^*$ derive trivially by proposition \ref{RAnCwtCw}(ii)
and and $L\Tr(S^{cw})^*:\CwDA^-(S^{cw},\mathbb Z)\to\CwDM^-(S^{cw},\mathbb Z)$ is the inverse of $\Tr(S^{cw})_*$ 
(c.f.remark \ref{RTrCWrem}), we have $\widetilde\Bti_0(S)^*=\widetilde\Bti(S)^*\circ L\Tr(S)^*$.
As in Ayoub's definition, it defines morphisms of homotopic 2-functors (c.f. theorem \ref{RCWBetti} :
\begin{itemize}
\item $S\in\Var(\mathbb C)\mapsto\widetilde\Bti_0(S):DA^-(S,\mathbb Z)\to D^-(S)$
\item $S\in\Var(\mathbb C)\mapsto\widetilde\Bti(S):DM^-(S,\mathbb Z)\to D^-(S)$
\end{itemize}
We finally prove in theorem \ref{Rmainthm} a relative version of the theorem \ref{mainthm},
that is the construction of the relative Betti realization functor via CW complexes coincide
with Ayoub's one via analytic spaces :
\begin{thm}
Let $S\in\Var(\mathbb C)$. 
Let $M\in\DM^-(S,\mathbb Z)$ is a constructible motive.
\begin{itemize}
\item[(i)] We have $\Bti^*M=\widetilde\Bti^*M$
\item[(ii)] Let $M_1,M_2\in\DM^-(S,\mathbb Z)$ contructible motives.
Let $F_1^{\bullet},F_2^{\bullet}\in PC^-(S)$ 
such that $M_i=D(\mathbb A^1,et)(S)(F_i^{\bullet})$
for $i=1,2$. The following diagram is commutative
\begin{equation*}
\xymatrix{\Hom_{\DM^{-}(S)}(M_1,M_2)\ar^{\widetilde\Cw(S)^*}[r]\ar[d]_{\An(S)^*} &  
\Hom_{\CwDM^-(S)}(\widetilde\Cw(S)^*M_1,\widetilde\Cw(S)^*M_2)\ar[d]^{Re^{tr}_{cw}(S)_*} \\  
\Hom_{\AnDM^-(S)}(\An(S)^*M_1,\An(S)^*M_2)\ar[r]^{Re^{tr}_{an}(S)_*} &  
\Hom_{D^-(S)}(\sing_{\mathbb I^*}\widetilde\Cw(S)^*F^{\bullet}_1,
\sing_{\mathbb I^*}\widetilde\Cw(S)^*F^{\bullet}).} 
\end{equation*}
\end{itemize}
\end{thm}
As in the absolute case, we define for $S\in\AnSp(\mathbb C)$, the morphism of functor $B(S)$, by associating
to $G^{\bullet}\in\PSh(\Cor^{fs}_{\mathbb Z}(\AnSp^{sm}(\mathbb C)/S),C^-(\mathbb Z))$, 
the morphism $B(S)(G^{\bullet})$ which is the composite 
\begin{equation*}
\xymatrix{
\underline{\sing}_{\bar{\mathbb D}^*}G^{\bullet}
\ar[d]_{\ad(\Cw(S)^*,\Cw(S)_*)(\underline{\sing}_{\bar{\mathbb D}^*}G^{\bullet})}\ar[rrrd]^{B(S)(G^{\bullet})} & \, & \, & \, \\
\Cw(S)_*\Cw(S)^*\underline{\sing}_{\bar{\mathbb D}^*}G^{\bullet}\ar[rrr]^{\Cw(S)_*(W(S)(G^{\bullet}))} & \, & \, &
\Cw(S)_*\underline{\sing}_{\mathbb I^*}\Cw(S)^*G^{\bullet}}
\end{equation*}
in $\PSh(\Cor^{fs}_{\mathbb Z}(\AnSp(\mathbb C)^{sm}/S),C^-(\mathbb Z))$,
and deduce this theorem from the point (ii) of following proposition :
\begin{prop}
Let $S\in\Var(\mathbb C)$ and $F^{\bullet}\in PC^-(S)$ 
such that $D(\mathbb A^1,et)(S)(F^{\bullet})\in\DM^-(S,\mathbb Z)$ is a constructible motive. Then,
\begin{itemize}
\item[(i)] 
$e^{tr}_{an}(S)_*B(S)(F^{\bullet}):\sing_{\bar{\mathbb D}^*}\An(S)^*F^{\bullet}\to
\sing_{\mathbb I^*}\widetilde\Cw(S)^*F^{\bullet}$
is an equivalence usu local in $C^{-}(S^{an})$.
\item[(ii)] 
$B(S)(F^{\bullet}):\underline{\sing}_{\bar{\mathbb D}^*}\An(S)^*F^{\bullet}\to
\Cw(S^{an})_*\underline{\sing}_{\mathbb I^*}\widetilde\Cw(S)^*F^{\bullet}$
is an equivalence $(\mathbb D^1,usu)$ local.
\end{itemize}
\end{prop}
As in the absolute case, we deduce the point (ii) of this proposition from point (i).
We prove the point (i) of this proposition using the absolute case 
(point (i) of the proposition \ref{Bettiprop}).

\section{Derived categories of motives of algebraic varieties, analytic spaces, and CW complexes}

\subsection{The derived category of mixed motives of algebraic varieties}

For $X\in\Var(k)$, and $\Lambda$ a commutative ring and $p\in\mathbb N$, 
we denote by $\mathcal Z_p(X,\Lambda)$ the free $\Lambda$ module 
generated by the irreducible closed subspaces of $X$ of dimension $p$,
by $\mathcal Z^p(X,\Lambda)=\mathcal Z_{d_X-p}(X,\Lambda)$ 
the free $\Lambda$ module generated by the irreducible closed subspaces of $X$ of codimension $p$.

An algebraic variety $X\in\Var(k)$ is aid to be proper (or complete) if the terminal map $a_X:X\to\Spec k$
is universally closed, that is for all $Y\in\Var(k)$,
the projection $p_Y:X\times_k Y\to Y$ is closed (note that the topology on $X\times_k Y$ is finer than the
product topology on the underlying topological spaces).
For $X\in\Var(k)$, a compactification of $X$ is a complete variety $\bar X\in\Var(k)$ such that $X\subset\bar X$
is an open subset. For $X\in\Var(k)$ quasi-projective, 
the projectivisation $\bar X\in\PVar(k)$ is a compactification of $X$.

For $X\in\Var(k)$, we have the Bloch cycle complex $\mathbb Z(X,*)$, with for $n\in\mathbb N$, 
$\mathcal Z^p(X,n)\subset \mathcal Z^p(X\times\square^n)$ is the subgroup of the $p$ codimentional irreducible closed
subspace of $X\times\square^n$ meeting all faces of $\square^n$ properly.

For $X,Y\in\Var(k)$ and $\Lambda$ a commutative ring, we will use the following notations :
\begin{itemize}
\item if $X$ is smooth connected $\mathcal{Z}^{fs/X}(X\times_k Y,\Lambda)\subset\mathcal{Z}_{d_X}(X\times_k Y,\Lambda)$ 
is the free  $\Lambda$ submodule generated by the irreducible closed subspaces
of $X\times_k Y$ which are finite and surjective over $X$,
\item if $X$ is smooth, $\mathcal{Z}^{fs/X}(X\times Y,\Lambda):=\oplus_i\mathcal{Z}^{fs/X_i}(X_i\times Y,\Lambda)$ 
where $X=\sqcup_i X_i$, with $X_i$ the connected components of $X$,
\item $\mathcal{Z}^{ed(r)/X}(X\times_k Y,\Lambda)\subset\mathcal{Z}^{d_X+r}(X\times_k Y,\Lambda)$ 
the free $\Lambda$ module generated by the irreducible closed subspaces of $X\times_kY$
dominant over $X$, and whose fibers over $X$ are either empty or equidimensional of relative dimension $r$.
\end{itemize}
By definition, for $X$ smooth, 
$\mathcal{Z}^{fs/X}(X\times_k Y,\Lambda)\subset\mathcal{Z}^{ed(0)/X}(X\times_k Y,\Lambda)$
and $\mathcal{Z}^{ed(p-d_X)/\square^n}(\square^n\times_k X,\Lambda)\subset\mathcal Z_p(X,n)\otimes_{\mathbb Z}\Lambda$.

For $Y\in\Var(k)$ irreducible and $j:V\hookrightarrow Y$ an open embedding, 
we denote by 
\begin{equation}\label{j}
j^*:\mathcal Z^p(Y,\Lambda)\hookrightarrow\mathcal Z^p(V,\Lambda) ; Z \mapsto Z\cap V 
\end{equation}

\begin{defi}\cite{VoeSF}
We define $\Cor^{fs}_{\Lambda}(\SmVar(k))$ to be the category whose objects are smooth algebraic varieties over $k$ 
and whose space of morphisms between
$X,Y\in\SmVar(k)$ is the free $\Lambda$ module $\mathcal{Z}^{fs/X}(X\times_k Y,\Lambda)$. 
The composition of morphisms is defined in \cite{VoeSF} 
\end{defi}

We have 
\begin{itemize}
\item the additive embedding of categories $\Tr:\mathbb Z(\SmVar(k))\hookrightarrow\Cor^{fs}_{\mathbb Z}(\SmVar(k))$ 
which gives the corresponding morphism of sites $\Tr:\Cor^{fs}_{\mathbb Z}(\SmVar(k))\to \mathbb Z(\SmVar(k))$.
\item the inclusion functor
$e_{var}:\left\{pt\right\}\hookrightarrow\SmVar(\mathbb C)$,
which gives the corresponding morphism of sites
$e_{var}:\SmVar(\mathbb C)\to\left\{pt\right\}$,
\item the inclusion functor $e^{tr}_{var}:=\Tr\circ e_{var}:\left\{pt\right\}\hookrightarrow\Cor^{fs}_{\mathbb Z}(\SmVar(\mathbb C))$
which gives the corresponding morphism of sites
$e^{tr}_{var}:=\Tr\circ e_{var}:\Cor^{fs}_{\mathbb Z}(\SmVar(\mathbb C))\to\left\{pt\right\}$.
\end{itemize}

We consider the following two big categories :
\begin{itemize}
\item $\PSh(\SmVar(k),C^-(\mathbb Z))=\PSh_{\mathbb Z}(\mathbb Z(\SmVar(k)),C^-(\mathbb Z))$, 
the category of bounded above complexes of presheaves on $\SmVar(k)$, or equivalently additive presheaves on $\mathbb Z(\SmVar(k))$, 
sometimes, we will write for short $P^-=\PSh(\SmVar(\mathbb C),C^-(\mathbb Z))$,
\item $\PSh_{\mathbb Z}(\Cor^{fs}_{\mathbb Z}(\SmVar(k)),C^-(\mathbb Z))$, 
the category of bounded above complexes of additive presheaves on $\Cor^{fs}_{\mathbb Z}(\SmVar(k))$
sometimes, we will write for short $PC^-=\PSh(\Cor^{fs}_{\mathbb Z}(\SmVar(\mathbb C)),C^-(\mathbb Z))$,
\end{itemize}
and the adjonctions :
\begin{itemize}
\item $(\Tr^*,\Tr_*):\PSh(\SmVar(k),C^-(\mathbb Z))\leftrightarrows\PSh_{\mathbb Z}(\Cor^{fs}_{\mathbb Z}(\SmVar(k)),C^-(\mathbb Z))$, 
\item $(e_{var}^*,e_{var*}):\PSh(\SmVar(\mathbb C),C^-(\mathbb Z))\leftrightarrows C^-(\mathbb Z)$,
\item $(e_{var}^{tr*},e_{var*}^{tr}):\PSh(\Cor^{fs}_{\mathbb Z}(\SmVar(k)),C^-(\mathbb Z))\leftrightarrows C^-(\mathbb Z)$,
\end{itemize}
given by $\Tr:\Cor^{fs}_{\mathbb Z}(\SmVar(k))\to \mathbb Z(\SmVar(k))$, $e_{var}:\SmVar(k)\to\left\{pt\right\}$
and $e^{tr}_{var}:\Cor^{fs}_{\mathbb Z}(\SmVar(k))\to\left\{pt\right\}$ respectively.
We denote by $a_{et}:\PSh_{\mathbb Z}(\SmVar(k),\Ab)\to\Sh_{\mathbb Z,et}(\SmVar(k),\Ab)$
the etale sheaftification functor.

For $X\in\SmVar(k)$, we denote by
\begin{equation}
\mathbb Z(X)\in\PSh(\SmVar(k),C^-(\mathbb Z)) \mbox{\;\;}, \mbox{\;\;}  
\mathbb Z_{tr}(X)\in\PSh_{\mathbb Z}(\Cor^{fs}_{\mathbb Z}(\SmVar(k)),C^-(\mathbb Z))
\end{equation}
the presheaves represented by $X$. They are etale sheaves.

For $X\in\Var(k)$, the have the presheaves 
\begin{itemize}
\item $\mathbb Z_{tr}(X)\in\PSh_{\mathbb Z}(\Cor^{fs}_{\mathbb Z}(\SmVar(k)),C^-(\mathbb Z))$, \; 
$Y\in\SmVar(k)\mapsto\mathcal{Z}^{fs/Y}(Y\times_k X,\mathbb Z)$, 
\item $\mathbb Z_{eq}(X,0)\in\PSh_{\mathbb Z}(\Cor^{fs}_{\mathbb Z}(\SmVar(k)),C^-(\mathbb Z))$, \; 
$Y\in\SmVar(k)\mapsto\mathcal{Z}^{eq(0)/Y}(Y\times_k X,\mathbb Z)$,
\end{itemize}
which are etale sheaves.
Of course, if $X\in\PVar(k)$, then $\mathbb Z_{eq}(X,0)=\mathbb Z_{tr}(X)$.

We consider 
\begin{itemize}
\item the usual monoidal strutcure on $\PSh(\SmVar(k)),C^-(\mathbb Z))$ and the associated internal $\Hom$
given by, for $F^{\bullet},G^{\bullet}\in\PSh(\SmVar(k)),C^-(\mathbb Z))$ and $Y\in\SmVar(k)$, 
\begin{equation}
F^{\bullet}\otimes G^{\bullet}(X):X\mapsto F^{\bullet}(X)\otimes_{\mathbb Z} G^{\bullet}(X) \; , \;
\underline{\Hom}(\mathbb Z(Y),F^{\bullet}):X\mapsto F^{\bullet}(X\times_kY),
\end{equation}
\item the unique monoidal strutcure on $\PSh_{\mathbb Z}(\Cor^{fs}_{\mathbb Z}(\SmVar(k)),C^-(\mathbb Z))$
such that,for $X,Y\in\SmVar(k)$, $\mathbb Z_{tr}(X)\otimes\mathbb Z_{tr}(Y):=\mathbb Z_{tr}(X\times_k Y)$
and wich commute with colimites.
It has an internal $\Hom$ which is given, for $X,Y\in\SmVar(k)$ and
$F^{\bullet}\in\PSh_{\mathbb Z}(\Cor^{fs}_{\mathbb Z}(\SmVar(k)),C^-(\mathbb Z))$,
$\underline{\Hom}(\mathbb Z_{tr}(Y),F^{\bullet}):X\mapsto F^{\bullet}(X\times_kY)$
\end{itemize}
Together with these monoidal structure, the functor
\begin{equation*}
\Tr^*:\PSh(\SmVar(k),C^-(\mathbb Z))\to \PSh_{\mathbb Z}(\Cor^{fs}_{\mathbb Z}(\SmVar(k)),C^-(\mathbb Z))
\end{equation*}
is monoidal.

\begin{defi}\cite{AyoubG2}\label{VaretTr}
\begin{itemize}
\item[(i)] We say that a morphism $\phi:G_1^{\bullet}\to G_2^{\bullet}$ in $\PSh(\SmVar(k),C^-(\mathbb Z))$ is an
etale local equivalence if $\phi_*:a_{et}H^k(G_1^{\bullet})\to a_{et}H^k(G_2^{\bullet})$ is an isomorphism for all $k\in\mathbb Z$.
The projective etale topology model structure on $\PSh(\SmVar(k),C^-(\mathbb Z))$ 
is the left Bousfield localization of the projective model structure $\mathcal{M}_P(\PSh_{\mathbb Z}(\SmVar(k),C^-(\mathbb Z)))$
with respect to the etale local equivalence.

\item[(ii)] We say that a morphism 
$\phi:G_1^{\bullet}\to G_2^{\bullet}$ in $\PSh_{\mathbb Z}(\Cor^{fs}_{\mathbb Z}(\SmVar(k)),C^-(\mathbb Z))$ is an
etale local equivalence if and only if its restriction to $\SmVar(\mathbb C)$ 
$\Tr_*\phi:\Tr_*G_1^{\bullet}\to\Tr_*G_2^{\bullet}$ is an etale local equivalence.
The projective etale topology model structure on $\PSh_{\mathbb Z}(\Cor^{fs}_{\mathbb Z}(\SmVar(k)),C^-(\mathbb Z))$ 
is the left Bousfield localization of the projective model structure 
$\mathcal{M}_P(\PSh_{\mathbb Z}(\Cor^{fs}_{\mathbb Z}(\SmVar(k)),C^-(\mathbb Z)))$
with respect to the etale local equivalence.
\end{itemize}
\end{defi}

\begin{defi}\label{Amodstr}
\begin{itemize}
\item[(i)] The projective $(\mathbb A_k^1,et)$ model structure on $\PSh(\SmVar(k),C^-(\mathbb Z))$ is the left Bousfield localization 
of the projective etale topology model structure (c.f. definition \ref{VaretTr}(i)) 
with respect to the class of maps
$\left\{\mathbb Z(X\times\mathbb A_k^1)[n]\to\mathbb Z(X)[n], \; X\in\SmVar(k),n\in\mathbb Z\right\}$. 

\item[(ii)] The projective $(\mathbb A_k^1,et)$ model structure on $\PSh_{\mathbb Z}(\Cor^{fs}_{\mathbb Z}(\SmVar(k)),C^-(\mathbb Z))$ 
is the left Bousfield localization of the projective etale topology model structure (c.f. definition \ref{VaretTr}(ii))
with respect to the class of maps 
$\left\{\mathbb Z_{tr}(X\times\mathbb A_k^1)[n]\to\mathbb Z(X)[n],\; X\in\SmVar(k),n\in\mathbb Z\right\}$. 
\end{itemize}
\end{defi}

\begin{defi}
\begin{itemize}
\item[(i)] We define 
$\DM^-(k,\mathbb Z)_{et}:=\Ho_{\mathbb A_k^1,et}(\PSh_{\mathbb Z}(\Cor^{fs}_{\mathbb Z}(\SmVar(k)),C^-(\mathbb Z)))$, 
to be the derived category of (effective) motives, it is
the homotopy category of the category $\PSh(\Cor^{fs}_{\mathbb Z}(\SmVar(k)),C^-(\mathbb Z))$ 
with respect to the projective $(\mathbb A^1,et)$ model structure (cf. definition \ref{Amodstr}(ii)).
We denote by
\begin{equation*}
D^{tr}(\mathbb A^1,et):\PSh_{\mathbb Z}(\Cor^{fs}_{\mathbb Z}(\SmVar(k),C^-(\mathbb Z)))\to\DM^{-}(k,\mathbb Z)_{et} \, , \,
D^{tr}(\mathbb A^1,et)(F^{\bullet})=F^{\bullet}
\end{equation*}
the canonical localisation functor.

\item[(ii)] By the same way, we denote 
$\DA^-(k,\mathbb Z)_{et}:=\Ho_{\mathbb A_k^1,et}(\PSh(\SmVar(k),C^-(\mathbb Z)))$(cf.\ref{Amodstr}(i)) and
\begin{equation*}
D(\mathbb A^1,et):\PSh_{\mathbb Z}(\Cor^{fs}_{\mathbb Z}(\SmVar(\mathbb C)),C^-(\mathbb Z)))\to\DA^{-}(k,\mathbb Z)_{et} \, , \,
D(\mathbb A^1,et)(F^{\bullet})=F^{\bullet}
\end{equation*}
the canonical localisation functor.
\end{itemize}
\end{defi}

For $X\in\Var(k)$, 
\begin{itemize}
\item the (derived) motive of $X$ is $M(X)=D(\mathbb A^1,et)(\mathbb Z_{tr}(X))\in\DM^-(k)$.
\item the (derived) motive with compact support of $X$ is $M^c(X)=D(\mathbb A^1,et)(\mathbb Z_{eq}(X,0))\in\DM^-(k,\mathbb Z)$. 
\end{itemize}
Of course, if $X\in\PVar(k)$, then $\mathbb Z_{eq}(X,0)=\mathbb Z_{tr}(X)$, so that $M^c(X)=M(X)$.


For $F^{\bullet}\in\PSh(\SmVar(k),\Ab)$ and $X\in\SmVar(k)$, 
we have the complex $F(X\times\square_k^*)$ associated to the cubical object 
$F(X\times\square_k^*)$ in the category of abelian groups.

\begin{itemize}
\item If $F^{\bullet}\in\PSh(\SmVar(k),C^-(\mathbb Z))$ , 
\begin{equation}
\underline{C}_*F^{\bullet}:=\Tot(\underline\Hom(\mathbb Z(\square^*_k),F^{\bullet}))\in\PSh(\SmVar(k),C^-(\mathbb Z)) 
\end{equation}
is the total complex of presheaves associated to the bicomplex of presheaves $X\mapsto F^{\bullet}(\square^*_k\times_kX)$,
and $C_*F^{\bullet}:=e_{var*}\underline{C}_*F^{\bullet}=F^{\bullet}(\square^*_k)\in C^-(\mathbb Z)$.
We denote by $S(F^{\bullet}):F^{\bullet}\to\underline{C}_*F^{\bullet}$, 
\begin{equation}\label{SFVar}
S(F^{\bullet}):\xymatrix{
\cdots\ar[r] & 0\ar[r]\ar[d] & 0\ar[r]\ar[d] & F^{\bullet}\ar[d]^{I}\ar[r] & 0\ar[r]\ar[d] & \cdots  \\
\cdots\ar[r] & \underline{C}_2F^{\bullet}\ar[r] & \underline{C}_1F^{\bullet}\ar[r] & F^{\bullet}\ar[r] & 0\ar[r] & \cdots} 
\end{equation}
the inclusion morphism of $\PSh(\SmVar(k),C^-(\mathbb Z))$ :
For $f:F_1^{\bullet}\to F_2^{\bullet}$ a morphism $\PSh(\SmVar(k),C^-(\mathbb Z))$, we denote by 
$S(f):\underline{C}_*F_1^{\bullet}\to\underline{C}_*F_2^{\bullet}$,
the morphism of $\PSh(\SmVar(k),C^-(\mathbb Z))$ given by for $X\in\SmVar(k)$,
\begin{equation}\label{SfVar}
S(f)(X):\xymatrix{
\cdots\ar[r] & F_1(\square^2\times X)\ar[r]\ar[d]^{f(\square^2\times X)} &
F_1(\square^1\times X)\ar[r]\ar[d]^{f(\square^1\times X)} & F_1^{\bullet}(X)\ar[d]^{f(X)}\ar[r] & 0\ar[r]\ar[d] & \cdots \\
\cdots\ar[r] & F_2(\square^2\times X)\ar[r] & F_2(\square^1\times X)\ar[r] & F_2^{\bullet}(X)\ar[r] & 0\ar[r] & \cdots.}  
\end{equation}

\item If $F^{\bullet}\in\PSh_{\mathbb Z}(\Cor^{fs}_{\mathbb Z}(\SmVar(k)),C^-(\mathbb Z))$, 
\begin{equation}
\underline{C}_*F^{\bullet}:=\underline\Hom(\mathbb Z_{tr}(\square^*_k),F^{\bullet})
\in\PSh_{\mathbb Z}(\Cor^{fs}_{\mathbb Z}(\SmVar(k)),C^-(\mathbb Z)), 
\end{equation}
is the complex of presheaves 
associated to the bicomplex of presheaves $X\mapsto F^{\bullet}(\square^*_k\times_kX)$,
and $C_*F^{\bullet}:=e^{tr}_{var*}\underline{C}_*F^{\bullet}=F^{\bullet}(\square^*_k)\in C^-(\mathbb Z)$.
We have the inclusion morphism (\ref{SFVar})
\begin{equation*}
S(\Tr_*F^{\bullet}):\Tr_*F^{\bullet}\to\underline{C}_*\Tr_*F^{\bullet}=\Tr_*\underline{C}_*F^{\bullet}
\end{equation*}
which is a morphism in $\PSh(\Cor^{fs}_{\mathbb Z}(\SmVar(k)),C^-(\mathbb Z))$ 
denoted the same way $S(F^{\bullet}):F^{\bullet}\to\underline{C}_*F^{\bullet}$.  
For $f:F_1^{\bullet}\to F_2^{\bullet}$ a morphism $PC^-$, we 
have the morphism (\ref{SfVar})
\begin{equation*} 
S(\Tr_*f):\Tr_*\underline{C}_*F_1^{\bullet}=\underline{C}_*\Tr_*F_1^{\bullet}\to
\underline{C}_*\Tr_*F_2^{\bullet}=\Tr_*\underline{C}_*F_2^{\bullet}
\end{equation*}
which is a morphism in $PC^-$ denoted the same way
$S(f):\underline{C}_*F_1^{\bullet}\to\underline{C}_*F_2^{\bullet}$.
\end{itemize}

For $F^{\bullet}\in\PSh(\Cor^{fs}_{\mathbb Z}(\SmVar(k)),C^-(\mathbb Z))$, we have by definition 
$\Tr_*\underline{C}_*F^{\bullet}=\underline{C}_*\Tr_*F^{\bullet}$ and $\Tr_*S(F^{\bullet})=S(\Tr_*F^{\bullet})$.

We now make the following definition

\begin{defi}\label{RMVar}
Let $X\in\Var(k)$ and $D\subset X$ a subvariety. Denote by  $l:D\hookrightarrow X$ the locally closed embbeding.
We define
\begin{itemize}
\item $\mathbb Z(X,D)=\coker(\mathbb Z(l))\in\PSh_{\mathbb Z}(\SmVar(k),C^-(\mathbb Z))$
to be the cokernel of the injective morphism $\mathbb Z(l):\mathbb Z(D)\hookrightarrow\mathbb Z(X)$.
By definition, we have the following exact sequence
\begin{equation}\label{XD}
0\to\mathbb Z(D)\xrightarrow{\mathbb Z(l)}\mathbb Z(X)\xrightarrow{c(X,D)}\mathbb Z(X,D)\to 0
\end{equation}
in $\PSh_{\mathbb Z}(\SmVar(k),C^-(\mathbb Z))$.

\item $\mathbb Z_{tr}(X,D)=\coker(\mathbb Z_{tr}(l))\in\PSh_{\mathbb Z}(\Cor^{fs}_{\mathbb Z}(\SmVar(k)),C^-(\mathbb Z))$
to be the cokernel of the injective morphism $\mathbb Z_{tr}(l):\mathbb Z_{tr}(D)\hookrightarrow\mathbb Z_{tr}(X)$.
By definition, we have the following exact sequence
\begin{equation}\label{XDtr}
0\to\mathbb Z_{tr}(D)\xrightarrow{\mathbb Z_{tr}(l)}\mathbb Z_{tr}(X)\xrightarrow{c(X,D)}\mathbb Z_{tr}(X,D)\to 0
\end{equation}
in $\PSh_{\mathbb Z}(\Cor^{fs}_{\mathbb Z}(\SmVar(k)),C^-(\mathbb Z))$.
By the exact sequences, we have $\Tr^*\mathbb Z(X,D)=\mathbb Z_{tr}(X,D)$
\end{itemize}
\end{defi}

In particular, we get the following exact sequence in $\PSh_{\mathbb Z}(\Cor^{fs}_{\mathbb Z}(\SmVar(k)),C^-(\mathbb Z))$
\begin{equation*}
0\to\underline{C}_*\mathbb Z_{tr}(D)\xrightarrow{S(\mathbb Z_{tr}(l))}\underline{C}_*\mathbb Z_{tr}(X)
\xrightarrow{S(c(X,D))}\underline{C}_*\mathbb Z_{tr}(X,D)\to 0
\end{equation*}
We define 
\begin{equation*}
M(Y,E)=D(\mathbb A^1,et)(\mathbb Z_{tr}(Y,E))\in\CwDM^-(\mathbb Z)
\end{equation*}
to be the relative motive of the pair $(Y,E)$.

We now look at the behavior of the functors mentionned above with respect to the $(\mathbb A^1,et)$ model structure

\begin{prop}\cite{AyoubG2}
\begin{itemize}
\item[(i)] $(\Tr^*,\Tr_*):\PSh(\SmVar(k),C^-(\mathbb Z))\leftrightarrows\PSh_{\mathbb Z}(\Cor^{fs}_{\mathbb Z}(\SmVar(k)),C^-(\mathbb Z))$
is a Quillen adjonction for the etale topology model structures (c.f. definition \ref{VaretTr} (i) and (ii) respectively)
and a Quillen adjonction for the $(\mathbb A^1,et)$ model structures (c.f. definition \ref{Amodstr} (i) and (ii) respectively). 
\item[(ii)] $(e_{var}^*,e_{var*}):\PSh(\SmVar(\mathbb C),C^-(\mathbb Z))\leftrightarrows C^-(\mathbb Z)$
is a Quillen adjonction for the etale topology model structure (c.f. definition \ref{VaretTr} (i))
and a Quillen adjonction for the $(\mathbb A^1,et)$ model structure (c.f. definition \ref{Amodstr} (i)). 
\item[(iii)] $(e_{var}^{tr*},e_{var*}^{tr}):\PSh(\Cor^{fs}_{\mathbb Z}(\SmVar(k)),C^-(\mathbb Z))\leftrightarrows C^-(\mathbb Z)$
is a Quillen adjonction for the etale topology model structure (c.f. definition \ref{VaretTr} (ii))
and a Quillen adjonction for the $(\mathbb A^1,et)$ model structure (c.f. definition \ref{Amodstr} (ii)). 
\end{itemize}
\end{prop}

\begin{prop}\cite{AyoubG2}\cite{VoeSF}
\begin{itemize}
\item[(i)] The functor 
$\Tr_*:\PSh_{\mathbb Z}(\Cor^{fs}_{\mathbb Z}(\SmVar(k)),C^-(\mathbb Z))\to\PSh(\SmVar(k),C^-(\mathbb Z))$ 
derive trivially.
\item[(ii)] For $K^{\bullet}\in C^-(\mathbb Z)$, $e_{var}^*K^{\bullet}$ is $\mathbb A_k^1$ local.
\item[(iii)] For $K^{\bullet}\in C^-(\mathbb Z)$, $e^{tr*}_{var}K^{\bullet}$ is $\mathbb A_k^1$ local.
\end{itemize}
\end{prop}

\begin{thm}\cite{VoeSF}\label{VarsingA}
\begin{itemize}
\item[(i)] For $F^{\bullet}\in\PSh(\SmVar(k),C^-(\mathbb Z))$,
$\underline{C}_*F^{\bullet}\in\PSh(\SmVar(k),C^-(\mathbb Z))$ is $\mathbb A_k^1$ local and 
the inclusion morphism $S(F^{\bullet}):F^{\bullet}\to\underline{C}_*F^{\bullet}$ is an $(\mathbb A_k^1,et)$ equivalence. 
\item[(ii)] For $F^{\bullet}\in\PSh_{\mathbb Z}(\Cor^{fs}_{\mathbb Z}(\SmVar(k)),C^-(\mathbb Z))$, 
$\underline{C}_*F^{\bullet}\in\PSh_{\mathbb Z}(\Cor^{fs}_{\mathbb Z}(\SmVar(k)),C^-(\mathbb Z))$ 
is $\mathbb A_k^1$ local and the inclusion morphism 
$S(F^{\bullet}):F^{\bullet}\to\underline{C}_*F^{\bullet}$ is an $(\mathbb A_k^1,et)$ equivalence. 
\end{itemize}
\end{thm}

\begin{thm}\cite{AyoubG1}
The adjonction
$(\Tr^*,\Tr_*):\PSh(\SmVar(k),C^-(\mathbb Z))\leftrightarrows\PSh_{\mathbb Z}(\Cor^{fs}_{\mathbb Z}(\SmVar(k)),C^-(\mathbb Z))$
is a Quillen equivalence for the $(\mathbb A^1,et)$ model structures. 
That is, the derived functor 
\begin{equation}
L\Tr^*:\DA^-(k,\mathbb Z)_{et}\xrightarrow{\sim}\DM^-(k,\mathbb Z)_{et}
\end{equation}
is an isomorphism and 
$\Tr_*:\DM^-(k,\mathbb Z)_{et}\xrightarrow{\sim}\DA^-(k,\mathbb Z)_{et}$ is it inverse. 
\end{thm}

The following proposition identify the relative motive of a closed pair $(Y,E)$ of projective varieties to the
motive of compact support of $V=Y\backslash E$

\begin{prop}\cite{VoeSF}\label{RMVarProp}
Let $V\in\Var(k)$ quasi projective. 
Let $Y\in\PVar(k)$ a compactification of $V$ with $E=Y\backslash V$.
Denote by $j:V\hookrightarrow Y$ the open embedding and $i:E\hookrightarrow Y$.
Then we have the following exact sequence in $\PSh_{\mathbb Z}(\Cor^{fs}_{\mathbb Z}(\SmVar(k)),C^-(\mathbb Z))$
\begin{equation*}
0\to\underline{C}_*\mathbb Z_{tr}(E)\xrightarrow{\mathbb Z_{tr}(i)}\underline{C}_*\mathbb Z_{tr}(Y)
\xrightarrow{\underline{j}^*}\underline{C}_*\mathbb Z_{eq}(V,0)
\end{equation*}
with $\underline{j}^*(X):=(I_X\times j)^*:\mathcal Z^{fs/X}(X\times Y,\mathbb Z)
\to\mathcal Z^{ed(0)/X}(X\times V,\mathbb Z)$.
This say (c.f. definition \ref{RMVar}) that $j^*$ 
induces a quasi-isomorphism in $\PSh_{\mathbb Z}(\Cor^{fs}_{\mathbb Z}(\SmVar(k)),C^-(\mathbb Z))$ 
\begin{equation*}
\underline{j}^*:\underline{C}_*\mathbb Z_{tr}(Y,E)\to\underline{C}_*\mathbb Z_{eq}(V,0)
\end{equation*}
\end{prop}

We finish this subsection by the following result of Suslin and Voevodsky :
\begin{prop}\label{DA1iso}
Let $X\in\SmVar(k)$, $Y\in\PVar(k)$ and $n\in\mathbb Z$, $n<0$.
The morphism of abelian group
\begin{equation*}
D(\mathbb A^1,et):\Hom_{PC^-}(\mathbb Z_{tr}(X),\underline{C}_*\mathbb Z_{tr}(Y)[n])\to\Hom_{DM^-(\mathbb Z)}(M(X),M(Y)[n])
\end{equation*}
of the functor 
$D(\mathbb A^1,et):\PSh_{\mathbb Z}(\Cor^{fs}_{\mathbb Z}(\SmVar(k)),C^-(\mathbb Z))\to\DM^-(k,\mathbb Z)$ 
is an isomorphism.
\end{prop}

\begin{proof}
For $X\in\SmVar(k)$ and $Y\in\Var(k)$, we have the commutative diagram
\begin{equation*}
\xymatrix{\Hom_{PC^-}(\mathbb Z_{tr}(X),\underline{C}_*\mathbb Z_{tr}(Y)[n])\ar[d]^{=}\ar[rr]^{D(\mathbb A^1,et)} 
& \, & \Hom_{DM^-(\mathbb Z)}(M(X),M(Y)[n])\ar[d]^{=} \\
H^n\underline{C}_*\mathbb Z_{tr}(Y)(X)\ar[rr]^{h(X)} & \, & \mathbb H_{et}^n(X,\underline{C}_*\mathbb Z_{tr}(Y))}
\end{equation*}
The equality on the right follow from the fact that
\begin{equation*}
\Hom_{DM^-(\mathbb Z)}(\mathbb Z_{tr}(X),\underline{C}_*\mathbb Z_{tr}(Y)[n])
=\Hom_{\Ho_{usu}(PC^-)}(\mathbb Z_{tr}(X),\underline{C}_*\mathbb Z_{tr}(Y)[n])
\end{equation*}
since, by theorem \ref{VarsingA}(ii)
\begin{itemize}
\item $\underline{C}_*\mathbb Z_{tr}(Y)$, $\underline{C}_*\mathbb Z_{tr}(Y)[n]$ are $\mathbb A^1$ local objects,
\item $S(\mathbb Z_{tr}(X)):\mathbb Z_{tr}(X)\to\underline{C}_*\mathbb Z_{tr}(X)$ is an equivalence $(\mathbb A^1,et)$ local.
\end{itemize}
Moreover, we have (c.f.\cite{VoeSF}) 
$\mathbb H_{et}^n(X,\underline{C}_*\mathbb Z_{tr}(Y))=\mathbb H_{Zar}^n(X,\underline{C}_*\mathbb Z_{tr}(Y))$
Now if $Y\in\PVar(k)$ is projective, $\underline{C}_*\mathbb Z_{tr}(Y)=\underline{C}_*\mathbb Z_{eq}(Y,0)$
satisfy Zariski descent (c.f.\cite{VoeSF}), hence $h(X)$ is an isomorphism.

\end{proof}

\subsection{The derived category of mixed motives of analytic spaces}

For $X\in\AnSp(\mathbb C)$, and $\Lambda$ a commutative ring and $p\in\mathbb N$, 
we denote by $\mathcal Z_p(X,\Lambda)$ the free $\Lambda$ module generated 
by the irreducible closed analytic subspaces of $X$ of dimension $p$.

For $X\in\AnSp(\mathbb C)$ irreducible, and $\Lambda$ a commutative ring and $p\in\mathbb N$, 
we denote by $\mathcal Z^p(X,\Lambda)=\mathcal Z_{d_X-p}(X,\Lambda)$ 
the free $\Lambda$ module generated by the irreducible closed analytic subspaces of $X$ of codimension $p$.

For $X,Y\in\AnSp(\mathbb C)$ and $\Lambda$ a commutative ring, we denote by 
\begin{itemize}
\item if $X$ is smooth connected, $\mathcal{Z}^{fs/X}(X\times Y,\Lambda)\subset\mathcal{Z}_{d_X}(X\times Y,\Lambda)$ 
the free $\Lambda$ submodule generated by the irreducible closed analytic subspaces of $X\times Y$ which are finite and surjective over $X$,
\item if $X$ is smooth, $\mathcal{Z}^{fs/X}(X\times Y,\Lambda):=\oplus_i\mathcal{Z}^{fs/X_i}(X_i\times Y,\Lambda)$ 
where $X=\sqcup_i X_i$, with $X_i$ the connected components of $X$.
\end{itemize}

\begin{defi}\cite{AyoubG2}
Let $\Cor^{fs}_{\Lambda}(\AnSm(\mathbb C))$ be the category whose objects are complex analytic varieties over $\mathbb C$ and whose space of morphisms between
$X,Y\in\AnSm(\mathbb C)$ is the free $\Lambda$ module $\mathcal{Z}^{fs/X}(X\times Y,\Lambda)$. 
The composition of morphisms is defined in \cite{AyoubG2}. 
\end{defi}

We have 
\begin{itemize}
\item the additive embedding of categories $\Tr:\mathbb Z(\AnSm(\mathbb C))\hookrightarrow\Cor^{fs}_{\mathbb Z}(\AnSm(\mathbb C))$ 
which gives the corresponding morphism of sites $\Tr:\Cor^{fs}_{\mathbb Z}(\AnSm(\mathbb C))\to \mathbb Z(\AnSm(\mathbb C))$,
\item the inclusion functor
$e_{an}:\left\{pt\right\}\hookrightarrow\AnSm(\mathbb C)$,
which gives the corresponding morphism of sites
$e_{an}:\AnSm(\mathbb C)\to\left\{pt\right\}$,
\item the inclusion functor $e^{tr}_{an}:=\Tr\circ e_{an}:\left\{pt\right\}\hookrightarrow\Cor^{fs}_{\mathbb Z}(\AnSm(\mathbb C))$
which gives the corresponding morphism of sites
$e^{tr}_{an}:=\Tr\circ e_{an}:\Cor^{fs}_{\mathbb Z}(\AnSm(\mathbb C))\to\left\{pt\right\}$.
\end{itemize}

We consider the following two big categories :
\begin{itemize}
\item $\PSh(\AnSm(\mathbb C),C^-(\mathbb Z))=\PSh_{\mathbb Z}(\mathbb Z(\AnSm(\mathbb C),C^-(\mathbb Z)))$,
the category of bounded above complexes of presheaves on $\AnSm(\mathbb C)$, 
or equivalently additive presheaves on $\mathbb Z(\AnSm(\mathbb C))$,
sometimes, we will write for short $P^-(\An)=\PSh(\AnSm(\mathbb C),C^-(\mathbb Z))$,
\item $\PSh(\Cor^{fs}_{\mathbb Z}(\AnSm(\mathbb C)),C^-(\mathbb Z))$, 
the category of bounded above complexes of additive presheaves on $\Cor^{fs}_{\mathbb Z}(\AnSm(\mathbb C))$
sometimes, we will write for short $PC^-(\An)=\PSh(\Cor^{fs}_{\mathbb Z}(\AnSm(\mathbb C)),C^-(\mathbb Z))$,
\end{itemize}
and the adjonctions :
\begin{itemize}
\item $(\Tr^*,\Tr_*):\PSh(\AnSm(\mathbb C),C^-(\mathbb Z))\leftrightarrows\PSh(\Cor^{fs}_{\mathbb Z}(\AnSm(\mathbb C)),C^-(\mathbb Z))$, 
\item $(e_{an}^*,e_{an*}):\PSh(\AnSm(\mathbb C),C^-(\mathbb Z))\leftrightarrows C^-(\mathbb Z)$,
\item $(e_{an}^{tr*},e_{an*}^{tr}):\PSh(\Cor^{fs}_{\mathbb Z}(\AnSm(\mathbb C)),C^-(\mathbb Z))\leftrightarrows C^-(\mathbb Z)$,
\end{itemize}
given by $\Tr:\Cor^{fs}_{\Lambda}(\AnSm(\mathbb C))\to \Lambda(\AnSm(\mathbb C))$, 
$e_{an}\AnSm(\mathbb C)\to\left\{pt\right\}$ and $e^{tr}_{an}:\Cor^{fs}_{\mathbb Z}(\AnSm(\mathbb C))\to\left\{pt\right\}$.
We denote by
$a_{usu}:\PSh_{\mathbb Z}(\AnSm(\mathbb C),\Ab)\to\Sh_{\mathbb Z,usu}(\AnSm(\mathbb C),\Ab)$
the sheaftification functor for the usual topology.

For $X\in\AnSm(\mathbb C)$, we denote by
\begin{equation}
\mathbb Z(X)\in\PSh(\AnSm(\mathbb C),C^-(\mathbb Z)) \mbox{\;\:}, \mbox{\;\;}
\mathbb Z_{tr}(X)\in\PSh_{\mathbb Z}(\Cor^{fs}_{\mathbb Z}(\AnSm(\mathbb C)),C^-(\mathbb Z))
\end{equation}
the presheaves represented by $X$. They are usu sheaves.
For $X\in\AnSp(\mathbb C)$ a complex analytic space non smooth,
\begin{equation}
\mathbb Z_{tr}(X)\in\PSh_{\mathbb Z}(\Cor^{fs}_{\mathbb Z}(\AnSm(\mathbb C)),C^-(\mathbb Z)), \;
Y\in\AnSm(\mathbb C)\to\mathcal{Z}^{fs/Y}(Y\times_k X,\mathbb Z) 
\end{equation}
is also an usu sheaf,

We consider 
\begin{itemize}
\item the usual monoidal strutcure on $\PSh(\AnSm(\mathbb C)),C^-(\mathbb Z))$ and the associated internal $\Hom$
given by, for $F^{\bullet},G^{\bullet}\in\PSh(\AnSm(\mathbb C)),C^-(\mathbb Z))$ and $Y\in\AnSm(\mathbb C)$, 
\begin{equation}
F^{\bullet}\otimes G^{\bullet}(X):X\mapsto F^{\bullet}(X)\otimes_{\mathbb Z} G^{\bullet}(X) \; , \;
\underline{\Hom}(\mathbb Z(Y),F^{\bullet}):X\mapsto F^{\bullet}(X\times Y),
\end{equation}
\item the unique monoidal strutcure on $\PSh_{\mathbb Z}(\Cor^{fs}_{\mathbb Z}(\AnSm(\mathbb C)),C^-(\mathbb Z))$
such that,for $X,Y\in\AnSm(\mathbb C)$, $\mathbb Z_{tr}(X)\otimes\mathbb Z_{tr}(Y):=\mathbb Z_{tr}(X\times Y)$
and wich commute with colimites.
It has an internal $\Hom$ which is given, for $X,Y\in\AnSm(\mathbb C)$ and
$F^{\bullet}\in\PSh_{\mathbb Z}(\Cor^{fs}_{\mathbb Z}(\AnSm(\mathbb C)),C^-(\mathbb Z))$,
$\underline{\Hom}(\mathbb Z_{tr}(Y),F^{\bullet}):X\mapsto F^{\bullet}(X\times Y)$
\end{itemize}
Together with these monoidal structure, the functor
\begin{equation*}
\Tr^*:\PSh(\AnSm(\mathbb C),C^-(\mathbb Z))\to \PSh_{\mathbb Z}(\Cor^{fs}_{\mathbb Z}(\AnSm(\mathbb C)),C^-(\mathbb Z))
\end{equation*}
is monoidal.

\begin{defi}\cite{AyoubT}\label{AnTrusu}
\begin{itemize}
\item[(i)] We say that a morphism $\phi:G_1^{\bullet}\to G_2^{\bullet}$ in $\PSh(\AnSm(\mathbb C),C^-(\mathbb Z))$ is an
usu local equivalence if $\phi_*:a_{usu}H^k(G_1^{\bullet})\to a_{usu}H^k(G_2^{\bullet})$ is an isomorphism for all $k\in\mathbb Z$.
The projective usual topology model structure on $\PSh(\AnSm(\mathbb C),C^-(\mathbb Z))$ 
is the left Bousfield localization of the projective model structure $\mathcal{M}_P(\PSh_{\mathbb Z}(\AnSm(\mathbb C),C^-(\mathbb Z)))$ 
with respect to the usu local equivalence.

\item[(ii)] We say that a morphism $\phi:G_1^{\bullet}\to G_2^{\bullet}$ 
in $\PSh_{\mathbb Z}(\Cor^{fs}_{\mathbb Z}(\AnSm(\mathbb C)),C^-(\mathbb Z))$ is an
usu local equivalence if and only if its restriction to $\AnSm(\mathbb C)$ 
$\Tr_*\phi:\Tr_*G_1^{\bullet}\to\Tr_*G_2^{\bullet}$ is an usu local equivalence.
The projective usual topology model structure on $\PSh_{\mathbb Z}(\Cor^{fs}_{\mathbb Z}(\AnSm(\mathbb C)),C^-(\mathbb Z))$ 
is the left Bousfield localization of the projective model structure  
$\mathcal{M}_P(\PSh_{\mathbb Z}(\Cor^{fs}_{\mathbb Z}(\AnSm(\mathbb C)),C^-(\mathbb Z)))$
with respect to the usu local equivalence.
\end{itemize}
\end{defi}


\begin{defi}\label{Dmodstr}
\begin{itemize}
\item[(i)] The projective $(\mathbb D^1,usu)$ model structure on $\PSh(\AnSm(\mathbb C),C^-(\mathbb Z))$ is the left Bousfield localization 
of the projective usual topology model structure (c.f. definition \ref{AnTrusu}(i)) 
with respect to the class of maps 
$\left\{\mathbb Z(X\times\mathbb D^1)[n]\to\mathbb Z(X)[n], \; X\in\AnSm(\mathbb C),n\in\mathbb Z\right\}$.

\item[(ii)] The projective $(\mathbb D^1,usu)$ model structure on $\PSh_{\mathbb Z}(\Cor^{fs}_{\mathbb Z}(\AnSm(\mathbb C)),C^-(\mathbb Z))$ is the left Bousfield localization of the projective usual topology model structure (c.f. definition \ref{AnTrusu}(ii)) 
with respect to the class of maps
$\left\{\mathbb Z_{tr}(X\times\mathbb D^1)[n]\to\mathbb Z(X)[n], \; X\in\AnSm(\mathbb C),n\in\mathbb Z\right\}$. 
\end{itemize}
\end{defi}

\begin{defi} 
\begin{itemize}
\item[(i)] We define 
$\AnDM^-(\mathbb Z):=\Ho_{\mathbb D^1,usu}(\PSh_{\mathbb Z}(\Cor^{fs}_{\mathbb Z}(\AnSm(\mathbb C)),C^-(\mathbb Z)))$, 
to be the derived category of motives of complex analytic space, it is 
the homotopy category of the category $\PSh_{\mathbb Z}(\Cor^{fs}_{\mathbb Z}(\AnSm(\mathbb C)),C^-(\mathbb Z))$ 
with respect to the projective $(\mathbb D^1,usu)$ model structure (c.f. definition \ref{Dmodstr}(ii)). 
We denote by
\begin{equation*}
D^{tr}(\mathbb D^1,usu):\PSh_{\mathbb Z}(\Cor^{fs}_{\mathbb Z}(\AnSm(\mathbb C)),C^-(\mathbb Z)))\to\AnDM^{-}(\mathbb Z) \, , \,
D^{tr}(\mathbb D^1,usu)(F^{\bullet})=F^{\bullet}
\end{equation*}
the canonical localization functor.

\item[(ii)] We denote by the same way 
$\AnDA^-(\mathbb Z):=\Ho_{\mathbb D^1,usu}(\PSh(\AnSm(\mathbb C),C^-(\mathbb Z)))$ (c.f. definition \ref{Dmodstr}(i)) and
\begin{equation*}
D(\mathbb D^1,usu):\PSh_{\mathbb Z}(\AnSm(\mathbb C),C^-(\mathbb Z)))\to\AnDA^{-}(\mathbb Z) \, , \,   
D(\mathbb D^1,usu)(F^{\bullet})=F^{\bullet}
\end{equation*}
the canonical localization functor.
\end{itemize}
\end{defi}

For $X\in\AnSp(\mathbb C)$, the (derived) motive of $X$ is 
$M(X)=D(\mathbb D^1,usu)(\mathbb Z_{tr}(X))\in\AnDM^-(\mathbb Z)$.

Recall $\bold{\Delta}$ denote the simplicial category
Let 
\begin{itemize}
\item $p_{\bold\Delta}:\AnSm(\mathbb C)\to\bold{\Delta}\times\AnSm(\mathbb C)$ 
be the morphism of site given by the projection functor 
$(X,i)\in\bold{\Delta}\times\AnSm(\mathbb C)\mapsto p_{\bold\Delta}((X,i))=X\in\AnSm(\mathbb C)$,
\item $p_{\bold\Delta}:\Cor_{\mathbb Z}(\AnSm(\mathbb C))\to\bold{\Delta}\times\Cor_{\mathbb Z}(\AnSm(\mathbb C))$ 
be the morphism of site given by the projection functor 
$(X,i)\in\bold{\Delta}\times\AnSm(\mathbb C)\mapsto p_{\bold\Delta}((X,i))=X\in\AnSm(\mathbb C)$
\end{itemize}
We have the adjonctions
\begin{itemize}
\item $(p_{\bold\Delta}^*,p_{\bold\Delta*}):\PSh(\bold{\Delta}\times\AnSm(\mathbb C),C^-(\mathbb Z))
\leftrightarrows\PSh(\AnSm(\mathbb C),C^-(\mathbb Z))$    
\item $p_{\bold\Delta}:\PSh(\bold{\Delta}\times\Cor_{\mathbb Z}(\AnSm(\mathbb C)),C^-(\mathbb Z))
\leftrightarrows\PSh(\Cor_{\mathbb Z}(\AnSm(\mathbb C)),C^-(\mathbb Z))$.   
\end{itemize}

In following lemma point (ii) is a generalization of point (i) to hypercoverings.
We will only use point (i) in this paper. 

\begin{lem}\label{HRAn}
\begin{itemize}
\item[(i)] Let $X\in\AnSm(\mathbb C)$ and $X=\cup_{i\in J}U_i$ a covering by open subsets $U_i\subset X$, $J$ being a countable set.
We denote by $j_i:U_i\hookrightarrow X$ the open embedding.
For $I\subset J$ a finite subset, let $U_I=\cap_{i\in I} U_i$.
Then, 
\begin{itemize}
\item $\left[\cdots\rightarrow\oplus_{\card I=r}\mathbb Z(U_I)
\rightarrow\cdots\rightarrow\oplus_{i\in J}\mathbb Z(U_i)\right]\xrightarrow{\oplus_{i\in J}\mathbb Z(j_i)}\mathbb Z(X)$ 
is an usu local equivalence in $P^-(\An)$,
\item $\left[\cdots\rightarrow\oplus_{\card I=r}\mathbb Z_{tr}(U_I)
\rightarrow\cdots\rightarrow\oplus_{i\in J}\mathbb Z_{tr}(U_i)\right]
\xrightarrow{\oplus_{i\in J}\mathbb Z_{tr}(j_i)}\mathbb Z_{tr}(X)$
is an usu local equivalence in $PC^-(\An)$
\end{itemize}
\item[(ii)] Let $X\in\AnSm(\mathbb C)$ and $j_{\bullet}:U_{\bullet}\to X$ an hypercovering of $X$ by open subset 
(i.e. for all $n\in\bold{\Delta}$,  $U_n$ is an open subset and $j_n:U_n\hookrightarrow X$ is the open embedding),
and $(U_{\bullet},\bullet):\cdots\to(U_{n},n)\to\cdots$ the associated complex in $\bold{\Delta}\times\AnSm(\mathbb C)$
Then,  
\begin{itemize}
\item $j:Lp_{\bold{\Delta}}^*(\mathbb Z(U_{\bullet},\bullet))\to\mathbb Z(X)$ is an isomorphism 
in $\Ho_{usu}(\PSh(\AnSm(\mathbb C),C^-(\mathbb Z)))$
\item $j:Lp_{\bold{\Delta}}^*(\mathbb Z_{tr}(U_{\bullet},\bullet))\to\mathbb Z_{tr}(X)$ is an isomorphism 
in $\Ho_{usu}(\PSh_{\mathbb Z}(\Cor_{\mathbb Z}(\AnSm(\mathbb C)),C^-(\mathbb Z)))$
\end{itemize}
\end{itemize}
\end{lem}

\begin{proof}

\noindent(i): It follows from the following two facts : 
\begin{itemize}
\item the sequence
$[\cdots\rightarrow\oplus_{\card I=r}\mathbb Z(U_I)\rightarrow\cdots\rightarrow\oplus_{i\in J}\mathbb Z(U_i)]$
is  clearly exact in $P^-(\An)$, 
\item the exactness of 
$[\oplus_{i\in J}\mathbb Z(U_i)\xrightarrow{\oplus_{i\in J}\mathbb Z(j_i)}\mathbb Z(X)\to 0]$
in $\Ho_{usu}(\PSh(\AnSm(\mathbb C),C^-(\mathbb Z)))$,
follows from the fact that for $Y\in\AnSm(\mathbb C)$, $y\in Y$ and $f:Y\to X$ a morphism, 
there exists an open subset $V(y)\subset Y$ containing $y$ such that $f(V(y))\subset U_i$, with
$U_i$ containing $f(y)$.  
\end{itemize}
Similary, the second point follows from the following two facts : 
\begin{itemize}
\item the sequence 
$[\cdots\rightarrow\oplus_{\card I=r}\mathbb Z_{tr}(U_I)\rightarrow\cdots\rightarrow\oplus_{i\in J}\mathbb Z_{tr}(U_i)]$
is clearly exact in $PC^-(\An)$, 
\item the exactness of 
$[\oplus_{i\in J}\mathbb Z_{tr}(U_i)\xrightarrow{\oplus_{i\in J}\mathbb Z_{tr}(j_i)}\mathbb Z_{tr}(X)\to 0]$
in $\Ho_{usu}(\PSh_{\mathbb Z}(\Cor^{fs}_{\mathbb Z}(\AnSm(\mathbb C)),C^-(\mathbb Z)))$,
follows from the fact that for $Y\in\AnSm(\mathbb C)$, $y\in Y$ and $\alpha\in\mathbb Z_{tr}(X)(Y)$ irreducible, 
there exists an open subset $V(y)\subset Y$ containing $y$ such that $\alpha_{|V(y)}\in\mathbb Z_{tr}(U_i)(V(y))$, with
$U_i$ containing $\alpha_*y$.  
\end{itemize}

\noindent(ii): see \cite{AyoubB} Proposition 1.4 Etape 1 with our notation 
$p_{\bold{\Delta}}^*$ for $p_{\bold{\Delta}}$.

\end{proof}

The following proposition is to use point (i) instead of point (ii) 
in the proof of the smooth case of point (i) of the proposition \ref{Bettiprop}.

\begin{prop}\label{Dcov}
Let $X\in\AnSm(\mathbb C)$ connected. Then there exist a countable open covering $X=\cup_{i\in J}D_i$ such that
for all finite subset $I\subset J$, $D_I:=\cap_{i\in I} D_i=\emptyset$ or $D_I\simeq\mathbb D^{d_X}$ is
bihomolomorphic to an open ball $\mathbb D^{d_X}\subset\mathbb C^{d_X}$.
\end{prop}

\begin{proof}
Let $x\in X$. As $X$ is smooth, there exist an open neighborhood $D_x\subset X$ of $x$ in $X$ such that
$D_x$ is geodesically convex and such that $D_x\hookrightarrow\mathbb C^{d_X}$. 
As $X$ a countable union of compact subset, we can extract a countable covering $X=\cup_{i\in J} D_J$ of 
the open covering $X=\cup_{x\in X}D_x$. As $D_I$ is geodesically convex and $D_I$ admits an open embedding
$D_I\hookrightarrow\mathbb C^{d_X}$ in $\mathbb C^{d_X}$, $D_I\simeq\mathbb D^{d_X}$.  

\end{proof}

We denote $\bar{\mathbb D}^n=\bar{D}(0,1)^n\subset\mathbb C^n$. We see it as a pro-objet of $\AnSm(\mathbb C)$ (\cite{AyoubG2}).
For $F^{\bullet}\in\PSh(\AnSm(\mathbb C),\Ab)$ and $X\in\AnSm(\mathbb C))$, 
we have the complex $F(X\times\bar{\mathbb D}^*)$ associated to the cubical object 
$F(X\times\bar{\mathbb D}^*)$ in the category of abelian groups.

\begin{itemize}
\item If $F^{\bullet}\in\PSh(\AnSm(\mathbb C),C^-(\mathbb Z))$, 
\begin{equation}
\underline{\sing}_{\bar{\mathbb D}^*}F^{\bullet}:=\underline\Hom(\mathbb Z(\bar{\mathbb D}^*),F^{\bullet})\in\PSh(\AnSm(\mathbb C),C^-(\mathbb Z))
\end{equation}
is the total complex of presheaves associated to the bicomplex of presheaves $X\mapsto F^{\bullet}(\bar{\mathbb D}^*\times X)$,
and 
$\sing_{\bar{\mathbb D}}F^{\bullet}:=
e_{an*}\underline{\sing}_{\bar{\mathbb D}^*}F^{\bullet}=F^{\bullet}(\bar{\mathbb D}^*)\in C^-(\mathbb Z)$.
We denote by 
$S(F^{\bullet}):F^{\bullet}\to\underline{\sing}_{\bar{\mathbb D}^*}F^{\bullet}$,
\begin{equation}\label{SFAn}
S(F^{\bullet}):\xymatrix{
\cdots\ar[r] & 0\ar[r]\ar[d] & 0\ar[r]\ar[d] & F^{\bullet}\ar[d]^{I}\ar[r] & 0\ar[r]\ar[d] & \cdots \\
\cdots\ar[r] & \underline{\sing}_{\bar{\mathbb D}^2}F^{\bullet}\ar[r] & 
\underline{\sing}_{\bar{\mathbb D^1}}F^{\bullet}\ar[r] & F^{\bullet}\ar[r] & 0\ar[r] & \cdots} 
\end{equation}
the inclusion morphism in $\PSh(\AnSm(\mathbb C),C^-(\mathbb Z))$. 
For $f:F_1^{\bullet}\to F_2^{\bullet}$ a morphism $\PSh(\AnSm(\mathbb C),C^-(\mathbb Z))$, we denote by 
$S(f):\underline{\sing}_{\bar{\mathbb D}^*}F_1^{\bullet}\to\underline{\sing}_{\bar{\mathbb D}^*}F_2^{\bullet}$
the morphism of $P^-(\An)$ given by, for $X\in\AnSm(\mathbb C)$,
\begin{equation}\label{SfAn}
S(f)(X):\xymatrix{
\cdots\ar[r] & F_1(\bar{\mathbb D}^2\times X)\ar[r]\ar[d]^{f(\bar{\mathbb D}^2\times X)} &
F_1(\bar{\mathbb D}^1\times X)\ar[r]\ar[d]^{f(\bar{\mathbb D}^1\times X)} & F_1^{\bullet}(X)\ar[d]^{f(X)}\ar[r] & 0\ar[r]\ar[d] & \cdots \\
\cdots\ar[r] & F_2(\bar{\mathbb D}^2\times X)\ar[r] & F_2(\bar{\mathbb D}^1\times X)\ar[r] & F_2^{\bullet}(X)\ar[r] & 0\ar[r] & \cdots.}  
\end{equation}

\item If $F^{\bullet}\in\PSh_{\mathbb Z}(\Cor^{fs}_{\mathbb Z}(\AnSm(\mathbb C)),C^-(\mathbb Z))$, 
\begin{equation}
\underline{\sing}_{\bar{\mathbb D}^*}F^{\bullet}:=\underline\Hom(\mathbb Z_{tr}(\bar{\mathbb D}^*),F^{\bullet})
\in\PSh_{\mathbb Z}(\Cor^{fs}_{\mathbb Z}(\AnSm(\mathbb C)),C^-(\mathbb Z)) 
\end{equation}
is the total complex of presheaves associated to the bicomplex of presheaves $X\mapsto F^{\bullet}(\bar{\mathbb D}^*\times X)$,
and 
$\sing_{\bar{\mathbb D}}F^{\bullet}:=
e^{tr}_{an*}\underline{\sing}_{\bar{\mathbb D}^*}F^{\bullet}=F^{\bullet}(\bar{\mathbb D}^*)\in C^-(\mathbb Z)$.
We have the inclusion morphism (\ref{SFAn})
\begin{equation*}
S(\Tr_*F^{\bullet}):\Tr_*F^{\bullet}\to
\underline{\sing}_{\bar{\mathbb D}^*}\Tr_*F^{\bullet}=\Tr_*\underline{\sing}_{\bar{\mathbb D}^*}F^{\bullet}
\end{equation*}
which is a morphism in $\PSh(\Cor^{fs}_{\mathbb Z}(\AnSm(\mathbb C)),C^-(\mathbb Z))$ 
denoted the same way $S(F^{\bullet}):F^{\bullet}\to\underline{\sing}_{\bar{\mathbb D}^*}F^{\bullet}$. 
For $f:F_1^{\bullet}\to F_2^{\bullet}$ a morphism $PC^-(\An)$, we 
have the morphism (\ref{SfAn})
\begin{equation*} 
S(\Tr_*f):\Tr_*\underline{\sing}_{\bar{\mathbb D}^*}F_1^{\bullet}=\underline{\sing}_{\bar{\mathbb D}^*}\Tr_*F_1^{\bullet}\to
\underline{\sing}_{\bar{\mathbb D}^*}\Tr_*F_2^{\bullet}=\Tr_*\underline{\sing}_{\bar{\mathbb D}^*}F_2^{\bullet}
\end{equation*}
which is a morphism in $PC^-(\An)$ 
denoted the same way
$S(f):\underline{\sing}_{\bar{\mathbb D}^*}F_1^{\bullet}\to\underline{\sing}_{\bar{\mathbb D}^*}F_2^{\bullet}$.
\end{itemize}

For $F^{\bullet}\in\PSh(\Cor^{fs}_{\mathbb Z}(\AnSm(\mathbb C)),C^-(\mathbb Z))$, we have by definition 
$\Tr_*\underline{\sing}_{\bar{\mathbb D}^*}F^{\bullet}=\underline{\sing}_{\bar{\mathbb D}^*}\Tr_*F^{\bullet}$
and $\Tr_*S(F^{\bullet})=S(\Tr_*F^{\bullet})$.

We now make the following definition

\begin{defi}\label{RMAn}
Let $X\in\AnSp(\mathbb C)$ and $D\subset X$ an analytic subspace. 
Denote by  $l:D\hookrightarrow X$ the locally closed embbeding.
We define
\begin{itemize}
\item[(i)] $\mathbb Z(X,D)=\coker(\mathbb Z(l))\in\PSh(\AnSm(\mathbb C),C^-(\mathbb Z))$
to be the cokernel of the injective morphism $\mathbb Z(l):\mathbb Z_{tr}(D)\hookrightarrow\mathbb Z(X)$.
By definition, we have the following exact sequence
\begin{equation}\label{AnXD}
0\to\mathbb Z(D)\xrightarrow{\mathbb Z(l)}\mathbb Z(X)\xrightarrow{c(X,D)}\mathbb Z(X,D)\to 0
\end{equation}
in $\PSh(\AnSm(\mathbb C),C^-(\mathbb Z))$.

\item[(ii)] $\mathbb Z_{tr}(X,D)=\coker(\mathbb Z_{tr}(l))\in\PSh_{\mathbb Z}(\Cor^{fs}_{\mathbb Z}(\AnSm(\mathbb C)),C^-(\mathbb Z))$
to be the cokernel of the injective morphism $\mathbb Z_{tr}(l):\mathbb Z_{tr}(D)\hookrightarrow\mathbb Z_{tr}(X)$.
By definition, we have the following exact sequence
\begin{equation}\label{AnXDtr}
0\to\mathbb Z_{tr}(D)\xrightarrow{\mathbb Z_{tr}(l)}\mathbb Z_{tr}(X)\xrightarrow{c(X,D)}\mathbb Z_{tr}(X,D)\to 0
\end{equation}
in $\PSh_{\mathbb Z}(\Cor^{fs}_{\mathbb Z}(\SmVar(k)),C^-(\mathbb Z))$.
By the exact sequences, we have $\Tr^*\mathbb Z(X,D)=\mathbb Z_{tr}(X,D)$
\end{itemize}
\end{defi}

In particular, we get the following exact sequence in $\PSh_{\mathbb Z}(\Cor^{fs}_{\mathbb Z}(\AnSm(\mathbb C)),C^-(\mathbb Z))$
\begin{equation}\label{AncXD}
0\to\underline{\sing}_{\mathbb D^*}\mathbb Z_{tr}(D)\xrightarrow{S(\mathbb Z_{tr}(l))}
\underline{\sing}_{\mathbb D^*}\mathbb Z_{tr}(X)
\xrightarrow{S(c(X,D))}\underline{\sing}_{\mathbb D^*}\mathbb Z_{tr}(X,D)\to 0
\end{equation}
We define 
\begin{equation*}
M(Y,E)=D(\mathbb D^1,usu)(\mathbb Z_{tr}(Y,E))\in\AnDM^-(\mathbb Z)
\end{equation*}
to be the relative motive of the pair $(Y,E)$.

We now look at the behavior of the functors mentionned above with respect to the $(\mathbb D^1,et)$ model structure

\begin{prop}\cite{AyoubG2}\cite{AyoubG1}
\begin{itemize}
\item[(i)] $(\Tr^*,\Tr_*):\PSh(\AnSm(k),C^-(\mathbb Z))\leftrightarrows\PSh_{\mathbb Z}(\Cor^{fs}_{\mathbb Z}(\AnSm(\mathbb C)),C^-(\mathbb Z))$
is a Quillen adjonction for the usual topology model structures (c.f. definition \ref{AnTrusu} (i) and (ii) respectively)
and a Quillen adjonction for the $(\mathbb D^1,usu)$ model structures (c.f. definition \ref{Dmodstr} (i) and (ii) respectively). 
\item[(ii)] $(e_{an}^*,e_{an*}):\PSh(\AnSm(\mathbb C),C^-(\mathbb Z))\leftrightarrows C^-(\mathbb Z)$
is a Quillen adjonction for the usual topology model structure (c.f. definition \ref{AnTrusu} (i))
and a Quillen adjonction for the $(\mathbb D^1,usu)$ model structure (c.f. definition \ref{Dmodstr} (i)). 
\item[(iii)] $(e_{an}^{tr*},e_{an*}^{tr}):\PSh(\Cor^{fs}_{\mathbb Z}(\AnSm(\mathbb C)),C^-(\mathbb Z))\leftrightarrows C^-(\mathbb Z)$
is a Quillen adjonction for the usual topology model structure (c.f. definition \ref{AnTrusu} (ii))
and a Quillen adjonction for the $(\mathbb D^1,usu)$ model structure (c.f. definition \ref{Dmodstr} (ii)). 
\end{itemize}
\end{prop}

\begin{lem}\label{DTr}
\begin{itemize}
\item[(i)] A complex of presheaves 
$F^{\bullet}\in\PSh(\Cor^{fs}_{\mathbb Z}(\AnSm(\mathbb C)),C^-(\mathbb Z))$ is 
$\mathbb D^1$ local if and only if 
$\Tr_*F^{\bullet}\in\PSh(\AnSm(\mathbb C),C^-(\mathbb Z))$ is 
$\mathbb D^1$ local.

\item[(ii)] A morphism 
$\phi:F^{\bullet}\to G^{\bullet}$ in $\PSh(\Cor^{fs}_{\mathbb Z}(\AnSm(\mathbb C)),C^-(\mathbb Z))$  
is an $(\mathbb D^1,usu)$ local equivalence if and only if
$\Tr_*\phi:\Tr_*F^{\bullet}\to\Tr_*G^{\bullet}$ is 
is an $(\mathbb D^1,usu)$ local equivalence.
\end{itemize}
\end{lem}

\begin{proof}

\noindent(i): Let $g:F^{\bullet}\to L^{\bullet}$ 
be an usu local equivalence in $\PSh(\Cor^{fs}_{\mathbb Z}(\AnSm(\mathbb C)),C^-(\mathbb Z))$
with $L^{\bullet}$ usu fibrant. Then,
\begin{itemize}
\item $F^{\bullet}$ is $\mathbb D^1$ local if and only if $L^{\bullet}$ is $\mathbb D^1$ local.

\item By definition $\Tr_*$ preserve usu local equivalence, hence
$\Tr_*g:\Tr_*F^{\bullet}\to\Tr_*L^{\bullet}$ is an usu local equivalence in
$\PSh(\AnSm(\mathbb C),C^-(\mathbb Z))$. 
Thus, $\Tr_*F^{\bullet}$ is $\mathbb D^1$ local if and only if $\Tr_*L^{\bullet}$ is $\mathbb D^1$ local. 

\item As indicated in \cite{AyoubG2}, $\Tr_*L^{\bullet}$ is also usu fibrant. 
By Yoneda lemma, for $X\in\AnSm(\mathbb C)$, we have
\begin{equation*}
\Hom_{P^-(\An)}(\mathbb Z(X)[n],\Tr_*L^{\bullet})=H^nL^{\bullet}(X)=\Hom_{PC^-(\An)}(\mathbb Z_{tr}(X)[n],L^{\bullet})
\end{equation*}
Hence $L^{\bullet}$ is $\mathbb D^1$ local if and only if $\Tr_*L^{\bullet}$ is $\mathbb D^1$ local.
\end{itemize}
This proves (i).

\noindent (ii): Let us prove (ii). 
\begin{itemize}
\item The only if part is proved in \cite{AyoubG2}.

\item Let $\phi:F^{\bullet}\to G^{\bullet}$ in $\PSh(\Cor^{fs}_{\mathbb Z}(\AnSm(\mathbb C)),C^-(\mathbb Z))$ such that
$\Tr_*\phi:\Tr_*F^{\bullet}\to\Tr_*G^{\bullet}$ is an equivalence $(\mathbb D^1,usu)$ local
in $\PSh(\AnSm(\mathbb C),C^-(\mathbb Z))$. Since by definition $\Tr_*$ detecte and preserve usu local equivalence,
we can, up to replace $F^{\bullet}$ and $G^{\bullet}$ by usu equivalent presheaves, 
assume that $F^{\bullet}$ and $G^{\bullet}$ are usu fibrant.
Let $K^{\bullet}\in\PSh(\Cor^{fs}_{\mathbb Z}(\AnSm(\mathbb C)),C^-(\mathbb Z))$ be an $\mathbb D^1$ local object.
By (i), $\Tr_*K^{\bullet}$ is $\mathbb D^1$ local.
Hence, we have,
\begin{eqnarray*}
\Hom_{PC^-(\An)}(G^{\bullet},K^{\bullet})=\Hom_{PC^-}(\Tr^*\Tr_*G^{\bullet},K^{\bullet})=\Hom_{P^-(\An)}(\Tr_*G^{\bullet},\Tr_*K^{\bullet})
\xrightarrow{\sim} \\ 
\Hom_{P^-(\An)}(\Tr_*F^{\bullet},\Tr_*K^{\bullet})=\Hom_{PC^-}(\Tr^*\Tr_*F^{\bullet},K^{\bullet})=\Hom_{PC^-(\An)}(F^{\bullet},K^{\bullet}).
\end{eqnarray*}
This proves the if part.
\end{itemize}

\end{proof}

\begin{prop}
\begin{itemize}
\item[(i)] The functor 
$\Tr_*:\PSh_{\mathbb Z}(\Cor^{fs}_{\mathbb Z}(\AnSm(\mathbb C)),C^-(\mathbb Z))\to\PSh(\AnSm(\mathbb C)),C^-(\mathbb Z))$ 
derive trivially.
\item[(ii)] For $K^{\bullet}\in C^-(\mathbb Z)$, $e_{an}^*K^{\bullet}$ is $\mathbb D^1$ local.
\item[(iii)] For $K^{\bullet}\in C^-(\mathbb Z)$, $e^{tr*}_{an}K^{\bullet}$ is $\mathbb D^1$ local.
\end{itemize}
\end{prop}

\begin{proof}

\noindent(i): By lemma \ref{DTr} (ii), $\Tr_*$ preserve $(\mathbb D^1,usu)$ local equivalence.

\noindent(ii): It is proved in \cite{AyoubB}.

\noindent(iii): We have $\Tr_*e^{tr*}_{an}K^{\bullet}=e_{an}^*K^{\bullet}$.
By (ii), $e_{an}^*K^{\bullet}$ is $\mathbb D^1$ local. By lemma \ref{DTr} (i), $\Tr_*$ detect $\mathbb D^1$ local object.
This proves (iii).

\end{proof}

\begin{thm}\label{AnS}
\begin{itemize}
\item[(i)] For $F^{\bullet}\in\PSh(\AnSm(\mathbb C),C^-(\mathbb Z))$,
$\underline{\sing}_{\bar{\mathbb D}^*}F^{\bullet}\in\PSh(\AnSm(\mathbb C),C^-(\mathbb Z))$ is $\mathbb D^1$ local and 
the inclusion morphism $S(F^{\bullet}):F^{\bullet}\to\underline{\sing}_{\bar{\mathbb D}^*}F^{\bullet}$ is an $(\mathbb D^1,usu)$ equivalence. 
\item[(ii)] For $F^{\bullet}\in\PSh_{\mathbb Z}(\Cor^{fs}_{\mathbb Z}(\AnSm(\mathbb C)),C^-(\mathbb Z))$, 
$\underline{\sing}_{\bar{\mathbb D}^*}F^{\bullet}\in\PSh_{\mathbb Z}(\Cor^{fs}_{\mathbb Z}(\AnSm(\mathbb C)),C^-(\mathbb Z))$ 
is $\mathbb D^1$ local and 
the inclusion morphism $S(F^{\bullet}):F^{\bullet}\to\underline{\sing}_{\bar{\mathbb D}^*}F^{\bullet}$ is an $(\mathbb D^1,usu)$ equivalence. 
\end{itemize}
\end{thm}

\begin{proof}

\noindent(i): It is proved in \cite{AyoubG1}.

\noindent(ii): 
By (i), $\Tr_*\underline{\sing}_{\bar{\mathbb D}^*}F^{\bullet}=\underline{\sing}_{\bar{\mathbb D}^*}\Tr_*F^{\bullet}$
is $\mathbb D^1$ local. 
By lemma \ref{DTr} (i), $\Tr_*$ detect $\mathbb D^1$ local object. 
Thus, 
$\underline{\sing}_{\bar{\mathbb D}^*}F^{\bullet}=\underline{\sing}_{\bar{\mathbb D}^*}F^{\bullet}$
is $\mathbb D^1$ local. 
This proves the first part of the assertion.

It follows from (i) that 
\begin{equation}
\Tr_*S(F^{\bullet}):\Tr_*F^{\bullet}\to\underline{\sing}_{\bar{\mathbb D}^*}\Tr_*F^{\bullet}
=\Tr_*\underline{\sing}_{\bar{\mathbb D}^*}F^{\bullet} 
\end{equation} 
is an $(\mathbb D^1,usu)$ equivalence. 
Since, by lemma \ref{DTr} (ii), $\Tr_*$ detect $(\mathbb D^1,usu)$ equivalence,
$S(F^{\bullet}):F^{\bullet}\to\underline{\sing}_{\bar{\mathbb D}^*}F^{\bullet}$
is an $(\mathbb D^1,usu)$ equivalence. 
 
\end{proof}

\begin{thm}\label{AnTr}\cite{AyoubG2}\cite{AyoubB}
\begin{itemize}
\item[(i)]The adjonction
$(\Tr^*,\Tr_*):\PSh(\AnSm(\mathbb C),C^-(\mathbb Z))\leftrightarrows\PSh_{\mathbb Z}(\Cor^{fs}_{\mathbb Z}(\AnSm(\mathbb C)),C^-(\mathbb Z))$
is a Quillen equivalence for the $(\mathbb D^1,usu)$ model structures. 
That is, the derived functor 
\begin{equation}
L\Tr^*:\AnDA^-(\mathbb Z)\xrightarrow{\sim}\AnDM^-(\mathbb Z) 
\end{equation}
is an isomorphism
and $\Tr_*:\AnDM^-(\mathbb Z)\xrightarrow{\sim}\AnDA^-(\mathbb Z)$ is it inverse. 

\item[(ii)] The adjonction
$(e_{an}^*,e_{an*}):C^-(\mathbb Z)\leftrightarrows\PSh_{\mathbb Z}(\AnSm(\mathbb C),C^-(\mathbb Z))$
is a Quillen equivalence for the $(\mathbb I^1,usu)$ model structures. 
That is, the derived functor
\begin{equation}
e_{an}^*:D^-(\mathbb Z)\xrightarrow{\sim}\AnDA^-(\mathbb Z)
\end{equation}
is an isomorphism
and $Re_{an*}:\AnDA^-(\mathbb Z)\xrightarrow{\sim}D^-(\mathbb Z)$ is it inverse. 

\item[(iii)] The adjonction
$(e^{tr*}_{an},e^{tr}_{an*}):C^-(\mathbb Z)\leftrightarrows\PSh_{\mathbb Z}(\Cor^{fs}_{\mathbb Z}(\AnSm(\mathbb C)),C^-(\mathbb Z))$
is a Quillen equivalence for the $(\mathbb D^1,usu)$ model structures. 
That is, the derived functor
\begin{equation}
e^{tr*}_{an}:D^-(\mathbb Z)\xrightarrow{\sim}\AnDM^-(\mathbb Z)
\end{equation}
is an isomorphism
and $Re^{tr}_{an*}:\AnDM^-(\mathbb Z)\xrightarrow{\sim}D^-(\mathbb Z)$ is it inverse. 
\end{itemize}
\end{thm}

\begin{defi}
\begin{itemize}
\item[(i)] We say that a morphism $\epsilon:X'\to X$ with $X,X'\in\AnSp(\mathbb C)$ is  
a proper modification (or abstract blow up) if it is proper and if there exist 
$Z\subset X$ a closed analytic subspace (the discriminant) such that $\epsilon$ 
induces an isomorphism $\epsilon_{X\backslash Z}:X'\backslash E\xrightarrow{\sim} X\backslash Z$
with $E=\epsilon^{-1}(Z)$.
\item[(ii)] The cdh topology on $\AnSp(\mathbb C)$ is the minimal Grothendieck topology generated by
open covering for the usual topology and covers $c=\epsilon\sqcup l:X'\sqcup Z\to X$ corresponding
to proper modifications $\epsilon:X'\to X$, with $l:Z\hookrightarrow X$ the closed embedding.
\item[(iii)] We say that a morphism $\epsilon:X'\to X$ in $\AnSp(\mathbb C)$ is  
a resolution of $X$ if it is a proper modification and if $X'$ is smooth.   
\end{itemize}
\end{defi}

We recall a version of Hironaka's desingulariation theorem
\begin{thm}\label{HirDes}\cite{Hironaka}
For all $X\in\AnSp(\mathbb C)$, there exist a proper modification $\epsilon:X'\to X$
with discriminant $Z\subset X$ such that $X'$ is smooth and $\epsilon^{-1}(Z)\subset X'$ is a normal crossing divisor. 
\end{thm}

We denote by $\iota_{an}:\AnSm(\mathbb C)\hookrightarrow\AnSp(\mathbb C)$ the full embedding functor,
which gives the morphism of sites 
$\iota_{an}:\AnSp(\mathbb C)\hookrightarrow\AnSm(\mathbb C)$.
If $F\in\PSh(\AnSm(\mathbb C),\Ab)$ is an usu sheaf, $\iota_{an}^*F$ is a cdh sheaf.
If $F\in\PSh(\AnSp(\mathbb C),\Ab)$ is an usu sheaf, we have $\iota_{an}^*\iota_{an*}F=a_{cdh}F$.

\begin{cor}\label{DesAnCor}
Let $F\in\PSh(\AnSm(\mathbb C),\Ab)$ be an usu sheaf, then $\iota_{an}^*F=0$
if and only for all $X\in\AnSm(\mathbb C)$ and all $\alpha\in F(X)$, there exist
a composition of blowups along smooth centers $e:X_r\to\cdots\to X$ such that 
$F(e)(\alpha)=0\in F(X_r)$.
\end{cor}

\begin{proof}
It follows from theorem \ref{HirDes}.
\end{proof}

\begin{defi}
Let $F\in\PSh(\AnSm(\mathbb C),\Ab)$. 
Let where $F\to I^{\bullet}$ is an usu local equivalence in $P^-(\An)$
with $I^{\bullet}$ a complex of injective objects (that is $I^{\bullet}$ is projectively fibrant).
\begin{itemize}
\item[(i)] We say that $F$ is homotopy invariant
if $F(p_X):F(X)\to F(X\times\mathbb D^1)$ is an isomorphism, 
where $p_X:X\times\mathbb D^1\to X$ is the projection.
\item[(ii)] We say that $F$ is strictly homotopy invariant
if the presheaf $H^n(I^{\bullet}):X\in\AnSm(\mathbb C)\mapsto\mathbb H^n(X,F)$, 
is homotopy invariant for all $n\in\mathbb N$.
\end{itemize}
\end{defi}

\begin{prop}\label{HSHDAn}
Let $F\in\PSh(\AnSm(\mathbb C),\Ab)$. 
Then $F$ is homotopy invariant if and only if it is strictly homotopy invariant. 
\end{prop}

\begin{lem}\label{DesAnlem}
\begin{itemize}
\item[(i)] Let $e:X'\to X$ be the blow up of a smooth $X\in\AnSm(\mathbb C)$ along a smooth center 
$Z\subset X$. Denote by $l:Z\hookrightarrow X$ the closed embedding.
Let $C$ (resp. $Q$) the cokernel of $\mathbb Z_{tr}(e):\mathbb Z_{tr}(X')\to\mathbb Z_{tr}(X)$   
(resp. the cokernel of 
$\mathbb Z_{tr}(l)\oplus\mathbb Z_{tr}(e):\mathbb Z_{tr}(Z)\oplus\mathbb Z_{tr}(X')\to\mathbb Z_{tr}(X)$).   
Then for any homotopy invariant sheaf $F\in\PSh(\Cor(\AnSm(\mathbb C)),\Ab)$,
$\Ext^n(C,F)=\Ext^n(Q,F)=0$ for all $n\in\mathbb N$.

\item[(ii)] Let $F\in\PSh(\Cor(\AnSm(\mathbb C)),\Ab)$ be an usu sheaf such that $\iota_{an}^*F=0$.
Then $a_{usu}H^n\underline{\sing}_{\mathbb D^*}F=0$ for all $n\in\mathbb Z$, $n\leq 0$.

\end{itemize}
\end{lem}

\begin{proof}

\noindent(i): As in the proof of \cite[proposition 13.19]{VoeSF},
it follows from the fact that locally for the usual topology, since $X$ and $Z$ are smooth, $X'$ is the blow up of 
$Z\times\mathbb A^{d_X-d_S}$ along $Z\times 0$.

\noindent(ii): As in the proof of \cite[theorem 13.25]{VoeSF},
(ii) follows from (i), corollary \ref{DesAnCor}, and the fact that $H^n\underline{\sing}_{\mathbb D^*}$ is homotopy
invariant, hence strictly homotopy invariant by proposition \ref{HSHDAn}.

\end{proof}

\begin{prop}\label{DesAn} 
Let $\epsilon:X'\to X$ be a proper modification, $X,X'\in\AnSp(\mathbb C)$.
Then, denoting
$\epsilon_{Z*}+l'_*=S(\mathbb Z_{tr}(\epsilon_Z)\oplus\mathbb Z_{tr}(l'))$ and
$\epsilon_{*}+l_*=S(\mathbb Z_{tr}(\epsilon)\oplus\mathbb Z_{tr}(l))$,
the morphism in $PC^-(\An)$
\begin{equation*}
\left[0\rightarrow\underline{\sing}_{\bar{\mathbb D}^*}\mathbb Z_{tr}(E)
\xrightarrow{\epsilon_{Z*}+l'_*}
\underline{\sing}_{\bar{\mathbb D}^*}\mathbb Z_{tr}(Z)\oplus\underline{\sing}_{\bar{\mathbb D}^*}\mathbb Z_{tr}(X')\right]
\xrightarrow{\epsilon_*+l_*}\underline{\sing}_{\bar{\mathbb D}^*}\mathbb Z_{tr}(X) 
\end{equation*}
is an usu local equivalence.

\end{prop}

\begin{proof}

Denote by $K\in\PSh(\Cor_{\mathbb Z}(\AnSm(\mathbb C)),\Ab)$, the cokernel of 
$\mathbb Z_{tr}(l)\oplus\mathbb Z_{tr}(\epsilon):\mathbb Z_{tr}(Z)\oplus\mathbb Z_{tr}(X')\to\mathbb Z_{tr}(X)$. 
The sequence 
\begin{equation*}
0\to\mathbb Z_{tr}(E)\xrightarrow{\mathbb Z_{tr}(\epsilon_Z)\oplus\mathbb Z_{tr}(l')}
\mathbb Z_{tr}(Z)\oplus\mathbb Z_{tr}(X') 
\xrightarrow{\mathbb Z_{tr}(\epsilon)\oplus\mathbb Z_{tr}(l)}\mathbb Z_{tr}(X)\to K\to 0 
\end{equation*}
is  clearly exact in $PC^-(\An)$, 
In particular, the sequence
\begin{equation*}
0\to\underline{\sing}_{\bar{\mathbb D}^*}\mathbb Z_{tr}(E)
\xrightarrow{\epsilon_{Z*}+l'_*}
\underline{\sing}_{\bar{\mathbb D}^*}\mathbb Z_{tr}(Z)\oplus\underline{\sing}_{\bar{\mathbb D}^*}\mathbb Z_{tr}(X')
\xrightarrow{\epsilon_*+l_*}
\underline{\sing}_{\bar{\mathbb D}^*}\mathbb Z_{tr}(X)\to\underline{\sing}_{\bar{\mathbb D}^*}K\to 0 
\end{equation*}
is exact in $PC^-(\An)$. 
Hence we have to show that 
$a_{usu}H^n\underline{\sing}_{\bar{\mathbb D}^*}K=\mathcal H^n\underline{\sing}_{\bar{\mathbb D}^*}a_{usu}K=0$ 
for all $n\in\mathbb Z$, $n\leq 0$, $a_{usu}$ being an exact functor..
By lemma \ref{DesAnlem}(ii), it suffice to show that $\iota_{an}^*a_{usu}K=0$.
But for all $U\in\AnSm(\mathbb C)$, $\alpha\in\mathbb Z_{tr}(X)(U)$ and 
$\alpha'\in\mathcal Z^{d_{X'}}(U\times X')$ the proper transform of $\alpha$ by $I_U\times\epsilon$, 
there exist, by platification and
resolution of singularities, a proper modification $e:U''\to U$, with $U''\in\AnSm(\mathbb C)$ such that
the proper transform $\alpha''\in\mathcal Z^{d_{X'}}(U''\times X')$ of $\alpha'$ by $e\times I_{X'}$ is finite
(and surjective) over $U''$, so that 
$\mathbb Z_{tr}(X)(e)(\alpha):=\alpha\circ e=\mathbb Z_{tr}\epsilon(X')(\alpha'')\in\mathbb Z_{tr}(X)(U'')$
with $\alpha''\in\mathbb Z_{tr}(X')(U'')$. It then follows from corollary \ref{DesAnCor}.

\end{proof}

We will use in the next section the following :

\begin{prop}\label{DusuEq}
\begin{itemize}
\item[(i)] Let $G_1^{\bullet},G_2^{\bullet}\in\PSh_{\mathbb Z}(\Cor^{fs}_{\mathbb Z}(\AnSm(\mathbb C)),C^-(\mathbb Z))$
and $f:G_1^{\bullet}\to G_2^{\bullet}$ a morphism. 
If 
\begin{equation*}
e^{tr}_{an*}S(f):\sing_{\bar{\mathbb D^*}}G_1^{\bullet}\to\sing_{\bar{\mathbb D^*}}G^{\bullet}_2 
\end{equation*}
is a quasi-isomorphism in $C^{-}(\mathbb Z)$,
then $f:G_1^{\bullet}\to G_2^{\bullet}$ is an  $(\mathbb D^1,usu)$ local equivalence. 

\item[(ii)] Let $G_1^{\bullet},G_2^{\bullet}\in\PSh_{\mathbb Z}(\Cor^{fs}_{\mathbb Z}(\AnSm(\mathbb C)),C^-(\mathbb Z))$
and $f:G_1^{\bullet}\to G_2^{\bullet}$ a morphism. 
If 
\begin{itemize}
\item $G_1^{\bullet}$ and $G_2^{\bullet}$ are $\mathbb D^1$ local and
\item $e^{tr}_{an*}f:e^{tr}_{an*}G_1^{\bullet}\to e^{tr}_{an*}G_2^{\bullet}$ 
is a quasi-isomorphism in $C^{-}(\mathbb Z)$,
\end{itemize}
then $f:G_1^{\bullet}\to G_2^{\bullet}$ is an  $(\mathbb D^1,usu)$ local equivalence. 
\end{itemize}
\end{prop}

\begin{proof}

\noindent(i): Consider the commutative diagrams
\begin{equation*}
\xymatrix{G_1\ar[rr]^{S(G^{\bullet}_1)}\ar[d]^{f} & \, & \underline{\sing}_{\bar{\mathbb D^*}}G^{\bullet}_1\ar[d]^{S(f)} \\
G_2^{\bullet}\ar[rr]^{S(G_2^{\bullet})} & \, & \underline{\sing}_{\bar{\mathbb D^*}}G_2} \; \; , \; \;
\xymatrix{
e_{an}^{tr*}e_{an*}^{tr}\underline{\sing}_{\bar{\mathbb D^*}}G^{\bullet}_1\ar[d]^{e_{an}^{tr*}e_{an*}^{tr*}S(f)}
\ar[rrrr]^{\ad(e_{an}^{tr*},e^{tr}_{an*})(\underline{\sing}_{\bar{\mathbb D^*}}G^{\bullet}_1)} & \, & \, & \, &
\underline{\sing}_{\bar{\mathbb D^*}}G^{\bullet}_1\ar[d]^{S(f)} \\
e_{an}^{tr*}e_{an*}^{tr}\underline{\sing}_{\bar{\mathbb D^*}}G^{\bullet}_2
\ar[rrrr]^{\ad(e_{an}^{tr*},e^{tr}_{an*})(\underline{\sing}_{\bar{\mathbb D^*}}G^{\bullet}_2)} & \, & \, & \, &
\underline{\sing}_{\bar{\mathbb D^*}}G^{\bullet}_2}
\end{equation*}
By theorem \ref{AnTr} (iii),
\begin{itemize}
\item $\ad(e_{an}^{tr*},e^{tr}_{an*})(\underline{\sing_{\bar{\mathbb D^*}}G^{\bullet}_1}) :
e_{an}^{tr*}e_{an*}^{tr}\underline{\sing}_{\bar{\mathbb D^*}}G^{\bullet}_1\to\underline{\sing}_{\bar{\mathbb D^*}}G^{\bullet}_1$
\item $\ad(e_{an}^{tr*},e^{tr}_{an*})(\underline{\sing_{\bar{\mathbb D^*}}G^{\bullet}_2}) :
e_{an}^{tr*}e_{an*}^{tr}\underline{\sing}_{\bar{\mathbb D^*}}G^{\bullet}_2\to\underline{\sing}_{\bar{\mathbb D^*}}G^{\bullet}_2$
\end{itemize}
are $(\mathbb D^1,usu)$ local equivalence. 
On the other side 
\begin{equation*}
e^{tr*}_{an}e^{tr}_{an*}S(f):
e_{an}^{tr*}e_{an*}^{tr}\underline{\sing}_{\bar{\mathbb D^*}}G^{\bullet}_1=e^{tr*}_{an}\sing_{\bar{\mathbb D^*}}G_1^{\bullet}
\to e_{an}^{tr*}e_{an*}^{tr}\underline{\sing}_{\bar{\mathbb D^*}}G^{\bullet}_2=e^{tr*}_{an}\sing_{\bar{\mathbb D^*}}G^{\bullet}_2
\end{equation*}
is an $(\mathbb D^1,usu)$ local equivalence (even a quasi-isomorphism),
since $e^{tr}_{an*}S(f):\sing_{\bar{\mathbb D^*}}G_1^{\bullet}\to\sing_{\bar{\mathbb D^*}}G^{\bullet}_2$ 
is a quasi-isomorphism in $C^{-}(\mathbb Z)$ by hypothesis.
Hence, by the second commutative diagram,
\begin{equation*}
S(f):\underline{\sing}_{\bar{\mathbb D^*}}G^{\bullet}_1\to\underline{\sing}_{\bar{\mathbb D^*}}G^{\bullet}_2
\end{equation*}
is an $(\mathbb D^1,usu)$ local equivalence.
Since $S(G_1^{\bullet})$ and $S(G_2^{\bullet})$ are  $(\mathbb D^1,usu)$ local equivalence by theorem \ref{AnS} (ii),
and $S(f)$ is a $(\mathbb D^1,usu)$ local equivalence, $f:G^{\bullet}_1\to G^{\bullet}_2$
is a $(\mathbb D^1,usu)$ local equivalence by the first commutative diagram.

\noindent(ii): Consider again the commutative diagram
\begin{equation*}
\xymatrix{G_1\ar[rr]^{S(G^{\bullet}_1)}\ar[d]^{f} & \, & \underline{\sing}_{\bar{\mathbb D^*}}G^{\bullet}_1\ar[d]^{S(f)} \\
G_2^{\bullet}\ar[rr]^{S(G_2^{\bullet})} & \, & \underline{\sing}_{\bar{\mathbb D^*}}G_2} \; \; , \; \;
\end{equation*}
Since $S(G_1^{\bullet})$ and is a $(\mathbb D^1,usu)$ local equivalence and $G_1^{\bullet}$ is $\mathbb D^1$ local,
$S(G_1^{\bullet})$ is an usu local equivalence.
Hence $e^{tr}_{an*}S(G_1^{\bullet})$ is a quasi isomorphism.
Since $S(G_2^{\bullet})$ and is a $(\mathbb D^1,usu)$ local equivalence and $G_2^{\bullet}$ is $\mathbb D^1$ local,
$S(G_2^{\bullet})$ is an usu local equivalence.
Hence $e^{tr}_{an*}S(G_2^{\bullet})$ is a quasi isomorphism.
On the other side, by hypothesis, $e^{tr}_{an*}f$ is a quasi isomorphism.
Hence, by the above diagram,
$e^{tr}_{an*}S(f):\sing_{\bar{\mathbb D^*}}G_1^{\bullet}\to\sing_{\bar{\mathbb D^*}}G^{\bullet}_2$ 
is a quasi-isomorphism. It follows then by (i) that $f:G_1\to G_2$ is a $(\mathbb D^1,usu)$ local equivalence.
 
\end{proof}

\subsection{Presheaves and transfers on the category of CW complexes}

Let $X,Y\in\Top$. A continous map $f:X\to Y$ is said to be proper if it is universally closed, that is for all $T\in\Top$, 
$f\times I_T:X\times T\to Y\times T$ is closed. 
A continous map $f:X\to Y$ is proper if and only if it is closed and, 
for all $y\in Y$, $X_y=f^{-1}(y)$ is quasi-compact (c.f.\cite{Bourbaki}).
A continous map $f:X\to Y$ is said to be finite if it is proper and, 
for all $y\in Y$, $X_y=f^{-1}(y)$ is a finite set.
A continous map $f:X\to Y$ is said to be dominant if $f(X)\subset Y$ has non empty interior. 
For $X\in\Top$, there exist a one point compactification $\bar{X}\in\Top$ of $X$, that is
$\bar{X}$ is quasi-compact and there is an open embedding $X\hookrightarrow\bar{X}$.
Moreover $\bar{X}$ is compact (i.e. quasi-compact and Hausdorf) if and only if 
$X$ is locally compact and Hausdorf.

Let $X,Y,Z\in\Top$ and denote $p_X:X\times Y\to X$ and $p_Y:Y\times Z\to Z$, 
$p_{XZ}:X\times Y\times Z\to X\times Z$ the projections. 
Assume that $Y$ is Hausdorf (equivalently the diagonal $\Delta_Y\subset Y\times Y$ is a closed subset). 
Let $\Gamma_1\subset X\times Y$, $\Gamma_2\subset Y\times Z$, closed subsets.
Hence, since $Y$ is Hausdorf, $\Gamma_1\times_Y\Gamma_2\subset X\times Y\times Z$ is a closed subset.  
Now assume that $p_{X|\Gamma_1}:\Gamma_1\to X$ and $p_{Y|\Gamma_2}:\Gamma_2\to Y$ are finite and surjective. 
Consider the commutative diagram 
\begin{equation}\label{compcorTop}
\xymatrix{\Gamma_1\times_Y\Gamma_2\ar[rrr]^{p_{XZ}|\Gamma_1\times_Y\Gamma_2}\ar[d]_{p_{\Gamma_1}} & \, & \, &
 p_{XZ}(\Gamma_1\times_Y\Gamma_2)\ar[d]^{p_{X|p_{XZ}(\Gamma_1\times_Y\Gamma_2)}} \\
\Gamma_1\ar[rrr]^{p_{X|\Gamma_1}} & \, & \, & X} 
\end{equation}
The projection $p_{\Gamma_1}:\Gamma_1\times_Y\Gamma_2\to\Gamma_1$ is finite and surjective 
(since $p_{Y|\Gamma_2}:\Gamma_2\to Y$ is finite and surjective by hypothesis and $Y$ is Hausdorf). 
The projection $p_{X|\Gamma_1}:\Gamma_1\to X$ is finite and surjective by hypothesis.
Hence, by the diagram (\ref{compcorTop}), the composition
\begin{equation*}
\Gamma_1\times_Y\Gamma_2\xrightarrow{p_{XZ}|\Gamma_1\times_Y\Gamma_2}
p_{XZ}(\Gamma_1\times_Y\Gamma_2)\xrightarrow{p_{X|p_{XZ}(\Gamma_1\times_Y\Gamma_2)}} X
\end{equation*}
is finite and surjective. Thus,
$p_{X|p_{XZ}(\Gamma_1\times_Y\Gamma_2)}:p_{XZ}(\Gamma_1\times_Y\Gamma_2)\to X$
is finite and surjective, that is
\begin{equation}\label{CorTop}
\Gamma_2\circ\Gamma_1:=p_{XZ|(X\times Y)\times_Y\Gamma_2}(\Gamma_1\times_Y\Gamma_2)
\end{equation}
is a closed subset of $X\times Z$ which is finite and surjective on $X$.

We consider now the full subcategory $\CW\subset\Top$ consisting of CW complexes. 
Recall that CW complexes are Hausdorf, locally contractible and locally compact.
By a CW subcomplex of $X\in\CW$, we mean a topological embbeding $Z\hookrightarrow X$ with $Z$ a CW complex. 
By a closed CW subcomplex of $X\in\CW$, we mean a topological closed embbeding $Z\hookrightarrow X$ with $Z$ a CW complex,
that is the image of the embedding is a closed subset of $X$. 
Let $X,S\in\CW$, $S$ connected, and $h:X\to S$ a finite and surjective morphism in $\CW$.
We say that $X/S=(X,h)\in\CW/S$ is reducible if $X=X_1\cup X_2$, with $X_1,X_2$ closed CW subcomplexes
finite and surjective over $S$ and $X_1,X_2\neq X$.
A pair $Y/S=(Y,h')\in\CW/S$ with $Y\in\CW$ and $h':Y\to S$ a finite and surjective morphism
is called irreducible if it is not reducible. In particular $Y$ is connected.
\begin{itemize}
\item For $X\in\CW$, $\Lambda$ a commutative ring and $p\in\mathbb N$, 
we denote by $\mathcal Z_p(X,\Lambda)$ the free $\Lambda$ module generated 
by the closed CW subcomplex of $X$ of dimension $p$ and by 
$\mathcal Z^p(X,\Lambda)=\mathcal Z_{d_X-p}(X,\Lambda)$ 
the free $\Lambda$ module generated by the closed CW subcomplex of $X$ of codimension $p$.

\item For $X,Y\in\CW$, $X$ connected, and $\Lambda$ a commutative ring, we define :
$\mathcal{Z}^{fs/X}(X\times Y,\Lambda)\subset\mathcal{Z}_{d_X}(X\times Y,\Lambda)$ 
the free $\Lambda$ module generated by the closed CW subcomplexes of $X\times Y$ finite and surjective over $X$
which are irreducible. 
Note that if $Z\subset X\times Y$ is a closed CW subcomplex finite and surjective over $X$ which is irreducible,
then if $Z=Z_1\cup Z_2$ with $Z_1,Z_2\subset Z$ closed CW subcomplex finite and surjective over $X$, 
then $Z_1=Z_2=Z$.

\item For $X,Y\in\CW$, and $\Lambda$ a commutative ring, we define :
$\mathcal{Z}^{fs/X}(X\times Y,\Lambda):=\oplus_i\mathcal{Z}^{fs/X_i}(X_i\times Y,\Lambda)$ 
where $X=\sqcup_i X_i$, with $X_i$ the connected components of $X$.
\end{itemize}

For $Y\in\CW$ and $j:V\hookrightarrow Y$ an open embedding, 
we denote by 
\begin{equation}\label{jCW}
j^*:\mathcal Z^p(Y,\Lambda)\hookrightarrow\mathcal Z^p(V,\Lambda) \; ; \; Z \mapsto Z\cap V 
\end{equation}

Let $X,Y,Z\in\CW$ and denote $p_X:X\times Y\to X$ and $p_Y:Y\times Z\to Z$, 
$p_{XZ}:X\times Y\times Z\to X\times Z$ the projections. 
Let $\Gamma_1\subset X\times Y$ be a closed CW subcomplex finite and surjective over $X$ which is irreducible. 
Let $\Gamma_2\subset Y\times Z$ be a closed CW subcomplex finite and surjective over $Y$ which is irreducible.
Hence, since $Y$ is Hausdorf,
$\Gamma_1\times_Y\Gamma_2\subset X\times Y\times Z$ is a closed CW subcomplex finite and surjective over $X$,
in particular all its irreducible components are finite and surjective over $X$. 
Thus, by diagram \ref{compcorTop}
\begin{equation}\label{CorTopCW}
\Gamma_2\circ\Gamma_1:=p_{XZ|(X\times Y)\times_Y\Gamma_2}(\Gamma_1\times_Y\Gamma_2)
\end{equation}
is a closed CW subcomplex of $X\times Z$ finite and surjective over $X$,
in particular all its irreducible components are finite and surjective over $X$ and
so define an element of $\mathcal Z^{fs/X}(X\times Z)$ . 

\begin{defi}
We define $\Cor^{fs}_{\Lambda}(\CW)$ to be the category whose objects are CW complexes and whose space of morphisms between
$X,Y\in\CW$ is the free $\Lambda$ module $\mathcal{Z}^{fs/X}(X\times Y,\Lambda)$.
The composition is the one given by (\ref{CorTopCW}).
\end{defi}

We have 
\begin{itemize}
\item the additive embedding of categories $\Tr:\mathbb Z(\CW)\hookrightarrow\Cor^{fs}_{\mathbb Z}(\CW)$ 
which gives the corresponding morphism of sites $\Tr:\Cor^{fs}_{\mathbb Z}(\CW)\to \mathbb Z(\CW)$,
\item the inclusion functor
$e_{cw}:\left\{pt\right\}\hookrightarrow\CW$,
which gives the corresponding morphism of sites
$e_{cw}:\CW\to\left\{pt\right\}$,
\item the inclusion functor $e^{tr}_{cw}:=\Tr\circ e_{cw}:\left\{pt\right\}\hookrightarrow\Cor^{fs}_{\mathbb Z}(\CW)$
which gives the corresponding morphism of sites
$e^{tr}_{cw}:=\Tr\circ e_{cw}:\Cor^{fs}_{\mathbb Z}(\CW)\to\left\{pt\right\}$.
\end{itemize}

We consider the following two big categories :
\begin{itemize}
\item $\PSh(\CW,C^-(\mathbb Z))=\PSh_{\mathbb Z}(\mathbb Z(\CW),C^-(\mathbb Z)))$,
the category of bounded above complexes of presheaves on $\CW$, 
or equivalently additive presheaves on $\mathbb Z(\CW)$, sometimes, we will write for short 
$P^-(CW)=\PSh(CW,C^-(\mathbb Z))$,
\item $\PSh(\Cor^{fs}_{\mathbb Z}(\CW),C^-(\mathbb Z))$, 
the category of bounded above complexes of additive presheaves on $\Cor^{fs}_{\mathbb Z}(\CW)$,
sometimes, we will write for short 
$PC^-(CW)=\PSh(\Cor^{fs}_{\mathbb Z}(CW),C^-(\mathbb Z))$,
\end{itemize}
and the adjonctions :
\begin{itemize}
\item $(\Tr^*,\Tr_*):\PSh(\CW,C^-(\mathbb Z))\leftrightarrows\PSh(\Cor^{fs}_{\mathbb Z}(\CW),C^-(\mathbb Z))$, 
\item $(e_{cw}^*,e_{cw*}):\PSh(\CW,C^-(\mathbb Z))\leftrightarrows C^-(\mathbb Z)$,
\item $(e_{cw}^{tr*},e_{cw*}^{tr}):\PSh(\Cor^{fs}_{\mathbb Z}(\CW),C^-(\mathbb Z))\leftrightarrows C^-(\mathbb Z)$,
\end{itemize}
given by $\Tr:\Cor^{fs}_{\Lambda}(\CW)\to \Lambda(\CW)$, 
$e_{cw}\CW\to\left\{pt\right\}$ and $e^{tr}_{cw}:\Cor^{fs}_{\mathbb Z}(\CW)\to\left\{pt\right\}$.

For $X\in\CW$, we denote by 
\begin{equation}
\mathbb Z(X)\in\PSh(\CW,C^-(\mathbb Z)) \mbox{\;\;},\mbox{\;\;} 
\mathbb Z_{tr}(X)\in\PSh_{\mathbb Z}(\Cor^{fs}_{\mathbb Z}(\CW),C^-(\mathbb Z)) 
\end{equation}
the presheaves represented by $X$. They are usu sheaves.
We denote by
$a_{usu}:\PSh_{\mathbb Z}(CW,\Ab)\to\Sh_{\mathbb Z,usu}(CW,\Ab)$ 
the usu sheaftification functor.
We consider 
\begin{itemize}
\item the usual monoidal strutcure on $\PSh(\CW,C^-(\mathbb Z))$ and the associated internal $\Hom$
given by, for $F^{\bullet},G^{\bullet}\in\PSh(\CW,C^-(\mathbb Z))$ and $Y\in\CW$, 
\begin{equation}
F^{\bullet}\otimes G^{\bullet}(X):X\mapsto F^{\bullet}(X)\otimes_{\mathbb Z} G^{\bullet}(X) \; , \;
\underline{\Hom}(\mathbb Z(Y),F^{\bullet}):X\mapsto F^{\bullet}(X\times Y),
\end{equation}
\item the unique monoidal strutcure on $\PSh_{\mathbb Z}(\Cor^{fs}_{\mathbb Z}(\CW),C^-(\mathbb Z))$
such that,for $X,Y\in\CW$, $\mathbb Z_{tr}(X)\otimes\mathbb Z_{tr}(Y):=\mathbb Z_{tr}(X\times Y)$
and wich commute with colimites.
It has an internal $\Hom$ which is given, for $X,Y\in\CW$ and
$F^{\bullet}\in\PSh_{\mathbb Z}(\Cor^{fs}_{\mathbb Z}(\CW),C^-(\mathbb Z))$,
$\underline{\Hom}(\mathbb Z_{tr}(Y),F^{\bullet}:X\mapsto F^{\bullet}(X\times Y)$
\end{itemize}
Together with these monoidal structure, the functor
\begin{equation*}
\Tr^*:\PSh(\CW,C^-(\mathbb Z))\to \PSh_{\mathbb Z}(\Cor^{fs}_{\mathbb Z}(\CW),C^-(\mathbb Z))
\end{equation*}
is monoidal.

\begin{defi}\label{CWTrusu}
\begin{itemize}
\item[(i)] We say that a morphism $\phi:G_1^{\bullet}\to G_2^{\bullet}$ in $\PSh(\CW,C^-(\mathbb Z))$ is an
usu local equivalence if $\phi_*:a_{usu}H^k(G_1^{\bullet})\to a_{usu}H^k(G_2^{\bullet})$ is an isomorphism for all $k\in\mathbb Z$.
The projective usual topology model structure on $\PSh(\CW,C^-(\mathbb Z))$ 
is the left Bousfield localization of the projective model structure $\mathcal{M}_P(\PSh_{\mathbb Z}(\CW,C^-(\mathbb Z)))$
with respect to the usu local equivalence.

\item[(ii)] We say that a morphism $\phi:G_1^{\bullet}\to G_2^{\bullet}$ in $\PSh_{\mathbb Z}(\Cor^{fs}_{\mathbb Z}(\CW),C^-(\mathbb Z))$ is an
usu local equivalence if and only if its restriction to $\CW$ $\Tr_*\phi:\Tr_*G_1^{\bullet}\to\Tr_*G_2^{\bullet}$ is an usu local equivalence.
The projective usual topology model structure on $\PSh_{\mathbb Z}(\Cor^{fs}_{\mathbb Z}(\CW),C^-(\mathbb Z))$ 
is the left Bousfield localization of the projective model structure  
$\mathcal{M}_P(\PSh_{\mathbb Z}(\Cor^{fs}_{\mathbb Z}(\CW),C^-(\mathbb Z)))$
with respect to the usu local equivalence.
\end{itemize}
\end{defi}

\begin{defi}\label{CWmodstr}
\begin{itemize}
\item[(i)] The projective $(\mathbb I^1,usu)$ model structure on $\PSh(\CW,C^-(\mathbb Z))$ is the left Bousfield localization 
of the projective usual topology model structure (c.f.\ref{CWTrusu}(i)) with respect to the class of maps 
$\left\{\mathbb Z(X\times\mathbb I^1)[n]\to\mathbb Z(X)[n], \; X\in\CW,n\in\mathbb Z\right\}$. 

\item[(ii)] The projective $(\mathbb I^1,usu)$ model structure on $\PSh_{\mathbb Z}(\Cor^{fs}_{\mathbb Z}(\CW),C^-(\mathbb Z))$ is the left Bousfield localization of the projective usual topology model structure (c.f.\ref{CWTrusu}(ii)) with respect to the class of maps
$\left\{\mathbb Z_{tr}(X\times\mathbb I^1)\to\mathbb Z(X), \; X\in\CW,n\in\mathbb Z\right\}$. 
\end{itemize}
\end{defi}

\begin{defi} 
\begin{itemize}
\item[(i)] We define $\CwDM^-(\mathbb Z):=\Ho_{\mathbb I^1,usu}(\PSh_{\mathbb Z}(\Cor^{fs}_{\mathbb Z}(\AnSm(\mathbb C)),C^-(\mathbb Z)))$, 
the derived category of motives of CW complexes, it is 
the homotopy category of $\PSh_{\mathbb Z}(\Cor^{fs}_{\mathbb Z}(\CW),C^-(\mathbb Z))$ 
with respect to the projective $(\mathbb I^1,usu)$ model structure (\ref{CWmodstr}(ii)). 
We denote by
\begin{equation*}
D^{tr}(\mathbb I^1,usu):\PSh_{\mathbb Z}(\Cor^{fs}_{\mathbb Z}(\CW),C^-(\mathbb Z))\to\CwDM^{-}(\mathbb Z) \,, \,
D^{tr}(\mathbb I^1,usu)(F^{\bullet})=F^{\bullet}
\end{equation*}
the canonical localisation functor.

\item[(ii)] We denote by the same way 
$\CwDA^-(\mathbb Z):=\Ho_{\mathbb I^1,usu}(\PSh(\CW,C^-(\mathbb Z)))$ (\ref{CWmodstr}(i)) and
\begin{equation}
D(\mathbb I^1,usu):\PSh(\CW,C^-(\mathbb Z))\to\CwDA^{-}(\mathbb Z) \,, \, 
D(\mathbb I^1,usu)(F^{\bullet})=F^{\bullet}
\end{equation}
the canonical localisation functor.
\end{itemize}
\end{defi}

We recall for convenience to the reader the definition of $\mathbb I^1$ homotopy :
\begin{itemize}
\item Let $X,Y\in\Top$. We say that two maps $f_0:X\to Y$, $f_1:X\to Y$ are $\mathbb I^1$ homotopic,
if there exist $h:X\times\mathbb I^1\to Y$ such that $f_0=h\circ(I_X\times i_0)$ and $f_1=h\circ(I_X\times i_1)$, with
$(I_X\times i_0):X\times\left\{0\right\}\hookrightarrow X\times\mathbb I^1$ 
and $(I_X\times i_1):X\times\left\{1\right\}\hookrightarrow X\times\mathbb I^1$ 
the inclusions.

\item Let $X,Y\in\Top$. We say that $X$ is $\mathbb I^1$ homotopy equivalent to $Y$ if there
exist two maps $f:X\to Y$, $g:Y\to X$ such that 
$g\circ f$ is $\mathbb I^1$ homotopic to $I_X$ and $f\circ g$ is $\mathbb I^1$ homotopic to $I_Y$.

\item Let $F^{\bullet},G^{\bullet}\in\PSh(\CW,C^-(\mathbb Z))$. 
We say that two maps $\phi_0:F^{\bullet}\to G^{\bullet}$ and $\phi_1:F^{\bullet}\to G^{\bullet}$ 
are $\mathbb I^1$ homotopic if there exist
$\tilde{\phi}:F^{\bullet}\to\underline{\Hom}(\mathbb Z(\mathbb I^1),G^{\bullet})$ such that
$\phi_0=\underline{G}^{\bullet}(i_0)\circ\tilde{\phi}$ and $\phi_1=\underline{G}^{\bullet}(i_1)\circ\tilde{\phi}$, where,
\begin{itemize}
\item $\underline{G}^{\bullet}(i_0):\underline{\Hom}(\mathbb Z(\mathbb I^1),G^{\bullet})\to G^{\bullet}$ is
induced by $i_0:\left\{\pt\right\}\to\mathbb I^1$, that is, for $X\in\CW$, 
$\underline{G}^{\bullet}(i_0)(X)=G^{\bullet}(I_X\times i_0):G^{\bullet}(X\times\mathbb I^1)\to G^{\bullet}(X)$,  
\item $\underline{G}^{\bullet}(i_1):\underline{\Hom}(\mathbb Z(\mathbb I^1),G^{\bullet})\to G^{\bullet}$ is
induced by $i_1:\left\{\pt\right\}\to\mathbb I^1$, that is, for $X\in\CW$, 
$\underline{G}^{\bullet}(i_1)(X)=G^{\bullet}(I_X\times i_1):G^{\bullet}(X\times\mathbb I^1)\to G^{\bullet}(X).$  
\end{itemize}

\item Let $F^{\bullet},G^{\bullet}\in\PSh(\CW,C^-(\mathbb Z))$. 
We say that $F^{\bullet}$ is $\mathbb I^1$ homotopy equivalent to $G^{\bullet}$ if there
exist two maps $\phi:F^{\bullet}\to G^{\bullet}$ and $\psi:F^{\bullet}\to G^{\bullet}$ such that 
$\psi\circ\phi$ is $\mathbb I^1$ homotopic to $I_{F^{\bullet}}$ 
and $\phi\circ \psi$ is $\mathbb I^1$ homotopic to $I_{G^{\bullet}}$.
\end{itemize} 

We have the following easy lemma

\begin{lem}\label{I1hptlem}
Let $X,Y\in\CW$ and $f_0:X\to Y$, $f_1:X\to Y$ two morphisms.
If $f_0$ and $f_1$ are $\mathbb I^1$ homotopic, then
\begin{itemize}
\item $\mathbb Z(f_0):\mathbb Z(X)\to\mathbb Z(Y)$ and $\mathbb Z(f_1):\mathbb Z(X)\to\mathbb Z(Y)$ 
are $\mathbb I^1$ homotopic in $\PSh(\CW,C^-(\mathbb Z))$,
\item $\Tr_*\mathbb Z_{tr}(f_0):\Tr_*\mathbb Z_{tr}(X)\to\Tr_*\mathbb Z_{tr}(Y)$ 
and $\Tr_*\mathbb Z_{tr}(f_1):\Tr_*\mathbb Z_{tr}(X)\to\Tr_*\mathbb Z_{tr}(Y)$ 
are $\mathbb I^1$ homotopic in $\PSh(\CW,C^-(\mathbb Z))$.
\end{itemize}
\end{lem}

\begin{proof}
Let $h:X\times\mathbb I^1\to Y$ such that $f_0=h\circ(I_X\times i_0)$ and $f_1=h\circ(I_X\times i_1)$, with
$(I_X\times i_0):X\times\left\{0\right\}\hookrightarrow X\times\mathbb I^1$ 
and $(I_X\times i_1):X\times\left\{1\right\}\hookrightarrow X\times\mathbb I^1$ 
the inclusions.
Then 
\begin{itemize}
\item $\phi(h):\mathbb Z(X)\to\underline{\Hom}(\mathbb Z(\mathbb I^1),\mathbb Z(Y))$, given by for $Z\in\CW$
\begin{eqnarray*}
\phi(h)(Z):\alpha\in\Hom_{\CW}(Z,X)
\mapsto(\alpha\times I_{\mathbb I^1})\in\Hom_{\CW}(Z\times\mathbb I^1,X\times\mathbb I^1) \\
\mapsto\mathbb Z(h)(Z\times\mathbb I^1)(\alpha\times I_{\mathbb I^1})=h\circ(\alpha\times I_{\mathbb I^1})
\in\Hom_{\CW}(Z\times\mathbb I^1,Y)
\end{eqnarray*}
satisfy 
$\underline{\mathbb Z(Y)}(i_0)\circ\phi(h)=\mathbb Z(f_0)$ and $\underline{\mathbb Z(Y)}(i_1)\circ\phi(h)=\mathbb Z(f_1)$.

\item $\phi(h):\mathbb Z_{tr}(X)\to\underline{\Hom}(\mathbb Z(\mathbb I^1),\mathbb Z_{tr}(Y))$, given by for $Z\in\CW$
\begin{eqnarray*}
\phi(h)(Z):\alpha\in\Hom_{\Cor^{fs}_{\mathbb Z}(\CW)}(Z,X)
\mapsto(\alpha\times I_{\mathbb I^1})\in\Hom_{\Cor^{fs}_{\mathbb Z}(\CW)}(Z\times\mathbb I^1,X\times\mathbb I^1) \\
\mapsto\mathbb Z_{tr}(h)(Z\times\mathbb I^1)(\alpha\times I_{\mathbb I^1})=h\circ(\alpha\times I_{\mathbb I^1})
\in\Hom_{\Cor^{fs}_{\mathbb Z}(\CW)}(Z\times\mathbb I^1,Y)
\end{eqnarray*}
satisfy 
$\underline{\mathbb Z_{tr}(Y)}(i_0)\circ\phi(h)=\mathbb Z_{tr}(f_0)$ and 
$\underline{\mathbb Z_{tr}(Y)}(i_1)\circ\phi(h)=\mathbb Z_{tr}(f_1)$.
\end{itemize}
\end{proof}

Recall $\bold{\Delta}$ denote the simplicial category
Let 
\begin{itemize}
\item $p_{\bold\Delta}:\CW\to\bold{\Delta}\times\CW$ 
be the morphism of site given by the projection functor 
$(X,i)\in\bold{\Delta}\times\CW\mapsto p_{\bold\Delta}((X,i))=X\in\CW$,
\item $p_{\bold\Delta}:\Cor_{\mathbb Z}(\CW)\to\bold{\Delta}\times\Cor_{\mathbb Z}(\CW)$ 
be the morphism of site given by the projection functor 
$(X,i)\in\bold{\Delta}\times\CW\mapsto p_{\bold\Delta}((X,i))=X\in\CW$
\end{itemize}
We have the adjonctions
\begin{itemize}
\item $(p_{\bold\Delta}^*,p_{\bold\Delta*}):\PSh(\bold{\Delta}\times\CW,C^-(\mathbb Z))
\leftrightarrows\PSh(\CW,C^-(\mathbb Z))$    
\item $(p_{\bold\Delta}^*,p_{\bold\Delta}):\PSh(\bold{\Delta}\times\Cor_{\mathbb Z}(\CW),C^-(\mathbb Z))
\leftrightarrows\PSh(\Cor_{\mathbb Z}(\CW),C^-(\mathbb Z))$.   
\end{itemize}

In the following lemma point (i) is a generalization of point (i) to usu hypercover.
We will only use point (i) in this paper in the proof of proposition \ref{CWadTr}(ii). 

\begin{lem}\label{HRCW}
\begin{itemize}
\item[(i)]: Let $X\in\CW$ and $X=\cup_{i\in J}U_i$ a covering by open subsets $U_i\subset X$, $J$ being a countable set.
We denote by $j_i:U_i\hookrightarrow X$ the open embedding.
For $I\subset J$ a finite subset, let $U_I=\cap_{i\in I} U_i$.
Then, 
\begin{itemize}
\item $\left[\cdots\rightarrow\oplus_{\card I=r}\mathbb Z(U_I)
\rightarrow\cdots\rightarrow\oplus_{i\in J}\mathbb Z(U_i)\right]\xrightarrow{\oplus_{i\in J}\mathbb Z(j_i)}\mathbb Z(X)$
is an usu local equivalence in $P^-(CW)$.
\item $\left[\cdots\rightarrow\oplus_{\card I=r}\mathbb Z_{tr}(U_I)
\rightarrow\cdots\rightarrow\oplus_{i\in J}\mathbb Z_{tr}(U_i)\right]\xrightarrow{\oplus_{i\in J}\mathbb Z_{tr}(j_i)}\mathbb Z_{tr}(X)$
is an usu local equivalence in $PC^-(CW)$.
\end{itemize}
\item[(ii)] Let $X\in\CW$ and $j_{\bullet}:U_{\bullet}\to X$ an hypercovering of $X$ by open subsets 
(i.e. for all $n\in\bold{\Delta}$,  $U_{n}$ is an open subset and $j_n:U_n\hookrightarrow X$ is the open embedding),
and $(U_{\bullet},\bullet):\cdots\to(U_{n},n)\to\cdots$ the associated complex in $\bold{\Delta}\times\CW$
Then,  
\begin{itemize}
\item $j:Lp_{\bold{\Delta}}^*(\mathbb Z(U_{\bullet},\bullet))\to\mathbb Z(X)$ is an isomorphism 
in $\Ho_{usu}(\PSh(\CW,C^-(\mathbb Z)))$
\item $j:Lp_{\bold{\Delta}}^*(\mathbb Z_{tr}(U_{\bullet},\bullet))\to\mathbb Z_{tr}(X)$ is an isomorphism 
in $\Ho_{usu}(\PSh_{\mathbb Z}(\Cor_{\mathbb Z}(\CW),C^-(\mathbb Z)))$
\end{itemize}
\end{itemize}
\end{lem}

\begin{proof}

\noindent(i): The first point follows from the following two facts :
\begin{itemize}
\item $[\cdots\rightarrow\oplus_{\card I=r}\mathbb Z(U_I)\rightarrow\cdots\rightarrow\oplus_{i\in J}\mathbb Z(U_i)]$
is  clearly exact in $P^-(CW)$. 
\item the exactness of 
$[\oplus_{i\in J}\mathbb Z(U_i)\xrightarrow{\oplus_{i\in J}\mathbb Z(j_i)}\mathbb Z(X)\to 0]$
in $\Ho_{usu}(\PSh(\CW,C^-(\mathbb Z)))$,
follows from the fact that for $Y\in\CW$, $y\in Y$ and $f:Y\to X$ a morphism, 
there exists an open subset $V(y)\subset Y$ containing $y$ such that $f(V(y))\subset U_i$, with
$U_i$ containing $f(y)$.  
\end{itemize}
Similary, the second point follows from the following two facts : 
\begin{itemize}
\item $[\cdots\rightarrow\oplus_{\card I=r}\mathbb Z_{tr}(U_I)\rightarrow\cdots\rightarrow\oplus_{i\in J}\mathbb Z_{tr}(U_i)]$
is clearly exact in $\PSh_{\mathbb Z}(\Cor^{fs}_{\mathbb Z}(\CW),C^-(\mathbb Z)))$. 
\item the exactness of 
$[\oplus_{i\in J}\mathbb Z_{tr}(U_i)\xrightarrow{\oplus_{i\in J}\mathbb Z_{tr}(j_i)}\mathbb Z_{tr}(X)\to 0]$
in $\Ho_{usu}(\PSh_{\mathbb Z}(\Cor^{fs}_{\mathbb Z}(\CW),C^-(\mathbb Z)))$,
follows from the fact that for $Y\in\CW$, $y\in Y$ and $\alpha\in\mathbb Z_{tr}(X)(Y)$ irreducible, 
there exists an open subset $V(y)\subset Y$ containing $y$ such that $\alpha_{|V(y)}\in\mathbb Z_{tr}(U_i)(V(y))$, with
$U_i$ containing $\alpha_*y$.  
\end{itemize}

\noindent(ii): Similar to \cite{AyoubB} Proposition 1.4 Etape 1.

\end{proof}

The following proposition is to use point (i) instead of point (ii) in the proof
of proposition \ref{CWadTr}(ii). Recall $\CS\subset\CW$ is the full subcategory of $\Delta$ complexes.

\begin{prop}\label{Icov}
\begin{itemize}
\item[(i)] Let $X\in\CW$. There exist $X'\in\CS$ homotopy equivalent to $X$,
that is there exist $g:X'\to X$ and $h:X\to X'$ such that $h\circ g$ is homotopic to $I_{X'}$
and $g\circ h$ is homotopic to $I_X$.

\item[(ii)] Let $X\in\CS$. There exist a countable open covering $X=\cup_{i\in J} U_i$ such that
for all finite subset $I\subset J$, $U_I:=\cap_{i\in I} U_i=\emptyset$ or $U_I$ is contractible.
\end{itemize}
\end{prop}

\begin{proof}

\noindent(i): See \cite[theorem 2C5]{Hatcher}.

\noindent(ii): Take an open star of each vertices of $X$. As $X$ is a countable union of compact set,
we can extract a countable subcovering of this open covering.

\end{proof}

We denote, for all $n\in\mathbb N$, $\mathbb I^n:=[0,1]^n\in\CW$.
It gives, together with the face maps 
$\partial^n_{\epsilon,i}:\mathbb I^{n-1}\hookrightarrow\mathbb I^n$, $\epsilon=0,1$, $i=1,\cdots n$
a cubical object of $\CW$, denoted $\mathbb I^*$.
In particular, for $n=1$, we have the canonical maps  
$i_0:\left\{0\right\}\hookrightarrow\mathbb I^1$ 
$i_1:\left\{1\right\}\hookrightarrow\mathbb I^1$, 
and the terminal map $a=a_{\mathbb I^1}:\mathbb I^1\to\left\{\pt\right\}$.

For $F^{\bullet}\in\PSh(\CW,\Ab)$ and $X\in\CW$, 
we have the complex $F(X\times\mathbb I^*)$ associated to the cubical object 
$F(X\times\mathbb I^*)$ in the category of abelian groups, whose differential maps are
\begin{equation*}
\partial^n_{\mathbb I}=\sum_{i=1}^n(-1)^i(F(I_X\times\partial^n_{0,i})-F(I_X\times\partial_{1,i}):
F(X\times\mathbb I^n)\to F(X\times\mathbb I^{n-1})
\end{equation*}
There is (\cite{Serre}) a canonical morphism $L:\mathbb I^*\to\Delta^*$ of complexes of $\mathbb Z(\CW)$.
This gives for $F\in\PSh(\CW,\Ab)$ and $X\in\CW$, the morphism of complexes of abelian groups
\begin{equation*}
F^{\bullet}(L\times I_X):F(\Delta^*\times X)\to F(\mathbb I^*\times X)
\end{equation*}
with $L\times I_X:\mathbb I^*\times X\to\Delta^*\times X$ 
the corresponding morphism of complexes of $\mathbb Z(\CW)$.
We have then following :

\begin{prop}\label{SerreP}\cite{Serre}
For $X\in\CW$, $\mathbb Z(X)(L):\mathbb Z\Hom(\Delta^*,X)\to\sing_{\mathbb I^*}\mathbb Z(X)$
is a quasi-isomorphism of complexes of abelian groups.
\end{prop}

We now introduce an explicit localization functor for the $(\mathbb I^1,usu)$ model structure.

\begin{itemize}
\item If $F^{\bullet}\in\PSh(\CW,C^-(\mathbb Z))$ , 
\begin{equation}
\underline{\sing}_{\mathbb I^*}F^{\bullet}:=\underline\Hom(\mathbb Z(\mathbb I^*),F^{\bullet})\in\PSh(\CW,C^-(\mathbb Z))
\end{equation}
is the total complex of presheaves associated to the bicomplex of presheaves $X\mapsto F^{\bullet}(\mathbb I^*\times X)$,
and $\sing_{\mathbb I^*}F^{\bullet}:=e_{cw*}\underline{\sing}_{\mathbb I^*}F^{\bullet}=F^{\bullet}(\mathbb I^*)\in C^-(\mathbb Z)$.
We denote by
$S(F^{\bullet}):F^{\bullet}\to\underline{\sing}_{\mathbb I^*}F^{\bullet}$ 
\begin{equation}\label{SFCW}
S(F^{\bullet}):
\xymatrix{\cdots\ar[r] & 0\ar[r]\ar[d] & 0\ar[r]\ar[d] & F^{\bullet}\ar[d]^{I}\ar[r] & 0\ar[r]\ar[d] & \cdots \\
\cdots\ar[r] & \underline{\sing}_{\mathbb I^2}F^{\bullet}\ar[r] & \underline{\sing}_{\mathbb I^1}F^{\bullet}\ar[r] 
& F^{\bullet}\ar[r] & 0\ar[r] & \cdots} 
\end{equation}
the inclusion morphism in $\PSh(\CW),C^-(\mathbb Z))$. 
For $f:F_1^{\bullet}\to F_2^{\bullet}$ a morphism of $\PSh(\CW,C^-(\mathbb Z))$, we denote by 
$S(f):\underline{\sing}_{\mathbb I^*}F_1^{\bullet}\to\underline{\sing}_{\mathbb I^*}F_2^{\bullet}$ 
the morphism of $P^-(CW)$ given by, for $X\in\CW$,
\begin{equation}\label{SfCW}
S(f)(X):\xymatrix{
\cdots\ar[r] & F^{\bullet}_1(\mathbb I^2\times X)\ar[r]\ar[d]^{f(\mathbb I^2\times X)} &
F^{\bullet}_1(\mathbb I^1\times X)\ar[r]\ar[d]^{f(\mathbb I^1\times X)} & F_1^{\bullet}(X)\ar[d]^{f(X)}\ar[r] & 0\ar[r]\ar[d] & \cdots \\
\cdots\ar[r] & F^{\bullet}_2(\mathbb I^2\times X)\ar[r] & F^{\bullet}_2(\mathbb I^1\times X)\ar[r] 
& F_2^{\bullet}(X)\ar[r] & 0\ar[r] & \cdots.}  
\end{equation}

\item If $F^{\bullet}\in\PSh_{\mathbb Z}(\Cor^{fs}_{\mathbb Z}(\CW),C^-(\mathbb Z))$, 
\begin{equation}
\underline{\sing}_{\mathbb I^*}F^{\bullet}:=\underline\Hom(\mathbb Z_{tr}(\mathbb I^*),F^{\bullet})
\in\PSh_{\mathbb Z}(\Cor^{fs}_{\mathbb Z}(\CW),C^-(\mathbb Z))
\end{equation}
is the total complex of presheaves associated to the bicomplex of presheaves  $X\mapsto F^{\bullet}(\mathbb I^*\times X)$,
and $\sing_{\mathbb I^*}F^{\bullet}:=e^{tr}_{cw*}\underline{\sing}_{\mathbb I^*}F^{\bullet}=F^{\bullet}(\mathbb I^*)\in C^-(\mathbb Z)$.
We have the inclusion morphism (\ref{SFCW})
\begin{equation*}
S(\Tr_*F^{\bullet}):\Tr_*F^{\bullet}\to
\underline{\sing}_{\mathbb I^*}\Tr_*F^{\bullet}=\Tr_*\underline{\sing}_{\mathbb I^*}F^{\bullet}
\end{equation*}
which is a morphism in $\PSh(\Cor^{fs}_{\mathbb Z}(\CW),C^-(\mathbb Z))$ 
denoted the same way $S(F^{\bullet}):F^{\bullet}\to\underline{\sing}_{\mathbb I^*}F^{\bullet}$. 
For $f:F_1^{\bullet}\to F_2^{\bullet}$ a morphism $PC^-(CW)$, we 
have the morphism (\ref{SfCW})
\begin{equation*} 
S(\Tr_*f):\Tr_*\underline{\sing}_{\mathbb I^*}F_1^{\bullet}=\underline{\sing}_{\mathbb I^*}\Tr_*F_1^{\bullet}\to
\underline{\sing}_{\mathbb I^*}\Tr_*F_2^{\bullet}=\Tr_*\underline{\sing}_{\mathbb I^*}F_2^{\bullet}
\end{equation*}
which is a morphism in $PC^-(CW)$ denoted the same way
$S(f):\underline{\sing}_{\mathbb I^*}F_1^{\bullet}\to\underline{\sing}_{\mathbb I^*}F_2^{\bullet}$.
\end{itemize}

For $F^{\bullet}\in\PSh(\Cor^{fs}_{\mathbb Z}(\CW),C^-(\mathbb Z))$, we have by definition 
$\Tr_*\underline{\sing}_{\mathbb I^*}F^{\bullet}=\underline{\sing}_{\mathbb I^*}\Tr_*F^{\bullet}$ 
and $\Tr_*S(F^{\bullet})=S(\Tr_*F^{\bullet})$.

We now make the following definition

\begin{defi}\label{RMCW}
Let $Y\in\CW$ and $E\subset Y$ a CW subcomplex. Denote by $l:E\hookrightarrow X$ the topological embedding. 
We define
\begin{itemize}
\item $\mathbb Z(Y,E)=\coker(\mathbb Z(l))\in\PSh_{\mathbb Z}(\CW,C^-(\mathbb Z))$
to be the cokernel of the injective morphism $\mathbb Z(l):\mathbb Z(E)\hookrightarrow\mathbb Z(X)$.

\item $\mathbb Z_{tr}(Y,E)=\coker(\mathbb Z_{tr}(l))\in\PSh_{\mathbb Z}(\Cor^{fs}_{\mathbb Z}(\CW),C^-(\mathbb Z))$
to be the cokernel of the injective morphism $\mathbb Z_{tr}(l):\mathbb Z_{tr}(E)\hookrightarrow\mathbb Z_{tr}(Y)$,
by definition, we have the following exact sequence in $PC^-(CW)$
\begin{equation*}
0\to\mathbb Z_{tr}(E)\xrightarrow{\mathbb Z_{tr}(l)}\mathbb Z_{tr}(Y)\xrightarrow{c(Y,E)}\mathbb Z_{tr}(Y,E)\to 0,
\end{equation*}

\item $\mathbb Z_{(Y,E)}=\coker(\ad(l_!,l^!)(\mathbb Z_X))\in\Sh(Y)$
the cokernel of the injective morphism 
$\ad(l_!,l^!)(\mathbb Z_Y):l_!l^!\mathbb Z_Y\hookrightarrow\mathbb Z_Y$ in $\Sh(Y)$
By definition, there is a distingushed triangle in $D^b(Y)$
\begin{equation*}
l_!l^!\mathbb Z_Y\to\mathbb Z_Y\to\mathbb Z_{(Y,E)}\to l_!l^!\mathbb Z_Y[1]
\end{equation*}
\end{itemize}
The relative Borel Moore homology of the pair $(Y,E)$ 
is $H^{BM}_p(Y,E,\mathbb Z)=\Hom(\mathbb Z,a_{X!}a_{X}^!\mathbb Z_{(Y,E)})$.
\end{defi}

In particular, we get from the second point the following exact sequence in 
$\PSh_{\mathbb Z}(\Cor^{fs}_{\mathbb Z}(\CW),C^-(\mathbb Z))$
\begin{equation}\label{CWcXD}
0\to\underline{\sing}_{\mathbb I^*}\mathbb Z_{tr}(E)\xrightarrow{\underline{\sing}_{\mathbb I^*}\mathbb Z_{tr}(l)}
\underline{\sing}_{\mathbb I^*}\mathbb Z_{tr}(Y)
\xrightarrow{\underline{\sing}_{\mathbb I^*}c(Y,E)}\underline{\sing}_{\mathbb I^*}\mathbb Z_{tr}(Y,E)\to 0
\end{equation}
We define 
\begin{equation*}
M(Y,E)=D(\mathbb I^1,usu)(\mathbb Z_{tr}(Y,E))\in\CwDM^-(\mathbb Z)
\end{equation*}
to be the relative motive of the pair $(Y,E)$.

We now look at the behavior of the functors mentionned above with respect to the $(\mathbb I^,usu)$ model structures.
We start we a lemma concerning the properties of $\mathbb I^1$ homotopy maps.
 
\begin{lem}\label{I1hpt}
Let $F^{\bullet}\in\PSh(\CW,C^-(\mathbb Z))$.
\begin{itemize}
\item[(i)] Let $X,Y\in\CW$ and $f_0:X\to Y$, $f_1:X\to Y$ be two maps. 
If $f_0$ and $f_1$ are $\mathbb I^1$ homotopic, then the maps of complexes 
\begin{itemize}
\item $\underline{\sing}_{\mathbb I^*}F^{\bullet}(f_0): 
\Tot F^{\bullet}(Y\times\mathbb I^*)\to\Tot F^{\bullet}(X\times\mathbb I^*)$  and
\item $\underline{\sing}_{\mathbb I^*}F^{\bullet}(f_1): 
\Tot F^{\bullet}(Y\times\mathbb I^*)\to\Tot F^{\bullet}(X\times\mathbb I^*)$ 
\end{itemize}
induces the same map on homology.
\item[(ii)] Let $X,Y\in\CW$, if $f:X\to Y$ is a $\mathbb I^1$ homotopy equivalence then
\begin{equation}
\underline{\sing}_{\mathbb I^*}F^{\bullet}(f): 
\Tot F^{\bullet}(Y\times\mathbb I^*)\to\Tot F^{\bullet}(X\times\mathbb I^*)
\end{equation}
is a quasi-isomorphism of complexes of abelian groups.
\end{itemize}
\begin{itemize}
\item[(iii)] Let $F^{\bullet},G^{\bullet}\in\PSh(\CW,C^-(\mathbb Z))$ and 
$\phi_0:F^{\bullet}\to G^{\bullet}$, $\phi_1:F^{\bullet}\to G^{\bullet}$ be two maps.
If $\phi_0$ and $\phi_1$ are $\mathbb I^1$ homotopic, then $\phi_0=\phi_1\in\CwDA^{-}$. 
\item[(iv)] Let $F^{\bullet},G^{\bullet}\in\PSh(\CW,C^-(\mathbb Z))$, if  
$\phi:F^{\bullet}\to G^{\bullet}$ is a $\mathbb I^1$ homotopy equivalence then $\phi$ is a $(\mathbb I^1,usu)$ local equivalence.
\end{itemize}
\end{lem}

\begin{proof}

\noindent (i): Let $h:X\times\mathbb I^1\to Y$ an homotopy from $f_0=h\circ(I_X\times i_0)$ to $f_1=h\circ(I_X\times i_1)$.
Let $\alpha\in F^k(Y\times\mathbb I^n)$. Then, we have
\begin{eqnarray*}
\partial (F^{\bullet}(h\times I_{\mathbb I^n})(\alpha))
&=&\partial_F(F^{\bullet}(h\times I_{\mathbb I^n})(\alpha))+\partial_{\mathbb I}(F^{\bullet}(h\times I_{\mathbb I^n})(\alpha)) \\
&=&F^{\bullet}(h\times I_{\mathbb I^n})(\partial_F(\alpha))+
F^{\bullet}(f_0)(\alpha)-F^{\bullet}(f_1)(\alpha)+F^{\bullet}(h\times I_{\mathbb I^n})(\partial_{\mathbb I}\alpha) \\
&=&F^{\bullet}(h\times I_{\mathbb I^n})(\partial\alpha)+F^{\bullet}(f_0)(\alpha)-F^{\bullet}(f_1)(\alpha)
\end{eqnarray*} 
This proves (i).

\noindent (ii): Follow imediately from (i).

\noindent (iii): Let $\tilde{\phi}:F^{\bullet}\to\underline{\Hom}(\mathbb Z(\mathbb I^1),G^{\bullet})$ an homotopy
from $\phi_0=\underline{G}^{\bullet}(i_0)\circ\tilde{\phi}$ to $\phi_1=\underline{G}^{\bullet}(i_1)\circ\tilde{\phi}$.
Let 
\begin{equation*} 
\underline{G}^{\bullet}(p):G^{\bullet}\to\underline{\Hom}(\mathbb Z(\mathbb I^1),G^{\bullet}) 
\end{equation*}
be the map induced by $a_{\mathbb I^1}:\mathbb I^1\to\left\{\pt\right\}$, that is, for $X\in\CW$, 
$\underline{G}^{\bullet}(p)(X)=G^{\bullet}(p_X):G^{\bullet}(X)\to G^{\bullet}(X\times\mathbb I^1)$,
with $p_X=I_X\times a_{\mathbb I^1}:X\times\mathbb I^1\to X$ is the projection.  
Then, we have 
\begin{equation}
\underline{G}^{\bullet}(i_0)\circ\underline{G}^{\bullet}(p)
=\underline{G}^{\bullet}(i_1)\circ\underline{G}^{\bullet}(p)=I_{G^{\bullet}}
\end{equation}  
Hence, it suffice to show that
$\underline{G}^{\bullet}(p):G^{\bullet}\to\underline{\Hom}(\mathbb Z(\mathbb I^1),G^{\bullet})$  
is an $(\mathbb I^1,usu)$ local equivalence, since
then $\underline{G}^{\bullet}(i_0)=\underline{G}^{\bullet}(i_1)\circ\underline{G}^{\bullet}(p)^{-1}\in\CwDA^-$.
So, consider $\psi:G^{\bullet}\to L^{\bullet}$ a morphism in  $\PSh(\CW,C^-(\mathbb Z))$ with $L^{\bullet}$ $\mathbb I^1$ local. 
Consider the commutative diagramm in $\Ho_{usu}\PSh(\CW,C^-(\mathbb Z))$
\begin{equation}\label{Iloc}
\xymatrix{
G^{\bullet}\ar[d]^{\psi}\ar[rr]^{\underline{G}^{\bullet}(p)} & \, & 
\underline{\Hom}(\mathbb Z(\mathbb I^1),G^{\bullet})\ar[d]^{\underline{\Hom}(\mathbb Z(\mathbb I^1),\psi)} \\
L^{\bullet}\ar[rr]^{\underline{L}^{\bullet}(p)} & \, & \underline{\Hom}(\mathbb Z(\mathbb I^1),L^{\bullet})}
\end{equation}
Since we 
now consider morphism in the homotopy category with respect to the usu model structure,
we can assume, up to replace $L^{\bullet}$ by an usu local equivalent object that $L^{\bullet}$ is usu fibrant.
Then, $L^{\bullet}\xrightarrow{\underline{L}^{\bullet}(p)}\underline{\Hom}(\mathbb Z(\mathbb I^1),L^{\bullet})$
is an usu local equivalence and 
\begin{equation}
\underline{L}^{\bullet}(p)^{-1}\circ\underline{\Hom}(\mathbb Z(\mathbb I^1),\psi):
\underline{\Hom}(\mathbb Z(\mathbb I^1),G^{\bullet})\to L^{\bullet}
\end{equation}
is the only map making the diagram \ref{Iloc} commute. This prove that 
$\underline{G}^{\bullet}(p)$ is a $(\mathbb I^1,usu)$ local equivalence.

\noindent (iv): Follows imediately from (iii).

\end{proof}

\begin{lem}\label{CWTr}
\begin{itemize}
\item[(i)] A complex of presheaves $F^{\bullet}\in\PSh(\Cor^{fs}_{\mathbb Z}(\CW),C^-(\mathbb Z))$ is 
$\mathbb I^1$ local if and only if 
$\Tr_*F^{\bullet}\in\PSh(\CW,C^-(\mathbb Z))$ is 
$\mathbb I^1$ local.

\item[(ii)] A morphism $\phi:F^{\bullet}\to G^{\bullet}$ in $\PSh(\Cor^{fs}_{\mathbb Z}(\CW),C^-(\mathbb Z))$ is 
an $(\mathbb I^1,usu)$ local equivalence if and only if
$\Tr_*\phi:\Tr_*F^{\bullet}\to\Tr_*G^{\bullet}$ is an $(\mathbb I^1,usu)$ local equivalence.
\end{itemize}
\end{lem}

\begin{proof}

\noindent(i): Let $h:F^{\bullet}\to L^{\bullet}$ 
be an usu local equivalence in $\PSh(\Cor^{fs}_{\mathbb Z}(\CW),C^-(\mathbb Z))$
with $L^{\bullet}$ usu fibrant. 
Then, 
\begin{itemize}
\item $F^{\bullet}$ is $\mathbb I^1$ local if and only if $L^{\bullet}$ is $\mathbb I^1$ local.

\item By definition $\Tr_*$ preserve usu local equivalence, hence
$\Tr_*h:\Tr_*F^{\bullet}\to\Tr_*L^{\bullet}$ is an usu local equivalence in $\PSh(\CW,C^-(\mathbb Z))$. 
Thus, $\Tr_*F^{\bullet}$ is $\mathbb I^1$ local if and only if $\Tr_*L^{\bullet}$ is $\mathbb I^1$ local. 

\item As indiquated in \cite{AyoubG2},since , $\Tr_*L^{\bullet}$ is also usu fibrant. 
By Yoneda lemma, for $X\in\CW$, we have
\begin{equation}
\Hom_{P^-(CW)}(\mathbb Z(X)[n],\Tr_*L^{\bullet})=H^nL^{\bullet}(X)=\Hom_{PC^-(CW)}(\mathbb Z_{tr}(X)[n],L^{\bullet})
\end{equation}
Hence $L^{\bullet}$ is $\mathbb I^1$ local if and only if $\Tr_*L^{\bullet}$ is $\mathbb I^1$ local.
\end{itemize}
This proves (i).

\noindent (ii): Let us prove (ii)
\begin{itemize}
\item The only if part is similar to \cite{AyoubG1}, lemma 2.111, we check that, for $Y\in\CW$,
\begin{equation*}
\Tr_*\mathbb Z_{tr}(p_Y):\Tr_*\mathbb Z_{tr}(Y\times\mathbb I^1)\to\Tr_*\mathbb Z_{tr}(Y) 
\end{equation*}
is an equivalence $(\mathbb I^1,usu)$ local, 
where $p_Y:Y\times\mathbb I^1\to Y$ is the projection.
By lemma \ref{I1hpt}(iv), it suffice to show that $\Tr_*\mathbb Z_{tr}(p_Y)$ 
is an $\mathbb I^1$ homotopy equivalence in $P^-(CW)$.
By lemma \ref{I1hptlem}, it suffice to show that $p_Y$ is an $\mathbb I^1$ homotopy equivalence in $\CW$.
But we have the map
\begin{equation}
\theta_{12}(Y):Y\times\mathbb I^2\to Y\times\mathbb I^1 \; , \; \theta_{12}(y,t_1,t_2)=(y,t_1-t_2t_1)
\end{equation}
which is an homotopy from $I_{Y\times\mathbb I^1}=\theta_{12}(Y)\circ (I_{Y\times\mathbb I^1}\times i_0)$ to
$I_Y\times 0=\theta_{12}(Y)\circ (I_{Y\times\mathbb I^1}\times i_1)$.
On the other hand, $p_Y\circ(I_Y\times 0)=I_Y$.
This proves the only if part. 

\item Let $\phi:F^{\bullet}\to G^{\bullet}$ in $\PSh(\Cor^{fs}_{\mathbb Z}(\CW),C^-(\mathbb Z))$ such that
$\Tr_*\phi:\Tr_*F^{\bullet}\to\Tr_*G^{\bullet}$ is an equivalence $(\mathbb I^1,usu)$ local
in $\PSh(\CW,C^-(\mathbb Z))$. Since by definition $\Tr_*$ preserve and detect usu local equivalence,
we can assume, up to replace $F^{\bullet}$ and $G^{\bullet}$ by usu equivalent presheaves, 
that $F^{\bullet}$ and $G^{\bullet}$ are usu fibrant.
Let $K^{\bullet}\in\PSh(\Cor^{fs}_{\mathbb Z}(\CW),C^-(\mathbb Z))$ an $\mathbb I^1$ local object.
By (i), $\Tr_*K^{\bullet}$ is $\mathbb I^1$ local.
Hence, we have,
\begin{eqnarray}
\Hom_{PC^-(CW)}(G^{\bullet},K^{\bullet})=\Hom_{PC^-}(\Tr^*\Tr_*G^{\bullet},K^{\bullet})=\Hom_{P^-(CW)}(\Tr_*G^{\bullet},\Tr_*K^{\bullet})
\xrightarrow{\sim} \\
\Hom_{P^-(CW)}(\Tr_*F^{\bullet},\Tr_*K^{\bullet})=\Hom_{PC^-}(\Tr^*\Tr_*F^{\bullet},G^{\bullet})=\Hom_{PC^-(CW)}(F^{\bullet},G^{\bullet}).
\end{eqnarray}
This proves the if part.
\end{itemize}

\end{proof}

\begin{prop}
\begin{itemize}
\item[(i)] $(\Tr^*,\Tr_*):\PSh(\CW,C^-(\mathbb Z))\leftrightarrows\PSh_{\mathbb Z}(\Cor^{fs}_{\mathbb Z}(\CW),C^-(\mathbb Z))$
is a Quillen adjonction for the usu topology model structures (c.f. definition \ref{CWTrusu} (i) and (ii) respectively)
and a Quillen adjonction for the $(\mathbb I^1,usu)$ model structures (c.f. definition \ref{CWmodstr} (i) and (ii) respectively). 
\item[(ii)] $(e_{cw}^*,e_{cw*}):\PSh(\CW,C^-(\mathbb Z))\leftrightarrows C^-(\mathbb Z)$.
is a Quillen adjonction for the usual topology model structure (c.f. definition \ref{CWTrusu} (i))
and a Quillen adjonction for the $(\mathbb I^1,usu)$ model structure (c.f. definition \ref{CWmodstr} (i)). 
\item[(iii)] $(e_{cw}^{tr*},e_{cw*}^{tr}):\PSh(\Cor^{fs}_{\mathbb Z}(\CW),C^-(\mathbb Z))\leftrightarrows C^-(\mathbb Z)$.
is a Quillen adjonction for the usual topology model structure (c.f. definition \ref{CWTrusu} (ii))
and a Quillen adjonction for the $(\mathbb I^1,usu)$ model structure (c.f. definition \ref{CWmodstr} (ii)). 
\end{itemize}
\end{prop}

\begin{proof}

\noindent (i): By lemma \ref{CWTr} (i), $\Tr_*$ preserve $\mathbb I^1$ local objects, hence preserve the fibrations of
the $(\mathbb I^1,usu)$ model structures. 
By lemma \ref{CWTr} (ii), $\Tr_*$ preserve $(\mathbb I^1,usu)$ equivalence. 
Thus, $\Tr_*$ preserve the trivial fibrations of the $(\mathbb I^1,usu)$ model structures. 

\noindent (ii): It is clear that $e^*_{cw}$ is a left Quillen functor, that is preserve cofibrations and trivial cofibrations.

\noindent (iii): It is clear that $e^{tr*}_{cw}$ is a left Quillen functor, that is preserve cofibrations and trivial cofibrations.

\end{proof}

\begin{prop}\label{actadjCW}
\begin{itemize}
\item[(i)] For $F^{\bullet}\in\PSh(\CW,C^-(\mathbb Z))$, the adjonction morphism 
\begin{equation}
\ad(e^*_{cw},e_{cw*})(\underline{\sing}_{\mathbb I^*}F^{\bullet}):
e_{cw}^*e_{cw*}\underline{\sing}_{\mathbb I^*}F^{\bullet}\to\underline{\sing}_{\mathbb I^*}F^{\bullet}
\end{equation}
is an equivalence usu local.

\item[(ii)] For $F^{\bullet}\in\PSh_{\mathbb Z}(\Cor^{fs}_{\mathbb Z}(\CW),C^-(\mathbb Z))$, the adjonction morphism 
\begin{equation}
\ad(e^{tr*}_{cw},e^{tr}_{cw*})(\underline{\sing}_{\mathbb I^*}F^{\bullet}):
e^{tr*}_{cw}e^{tr}_{cw*}\underline{\sing}_{\mathbb I^*}F^{\bullet}\to\underline{\sing}_{\mathbb I^*}F^{\bullet}
\end{equation}
is an equivalence usu local.

\item[(iii)] For $K^{\bullet}\in C^-(\mathbb Z)$, the adjonction morphisms
\begin{equation}
\ad(e^*_{cw},e_{cw*})(K^{\bullet}):K^{\bullet}\to e_{cw*}e^*_{cw}K^{\bullet} \mbox{\; \; and \; \;}
\ad(e^{tr*}_{cw},e^{tr}_{cw*})(K^{\bullet}):K^{\bullet}\to e^{tr}_{cw*}e^{tr*}_{cw}K^{\bullet}
\end{equation}
are isomorphisms.
\end{itemize}
\end{prop}

\begin{proof}

\noindent (i): The question is local. It then follows from the fact that the CW complexes are locally contractible. 
Indeed, let $X\in\CW$. Since the question is local and CW complexes are locally contractile, 
we can assume after shrinking $X$ that $X$ is contractible.
Let $x\in X$. Denoting $a_X:X\to\left\{\pt\right\}$ the terminal map 
and  $e_x:\left\{\pt\right\}\to X$ the point map, there exist an homotopy between $e_x\circ a_X$ and  the identity $I_X$ of $X$,
that is $p_X$ is a $\mathbb I^1$ homotopy equivalence. 
Hence, by lemma \ref{I1hpt} (ii), 
\begin{equation}
\ad(e^*_{cw},e_{cw*})(\underline{\sing}_{\mathbb I^*}F^{\bullet})(X)=\underline{\sing}_{\mathbb I^*}F^{\bullet}(a_X):
F^{\bullet}(\mathbb I^*)\to F^{\bullet}(\mathbb I^*\times X)
\end{equation}
is an homotopy equivalence of complexes of abelian groups.

\noindent (ii): It is a particular case of point (i) since, by definition (\ref{CWTrusu}), $\Tr_*$ detect and preserve usu local equivalence. 
Indeed, by (i)   
\begin{equation}
\xymatrix{ 
e_{cw}^*e_{cw*}\underline{\sing}_{\mathbb I^*}\Tr_*F^{\bullet}
\ar[rrrr]^{\ad(e^*_{cw},e_{cw*})(\underline{\sing}_{\mathbb I^*}\Tr_*F^{\bullet})}
\ar[d]_{=} & \, & \, & \, & \underline{\sing}_{\mathbb I^*}\Tr_*F^{\bullet}\ar[d]^{=} \\
\Tr_*e^{\Tr*}_{cw}e^{\Tr}_{cw*}\underline{\sing}_{\mathbb I^*}F^{\bullet}\ar[rrrr] 
& \, & \, & \, & \Tr_*\underline{\sing}_{\mathbb I^*}F^{\bullet}}
\end{equation}
is an equivalence usu local. Since $\Tr_*$ preserve usu local equivalence, this prove (ii).

\noindent (iii): Trivial.

\end{proof}

\begin{prop}\label{cle}
\begin{itemize}
\item[(i)] The functor 
$\Tr_*:\PSh_{\mathbb Z}(\Cor^{fs}_{\mathbb Z}(\CW),C^-(\mathbb Z))\to\PSh(\CW,C^-(\mathbb Z))$ 
derive trivially.
\item[(ii)] For $K^{\bullet}\in C^-(\mathbb Z)$, $e_{cw}^*K^{\bullet}$ is $\mathbb I^1$ local.
\item[(iii)] For $K^{\bullet}\in C^-(\mathbb Z)$, $e^{tr*}_{cw}K^{\bullet}$ is $\mathbb I^1$ local.
\end{itemize}
\end{prop}

\begin{proof}

\noindent (i): By lemma \ref{CWTr}, $\Tr_*$ preserve $(\mathbb I^1,usu)$ local equivalences.

\noindent (ii): Let $e_{cw}^*K^{\bullet}\to L^{\bullet}$ an usu local equivalence, with $L^{\bullet}$ usu fibrant.
Since $e_{cw}^*K^{\bullet}$ is usu equivalent to $ L^{\bullet}$ it suffices to prove that $L^{\bullet}$ is $\mathbb I^1$ local.
Since $L^{\bullet}$ is usu fibrant, we have to prove that 
\begin{equation*}
\underline{L}(p):L^{\bullet}\to\underline\Hom(\mathbb Z(\mathbb I^1),L^{\bullet}) 
\end{equation*}
is an equivalence usu local
The proof is now similar to \cite{AyoubB} proposition 1.6 etape B.

\noindent (iii): We have $\Tr_*e^{tr*}_{cw}K^{\bullet}=e_{cw}^*K^{\bullet}$.
By (ii), $e_{cw}^*K^{\bullet}$ is $\mathbb I^1$ local. By lemma \ref{CWTr} (i), $\Tr_*$ detect $\mathbb I^1$ local object.
This proves (iii).

\end{proof}

\begin{thm}\label{CWsingI}
\begin{itemize}
\item[(i)] For $F^{\bullet}\in\PSh(\CW,C^-(\mathbb Z))$,
$\underline{\sing}_{\mathbb I^*}F^{\bullet}\in\PSh(\CW,C^-(\mathbb Z))$ is $\mathbb I^1$ local and the inclusion morphism
$S(F^{\bullet}):F^{\bullet}\to\underline{\sing}_{\mathbb I^*}F^{\bullet}$ is an $(\mathbb I^1,usu)$ equivalence. 
\item[(ii)] For $F^{\bullet}\in\PSh_{\mathbb Z}(\Cor^{fs}_{\mathbb Z}(\CW),C^-(\mathbb Z))$, 
$\underline{\sing}_{\mathbb I^*}F^{\bullet}\in\PSh_{\mathbb Z}(\Cor^{fs}_{\mathbb Z}(\CW),C^-(\mathbb Z))$ 
is $\mathbb I^1$ local and the inclusion morphism 
$S(F^{\bullet}):F^{\bullet}\to\underline{\sing}_{\mathbb I^*}F^{\bullet}$ is an $(\mathbb I^1,usu)$ equivalence. 
\end{itemize}
\end{thm}

\begin{proof}

\noindent(i): Let us prove (i)
\begin{itemize}
\item By proposition \ref{actadjCW} (i) and proposition \ref{cle} (ii),
$\underline{\sing}_{\mathbb I^*}F^{\bullet}\in\PSh(\CW,C^-(\mathbb Z))$ is $\mathbb I^1$ local. 
This proves the first part of the assertion.

\item Consider the commutative diagram
\begin{equation}\label{singH}
\xymatrix{
\cdots\ar[r]  & 0\ar[r]\ar[d] & 0\ar[r]\ar[d] & F^{\bullet}\ar[r]\ar[d]^{I} & 0\ar[r] & \cdots \\
\cdots\ar[r]^{0}  & F^{\bullet}\ar[r]^{I}\ar[d]^{\underline{F}^{\bullet}(p_2)} & 
F^{\bullet}\ar[r]^{0}\ar[d]^{\underline{F}^{\bullet}(p_1)} & F^{\bullet}\ar[r]\ar[d]^{I} 
& 0\ar[r] & \cdots \\
\cdots\ar[r]^{0}  & \underline{\sing}_{\mathbb I^2}F^{\bullet}\ar[r] & 
\underline{\sing}_{\mathbb I^1}F^{\bullet}\ar[r] & F^{\bullet}\ar[r] & 0\ar[r] & \cdots }
\end{equation}
By the diagram (\ref{singH}), to prove that 
$S(F^{\bullet}):F^{\bullet}\to\underline{\sing}_{\mathbb I^*}F^{\bullet}$
is an $(\mathbb I^1,usu)$ equivalence, it suffices to show, as in \cite{AyoubG1}, theorem 2.23 etape 2,
that for all $n\in\mathbb Z$, the morphism 
\begin{equation*}
\underline{F}^{\bullet}(p_n):F^{\bullet}\to\underline{\Hom}(\mathbb Z(\mathbb I^n),F^{\bullet}), \;
X\in\CW\mapsto\underline{F}^{\bullet}(p_n)(X)=F^{\bullet}(p_X):
F^{\bullet}(X)\to F^{\bullet}(\mathbb I^n\times X) 
\end{equation*}
is an equivalence $(\mathbb I^1,usu)$ local.
For $X\in\CW$, consider the map
\begin{equation*}
\theta_{1,n}(X):\mathbb I^n\times\mathbb I^1\times X\to\mathbb I^n\times X \; ; \; 
(t_1,\cdots,t_n,t_{n+1},x)\mapsto(t_1-t_{n+1}t_1,\cdots, t_n-t_{n+1}t_n,x)
\end{equation*}
We have, 
\begin{equation*}
\theta_{1,n}(X)\circ (I_{\mathbb I^n}\times i_0\times I_X)=I_{\mathbb I^n\times X} \; , \;
\theta_{1,n}(X)\circ (I_{\mathbb I^n}\times i_1\times I_X)=0\times I_X,
\end{equation*}
that is $\theta_{1,n}(X)$
define an $\mathbb I^1$ homotopy from $I_{\mathbb I^n\times X}$ to $0\times I_X$ ; 
on the other side
$p_X\circ(0\times I_X)=I_X$, with $p_X:\mathbb I^n\times X\to X$ the projection.
Thus,
\begin{equation*}
\underline{F}^{\bullet}(\theta_{1,n})\circ\underline{F}^{\bullet}(I_{\mathbb I^n}\times i_0)
=\underline{F}^{\bullet}(I_{\mathbb I^n}) \; , \;
\underline{F}^{\bullet}(\theta_{1,n})\circ\underline{F}^{\bullet}(I_{\mathbb I^n}\times i_1)
=\underline{F}^{\bullet}(0), 
\end{equation*}
that is 
$\underline{F}^{\bullet}(\theta_{1,n})$
define an $\mathbb I^1$ homotopy from $\underline{F}^{\bullet}(I_{\mathbb I^n})$ to
$\underline{F}^{\bullet}(0)$ ; on the other side,
$\underline{F}^{\bullet}(0)\circ\underline{F}^{\bullet}(p_n)=I$,
with
\begin{itemize}
\item $\underline{F}^{\bullet}(0):\underline{\Hom}(\mathbb Z(\mathbb I^n),F^{\bullet})
\to\underline{\Hom}(\mathbb Z(\mathbb I^n),F^{\bullet})$, given by
$X\in\CW\mapsto\underline{F}^{\bullet}(0)(X)=F^{\bullet}(0\times I_X):
F^{\bullet}(\mathbb I^n\times X)\to F^{\bullet}(\mathbb I^n\times X)$ 
\item $\underline{F}^{\bullet}(I_X):\underline{\Hom}(\mathbb Z(\mathbb I^n),F^{\bullet})\to
\underline{\Hom}(\mathbb Z(\mathbb I^n),F^{\bullet})$, given by
$X\in\CW\mapsto\underline{F}^{\bullet}(I_X)(X)=F^{\bullet}(I_{\mathbb I^n\times X})$ 
\item $\underline{F}^{\bullet}(\theta_{1,n}):
\underline{\Hom}(\mathbb Z(\mathbb I^1),\underline{\Hom}(\mathbb Z(\mathbb I^{n}),F^{\bullet}))\to
\underline{\Hom}(\mathbb Z(\mathbb I^n),F^{\bullet})$, given by
$X\in\CW\mapsto\underline{F}^{\bullet}(\theta_{1,n})(X)=F^{\bullet}(\theta_{1,n}(X)):
F^{\bullet}(\mathbb I^n\times X)\to F^{\bullet}(\mathbb I^n\times\mathbb I^1\times X)$. 
\end{itemize}
Hence,
$\underline{F}^{\bullet}(p_{\mathbb I^n}):F^{\bullet}\to\underline{\Hom}(\mathbb Z(\mathbb I^n),F^{\bullet})$
is an homotopy equivalence, hence an equivalence $(\mathbb I^1,usu)$ local by lemma \ref{I1hpt} (iv). 
\end{itemize}
 
\noindent(ii): Let us prove (ii)
\begin{itemize}
\item By (i), 
$\Tr_*\underline{\sing}_{\mathbb I^*}F^{\bullet}=\underline{\sing}_{\mathbb I^*}\Tr_*F^{\bullet}$
is $\mathbb I^1$ local. 
By lemma \ref{CWTr} (i), $\Tr_*$ detect $\mathbb I^1$ local object. 
Thus, 
$\underline{\sing}_{\mathbb I^*}F^{\bullet}=\underline{\sing}_{\mathbb I^*}F^{\bullet}$
is $\mathbb I^1$ local. 
This proves the first part of the assertion.

\item It follows from (i) that 
\begin{equation}
\Tr_*F^{\bullet}\xrightarrow{\Tr_*S(F^{\bullet})}
\underline{\sing}_{\mathbb I^*}\Tr_*F^{\bullet}=\Tr_*\underline{\sing}_{\mathbb I^*}F^{\bullet}
\end{equation}
is an $(\mathbb I^1,usu)$ equivalence. 
Since, by lemma \ref{CWTr} (ii), $\Tr_*$ detect $(\mathbb I^1,usu)$ equivalence,
$S(F^{\bullet}):F^{\bullet}\to\underline{\sing}_{\mathbb I^*}F^{\bullet}$
is an $(\mathbb I^1,usu)$ equivalence. 
\end{itemize}

\end{proof}

\begin{prop}\label{CWadTr}
\begin{itemize}
\item[(i)] For $F^{\bullet}\in\PSh(\Cor_{\mathbb Z}(\CW),C^-(\mathbb Z))$, 
\begin{equation*}
\ad(\Tr^*,\Tr_*)(F^{\bullet}):F^{\bullet}\to\Tr^*\Tr_*F^{\bullet}
\end{equation*}
is an isomorphism in $\PSh(\Cor_{\mathbb Z}(\CW),C^-(\mathbb Z))$.

\item[(ii)] For $X\in\CW$, the embedding
\begin{equation*}
\ad(\Tr^*,\Tr_*)(\underline{\sing}_{\mathbb I^*}\mathbb Z(X)):
\underline{\sing}_{\mathbb I^*}\mathbb Z(X)\to\Tr_*\underline{\sing}_{\mathbb I^*}\mathbb Z_{tr}(X)
\end{equation*}
in $\PSh(\CW,C^-(\mathbb Z))$ is an equivalence usu local
\end{itemize}
\end{prop}

\begin{proof}

\noindent(i): Obvious

\noindent(ii): 
By proposition \ref{Icov}(i), there exist an homotopy equivalence $g:X'\to X$, with $X'\in\CS$ a $\Delta$ complex.
Then,
\begin{itemize}
\item $S(\mathbb Z(g)):\underline{\sing}_{\mathbb I^*}\mathbb Z(X')\to\underline{\sing}_{\mathbb I^*}\mathbb Z(X)$ 
is an $\mathbb I^1$ homotopy equivalence by lemma \ref{I1hptlem}, 
thus an usu local equivalence by lemma \ref{I1hpt}(iv) and theorem \ref{CWsingI}(i),
\item $\Tr_*S(\mathbb Z_{tr}(g)):\Tr_*\underline{\sing}_{\mathbb I^*}\mathbb Z_{tr}(X')
\to\Tr_*\underline{\sing}_{\mathbb I^*}\mathbb Z(X)$ 
is an $\mathbb I^1$ homotopy equivalence by lemma \ref{I1hptlem}, 
thus an usu local equivalence by lemma \ref{I1hpt}(iv) and theorem \ref{CWsingI}(ii)
and the fact that $\Tr_*$ preserve $\mathbb I^1$ local object by lemma \ref{CWTr}.
\end{itemize}
Consider the following commutative diagram in $\PSh(\CW,C^-(\mathbb Z))$
\begin{equation}\label{CWCS}
\xymatrix{\underline{\sing}_{\mathbb I^*}\mathbb Z(X')\ar[rrrr]^{\ad(\Tr^*,\Tr_*)(\underline{\sing}_{\mathbb I^*}\mathbb Z(X'))}
\ar[d]^{S(\mathbb Z(g))} & \, & \, & \, & \Tr_*\underline{\sing}_{\mathbb I^*}\mathbb Z_{tr}(X')\ar[d]^{\Tr_*S(\mathbb Z_{tr}(g))} \\
\underline{\sing}_{\mathbb I^*}\mathbb Z(X)\ar[rrrr]^{\ad(\Tr^*,\Tr_*)(\underline{\sing}_{\mathbb I^*}\mathbb Z(X))}
 & \, & \, & \, & \Tr_*\underline{\sing}_{\mathbb I^*}\mathbb Z_{tr}(X)}
\end{equation} 
By diagram (\ref{CWCS}) and the fact that $S(\mathbb Z(g))$ and $\Tr_*S(\mathbb Z_{tr}(g))$ are usu local
equivalence, it suffice to show that
\begin{equation*}
\ad(\Tr^*,\Tr_*)(\underline{\sing}_{\mathbb I^*}\mathbb Z(X')):
\underline{\sing}_{\mathbb I^*}\mathbb Z(X')\to\Tr_*\underline{\sing}_{\mathbb I^*}\mathbb Z_{tr}(X')
\end{equation*}
is an usu local equivalence.
By proposition \ref{Icov}(ii), there exist a countable open covering $X'=\cup_{i\in J}U_i$ such that for all
finite subset $I\subset J$, $U_I:=\cap_{i\in I}U_i=\emptyset$ or $U_I$ is contractible.
Let $(U_{\bullet},\bullet)$ be the complex in $\bold\Delta\times\CW$ associated to this covering. 
Let
\begin{itemize}
\item $h:\mathbb Z(U_{\bullet},\bullet)\to K^{\bullet}$ be an equivalence usu local 
in $\PSh(\bold\Delta\times\CW,C^-(\mathbb Z))$ with $K^{\bullet}$ usu fibrant
\item $l:\mathbb Z_{tr}(U_{\bullet},\bullet)\to L^{\bullet}$ be an equivalence usu local 
in $\PSh(\bold\Delta\times\Cor_{\mathbb Z}(\CW),C^-(\mathbb Z))$ with $L^{\bullet}$ usu fibrant
\end{itemize}
We have then a unique morphism $m:K^{\bullet}\to\Tr_*L^{\bullet}$ 
such that the following diagram in $\Ho_{usu}(\PSh(\bold{\Delta}\times\CW,C^-(\mathbb Z)))$ commutes
\begin{equation}\label{mMor}
\xymatrix{
\mathbb Z((U_{\bullet},\bullet))\ar[rrrr]^{\ad(\Tr^*,\Tr_*)(\mathbb Z((U_{\bullet},\bullet)))}\ar[d]_{h} \, & \, & \, & \, &
\Tr_*\mathbb Z_{tr}((U_{\bullet},\bullet))\ar[d]^{\Tr_*l} \\
K^{\bullet}\ar[rrrr]^{m} & \, & \, & \, & \Tr_*L^{\bullet}}
\end{equation}
We have then the following commutative diagram in $\Ho_{usu}(\PSh(\CW,C^-(\mathbb Z)))$ :
\begin{equation}\label{trMor}
\xymatrix{
p^*_{\bold\Delta}K^{\bullet}\ar[d]_{j}\ar[rrr]^{p^*_{\bold{\Delta}}(m)} & \, & \, & 
p^*_{\bold\Delta}\Tr_*L^{\bullet}=\Tr_*p^*_{\bold\Delta}L^{\bullet}\ar[d]^{\Tr_*(j)} \\
\mathbb Z(X)\ar[rrr]^{\ad(\Tr^*,\Tr_*)(\mathbb Z(X'))} \, & \, & \, & \Tr_*\mathbb Z_{tr}(X')}
\end{equation}
By lemma \ref{HRCW} (i), and the fact that $\Tr_*$ preserve usu local equivalence,
\begin{itemize}
\item $j:Lp^*_{\bold{\Delta}}\mathbb Z((U_{\bullet},\bullet))=p^*_{\bold\Delta}K^{\bullet}\to\mathbb Z(X)$, and
\item $\Tr_*(j):\Tr_*Lp^*_{\bold{\Delta}}\mathbb Z_{tr}((U_{\bullet},\bullet))=p^*_{\bold\Delta}\Tr_*L^{\bullet}
\to\Tr_*\mathbb Z_{tr}(X)$,
\end{itemize}
are isomorphisms.
On the other hand, $U_{[n]}$ is contractible for all $[n]\in\bold\Delta$, this means that the canonical morphism
$a_{U_{[n]}}:U_{[n]}\to\left\{\pt\right\}$ in $\CW$ is an homotopy equivalence.
Hence, 
\begin{itemize}
\item $\mathbb Z(a_{U_{[n]}}):\mathbb Z(U_{[n]})\to\mathbb Z(\left\{\pt\right\})$ is an
homotopy equivalence in $\PSh(\CW,C^-(\mathbb Z))$, thus an $(\mathbb I^1,usu)$ local equivalence
by lemma \ref{I1hpt}(iv) ;
\item $\Tr_*\mathbb Z_{tr}(a_{U_{[n]}}):\Tr_*\mathbb Z_{tr}(U_{[n]})\to\Tr_*\mathbb Z_{tr}(\left\{\pt\right\})$ is an
homotopy equivalence in $\PSh(\CW,C^-(\mathbb Z))$, thus an $(\mathbb I^1,usu)$ local equivalence
by lemma \ref{I1hpt}(iv).
\end{itemize}
Thus,
\begin{equation*}
\ad(\Tr^*,\Tr_*)(\mathbb Z((U_{\bullet},\bullet))):\mathbb Z((U_{\bullet},\bullet))\to
\Tr_*\mathbb Z_{tr}((U_{\bullet},\bullet))
\end{equation*} 
is an $(\mathbb I^1,usu)$ local equivalence in $\PSh(\bold\Delta\times\CW,C^-(\mathbb Z))$.
The diagram (\ref{mMor}) then shows that $m:K^{\bullet}\to\Tr_*L^{\bullet}$ 
is an isomorphism in $\Ho_{\mathbb I^1,usu}(\PSh(\bold\Delta\times\CW,C^-(\mathbb Z))$. 
Since $p_{\bold\Delta}^*$ derive trivially by definition, this implies that
\begin{equation*} 
p_{\bold\Delta}^*(m):p_{\bold\Delta}^*K^{\bullet}\to p_{\bold\Delta}^*\Tr_*L^{\bullet}
\end{equation*}
is an isomorphism in $\CwDM(\mathbb Z)^-$
The diagram (\ref{trMor}), then shows that
$\ad(\Tr^*,\Tr_*)(\mathbb Z(X')):\mathbb Z(X)\to\Tr_*\mathbb Z_{tr}(X')$
is an isomorphism in $\CwDM(\mathbb Z)^-$.

\end{proof}

\begin{thm}\label{CWMot}
\begin{itemize} 
\item[(i)]The adjonction
$(\Tr^*,\Tr_*):\PSh(\CW,C^-(\mathbb Z))\leftrightarrows\PSh_{\mathbb Z}(\Cor^{fs}_{\mathbb Z}(\CW),C^-(\mathbb Z))$
is a Quillen equivalence for the $(\mathbb I^1,usu)$ model structures. 
That is, the derived functor 
\begin{equation}
L\Tr^*:\CwDA^-(\mathbb Z)\xrightarrow{\sim}\CwDM^-(\mathbb Z) 
\end{equation}
is an isomorphism
and $\Tr_*:\CwDM^-(\mathbb Z)\xrightarrow{\sim}\CwDA^-(\mathbb Z)$ is it inverse. 

\item[(ii)] The adjonction
$(e_{cw}^*,e_{cw*}):C^-(\mathbb Z)\leftrightarrows\PSh_{\mathbb Z}(\CW,C^-(\mathbb Z))$
is a Quillen equivalence for the $(\mathbb I^1,usu)$ model structures. 
That is, the derived functor
\begin{equation}
e_{cw}^*:D^-(\mathbb Z)\xrightarrow{\sim}\CwDA^-(\mathbb Z)
\end{equation}
is an isomorphism
and $Re_{cw*}:\CwDA^-(\mathbb Z)\xrightarrow{\sim}D^-(\mathbb Z)$ is it inverse. 

\item[(iii)] The adjonction
$(e^{tr*}_{cw},e^{tr}_{cw*}):C^-(\mathbb Z)\leftrightarrows\PSh_{\mathbb Z}(\Cor^{fs}_{\mathbb Z}(\CW),C^-(\mathbb Z))$
is a Quillen equivalence for the $(\mathbb I^1,usu)$ model structures. 
That is, the derived functor
\begin{equation}
e^{tr*}_{cw}:D^-(\mathbb Z)\xrightarrow{\sim}\CwDM^-(\mathbb Z)
\end{equation}
is an isomorphism
and $Re^{tr}_{cw*}:\CwDM^-(\mathbb Z)\xrightarrow{\sim}D^-(\mathbb Z)$ is it inverse.  
\end{itemize}
\end{thm}

\begin{proof}

\noindent (i): It follows from proposition \ref{CWadTr} and theorem \ref{CWsingI}.

\noindent (ii): It follows from proposition \ref{actadjCW}(i) and theorem \ref{CWsingI}(i). 

\noindent (iii): It follows from proposition \ref{actadjCW}(ii) and theorem \ref{CWsingI}(ii). 
It also follows from (i) and (ii).

\end{proof}

We deduce from proposition \ref{CWadTr} (ii) and proposition \ref{SerreP} the following :

\begin{prop}\label{CWadTr(iii)}
For $Y\in\CW$ and $l:E\hookrightarrow Y$ a CW subcomplex, the followings embeddings are quasi-isomorphism :
\begin{equation*}
C^{\sing}_*(Y,E,\mathbb Z)\xrightarrow{\mathbb Z(Y,E)(L)}\sing_{\mathbb I^*}\mathbb Z(Y,E)
\xrightarrow{e_{cw*}\ad(\Tr^*,\Tr_*)(\underline{\sing}_{\mathbb I^*}\mathbb Z(Y,E))}
\sing_{\mathbb I^*}\mathbb Z_{tr}(Y,E)
\end{equation*}
where $C^{\sing}_*(Y,E,\mathbb Z)=\coker{l_*}$ is the relative cohomology, with
$l_*:\mathbb Z\Hom_{\CW}(\Delta^*,E)\hookrightarrow\mathbb Z\Hom_{\CW}(\Delta^*,Y)$.
\end{prop}

\begin{proof}
Follows from proposition \ref{CWadTr} (ii) and proposition \ref{SerreP}.
\end{proof}

Let $X,Y\in\CW$ and $E\subset Y$ a subcomplex, we have
\begin{itemize}
\item the morphism 
\begin{equation*}
S_n(X,(Y,E)):\Hom_{PC^-(CW)}(\underline{\sing}_{\mathbb I^*}\mathbb Z_{tr}(X),
\underline{\sing}_{\mathbb I^*}\mathbb Z_{tr}(Y,E)[n])\to
\Hom_{PC^-(CW)}(\mathbb Z_{tr}(X),\underline{\sing}_{\mathbb I^*}\mathbb Z_{tr}(Y,E)[n]), 
\end{equation*}
given by the composition
$S_n(X,(Y,E))(H)=H\circ S(\mathbb Z_{tr}(X))$ with the inclusion morphism $S(\mathbb Z_{tr}(X))$,

\item the morphism 
\begin{equation*}
R_n(X,(Y,E)):\Hom_{PC^-(CW)}(\mathbb Z_{tr}(X),\underline{\sing}_{\mathbb I^*}\mathbb Z_{tr}(Y,E)[n])\to 
\Hom_{PC^-(CW)}(\underline{\sing}_{\mathbb I^*}\mathbb Z_{tr}(X),
\underline{\sing}_{\mathbb I^*}\mathbb Z_{tr}(Y,E)[n]) 
\end{equation*}
given by, 
for $[T]\in\Hom_{PC^-(CW)}(\mathbb Z_{tr}(X),\underline{\sing}_{\mathbb I^*}\mathbb Z_{tr}(Y,E)[n])$, and $Z\in\CW$, 
\begin{equation*}
R_n(X,(Y,E))([T])(Z):\gamma\in\underline{\sing}_{\mathbb I^p}\mathbb Z_{tr}(X)(Z)\mapsto 
T\circ(I_{\mathbb I^n}\times\gamma)\in\underline{\sing}_{\mathbb I^{n+p}}\mathbb Z_{tr}(Y,E)(Z)
\end{equation*}
with $I_{\mathbb I^n}\times\gamma\in
\mathcal Z^{fs/\mathbb I^{n+p}\times Z}((\mathbb I^n\times\mathbb I^p\times Z)\times(\mathbb I^n\times X))$.
\end{itemize}

\begin{lem}\label{SRCW}
Let $X,Y\in\CW$ and $E\subset Y$ a subcomplex, we have 
\begin{itemize}
\item[(i)] $S_n(X,(Y,E))\circ R_n(X,(Y,E))=I$,
\item[(ii)] for $[T]\in\Hom_{PC^-(CW)}(\mathbb Z_{tr}(X),\underline{\sing}_{\mathbb I^*}\mathbb Z_{tr}(Y,E)[n])$,
the equality of morphism of $\CwDM^-(\mathbb Z)$
\begin{equation*}
D(\mathbb I^1,usu)([T])\circ S(\mathbb Z_{tr}(X))^{-1}=D(\mathbb I^1,usu)(R_n(X,(Y,E))([T])). 
\end{equation*}
\end{itemize}
\end{lem}

\begin{proof}

\noindent(i): Obvious

\noindent(ii): Follows from (i).

\end{proof}

Recall $\TM(\mathbb R)\subset\CW$ is the full subcategory of topological manifolds.
Let $X\in\TM(\mathbb R)$ connected topological manifold, $Y\in\CW$ a CW complex, 
and $l:E\hookrightarrow Y$ a CW subcomplex.
Denote by $p_X:\mathbb I^n\times X\times Y\to X$, $p_Y:\mathbb I^n\times X\times Y\to Y$ and
$p_{X\times Y}:\mathbb I^n\times X\times Y\to X\times Y$ the projections.
Let 
\begin{equation*}
T=\sum_i n_i T_i\in\underline{\sing}_{\mathbb I^n}\mathbb Z_{tr}(Y,E)(X)
\end{equation*}
such that $\partial_{\mathbb I^*}T=0$, 
where $m_i:T_i\hookrightarrow\mathbb I^n\times X\times Y$ is a closed CW subcomplex for all $i$.
Denote $p_{Xi}=p_X\circ m_i:T_i\to X$ and $p_{Yi}=p_X\circ m_i:T_i\to Y$.  
Consider the class of $T$ : 
\begin{equation}
[T]\in H^n\underline{\sing}_{\mathbb I^*}\mathbb Z_{tr}(Y,E)(X)
=\Hom_{PC^-(CW)}(\mathbb Z_{tr}(X),\underline{\sing}_{\mathbb I^*}\mathbb Z_{tr}(Y,E)[n]) 
\end{equation}
We have then
\begin{itemize}
\item its image
\begin{eqnarray*}
Re^{tr}_{cw*}\circ D(\mathbb I^1,usu)([T])=e^{tr}_{cw*}([T]\circ S(\mathbb Z_{tr}(X))^{-1}) \\ 
\in\Hom_{D^-(\mathbb Z)}(\sing_{\mathbb I^*}\mathbb Z_{tr}(X),\sing_{\mathbb I^*}\mathbb Z_{tr}(Y,E)[n])
\end{eqnarray*}
by the composite functor
$Re^{tr}_{cw*}\circ D(\mathbb I^1,usu):\PSh(\Cor_{\mathbb Z}(\CW),C^-(\mathbb Z))\to D^-(\mathbb Z)$,
where the last equality follows from the fact that by theorem \ref{CWsingI}(ii) 
\begin{itemize}
\item $S(\mathbb Z_{tr}(X))):\mathbb Z_{tr}(X)\to\underline{\sing}_{\mathbb I^*}\mathbb Z_{tr}(X)$
is an equivalence $(\mathbb I^1,usu)$ local and
\item $\underline{\sing}_{\mathbb I^*}\mathbb Z_{tr}(X)$ and $\underline{\sing}_{\mathbb I^*}\mathbb Z_{tr}(Y)$
are $\mathbb I^1$ local objects,
\end{itemize}
\item the action of $p_{X\times Y}(T)\in\mathcal Z_{d_X+n}(X\times Y,\mathbb Z)$
\begin{equation*}
K_n(X,(Y,E))(p_{X\times Y}(T)):\sum_i n_i(c_{Y,E}[n])\circ (p_{Yi*}[n])\circ p_{Xi}^*
\in\Hom_{D^-(\mathbb Z)}(C_*(X,\mathbb Z),C_*(Y,E,\mathbb Z)[n]), \
\end{equation*}
on homology, where, for each $i$ :
\begin{itemize}
\item $p_{Xi}^*\in\Hom_{D^-(\mathbb Z)}(C_*(X,\mathbb Z),C_*(T_i,\mathbb Z)[n])$ 
is the Gynsin morphism ($p_{Xi}$ is proper and $X\in\TM(\mathbb R)$ is a topological manifold),
\item $p_{Yi*}=\mathbb Z(p_{Yi})(\Delta^*):C_*(T_i,\mathbb Z)\to C_*(Y,\mathbb Z)$ is the classical map on singular chain, 
\item $c_{Y,E}:C_*(Y,\mathbb Z)\to C_*(Y,E,\mathbb Z)$ is the quotient map.
\end{itemize}
Here, since $X$ is a connected topological manifold, we have $H^{BM}_{d_X}(X,\mathbb Z)=[X]=Z$ and 
since $T_i\subset\mathbb I^n\times X\times Y$ is finite surjective and irreducible over $\mathbb I^n\times X$,
we have $H^{BM}_{d_X+n}(T_i\backslash\partial(T_i),\mathbb Z)=[T_i]=\mathbb Z$.
\end{itemize}

We have the following :

\begin{lem}\label{actCWFCl}
Let $X\in\TM(\mathbb R)$ connected, $Y\in\CW$ and $E\subset Y$ a subcomplex. 
Let $T=\sum_i n_iT_i\in\underline{\sing}_{\mathbb I^n}\mathbb Z_{tr}(Y,E)(X)$ such that $\partial_{\mathbb I^*}T=0$ 
and let $\gamma=\sum_j m_j\gamma_j\in C_p(X,\mathbb Z)$ such that $\partial\gamma=0$.
Then 
\begin{equation*}
T\circ(I_{\mathbb I^n}\times\gamma)=K_n(X,(Y,E))(p_{X,Y}(T))(\gamma). 
\end{equation*}
\end{lem}

\begin{proof}
Denote by $T_{i,\gamma_j}:=p_Y(T_i\circ(I_{\mathbb I^n}\times\gamma_j))\subset Y$ 
where $p_Y:\mathbb I^{p+n}\times Y\to Y$ is the projection.
We have then $\dim T_{i,\gamma_j}\leq\dim(T_i\circ(I_{\mathbb I^n}\times\gamma_j))=p+n$ 
and the factorization 
$T_i\circ(I_{\mathbb I^n}\times\gamma_j)=i_{\gamma_j}\circ T_{i,\gamma_j}^o$ 
where, 
\begin{itemize}
\item $T_{i,\gamma_j}^o\subset\mathbb I^{p+n}\times T_{i,\gamma_j}$ is finite and surjective over $\mathbb I^{p+n}$,
so that $\dim T^o_{i,\gamma_j}=p+n\geq\dim T_{i,\gamma_j}$, 
\item $i_{\gamma_j}:T_{i,\gamma_j}\hookrightarrow Y$ is the closed embedding.
\end{itemize}
We can assume without loss of generality that $\dim T_{i,\gamma_j}=p+n$. We have then,
\begin{eqnarray*}
p_{Yi*}p_{Xi}^*[\Im(\gamma_j)\backslash\partial]
&=&[p_Y(p_{Xi}^{-1}(\Im(\gamma_j)\backslash\partial))] \\
&=&[i_{\gamma_j}\circ p_{T_{i,\gamma_j}}(T_{i,\gamma_j}^o)] 
\mbox{\, since \,} p_Y(p_{Xi}^{-1}(\Im(\gamma_j)\backslash(\partial)))=T_{i,\gamma_j}\backslash\partial \\
&=&i_{\gamma_j*}[T_{i,\gamma_j}^o\backslash\partial] \mbox{\, since \,} \dim T_{i,\gamma_j}=p+n \\
&=&[(T_i\circ(I_{\mathbb I^n}\times\gamma_j))\backslash\partial] \mbox{\, by \, definition,}
\end{eqnarray*}
where $[\cdot]$ denote the fundamental class in Borel Moore homology and $\partial$ the boundary.

\end{proof}
 
We deduce from lemma \ref{SRCW}(ii) and lemma \ref{actCWFCl} the following :

\begin{prop}\label{Kunneth}
Let $X\in\TM(\mathbb R)$ connected and $Y\in\CW$. Let $l:E\hookrightarrow Y$ a CW subcomplex. 
\begin{itemize}
\item[(i)]Let $[T]=[\sum_i n_iT_i]\in\Hom_{PC^-(CW)}(\mathbb Z_{tr}(X),\sing_{\mathbb I^*}(\mathbb Z_{tr}(Y,E))[n])$  
Then, 
\begin{eqnarray*}
Re^{tr}_{cw*}\circ D(\mathbb I^1,usu)([T])=K_n(X,(Y,E))(p_{X\times Y}(T)) \\
\in\Hom_{D^-(\mathbb Z)}(\sing_{\mathbb I^*}\mathbb Z_{tr}(X),\sing_{\mathbb I^*}\mathbb Z_{tr}(Y,E)[n])
=\Hom_{D^-(\mathbb Z)}(C_*(X,\mathbb Z),C_*(Y,E,\mathbb Z)[n]), 
\end{eqnarray*}

\item[(ii)] If $Y$ is compact and $E\subset Y$ is closed then the factorization 
\begin{equation*}
K_n(X,(Y,E))=\overline{K_n(X,(Y,E))}\circ [\cdot]:\mathcal Z_{d_X+n}(X\times Y,\mathbb Z)\to
\Hom_{D^-(\mathbb Z)}(C_*(X,\mathbb Z),C_*(Y,E,\mathbb Z)[n])
\end{equation*}
where 
$[\cdot]:\mathcal Z_{d_X+n}(X\times Y,X\times E,\mathbb Z)\to H^{BM}_{d_X+n}(X\times Y,X\times E,\mathbb Z)$
is the fundamental class gives the classical isomorphism
\begin{equation}
\overline{K_n(X,(Y,E))}:H^{BM}_{d_X+n}(X\times Y,X\times E,\mathbb Z)
\xrightarrow{\sim}\Hom_{D^-(\mathbb Z)}(C_*(X,\mathbb Z),C_*(Y,E,\mathbb Z)[n])
\end{equation}
of \cite{CH}. 
\end{itemize}
\end{prop}

\begin{proof}

\noindent(i): The equality 
\begin{equation*}
\Hom_{D^-(\mathbb Z)}(\sing_{\mathbb I^*}\mathbb Z_{tr}(X),\sing_{\mathbb I^*}\mathbb Z_{tr}(Y,E)[n])
=\Hom_{D^-(\mathbb Z)}(C_*(X,\mathbb Z),C_*(Y,E,\mathbb Z)[n])
\end{equation*}
follows from proposition \ref{CWadTr(iii)}.  
For $p\in\mathbb N$ and $\gamma\in H_p(X,\mathbb Z)$ we have 
\begin{eqnarray*}
Re^{tr}_{cw*}\circ D(\mathbb I^1,usu)([T])(\gamma)
&=&e^{tr}_{cw*}([T]\circ S(\mathbb Z_{tr}(X))^{-1})(\gamma) \\
&=&e^{tr}_{cw*}R_n(X,Y)([T])(\gamma) \mbox{\; by \;  lemma \ref{SRCW} (ii)} \\
&:=& T\circ(I_{\mathbb I^n}\times\gamma) \mbox{ \; with \;} 
I_{\mathbb I^n}\times\gamma\in\mathcal Z^{fs/\mathbb I^{p+n}}((\mathbb I^n\times\mathbb I^p)\times(\mathbb I^n\times X)) \\
&=&K_n(X,(Y,E))(p_{X\times Y}(T))(\gamma)\in H_{p+n}(Y,E,\mathbb Z) \mbox{\; by \; lemma \ref{actCWFCl}} 
\end{eqnarray*}

\noindent(ii):Let $p\in\mathbb N$ and $\gamma\in H_p(X,\mathbb Z)$. If $Y$ is compact,
$p_X:X\times Y\to X$ is proper and we have
\begin{equation*}
K_n(X,(Y,E))(p_{X\times Y}(T))(\gamma)=p_{Y*}(p_X^*\gamma.[p_{X\times Y}(T)])\in H_{p+n}(Y,E,\mathbb Z) 
\end{equation*}
where $p_{X}^*:H_{p+n}(X,\mathbb Z)\to H_{p+n}(Y,\mathbb Z)$ is Gysin morphism,
and $[p_{X\times Y}(T)]\in H^{BM}_{d_X+n}(X\times Y, X\times E)$ is the Borel Moore class
of $p_{X\times Y}(T)\in\mathcal Z_{d_X+n}(X\times Y,\mathbb Z)$.

\end{proof}

We finish this subsection by the following proposition :
\begin{prop}\label{DI1iso}
Let $X,Y\in\CW$, $E\subset Y$ a CW subcomplex and $n\in\mathbb Z$, $n<0$.
The morphism of abelian group
\begin{equation*}
D(\mathbb I^1,usu):
\Hom_{PC^-(CW)}(\mathbb Z_{tr}(X),\underline{\sing}_{\mathbb I^*}\mathbb Z_{tr}(Y,E)[n])\to\Hom_{\CwDM^-(\mathbb Z)}(M(X),M(Y,E)[n])
\end{equation*}
of the functor 
$D(\mathbb I^1,usu):\PSh_{\mathbb Z}(\Cor^{fs}_{\mathbb Z}(\CW),C^-(\mathbb Z))\to\CwDM^-(\mathbb Z)$ 
is an isomorphism if $X\in\TM(\mathbb R)$ and $Y\in\CW$ is compact and $E\subset Y$ is closed.

\end{prop}

\begin{proof}
By proposition \ref{Kunneth}
\begin{eqnarray*}
Re^{tr}_{cw*}\circ D(\mathbb I^1,usu)=K_n(X,(Y,E)):
\Hom_{PC^-(CW)}(\mathbb Z_{tr}(X),\underline{\sing}_{\mathbb I^*}\mathbb Z_{tr}(Y,E)[n]) \\
\to\Hom_{D^-(\mathbb Z)}(\sing_{\mathbb I^*}\mathbb Z_{tr}(X),\sing_{\mathbb I^*}\mathbb Z_{tr}(Y,E)[n]).
\end{eqnarray*}
Hence $Re^{tr}_{cw*}\circ D(\mathbb I^1,usu)$
is an isomorphism if $X\in\TM$, $Y$ is compact and $E$ is closed 
since $K_n(X,(Y,E))$ is an isomorphism by the six functor formalism as indiquated in \cite{CH} 
if $X\in\TM$, $Y$ is compact and $E$ is closed. 
On the other hand, by theorem \ref{CWMot} (iii)
\begin{equation*}
Re^{tr}_{cw*}:\Hom_{\CwDM^-(\mathbb Z)}(M(X),M(Y,E)[n])
\to\Hom_{D^-(\mathbb Z)}(\sing_{\mathbb I^*}\mathbb Z_{tr}(X),\sing_{\mathbb I^*}\mathbb Z_{tr}(Y,E)[n])
\end{equation*}
is an isomorphism.
This proves the proposition.

\end{proof}

\section{The Betti realization functor on the derived category of motives of complex algebraic varieties}

We will consider,  
\begin{itemize}

\item the analytical functor $\An:\Var(\mathbb C)\to\AnSp(\mathbb C)$ 
given by $\An(V)=V^{an}$, $\An(g)=g^{an}$,
and the analytical functor on transfers
$\An:\Cor^{fs}_{\Lambda}(\SmVar(\mathbb C))\to\Cor^{fs}_{\mathbb Z}(\AnSm(\mathbb C))$ 
given by $\An(V)=V^{an}$,$\An(\Gamma)=\Gamma^{an}$,

\item the forgetful functor $\Cw:\AnSp(\mathbb C)\to\CW$ 
given by $\Cw(W)=W^{cw}$, $\Cw(g)=g^{cw}$,
and the forgetful functor on transfers 
$\Cw:\Cor^{fs}_{\mathbb Z}(\AnSm(\mathbb C))\to\Cor^{fs}_{\mathbb Z}(\CW)$ given by $\Cw(W)=W^{cw}$,
$\Cw(\Gamma)=\Gamma^{cw}$,

\item the composites $\widetilde\Cw=\Cw\circ\An:\Var(\mathbb C)\to\CW$, 
given by $\widetilde\Cw(V)=V^{cw}$, $\widetilde\Cw(g)=g^{cw}$,
and $\widetilde\Cw=\Cw\circ\An:\Cor^{fs}_{\mathbb Z}(\SmVar(\mathbb C))\to\Cor^{fs}_{\mathbb Z}(\CW)$, 
given by $\widetilde\Cw(V)=V^{cw}$, $\widetilde\Cw(\Gamma)=\Gamma^{cw}$.

\item the embeddings of categories
$\iota_{an}:\AnSm(\mathbb C)\to\AnSp(\mathbb C)$ and 
$\iota_{var}:\SmVar(\mathbb C)\to\Var(\mathbb C)$.

\end{itemize}

By definition, we have the following commutative diagram of sites
\begin{equation}
\xymatrix{
\widetilde\Cw: \Cor^{fs}_{\mathbb Z}(\CW)\ar[r]^{\Cw}\ar[d]^{\Tr} & \Cor^{fs}_{\mathbb Z}(\AnSm(\mathbb C))\ar[r]^{\An}\ar[d]^{\Tr} 
& \Cor^{fs}_{\mathbb Z}(\SmVar(\mathbb C))\ar[d]^{\Tr} \\
\widetilde\Cw: \mathbb Z(CW)\ar[r]^{\Cw}\ar[rd]^{\Cw} & \mathbb Z(\AnSm(\mathbb C))\ar[r]^{\An} & \mathbb Z(\SmVar(\mathbb C)) \\
 \, & \mathbb Z(\AnSp(\mathbb C))\ar[r]^{\An}\ar[u]^{\iota_{an}} & \mathbb Z(\Var(\mathbb C))\ar[u]^{\iota_{var}} } 
\end{equation}

We note that we have the following :

\begin{prop}\label{AnCwtCw}
\begin{itemize}
\item[(i)]The functors 
\begin{itemize}
\item $\An^*:\PSh(\SmVar(\mathbb C),C^-(\mathbb Z))\to\PSh(\AnSm(\mathbb C),C^-(\mathbb Z))$ and
\item $\An^*:\PSh(\Cor^{fs}_{\mathbb Z}(\SmVar(\mathbb C)),C^-(\mathbb Z))\to\PSh(\Cor^{fs}_{\mathbb Z}(\AnSm(\mathbb C)),C^-(\mathbb Z))$, 
\end{itemize}
derive trivially for the $(\mathbb A^1,et)$ and $(\mathbb D^1,usu)$ model structures.

\item[(ii)] 
\begin{itemize}
\item Let  $K^{\bullet}\in\PSh(\Cor^{fs}_{\mathbb Z}(\CW),C^-(\mathbb Z))$. If $K^{\bullet}$ is $\mathbb I^1$ local,
then $\Cw_*K^{\bullet}$ is $\mathbb D^1$ local.
\item Let $K^{\bullet}\in\PSh(\CW,C^-(\mathbb Z))$. 
If $K^{\bullet}$ is $\mathbb I^1$ local, then $\Cw_*K^{\bullet}$ is $\mathbb D^1$ local.
\end{itemize}
Moreover, the functors 
\begin{itemize}
\item $\Cw^*:\PSh(\AnSm(\mathbb C),C^-(\mathbb Z))\to\PSh(\CW,C^-(\mathbb Z))$ and
\item $\Cw^*:\PSh(\Cor^{fs}_{\mathbb Z}(\AnSm(\mathbb C)),C^-(\mathbb Z))\to\PSh(\Cor^{fs}_{\mathbb Z}(\CW),C^-(\mathbb Z))$, 
\end{itemize}
derive trivially for the $(\mathbb D^1,usu)$ and $(\mathbb I^1,usu)$ model structures.

\item[(iii)]The functors 
\begin{itemize}
\item $\widetilde\Cw^*:\PSh(\SmVar(\mathbb C),C^-(\mathbb Z))\to\PSh(\CW,C^-(\mathbb Z))$ and
\item $\widetilde\Cw^*:\PSh(\Cor^{fs}_{\mathbb Z}(\SmVar(\mathbb C)),C^-(\mathbb Z))\to\PSh(\Cor^{fs}_{\mathbb Z}(\CW),C^-(\mathbb Z))$, 
\end{itemize}
derive trivially for the $(\mathbb A^1,et)$ and $(\mathbb I^1,usu)$ model structures.
\end{itemize}
\end{prop}

\begin{proof}

\noindent(i): It is proved in \cite{AyoubG2}.

\noindent(ii): Let us prove (ii).
\begin{itemize}
\item Let $K^{\bullet}\in\PSh(\CW,C^-(\mathbb Z))$ be an  $\mathbb I^1$ local object.
Let $h:K^{\bullet}\to L^{\bullet}$ an equivalence usu local with $L^{\bullet}$ usu fibrant.
Then, 
\begin{itemize}
\item $L^{\bullet}$ is $\mathbb I^1$ local,
$S(L^{\bullet}):L^{\bullet}\to\underline{\sing}_{\mathbb I^*}L^{\bullet}$ is an equivalence usu local, 
\item $\underline{\sing}_{\mathbb I^*}L^{\bullet}$ is usu fibrant.
\end{itemize}
Since $\Cw_*$ preserve usu local equivalence and usu fibrant object 
(for $X\in\AnSm(\mathbb C)$, the restriction of the functor $\Cw$ to the small site of open subset of $X$ is fully faithfull), 
\begin{itemize}
\item $\Cw_*(S(L^{\bullet})\circ h)=(\Cw_*S(L^{\bullet}))\circ(\Cw_*h):
\Cw_*K^{\bullet}\to\Cw_*\underline{\sing}_{\mathbb I^*}L^{\bullet}$ is an equivalence usu local
\item $\Cw_*\underline{\sing}_{\mathbb I^*}L^{\bullet}$ is usu fibrant.
\end{itemize}
Now, for $X\in\AnSm(\mathbb C)$,
\begin{eqnarray*}
H^n(\underline{\sing}_{\mathbb I^*}L^{\bullet}(p_{X^{cw}})):
\Hom_{P^-}(\mathbb Z(X)[n],\Cw_*\underline{\sing}_{\mathbb I^*}L^{\bullet})
=H^n(\underline{\sing}_{\mathbb I^*}L^{\bullet}(X^{cw}))\to \\
\Hom(\mathbb Z(X\times\mathbb D^1)[n],\Cw_*\underline{\sing}_{\mathbb I^*}L^{\bullet})
=\underline{\sing}_{\mathbb I^*}L^{\bullet}(X^{cw}\times\mathbb D^1)[n]
\end{eqnarray*}
Moreover, the map 
$p_{X^{cw}}=I_{X^{cw}}\times a_{\mathbb D^1}:X^{cw}\times\mathbb D^1\to X^{cw}$ is an homotopy equivalence.
Hence, by lemma \ref{I1hpt}(ii), 
\begin{equation*}
\underline{\sing}_{\mathbb I^*}L^{\bullet}(p_{X^{cw}}):
\underline{\sing}_{\mathbb I^*}L^{\bullet}(X^{cw})\to\underline{\sing}_{\mathbb I^*}L^{\bullet}(X^{cw}\times\mathbb D^1)
\end{equation*}
is a quasi-isomorphism. 
In particular $H^n(\underline{\sing}_{\mathbb I^*}L^{\bullet}(p_{X^{cw}}))$ is an isomorphism.
This proves that $\Cw_*\underline{\sing}_{\mathbb I^*}L^{\bullet}$ is $\mathbb D^1$ local. 
Now, since 
\begin{itemize}
\item $\Cw_*(S(L^{\bullet})\circ h)$  is an equivalence usu local
\item $\Cw_*\underline{\sing}_{\mathbb I^*}L^{\bullet}$ is $\mathbb D^1$ local, 
\end{itemize}
$\Cw_*K^{\bullet}$ is $\mathbb D^1$ local. 

\item Let $K^{\bullet}\in\PSh(\Cor^{fs}_{\mathbb Z}(\CW),C^-(\mathbb Z))$ be an  $\mathbb I^1$ local object.
By lemma \ref{CWTr}, $\Tr_*K$ is $\mathbb I^1$ local
Hence by above, $\Tr_*\Cw_*K^{\bullet}=\Cw_*\Tr_*K^{\bullet}$ is $\mathbb D^1$ local.
Thus, by lemma \ref{CWTr} $\Tr_*K^{\bullet}$ is $\mathbb D^1$ local.

\item Let $f:G_1\to G_2$ an equivalence $(\mathbb D^1,usu)$ local in $\PSh(\AnSm(\mathbb C),C^-(\mathbb Z))$.
Let $K\in\PSh(\CW,C^-(\mathbb Z))$ a $\mathbb I^1$ local object. Up to replace $K$ by an usu equivalent presheaf,
we may assume that $K$ is usu fibrant. Then $\Cw_*K$ is also usu fibrant.
Consider the following commutative diagram :
\begin{equation}
\xymatrix{
\Hom_{P^-(CW)}(\Cw^*G_2,K)\ar[d]^{\sim}\ar[rr]^{\mathbb Z(K)(\Cw^*f)} & \, & \Hom_{P^-(CW)}(\Cw^*G_1,K)\ar[d]^{\sim} \\
\Hom_{P^-(\An)}(G_2,\Cw_*K)\ar[rr]^{\mathbb Z(\Cw_*K)(f)} & \, & \Hom_{P^-(\An)}(G_1,\Cw_*K)}
\end{equation}
By above $\Cw_*K$ is $\mathbb D^1$ local, and $f$ is an equivalence $(\mathbb D^1,usu)$ local.
Hence, $\mathbb Z(\Cw_*K)(f)$ is an isomorphism since $\Cw_*K$ is usu fibrant. 
The diagram shows then that $\mathbb Z(K)(\Cw^*f)$ is an isomorphism.
This proves that $\Cw^*f:\Cw^*G_1\to\Cw^*G_2$ is an equivalence $(\mathbb I^1,usu)$ local since $K$ is usu fibrant.

\item Let $f:G_1\to G_2$ an equivalence $(\mathbb D^1,usu)$ local in $\PSh(\Cor^{fs}_{\mathbb Z}(\AnSm(\mathbb C)),C^-(\mathbb Z))$.
Let $K\in\PSh(\Cor^{fs}_{\mathbb Z}(\CW),C^-(\mathbb Z))$ a $\mathbb I^1$ local object.
Up to replace $K$ by an usu equivalent presheaf, we may assume that $K$ is usu fibrant. Then $\Cw_*K$ is also usu fibrant.
Consider the following commutative diagram :
\begin{equation*}
\xymatrix{
\Hom_{PC^-(CW)}(\Cw^*G_2,K)\ar[d]^{\sim}\ar[rr]^{\mathbb Z(K)(\Cw^*f)} & \, & \Hom_{PC^-(CW)}(\Cw^*G_1,K)\ar[d]^{\sim} \\
\Hom_{PC^-(\An)}(G_2,\Cw_*K)\ar[rr]^{\mathbb Z(\Cw_*K)(f)} & \, & \Hom_{PC^-(\An)}(G_1,\Cw_*K)}
\end{equation*}
By above $\Cw_*K$ is $\mathbb D^1$ local, and $f$ is an equivalence $(\mathbb D^1,usu)$ local.
Hence, $\mathbb Z(\Cw_*K)(f)$ is an isomorphism since $\Cw_*K$ is usu fibrant. 
The diagram shows then that $\mathbb Z(K)(\Cw^*f)$ is an isomorphism.
This proves that $\Cw^*f:\Cw^*G_1\to\Cw^*G_2$ is an equivalence $(\mathbb I^1,usu)$ local
since $K$ is usu fibrant.
\end{itemize}

\noindent(iii): Follows from (i) and (ii) since by definition $\widetilde\Cw=\Cw\circ\An$.

\end{proof}

\subsection{Ayoub's Betti realization functor and the Betti realisation functor via CW commplexes}

We recall the definition of Ayoub's realization functor :

\begin{defi}\cite{AyoubB}\cite{AyoubG2}
\begin{itemize}
\item[(i)] The Betti realisation functor (without transfers) is the composite :
\begin{equation}
\Bti_0^*:\DA^{-}(\mathbb C,\mathbb Z)\xrightarrow{\An^*}\AnDA^{-}(\mathbb Z)\xrightarrow{Re_{an*}}D^{-}(\mathbb Z)
\end{equation}

\item[(ii)] The Betti realisation functor with transfers is the composite :
\begin{equation}
\Bti^*:\DM^{-}(\mathbb C,\mathbb Z)\xrightarrow{\An^*}\AnDM^{-}(\mathbb Z)\xrightarrow{Re^{tr}_{an*}}D^{-}(\mathbb Z)
\end{equation}
\end{itemize}
Since $\An^*$ derive trivially by proposition \ref{AnCwtCw}(i)
and $L\Tr^*:\AnDA^-(\mathbb Z)\to\AnDM^-(\mathbb Z)$ is the inverse of $\Tr_*$ by theorem \ref{AnTr}(i), we have 
$\Bti_0^*=\Bti^*\circ L\Tr^*$.
\end{defi}

We now define a Betti realization functor via CW complexes. 
The main result of this subsection will be that this functor coincide with Ayoub's one.

\begin{defi}
\begin{itemize}
\item[(i)] The CW-Betti realization functor (without transfers) is the composite :
\begin{equation}
\widetilde{\Bti}_0^*:\DA^{-}(\mathbb C,\mathbb Z)\xrightarrow{\widetilde\Cw^*}\CwDA^{-}(\mathbb Z)
\xrightarrow{Re_{cw*}}D^{-}(\mathbb Z)
\end{equation}

\item[(ii)] The CW-Betti realisation functor with transfers is the composite :
\begin{equation}
\widetilde\Bti^*:\DM^{-}(\mathbb C,\mathbb Z)\xrightarrow{\widetilde\Cw^*}\CwDM^{-}(\mathbb Z)\xrightarrow{Re^{tr}_{cw*}}D^{-}(\mathbb Z)
\end{equation}
\end{itemize}
Since $\widetilde\Cw^*$ derive trivially by proposition \ref{AnCwtCw}(ii)
and $L\Tr^*:\CwDA^-(\mathbb Z)\to\CwDM^-(\mathbb Z)$ is the inverse of $\Tr_*$ by theorem \ref{CWMot}(i), we have 
$\widetilde\Bti_0^*=\widetilde\Bti^*\circ L\Tr^*$.
\end{defi}

Denote by $i_*:\mathbb I^*_{et}\hookrightarrow\square^*_{\mathbb C}=\square^*$ the embeddings
of  pro-objects in $\SmVar(\mathbb C)$ 
\begin{equation*}
\mathbb I^*_{et}:=(U_j,a_{jk})_{j,k\in\mathcal V_{et}([0,\infty]^*,\square^*_{\mathbb C}),j\leq k}
\end{equation*}
indexed by the filtrant category of etale neighborhood of $[0,\infty]^*$ in $\square^*_{\mathbb C}$, 
The morphism $i_*$ is given by the etale morphisms $i_*(l):U_{l}\to\square^*_{\mathbb C}$ associated to
$l\in\mathcal V_{et}([0,\infty]^*,\square^*_{\mathbb C})$.

Denote by $i'_*:\mathbb I^*_{an}\xrightarrow{i'_{1*}}\bar{\mathbb D}^*\xrightarrow{i'_{2*}}\mathbb A^*_{\mathbb C}$ 
the embeddings of the pro-objects in $\AnSm(\mathbb C)$ 
\begin{itemize}
\item $\mathbb I^*_{an}:=(U_j,a_{jk})_{j,k\in\mathcal V_{an}(\mathbb I^*,\mathbb A^*_{\mathbb C}),j\leq k}$ indexed
by the filtrant category of etale analytic neighborhood of $\mathbb I^*$ in $\mathbb A^*_{\mathbb C}$, and
\item $\bar{\mathbb D}^*:=(U_l,a_{lm})_{l,m\in\mathcal V_{an}(\bar{\mathbb D}^*,\mathbb A^*_{\mathbb C}),l\leq m}$ indexed
by the filtrant category of etale analytic neighborhood of $\mathbb I^*$ in $\mathbb A^*_{\mathbb C}$.
\end{itemize}
The morphism $i'_{1*}$ is given by the identities $i'_{1*}(l):U_{\tau(l)}=U_l\xrightarrow{I_{U_l}}U_l$, where 
$\tau:\mathcal V_{an}(\bar{\mathbb D}^*,\mathbb A^*_{\mathbb C})\to\mathcal V_{an}(\mathbb I^*,\mathbb A^*_{\mathbb C})$ 
is the natural embedding of categories.

The embedding $i_*:\mathbb I_{et}^*\hookrightarrow\square^*$ of pro-objets in $\SmVar(\mathbb C)$ gives for, 
$F^{\bullet}\in\PSh(\Cor^{fs}_{\mathbb Z}(\SmVar(\mathbb C)),C^-(\mathbb Z))$, 
the following morphism in $\PSh(\Cor^{fs}_{\mathbb Z}(\SmVar(\mathbb C)),C^-(\mathbb Z))$, 
\begin{equation}
\underline{F}^{\bullet}(i):\underline{C}_*F^{\bullet}\to\underline{\sing}_{\mathbb I_{et}^*}F^{\bullet}
\end{equation}
given by for $X\in\SmVar(\mathbb C)$, the morphism of complexes
\begin{equation}
\underline{F}^{\bullet}(i)(X):F^{\bullet}(X\times\square^*)\xrightarrow{F^{\bullet}(I_X\times i_*)}
F^{\bullet}(X\times\mathbb I_{et}^*)
\end{equation}

The morphism $i'_{1*}:\mathbb I^*_{an}\to\bar{\mathbb D}^*$ of pro-objet in $\AnSm(\mathbb C)$ gives for, 
$G^{\bullet}\in\PSh(\Cor^{fs}_{\mathbb Z}(\AnSm(\mathbb C)),C^-(\mathbb Z))$, 
the following morphism in $\PSh(\Cor^{fs}_{\mathbb Z}(\AnSm(\mathbb C)),C^-(\mathbb Z))$, 
\begin{equation}
\underline{F}^{\bullet}(i'_{1*}):
\underline{\sing}_{\bar{\mathbb D}^*}F^{\bullet}\to\underline{\sing}_{\mathbb I_{an}^*}F^{\bullet}
\end{equation}
given by for $X\in\SmVar(\mathbb C)$, the morphism of complexes of abelian groups
\begin{equation}
\underline{F}^{\bullet}(i'_{1*})(X)=F^{\bullet}(I_X\times i'_{1*}):
F^{\bullet}(X\times\bar{\mathbb D}^*)\to F^{\bullet}(X\times\mathbb I_{an}^*).
\end{equation}

We have two canonical morphism of functors :
\begin{itemize}
\item the morphism $\psi^{\Cw}$, which,
for $G^{\bullet}\in\PSh(\Cor^{fs}_{\mathbb Z}(\AnSm(\mathbb C)),C^-(\mathbb Z))$, associate the morphism 
\begin{equation*}
\psi^{\Cw}(G^{\bullet}):\Cw^*(\underline{\sing}_{\mathbb I_{an}^*}G^{\bullet})\to\underline{\sing}_{\mathbb I^*}\Cw^*G^{\bullet}
\end{equation*}
in $\PSh(\Cor^{fs}_{\mathbb Z}(\CW),C^-(\mathbb Z))$ ;
the morphism $\psi^{\Cw}(G^{\bullet})$ is given by, for $Z\in\CW$,
\begin{equation}
\psi^{\Cw}(G^{\bullet})(Z):
\lim_{X^{cw}\to Z}G^{\bullet}(X\times\mathbb I_{an}^*)\to\lim_{Y^{cw}\to Z\times\mathbb I^*}G^{\bullet}(Y)
\end{equation}
given by 
$(f:X^{cw}\to Z)\mapsto(f\times I_{\mathbb I^*}:(X\times\mathbb I_{an}^*)^{cw}\to Z\times\mathbb I^*)$
and the identity of $G^{\bullet}(X\times\mathbb I_{an}^*)$ ;

\item the morphism $\psi^{\widetilde{\Cw}}$, which,
for $F^{\bullet}\in\PSh(\Cor^{fs}_{\mathbb Z}(\SmVar(\mathbb C)),C^-(\mathbb Z))$, associate the morphism 
\begin{equation*}
\psi^{\widetilde{\Cw}}(F^{\bullet}):\widetilde{\Cw}^*(\underline{\sing}_{\mathbb I_{et}^*}F^{\bullet})
\to\underline{\sing}_{\mathbb I^*}\widetilde{\Cw}^*F^{\bullet}
\end{equation*}
in $\PSh(\Cor^{fs}_{\mathbb Z}(\CW),C^-(\mathbb Z))$ ;
the morphism $\psi^{\widetilde{\Cw}}(F^{\bullet})$ is given by, for $Z\in\CW$,
\begin{equation}
\psi^{\widetilde{\Cw}}(F^{\bullet})(Z):
\lim_{X^{cw}\to Z}F^{\bullet}(X\times\mathbb I_{et}^*)\to\lim_{Y^{cw}\to Z\times\mathbb I^*}F^{\bullet}(Y)
\end{equation}
given by 
$(f:X^{cw}\to Z)\mapsto(f\times I_{\mathbb I^*}:(X\times\mathbb I_{et}^*)^{cw}\to Z\times\mathbb I^*)$
and the identity of $F^{\bullet}(X\times\mathbb I_{et}^*)$.
\end{itemize}

\begin{defi}
We define the following two morphism of functors :
\begin{itemize}
\item[(i)] the morphism $W$, which,
for $G^{\bullet}\in\PSh(\Cor^{fs}_{\mathbb Z}(\AnSm(\mathbb C)),C^-(\mathbb Z))$, associate the composition
\begin{equation*}
W(G^{\bullet}):
\Cw^*(\underline{\sing}_{\bar{\mathbb D}^*}G^{\bullet})\xrightarrow{\Cw^*(\underline{G}^{\bullet}(i'_1))}
\Cw^*(\underline{\sing}_{\mathbb I_{an}^*}G^{\bullet})\xrightarrow{\psi^{\Cw}(G^{\bullet})}
\underline{\sing}_{\mathbb I^*}\Cw^*G^{\bullet}
\end{equation*}
in $\PSh(\Cor^{fs}_{\mathbb Z}(\CW),C^-(\mathbb Z))$,

\item[(ii)] the morphism $\widetilde W$, which,
for $F^{\bullet}\in\PSh(\Cor^{fs}_{\mathbb Z}(\SmVar(\mathbb C)),C^-(\mathbb Z))$, associate the composition
\begin{equation*}
\widetilde W(F^{\bullet}):
\widetilde\Cw^*(\underline{C}_*F^{\bullet})\xrightarrow{\widetilde\Cw^*(\underline{F}^{\bullet}(i))}
\widetilde\Cw^*(\underline{\sing}_{\mathbb I_{et}^*}F^{\bullet})\xrightarrow{\psi^{\widetilde\Cw}(F^{\bullet})}
\underline{\sing}_{\mathbb I^*}\widetilde\Cw^*F^{\bullet}
\end{equation*}
in $\PSh(\Cor^{fs}_{\mathbb Z}(\CW),C^-(\mathbb Z))$.
\end{itemize}
\end{defi}

\begin{prop}\label{AWtW}
\begin{itemize}
\item[(i)] For $G^{\bullet}\in\PSh(\Cor^{fs}_{\mathbb Z}(\AnSm(\mathbb C)),C^-(\mathbb Z))$, 
$W(G^{\bullet}):\Cw^*(\underline{\sing}_{\bar{\mathbb D}^*}G^{\bullet})\to\underline{\sing}_{\mathbb I^*}\Cw^*G^{\bullet}$
is an equivalence $(\mathbb I^1,usu)$ local in $\PSh(\Cor^{fs}_{\mathbb Z}(\CW),C^-(\mathbb Z))$,

\item[(ii)] For $F^{\bullet}\in\PSh(\Cor^{fs}_{\mathbb Z}(\SmVar(\mathbb C)),C^-(\mathbb Z))$,
$\widetilde W(F^{\bullet}):\widetilde\Cw^*(\underline{C}_*F^{\bullet})\to\underline{\sing}_{\mathbb I^*}\widetilde\Cw^*F^{\bullet}$
is an $(\mathbb I^1,usu)$ local equivalence in $\PSh(\Cor^{fs}_{\mathbb Z}(\CW),C^-(\mathbb Z))$.
\end{itemize}
\end{prop}

\begin{proof}

\noindent(i):Consider the following commutative diagram
\begin{equation}\label{WG}
\xymatrix{\Cw^*G^{\bullet}\ar[d]_{\Cw^*(S(G^{\bullet}))}\ar[rrd]^{S(\Cw^*G^{\bullet})} & \, & \, \\
\Cw^*(\underline{\sing}_{\bar{\mathbb D}^*}G^{\bullet})\ar[rr]^{W(G^{\bullet})} & \, & 
\underline{\sing}_{\mathbb I^*}\Cw^*G^{\bullet}}
\end{equation}
\begin{itemize}
\item By theorem \ref{AnS}(ii) $S(G^{\bullet})$ is a $(\mathbb D^1,et)$ equivalence. 
Hence by proposition \ref{AnCwtCw} (ii) $\Cw^*S(G^{\bullet})$ is a $(\mathbb I^1,usu)$ equivalence.
\item By theorem \ref{CWsingI} (ii) $S(\Cw^*G^{\bullet})$ is a $(\mathbb I^1,usu)$ equivalence.
\end{itemize}
Now, since $\Cw^*S(G^{\bullet})$ and $S(\Cw^*G^{\bullet})$ are $(\mathbb I^1,usu)$ equivalence, 
the diagram (\ref{WG}) shows that $W(G^{\bullet})$ is a $(\mathbb I^1,usu)$ equivalence.

\noindent(ii): Consider the following commutative diagram
\begin{equation}\label{WtF}
\xymatrix{\widetilde\Cw^*F^{\bullet}\ar[d]_{\widetilde\Cw^*(S(F^{\bullet}))}\ar[rrd]^{S(\widetilde\Cw^*F^{\bullet})} \, & \, \\
\widetilde\Cw^*(\underline{C}_*F^{\bullet})\ar[rr]^{\widetilde W(F^{\bullet})} & \, & 
\underline{\sing}_{\mathbb I^*}\widetilde\Cw^*F^{\bullet}}
\end{equation}
\begin{itemize}
\item By theorem \ref{VarsingA}(ii) $S(F^{\bullet})$ is a $(\mathbb A^1,et)$ equivalence. 
Hence, by proposition \ref{AnCwtCw} (iii) $\widetilde\Cw^*S(F^{\bullet})$ is a $(\mathbb I^1,usu)$ equivalence.
\item By theorem \ref{CWsingI} (ii) $S(\widetilde\Cw^*F^{\bullet})$ is a $(\mathbb I^1,usu)$ equivalence.
\end{itemize}
Now, since $\widetilde\Cw^*S(F^{\bullet})$ and $S(\widetilde\Cw^*F^{\bullet})$ are $(\mathbb I^1,usu)$ equivalence, 
the diagram (\ref{WtF}) shows that $\widetilde W(F^{\bullet})$ is a $(\mathbb I^1,usu)$ equivalence.

\end{proof}

\begin{defi}
We define the morphism of functor $B$, by associating
to $G^{\bullet}\in\PSh(\Cor^{fs}_{\mathbb Z}(\AnSm(\mathbb C)),C^-(\mathbb Z))$, the composite 
\begin{equation*}
B(G^{\bullet}):
\underline{\sing}_{\bar{\mathbb D}^*}G^{\bullet}\xrightarrow{\ad(\Cw^*,\Cw_*)(\underline{\sing}_{\bar{\mathbb D}^*}G^{\bullet})}
\Cw_*\Cw^*\underline{\sing}_{\bar{\mathbb D}^*}G^{\bullet}\xrightarrow{\Cw_*(W(G^{\bullet}))}
\Cw_*\underline{\sing}_{\mathbb I^*}\Cw^*G^{\bullet}
\end{equation*}
in $\PSh(\Cor^{fs}_{\mathbb Z}(\AnSm(\mathbb C)),C^-(\mathbb Z))$
\end{defi}

We have now the following key proposition.

\begin{prop}\label{Bettiprop}
\begin{itemize}
\item[(i)] For $Y\in\AnSp(\mathbb C)$ and $E\subset Y$ an analytic subset,
\begin{equation*}
e^{tr}_{an*}(B(\mathbb Z_{tr}(Y,E))):
\sing_{\bar{\mathbb D^*}}\mathbb Z_{tr}(Y,E)\to\sing_{\mathbb I^*}\mathbb Z_{tr}(Y^{cw},E^{cw})
\end{equation*}
is a quasi isomorphism in $C^{-}(\mathbb Z)$.

\item[(ii)] For $Y\in\AnSp(\mathbb C)$, and $E\subset Y$ an analytic subset, the morphism 
\begin{equation*}
B(\mathbb Z_{tr}(Y,E)):
\underline{\sing}_{\bar{\mathbb D}^*}\mathbb Z_{tr}(Y,E)\to\Cw_*\underline{\sing}_{\mathbb I^*}\mathbb Z_{tr}(Y^{cw},E^{cw})
\end{equation*}
is an equivalence $(\mathbb D^1,usu)$ local in $\PSh(\Cor^{fs}_{\mathbb Z}(\AnSm(\mathbb C)),C^-(\mathbb Z))$.
\end{itemize}
\end{prop}

\begin{proof}

\noindent(i):
\begin{itemize}
\item Consider first $Y\in\AnSm(\mathbb C)$.
By proposition \ref{Dcov}, there exist a covering $Y=\cup_{i\in J}D_i$ by a countable family of open balls 
$D_i\simeq\mathbb D^{d_Y}$ such that $D_{I}=\emptyset$ or $D_{I}\simeq\mathbb D^{d_Y}$, 
for all $I=\left\{i_1,\cdots i_l\right\}\subset J$. Denote by $j_{I}:D_I\hookrightarrow Y$ the open embedding. 
We have then the following commutative diagram in $\PSh(\Cor^{fs}_{\mathbb Z}(\AnSm(\mathbb C)),C^-(\mathbb Z))$ :
\begin{equation}\label{bettidia1}
\xymatrix{
\underline{\sing}_{\bar{\mathbb D}^*} \mathbb Z_{tr}(Y)\ar[rrr]^{B(\mathbb Z_{tr}(Y))} & \, & \, &
 \Cw_*\underline{\sing}_{\mathbb I^*}\mathbb Z_{tr}(Y^{cw}) \\
\Tot_{\bullet,*}(\oplus_{\card I=\bullet}\underline{\sing}_{\bar{\mathbb D}^*}\mathbb Z_{tr}(D_{I}))
\ar[rrr]^{\Tot(B(\mathbb Z_{tr}(D_{I})))}\ar[u]^{\Tot(S(\mathbb Z(j_I)))} & \, & \, & \, 
\Tot_{\bullet,*}(\oplus_{\card I=\bullet}\Cw_*\underline{\sing}_{\mathbb I^*}\mathbb Z_{tr}(D^{cw}_{I}))
\ar[u]^{\Tot(\Cw_*S(\mathbb Z(j^{cw}_I)))}}
\end{equation}
This gives, after applying the functor $e^{tr}_{an*}$ to (\ref{bettidia1}), the commutative diagram in $C^-(\mathbb Z)$:
\begin{equation}\label{betti}
\xymatrix{
\sing_{\bar{\mathbb D}^*}\mathbb Z_{tr}(Y)\ar[rrr]^{e^{tr}_{an*}B(\mathbb Z_{tr}(Y))} & \, & \, &
 \sing_{\mathbb I^*}\mathbb Z_{tr}(Y^{cw}) \\
\Tot_{\bullet,*}(\oplus_{\card I=\bullet}\sing_{\bar{\mathbb D}^*}\mathbb Z_{tr}(D_{I}))
\ar[rrr]^{\Tot(e^{tr}_{an*}B(\mathbb Z_{tr}(D_{I})))}\ar[u]^{\Tot(e^{tr}_{an*}S(\mathbb Z(j_I)))}
\ar[d]_{\Tot(H_0(\sing_{\bar{\mathbb D}^*}\mathbb Z_{tr}(D_{I})))}& \, & \, &
\Tot_{\bullet,*}(\oplus_{\card I=\bullet}\sing_{\mathbb I^*}\mathbb Z_{tr}(D^{cw}_{I}))
\ar[u]^{\Tot(e^{tr}_{cw*}S(\mathbb Z(j^{cw}_I)))}\ar[d]^{\Tot(H_0(\sing_{\mathbb I^*}\mathbb Z_{tr}(D^{cw}_{I})))} \\
\oplus_{\card I=\bullet}H_0(\sing_{\bar{\mathbb D}^*}\mathbb Z_{tr}(D_{I}))\ar[rrr]^{H_0(e^{tr}_{an*}B(\mathbb Z_{tr}(D_{I})))} & 
\,& \, & \oplus_{\card I=\bullet}H_0(\sing_{\mathbb I^*}\mathbb Z_{tr}(D^{cw}_{I}))} 
\end{equation}
Since $D_I\simeq\mathbb D^{d_Y}$,
\begin{itemize}
\item $\Tot(H_0(\sing_{\mathbb I^*}\mathbb Z_{tr}(D^{cw}_{I}))):
\Tot_{\bullet,*}(\oplus_{\card I=\bullet}\sing_{\mathbb I^*}\mathbb Z_{tr}(D^{cw}_{I}))\to
\oplus_{\card I=\bullet}H_0(\sing_{\mathbb I^*}\mathbb Z_{tr}(D^{cw}_{I}))$
is a quasi-isomorphism in $C^{-}(\mathbb Z)$ by proposition \ref{CWadTr(iii)} ;
\item $\Tot(H_0(\sing_{\bar{\mathbb D}^*}\mathbb Z_{tr}(D_I))):
\Tot_{\bullet,*}(\oplus_{\card I=\bullet}\sing_{\bar{\mathbb D}^*}\mathbb Z_{tr}(D_{I}))\to
\oplus_{\card I=\bullet}H_0(\sing_{\bar{\mathbb D}^*}\mathbb Z_{tr}(D_{I}))$ 
is a quasi-isomorphism in $C^{-}(\mathbb Z)$
since $\sing_{\bar{\mathbb D}^*}\mathbb Z_{tr}(D_I)\to\mathbb Z_{tr}(\pt)$ is an equivalence
usu local, since, by theorem \ref{AnS}(ii), it is an equivalence $(\mathbb D^1,usu)$ local between
$\mathbb D^1$ local objects ; 
\item $H_0(e^{tr}_{an*}B(\mathbb Z_{tr}(D_{I}))):
\oplus_{\card I=\bullet}H_0(\sing_{\bar{\mathbb D}^*}\mathbb Z_{tr}(D_{I}))\to
\oplus_{\card I=\bullet}H_0(\sing_{\mathbb I^*}\mathbb Z_{tr}(D^{cw}_{I}))$
is a quasi-isomorphism in $C^{-}(\mathbb Z)$ by proposition by the two preceding points.  
\end{itemize}
Hence,
\begin{equation*}
\Tot(e^{tr}_{an*}B(\mathbb Z_{tr}(D_{I}))):
\Tot_{\bullet,*}(\oplus_{\card I=\bullet}\sing_{\bar{\mathbb D}^*}\mathbb Z_{tr}(D_{I}))\to
\Tot_{\bullet,*}(\oplus_{\card I=\bullet}\sing_{\mathbb I^*}\mathbb Z_{tr}(D^{cw}_{I}))
\end{equation*}
is a quasi-isomorphism in $C^{-}(\mathbb Z)$. 
On the other hand,
\begin{itemize}
\item $\Tot(e^{tr}_{an*}S(\mathbb Z(j_I))):
\Tot_{\bullet,*}(\oplus_{\card I=\bullet}\sing_{\bar{\mathbb D}^*}\mathbb Z_{tr}(D_{I}))\to
\sing_{\bar{\mathbb D}^*} \mathbb Z_{tr}(Y)$ is a quasi-isomorphism in $C^{-}(\mathbb Z)$
since 
\begin{equation*}
\Tot(S(\mathbb Z(j_I))):
\Tot_{\bullet,*}(\oplus_{\card I=\bullet}\underline{\sing}_{\bar{\mathbb D}^*}\mathbb Z_{tr}(D_{I}))\to
\underline{\sing}_{\bar{\mathbb D}^*} \mathbb Z_{tr}(Y)
\end{equation*}
is a $(\mathbb D^1,usu)$ local equivalence in $\PSh(\Cor^{fs}_{\mathbb Z}(\AnSm(\mathbb C)),C^-(\mathbb Z))$
by lemma \ref{HRAn} and theorem \ref{AnS}(ii),
\item $\Tot(e^{tr}_{cw*}S(\mathbb Z(j^{cw}_I))):
\Tot_{\bullet,*}(\oplus_{\card I=\bullet}\sing_{\mathbb I^*}\mathbb Z_{tr}(D^{cw}_{I}))\to
\sing_{\mathbb I^*}\mathbb Z_{tr}(Y^{cw})$ is a quasi-isomorphism in $C^{-}(\mathbb Z)$
since 
\begin{equation*}
\Tot(S(\mathbb Z(j^{cw}_I))):
\Tot_{\bullet,*}(\oplus_{\card I=\bullet}\underline{\sing}_{\mathbb I^*}\mathbb Z_{tr}(D^{cw}_{I}))\to
\underline{\sing}_{\mathbb I^*}\mathbb Z_{tr}(Y^{cw})
\end{equation*}
is an equivalence $(\mathbb I^1,usu)$ local in $\PSh(\Cor^{fs}_{\mathbb Z}(\CW),C^-(\mathbb Z))$ 
by lemma \ref{HRCW} and theorem \ref{CWsingI}(ii)  
\end{itemize}
The diagram (\ref{betti}) then shows that
\begin{equation*} 
e^{tr}_{an*}B(\mathbb Z_{tr}(Y)):
\sing_{\bar{\mathbb D}^*}\mathbb Z_{tr}(Y)\to\sing_{\mathbb I^*}\mathbb Z_{tr}(Y^{cw}) 
\end{equation*}
is a quasi-isomorphism.

\item Consider now $Y\in\AnSp(\mathbb C)$. 
We prove that $e^{tr}_{an*}B(\mathbb Z_{tr}(Y))$ is a quasi-isomorphism by induction on $d_Y=\dim Y$.
If $d_Y=0$ there is nothing to prove.
By theorem \ref{HirDes}, there exist a proper modification
$\epsilon:Y'\to Y$ such that $Y'\in\AnSm(\mathbb C)$ and $E=\epsilon^{-1}(Z)\subset Y'$ is a normal crossing divisor.
Denote by $l:Z\hookrightarrow Y$ and $l':E\hookrightarrow Y'$ the closed embeddings and 
$\epsilon_Z:E\to Z$ the morphism such that $l\circ\epsilon_Z=\epsilon\circ l'$.
We have then, with 
\begin{itemize}
\item $b=\epsilon_{Z*}+l'_*=S(\mathbb Z_{tr}(\epsilon_Z)\oplus\mathbb Z_{tr}(l'))$ and
$a=\epsilon_{*}+l_*=S(\mathbb Z_{tr}(\epsilon)\oplus\mathbb Z_{tr}(l))$,
\item $d=\epsilon^{cw}_{Z*}+l'^{cw}_*=S(\mathbb Z_{tr}(\epsilon^{cw}_Z)\oplus\mathbb Z_{tr}(l'^{cw}))$ and
$c=\epsilon^{cw}_{*}+l^{cw}_*=S(\mathbb Z_{tr}(\epsilon^{cw})\oplus\mathbb Z_{tr}(l^{cw}))$,
\end{itemize}
the following commutative diagram in $\PSh(\Cor^{fs}_{\mathbb Z}(\AnSm(\mathbb C)),C^-(\mathbb Z))$ :
\begin{equation}\label{bettidia2}
\xymatrix{
0\ar[r] & \underline{\sing}_{\bar{\mathbb D}^*}\mathbb Z_{tr}(E)\ar[r]^{b}\ar[d]^{B(\mathbb Z_{tr}(E))} & 
\underline{\sing}_{\bar{\mathbb D}^*}\mathbb Z_{tr}(Z)\oplus\underline{\sing}_{\bar{\mathbb D}^*}\mathbb Z_{tr}(Y')
\ar[r]^{a}\ar[d]^{B(\mathbb Z_{tr}(Z)\oplus\mathbb Z_{tr}(Y'))} & 
\underline{\sing}_{\bar{\mathbb D}^*}\mathbb Z_{tr}(Y)\ar[d]^{B(\mathbb Z_{tr}(Y))} \\
0\ar[r] & \Cw_*\underline{\sing}_{\mathbb I^*}\mathbb Z_{tr}(E^{cw})\ar[r]^{\Cw_*(d)} &  
\Cw_*\underline{\sing}_{\mathbb I^*}\mathbb Z_{tr}(Z^{cw})\oplus\underline{\sing}_{\mathbb I^*}\mathbb Z_{tr}(Y'^{cw})
\ar[r]^{\Cw_*(c)} &  \Cw_*\underline{\sing}_{\mathbb I^*}\mathbb Z_{tr}(Y^{cw})}
\end{equation}
But, 
\begin{itemize}
\item $\left[0\to\sing_{\bar{\mathbb D}^*}\mathbb Z_{tr}(E)\xrightarrow{b}
\sing_{\bar{\mathbb D}^*}\mathbb Z_{tr}(Z)\oplus\sing_{\bar{\mathbb D}^*}\mathbb Z_{tr}(Y')\right]
\xrightarrow{a}\sing_{\bar{\mathbb D}^*}\mathbb Z_{tr}(Y)$
is a quasi-isomorphism by proposition \ref{DesAn}, 
\item $\left[0\to\sing_{\mathbb I^*}\mathbb Z_{tr}(E^{cw})\xrightarrow{d}
\sing_{\mathbb I^*}\mathbb Z_{tr}(Z^{cw})\oplus\sing_{\mathbb I^*}\mathbb Z_{tr}(Y'^{cw})\right]
\xrightarrow{c}\sing_{\mathbb I^*}\mathbb Z_{tr}(Y^{cw})$
is a quasi-isomorphism by proposition \ref{SerreP} and the fact that the 2-cubical
diagram associated to a proper modification is of cohomological descent (see e.g. \cite{PS}), 
\item as $Y'\in\SmVar(\mathbb C)$ is smooth and  $\dim E=d_Y-1<d_Y$, $\dim Z<d_Y$,
we have, by the smooth case we proved above and by the induction hypothesis, 
that $e^{tr}_{an*}B(\mathbb Z_{tr}(Z)\oplus\mathbb Z_{tr}(Y'))$ and $e^{tr}_{an*}B(\mathbb Z_{tr}(E))$
are quasi-isomorphisms,
\end{itemize}
thus, the diagram (\ref{bettidia2}) shows that $e^{tr}_{an*}B(\mathbb Z_{tr}(Y))$ is a quasi-isomorphism.

\item Consider now $Y\in\AnSp(\mathbb C)$ as before and $E\subset Y$ an analytic subspace.
Denote by $l:E\hookrightarrow Y$ the locally closed embedding. 
We have then the following commutative diagram in $PC^-(\An)$ :
\begin{equation}\label{bettidia3}
\xymatrix{
0\ar[r] & \underline{\sing}_{\bar{\mathbb D}^*}\mathbb Z_{tr}(E)
\ar[rr]^{S(\mathbb Z(l))}\ar[d]^{B(\mathbb Z_{tr}(E))} & \, &
\underline{\sing}_{\bar{\mathbb D}^*}\mathbb Z_{tr}(Y)
\ar[rr]^{S(c_{Y,E})}\ar[d]^{B(\mathbb Z_{tr}(Y))} & \, &
\underline{\sing}_{\bar{\mathbb D}^*}\mathbb Z_{tr}(Y,E)\ar[d]^{B(\mathbb Z_{tr}(Y,E))}\ar[r] & 0 \\
0\ar[r] & \Cw_*\underline{\sing}_{\mathbb I^*}\mathbb Z_{tr}(E)\ar[rr]^{\Cw_*S(\mathbb Z(l^{cw}))} & \, &
\Cw_*\underline{\sing}_{\mathbb I^*}\mathbb Z_{tr}(Y)\ar[rr]^{\Cw_*S(c_{Y^{cw},E^{cw}})} & \, &
\Cw_*\underline{\sing}_{\mathbb I^*}\mathbb Z_{tr}(Y,E)\ar[r] & 0}
\end{equation}
where the first row is the exact sequence (\ref{AncXD}), and the second row is an exact sequence
since the sequence (\ref{CWcXD}) is exact and $\Cw_*$ is an exact functor.
Since we just showed that $e^{tr}_{an*}B(\mathbb Z_{tr}(Y))$ and $e^{tr}_{an*}B(\mathbb Z_{tr}(E))$
are quasi-isomorphism, the diagram (\ref{bettidia3}) implies that $e^{tr}_{an*}B(\mathbb Z_{tr}(Y,E))$ 
is a quasi isomorphism.
\end{itemize}

\noindent(ii): Follows by (i).Let us explain.
\begin{itemize}
\item On the one hand,
\begin{itemize}
\item By theorem \ref{AnS} (ii), $\underline{\sing}_{\bar{\mathbb D}^*}\mathbb Z_{tr}(Y,E)$ is $\mathbb D^1$ local.
\item By theorem \ref{CWsingI}(ii), $\underline{\sing}_{\mathbb I^*}\mathbb Z_{tr}(Y^{cw},E^{cw})$ is $\mathbb I^1$ local.
Hence, $\Cw_*\underline{\sing}_{\mathbb I^*}\mathbb Z_{tr}(Y^{cw},E^{cw})$ is $\mathbb D^1$ local,
by proposition \ref{AnCwtCw} (ii). 
\end{itemize}
\item On the other hand by (i)
$e^{tr}_{an*}(B(\mathbb Z_{tr}(Y,E))):
\sing_{\bar{\mathbb D^*}}\mathbb Z_{tr}(Y,E)\to\sing_{\mathbb I^*}\mathbb Z_{tr}(Y^{cw},E^{cw})$
is a quasi isomorphism in $C^{-}(\mathbb Z)$.
\end{itemize}
Hence, by proposition \ref{DusuEq} (ii), 
\begin{equation*}
B(\mathbb Z_{tr}(Y,E)):
\underline{\sing}_{\bar{\mathbb D}^*}\mathbb Z_{tr}(Y,E)\to\Cw_*\underline{\sing}_{\mathbb I^*}\mathbb Z_{tr}(Y^{cw},E^{cw})
\end{equation*}
is an equivalence $(\mathbb D^1,usu)$ local.

\end{proof}

The main result of this subsection is the following :

\begin{thm}\label{mainthm}
\begin{itemize}
\item[(i)] For $Y\in\Var(\mathbb C)$, and $E\subset Y$ a subvariety, we have $\Bti^*M(Y,E)=\widetilde\Bti^*M(Y,E)$

\item[(ii)] For $X,Y\in\Var(\mathbb C)$, $D\subset X$, $E\subset Y$ subvarieties, 
and $n\in\mathbb Z$, $n\leq 0$, the following diagram is commutative
\begin{equation*}
\xymatrix{
\Hom_{\DM^{-}(\mathbb C,\mathbb Z)}(M(X,D),M(Y,E)[n])\ar^{\widetilde\Cw^*}[rr]\ar[d]_{\An^*} & \, & 
\Hom_{\CwDM^-(\mathbb Z)}(M(X,D),M(Y,E)[n])\ar[d]^{Re^{tr}_{cw*}} \\  
\Hom_{\AnDM^-(\mathbb Z)}(M(X,D),M(Y,E)[n])\ar[rr]^{Re^{tr}_{an*}} & \, & 
\Hom_{D^-(\mathbb Z)}(\sing_{\mathbb I^*}\mathbb Z_{tr}(X,D),\sing_{\mathbb I^*}\mathbb Z_{tr}(Y,E)[n])} 
\end{equation*}
where we denoted for simplicity $X$ for $X^{an}$ and $X^{cw}$, and similarly for $D$, $Y$ and $E$. 
\end{itemize}
\end{thm}

\begin{proof}

\noindent(i): By definition, $\Bti^*M(Y,E)=Re^{tr}_{an*}(\mathbb Z_{tr}(Y^{an},E^{an}))$.
Since, by theorem \ref{AnS}(ii), 
\begin{itemize}
\item $S(\mathbb Z_{tr}(Y^{an},E^{an})):\mathbb Z_{tr}(Y^{an},E^{an})\to\underline{\sing}_{\mathbb D^*}\mathbb Z_{tr}(Y^{an},E^{an})$
is an equivalence $(\mathbb D^1,usu)$ local in $PC^-(\An)$ and
\item $\underline{\sing}_{\mathbb D^*}\mathbb Z_{tr}(Y^{an},E^{an})$
is a $\mathbb D^1$ local object,
\end{itemize}
we have
$\Bti^*M(Y,E)=e^{tr}_{an*}(\underline{\sing}_{\mathbb D^*}\mathbb Z_{tr}(Y^{an},E^{an}))
=\sing_{\mathbb D^*}\mathbb Z_{tr}(Y^{an},E^{an})$.
Since
\begin{itemize}
\item $B(\mathbb Z_{tr}(Y^{an},E^{an})):
\underline{\sing}_{\mathbb D^*}\mathbb Z_{tr}(Y^{an},E^{an})\to\Cw_*\underline{\sing}_{\mathbb I^*}\mathbb Z_{tr}(Y^{cw},E^{cw})$
is an equivalence $(\mathbb D^1,usu)$ local in $PC^-(\An)$ 
by proposition \ref{Bettiprop} (ii), and
\item $\Cw_*\underline{\sing}_{\mathbb I^*}\mathbb Z_{tr}(Y^{cw},E^{cw})$
is a $\mathbb D^1$ local object by theorem \ref{CWsingI}(ii) and proposition \ref{AnCwtCw}(ii),
\end{itemize}
we have
\begin{equation}\label{maineq}
\Bti^*M(X)=e^{tr}_{an*}(\Cw_*\underline{\sing}_{\mathbb I^*}\mathbb Z_{tr}(Y^{cw},E^{cw}))
=\sing_{\mathbb I^*}\mathbb Z_{tr}(Y^{cw},E^{cw})
\end{equation}

By definition, $\widetilde\Bti^*M(X)=Re^{tr}_{cw*}(\mathbb Z_{tr}(Y^{cw},E^{cw}))$.
Since, by theorem \ref{CWsingI}(ii),
\begin{itemize}
\item $S(\mathbb Z_{tr}(Y^{cw},E^{cw})):\mathbb Z_{tr}(Y^{cw},E^{cw})\to\underline{\sing}_{\mathbb I^*}\mathbb Z_{tr}(Y^{cw},E^{cw})$
is an equivalence $(\mathbb I^1,usu)$ local in $PC^-(\CW)$ and 
\item $\underline{\sing}_{\mathbb I^*}\mathbb Z_{tr}(Y^{cw},E^{cw})$ is an $\mathbb I^1$ local object,
\end{itemize}
we have
\begin{eqnarray*}
\widetilde\Bti^*M(Y,E)=e^{tr}_{cw*}(\underline{\sing}_{\mathbb I^*}\mathbb Z_{tr}(Y^{cw},E^{cw}))
&=&\sing_{\mathbb I^*}\mathbb Z_{tr}(Y^{cw},E^{cw}) \\
&=&\Bti^*M(Y,E) \mbox{\; by (\ref{maineq}) }
\end{eqnarray*}
This proves (i).

\noindent(ii): 
Let $\alpha\in\Hom_{\DM^-(\mathbb C,\mathbb Z)}(M(X,D),M(Y,E)[n])$.
Consider the commutative diagram in $\AnDM(\mathbb Z)$ 
\begin{equation*}
\xymatrix{
\underline{\sing}_{\mathbb D^*}\mathbb Z_{tr}(X^{an},D^{an})\ar[rr]^{\An^*\alpha}\ar[d]_{B(\mathbb Z_{tr}(X,D))} & \, &
\underline{\sing}_{\mathbb D^*}\mathbb Z_{tr}(Y^{an},E^{an})[n]\ar[d]^{B(\mathbb Z_{tr}(Y,E))} \\
\Cw_*\underline{\sing}_{\mathbb I^*}\mathbb Z_{tr}(X^{cw},D^{cw})\ar[rr]^{\Cw_*\widetilde\Cw^*\alpha} & \, &
\Cw_*\underline{\sing}_{\mathbb I^*}\mathbb Z_{tr}(Y^{cw},E^{cw})[n]}
\end{equation*}
Since
$\underline{\sing}_{\mathbb D^*}\mathbb Z_{tr}(X^{an},D^{an})$ and $\underline{\sing}_{\mathbb D^*}\mathbb Z_{tr}(Y^{an},E^{an})$
are $\mathbb D^1$ local objects by theorem \ref{AnS}(ii), 
\begin{equation*}
\An^*\alpha\in\Hom_{\Ho_{usu}(PC^-(\An))} 
(\underline{\sing}_{\mathbb D^*}\mathbb Z_{tr}(X^{an},D^{an}),\underline{\sing}_{\mathbb D^*}\mathbb Z_{tr}(Y^{cw},E^{cw}))
\end{equation*}
Thus,
\begin{equation}\label{Anaplha}
\Bti^*(\alpha):=Re^{tr}_{an*}(\An^*\alpha)=e^{tr}_{an*}\An^*\alpha
\end{equation}
Since
$\underline{\sing}_{\mathbb I^*}\mathbb Z_{tr}(X^{cw},D^{cw})$ and $\underline{\sing}_{\mathbb I^*}\mathbb Z_{tr}(Y^{cw},E^{cw})$
are $\mathbb I^1$ local objects by theorem \ref{CWsingI}(ii), 
\begin{equation*}
\widetilde\Cw^*\alpha\in\Hom_{\Ho_{usu}(PC^-(CW))} 
(\underline{\sing}_{\mathbb I^*}\mathbb Z_{tr}(X^{cw},D^{cw}),\underline{\sing}_{\mathbb I^*}\mathbb Z_{tr}(Y^{cw},E^{cw}))
\end{equation*}
Thus
\begin{equation}\label{Cwalpha1}
\widetilde\Bti^*(\alpha):=Re^{tr}_{cw*}(\widetilde\Cw^*\alpha)=e^{tr}_{cw*}\widetilde\Cw^*\alpha
\end{equation}
Since
$\Cw_*\underline{\sing}_{\mathbb I^*}\mathbb Z_{tr}(X^{cw},D^{cw})$ and $\Cw_*\underline{\sing}_{\mathbb I^*}\mathbb Z_{tr}(Y^{cw},E^{cw})$
are $\mathbb D^1$ local objects by theorem \ref{CWsingI}(ii) and proposition \ref{AnCwtCw} (ii), 
\begin{equation}\label{adcwusu}
\Cw_*\widetilde\Cw^*\alpha\in\Hom_{\Ho_{usu}(PC^-(\An))} 
(\Cw_*\underline{\sing}_{\mathbb I^*}\mathbb Z_{tr}(X^{cw},D^{cw}),\Cw_*\underline{\sing}_{\mathbb I^*}\mathbb Z_{tr}(Y^{cw},E^{cw}))
\end{equation}
Thus
\begin{eqnarray*}
\Bti^*(\alpha)&=&Re^{tr}_{an*}(\Cw_*\widetilde\Cw^*\alpha) 
\mbox{\; since \;} B(\mathbb Z_{tr}(X^{an},D^{an})) \mbox{\; and \; } B(\mathbb Z_{tr}(Y^{an},E^{an}))[n] \\
&\, &\mbox{\; are \;} (\mathbb D^1,usu) \mbox{\; local \;  equivalence \; by \; proposition \ref{Bettiprop}(ii)} \\
&=&e^{tr}_{an*}(\Cw_*\widetilde\Cw^*\alpha) \mbox{\; by \; (\ref{adcwusu})} \\
&=&e^{tr}_{cw}(\widetilde\Cw^*\alpha) \\
&=&\widetilde\Bti^*(\alpha) \mbox{\; by \; (\ref{Cwalpha1}).}
\end{eqnarray*}

\end{proof}

\subsection{The image of algebraic correspondences after localization by Ayoub's Betti realization functor}

\begin{defi}
Let $V\in\Var(\mathbb C)$ quasi projective and $p,q\in\mathbb N$. 
Recall $\mathcal Z^p(V,*)\subset\mathcal Z^p(\square^*\times V,\mathbb Z)$ is the Bloch cycle complex
consisting of closed subset meeting the face of $\square^*$ properly.
\begin{itemize}
\item[(i)] Let $\mathcal Z^q(V^{cw},*)\subset\mathcal Z^{q}(\mathbb I^*\times V^{cw},\mathbb Z)$
be the abelian subgroup consisting of closed subset meeting the face of $\mathbb I^n$ properly.
We have the restriction map induced by the closed embedding of CW complexes 
$i_n\times I_V:\mathbb I^n\times V^{cw}\hookrightarrow\square^n\times V^{cw}$, 
which gives the morphism of complexes of abelian groups
\begin{equation*}
\hat{\mathcal T}^p:\mathcal Z^p(V,*)\to\mathcal Z^{2p}(V^{cw},*), \; 
\alpha\mapsto \hat{T}_{\alpha}=\alpha^{cw}_{|\mathbb I^{*}\times V^{cw}}.
\end{equation*}
\item[(ii)] Let $Y\in\PVar(\mathbb C)$ be a compactification of $V$ and $E=Y\backslash V$.
The higher cycle class map (\cite{MKerr}) is the morphism of complexes of abelian groups :
\begin{equation*}
\mathcal T^p:\mathcal Z^p(V,*)\to C^{sing}_{2d_Y-2p+*}(Y^{cw},E^{cw},\mathbb Z), \; 
\alpha\mapsto [p_Y(\hat{T}_{\alpha})]=[p_{Y}(\bar{\alpha}^{cw}_{|\mathbb I^*\times Y^{cw}})].
\end{equation*}
$\bar\alpha\in\mathcal Z^p(Y\times\square^n)$ is the closure of $\alpha$ and
$p_Y:Y^{cw}\times\mathbb I^*\to Y^{cw}$ is the projection
\end{itemize}
\end{defi}

Let $X\in\SmVar(\mathbb C)$ and $Y\in\Var(\mathbb C)$.
Let $E\subset Y$ be a closed subset and  $V=Y\backslash E$.
Denote by $j:V\hookrightarrow Y$ the open embeddings.
We have the open embedding $I_X\times j:X\times V\hookrightarrow X\times Y$. 
The map of complexes $\hat{\mathcal T}^{d_Y}_{X\times Y}$ induces a map 
denoted by the same way on the subcomplexes indicated in the following diagram
in $C^-(\mathbb Z)$:
\begin{equation*}
\xymatrix{
\mathcal Z^{d_V}(X\times V,*)\ar[rr]^{\hat{\mathcal T}^{d_V}_{X\times V}} & \, & 
\mathcal Z^{2d_V}(X^{cw}\times V^{cw},*) \\
\underline{C}_*\mathbb Z_{tr}(Y,E)(X)\ar[rr]^{\hat{\mathcal T}^{d_V}_{X\times V}}\ar@{^{(}->}[u]^{(I_X\times j)^*} & \, &
\underline{\sing}_{\mathbb I^*}\mathbb Z_{tr}(Y^{cw},E^{cw})(X^{cw})\ar@{^{(}->}[u]^{(I_X\times j)^{cw*}} \\
\underline{C}_*\mathbb Z_{tr}(V)(X)
\ar[rr]^{\hat{\mathcal T}^{d_V}_{X\times V}}\ar@{^{(}->}[u]^{S(c(Y,E)\circ\mathbb Z_{tr}(j))(X)} & \, &
\underline{\sing}_{\mathbb I^*}\mathbb Z_{tr}(V^{cw})(X^{cw})
\ar@{^{(}->}[u]_{S(c(Y^{cw},E^{cw})\circ\mathbb Z_{tr}(j^{cw}))(X^{cw})}}
\end{equation*}

\begin{lem}\label{TAnCWlem}

Let $Y\in\Var(\mathbb C)$, $X\in\SmVar(\mathbb C)$, $E\subset Y$ a closed subset and $V= Y\backslash E$
the open complementary. Let $n\in\mathbb Z$, $n\leq 0$. Then, for 
\begin{equation*}
\alpha\in\Hom_{PC^-}(\mathbb Z_{tr}(X),\underline{C}_*\mathbb Z_{tr}(Y,E)[n])=H^n\underline{C}_*\mathbb Z_{tr}(Y,E)(X).
\end{equation*}
we have the following equality of morphisms of $\PSh(\Cor_{\mathbb Z}(\CW),C^-(\mathbb Z))$.
\begin{eqnarray*}
(\widetilde{W}(\mathbb Z_{tr}(Y,E))[n])\circ(\widetilde\Cw^*\alpha)=H^n(\hat{\mathcal T}_{X\times V})(\alpha) : \\
\mathbb Z_{tr}(X^{cw})\xrightarrow{\widetilde\Cw^*\alpha} 
\widetilde\Cw^*(\underline{C}_*\mathbb Z_{tr}(Y,E)[n])\xrightarrow{\widetilde{W}(\mathbb Z_{tr}(Y,E))[n]}
\underline{\sing}_{\mathbb I^*}\mathbb Z_{tr}(Y^{cw},E^{cw})[n]
\end{eqnarray*}
where $H^n\hat{\mathcal T}_{X\times V}$ is the map induced in cohomology of the map of complex $\hat{\mathcal T}_{X\times V}$.

\end{lem}

\begin{proof}

Let $\alpha_n\in\underline{C}_n\mathbb Z_{tr}(Y,E)(X)$ such that 
$\alpha=[\alpha_n]\in H^n\underline{C}_*\mathbb Z_{tr}(Y,E)(X)$.
We have 
\begin{eqnarray*}
(\widetilde{W}(\mathbb Z_{tr}(Y,E))[n])\circ(\widetilde\Cw^*\alpha) 
&=&(\psi^{\widetilde\Cw}(\mathbb Z_{tr}(Y,E))[n])\circ(\widetilde\Cw^*(\underline{\mathbb Z^{tr}(Y,E)}(i))[n])\circ([\widetilde\Cw^*\alpha_n]) \\
&=&\psi^{\widetilde\Cw}(\mathbb Z_{tr}(Y,E)[n])\circ([\widetilde\Cw^*(\alpha_{n|X\times\mathbb I_{et}^n})]) \\
&=&[(\alpha_{n|X\times\mathbb I_{et}^n})^{cw}]=[\hat{T}_{\alpha_n}]=H^n\hat{\mathcal{T}}(\alpha)
\end{eqnarray*}

\end{proof}

We deduce from this lemma \ref{TAnCWlem} and proposition \ref{AWtW}(ii) the following

\begin{prop}\label{TAnCW}
Let $X\in\SmVar(\mathbb C)$, $Y\in\Var(\mathbb C)$, $E\subset Y$ a closed subset
and $V=Y\backslash E$ the open complementary.
Let $n\in\mathbb Z$, $n\leq 0$. Then, the following diagram is commutative
\begin{equation}\label{mainthm(iii)}
\xymatrix{
\Hom_{PC^-}(\mathbb Z_{tr}(X),\underline{C}_*\mathbb Z_{tr}(Y,E)[n])
\ar[d]_{D(\mathbb A^1,et)}\ar[rr]^{H^n(\hat{\mathcal T}_{X\times V}^{d_X})} & \, & 
\Hom_{PC^-(\CW)}(\mathbb Z_{tr}(X^{cw}),\underline{\sing}_{\mathbb I^*}\mathbb Z_{tr}(Y^{cw},E^{cw})[n])
\ar[d]^{D(\mathbb I^1,usu)} \\
\Hom_{\DM^{-}(\mathbb C,\mathbb Z)}(M(X),M(Y,E)[n])\ar^{\widetilde\Cw^*}[rr] & \, & 
\Hom_{\CwDM^-(\mathbb Z)}(M(X^{cw}),M(Y^{cw},E^{cw})[n])} 
\end{equation}
\end{prop}

\begin{proof}

Let $\alpha\in\Hom_{PC^-}(\mathbb Z_{tr}(X),\underline{C}_*\mathbb Z_{tr}(Y,E)[n])$.
We have 
\begin{eqnarray*}
\widetilde\Cw^*D(\mathbb A^1,et)(\alpha)&=&D(\mathbb I^1,usu)(\widetilde\Cw^*\alpha), 
\mbox{\; since \;} \widetilde\Cw^* \mbox{\; derive \; trivially \;by \; proposition \; \ref{AnCwtCw}(iii)} \\
&=&D(\mathbb I^1,usu)(\widetilde W(\mathbb Z_{tr}(Y,E))\circ\widetilde\Cw^*\alpha) 
\mbox{\; since \;} \widetilde W(\mathbb Z_{tr}(Y,E)) 
\mbox{\; is \; an \;} (\mathbb I^1,usu) \\
&\,&\mbox{\; local \; equivalence \; by \; lemma \: \ref{AWtW}(ii)} \\
&=&D(\mathbb I^1,usu)(H^n\hat{\mathcal T}(\alpha))
\mbox{ \; by \; lemma \; \ref{TAnCWlem}}
\end{eqnarray*}

\end{proof}

Using proposition \ref{TAnCW} and proposition \ref{Kunneth} (i), 
we immediately deduce from theorem \ref{mainthm} the following :

\begin{cor}\label{cormain} 
Let $X\in\SmVar(\mathbb C)$, $Y\in\Var(\mathbb C)$, $E\subset Y$ a closed subvariety 
and $V=Y\backslash E$ the open complementary. 
Let $n\in\mathbb Z$, $n\leq 0$. Then,
\begin{itemize}
\item[(i)] the following diagram is commutative 
\begin{equation}\label{cormaindia}
\xymatrix{
\Hom_{PC^-}(\mathbb Z_{tr}(X),\underline{C}_*\mathbb Z(Y,E)[n])
\ar[d]_{\An^*\circ D(\mathbb A^1,et)}\ar[rr]^{H^n(\hat{\mathcal T}_{X\times V}^{d_X})} & \, & 
\Hom_{PC^-(\CW)}(\mathbb Z_{tr}(X),\underline{\sing}_{\mathbb I^*}\mathbb Z_{tr}(Y,E)[n])
\ar[d]^{Re^{tr}_{cw}\circ D(\mathbb I^1,usu)} \\  
\Hom_{\AnDM^-(\mathbb Z)}(M(X),M(Y,E)[n])\ar[rr]^{Re^{tr}_{an*}} & \, & 
\Hom_{D^-(\mathbb Z)}(\sing_{\mathbb I^*}\mathbb Z_{tr}(X),\sing_{\mathbb I^*}\mathbb Z_{tr}(Y,E)[n])} 
\end{equation}
where we denoted for simplicity $X$ for $X^{an}$ and $X^{cw}$, and similarly for $Y$ and $E$. 

\item[(ii)] for $\alpha\in\Hom_{PC^-}(\mathbb Z_{tr}(X),\underline{C}_*\mathbb Z(Y,E)[n])$,
we have
\begin{equation}\label{cormaineq}
\Bti^*\circ D(\mathbb A^1,et)(\alpha)=K_n(X,Y)(p_{X\times Y}(H^n\hat{\mathcal T}_{X\times V}(\alpha)))
\end{equation}
where $p_{X\times Y}:X^{cw}\times Y^{cw}\times\mathbb I^n\to X^{cw}\times Y^{cw}$ is the projection.
\end{itemize}
\end{cor}

\begin{proof}
\noindent(i): By theorem \ref{mainthm} (ii) and proposition \ref{TAnCW}, 
the following diagram is commutative
\begin{equation*}
\xymatrix{
\Hom_{PC^-}(\mathbb Z_{tr}(X),\underline{C}_*\mathbb Z(Y,E)[n])
\ar[d]_{D(\mathbb A^1,et)}\ar[rr]^{H^n(\hat{\mathcal T}_{X\times V}^{d_X})} & \, & 
\Hom_{PC^-(\CW)}(\mathbb Z_{tr}(X),\underline{\sing}_{\mathbb I^*}\mathbb Z_{tr}(Y,E)[n])
\ar[d]^{D(\mathbb I^1,usu)} \\
\Hom_{\DM^{-}(\mathbb C,\mathbb Z)}(M(X),M(Y,E)[n])\ar^{\widetilde\Cw^*}[rr]\ar[d]_{\An^*} & \, & 
\Hom_{\CwDM^-(\mathbb Z)}(M(X),M(Y,E)[n])\ar[d]^{Re^{tr}_{cw}} \\  
\Hom_{\AnDM^-(\mathbb Z)}(M(X),M(Y,E)[n])\ar[rr]^{Re^{tr}_{an*}} & \, & 
\Hom_{D^-(\mathbb Z)}(\sing_{\mathbb I^*}\mathbb Z_{tr}(X),\sing_{\mathbb I^*}\mathbb Z_{tr}(Y,E)[n])} 
\end{equation*}
where we denote for simplicity $X$ for $X^{an}$ and $X^{cw}$, and similary for $Y$ and $E$. 
This proves (i).

\noindent(ii) : We have,
\begin{eqnarray*}
\Bti^*(D(\mathbb A^1,et)(\alpha)):&=&Re^{tr}_{an*}\circ\An^*(D(\mathbb A^1,et)(\alpha)) \\
&=&Re^{tr}_{cw*}\circ D(\mathbb I^1,usu)(H^n\hat{\mathcal T}_{X\times V}(\alpha)) \mbox{\; by (i) } \\
&=&K_n(X,Y)(p_{X\times Y}(H^n\hat{\mathcal T}_{X\times V}(\alpha))) \mbox{\; by proposition \; \ref{Kunneth}(i)}
\end{eqnarray*}
This proves (ii).

\end{proof}

\subsection{Ayoub's Betti realization functor and Nori motives}

We denote by $C^b(\Cor_{\mathbb Z}(\SmVar(\mathbb C)))$ the category of bounded complexes of the
category $\Cor_{\mathbb Z}(\SmVar(\mathbb C))$. We have the Yoneda embedding
\begin{equation*}
C^b(\Cor_{\mathbb Z}(\SmVar(\mathbb C)))\hookrightarrow\PSh(\Cor_{\mathbb Z}(\SmVar(\mathbb C)),C^-(\mathbb Z), \;
X\in\SmVar(\mathbb C)\mapsto\mathbb Z_{tr}(X).
\end{equation*}
A (small) diagram is a 1-simplicial set, that is a functor from $\bold\Delta^1$ to Set
We denote by $\Var^2(\mathbb C)$ the diagram of pairs 
whose set of vertices are the triplet $(X,D,i)$, $X,D\in\Var(\mathbb C)$, 
$D\subset X$ a closed subvariety and $i\in\mathbb N$, and whose set of edges
are the morphism of pairs $(X,D,i)\to(Y,E,i)$ and if $(Z,K)\subset(X,D)$,
there is an edge $(X,D,i)\to(Z,K,i-1)$. 
We have a morphism of diagram
\begin{equation*}
H_*:\Var^2(\mathbb C)\to\Ab \; ; \; (X,D,i)\mapsto H_i(X,D,\mathbb Z)
\end{equation*}
We denote by $\GVar^2(\mathbb C)\subset\Var^2(\mathbb C)$ the subdiagram of good pairs
that is $(X,D,i)$ is a good pair if $X,D$ are affine, $\dim X=i$ and if $H_p(X,D,\mathbb Z)=0$ is $p\neq i$.
We recall the definition of Nori motives.

\begin{defiprop}\cite{Nori}
\begin{itemize}
\item[(i)] We have a well defined functor $\mathcal N:\Var(\mathbb C))\to D^b(\mathcal N)$
given by, for $X\in\Var(\mathbb C)$, considering
\begin{itemize}
\item the open coverings $X=\cup_{i\in J}U_i$ by affine varieties, 
\item the filtrations by closed subvarieties $U^0_I\subset\cdots\subset U_I$ such that 
$(U^p_I,U^{p-1}_I)$ is a good pair, where $I\subset J$ is a finite set,
\end{itemize}
we have 
\begin{equation*}
o_N\circ\mathcal N(X)=\lim\Tot_{p,I}(H_{p}(U^{p,cw}_I,U^{p-1,cw}_I,p)),
\end{equation*}
where the limit is indexed by open coverings and filtrations by good pairs,$\Tot$ is the total complex,
and $o_N:D^b(\mathcal N)\to D^b(\mathbb Z)$ is the forgetful functor.

\item[(ii)] The restriction of the functor $\mathcal N:\Var(\mathbb C))\to D^b(\mathcal N)$ 
to $\SmVar(\mathbb C)\subset\Var(\mathbb C)$ extends to a functor
\begin{equation*}
\mathcal N:C^b(\Cor_{\mathbb Z}(\SmVar(\mathbb C)))\to D^b(\mathcal N)
\end{equation*}

\item[(iii)] The functor $\mathcal N:C^b(\Cor_{\mathbb Z}(\SmVar(\mathbb C)))\to D^b(\mathcal N)$
factorize trough to get the composition
\begin{equation*}
\mathcal N=\bar{\mathcal N}\circ D(\mathbb A^1,et) \;, \; \bar{\mathcal N}:DM^{gm}\to D^b(\mathcal N)
\end{equation*}
\end{itemize}
\end{defiprop}

Using the property of good pairs, Mayer Vietoris property for open coverings for singular cohomoloy,
Nori showed the following :

\begin{prop}\label{Noriprop}\cite{Nori}
\begin{itemize}
\item[(i)] Let $X\in\Var(\mathbb C)$,
we have the following isomorphism in $D^b(\mathbb Z)$ :
\begin{equation*}
o_N\circ\mathcal N(X)\simeq C_*(X,\mathbb Z)\simeq\sing_{\mathbb I^*}\mathbb Z_{tr}(X^{cw})
\end{equation*}

\item[(ii)] Let $X,Y\in\SmVar(\mathbb C)$, then the following diagram commutes
\begin{equation*}
\xymatrix{\Hom_{PC^-}(\mathbb Z_{tr}(X),\mathbb Z_{tr}(Y))\ar^{H^0(\mathcal T_{X\times Y})}[rr]\ar_{\mathcal N}[d] & \, &
\Hom_{PC^-(CW)}(\mathbb Z_{tr}(X^{cw}),\mathbb Z_{tr}(Y^{cw}))\ar[d]^{K_0(X,Y)} \\
\Hom_{D^b(\mathcal N)}(\mathcal N(X),\mathcal N(Y))\ar[rr]^{o_N} & \, &
\Hom_{D^b(\mathbb Z)}(\sing_{\mathbb I^*}\mathbb Z_{tr}(X^{cw}),\sing_{\mathbb I^*}\mathbb Z_{tr}(Y^{cw}))}
\end{equation*}
\end{itemize}
\end{prop}

We will use the following lemma
\begin{lem}\label{Dfactor}
Let $\mathcal C,\mathcal S$ two categories, and $D:\mathcal C\to\mathcal C'$ a localization functor.
Let $F_1,F_2:\mathcal C'\to\mathcal S$
two functors. If
\begin{equation*}
F_1\circ D=F_2\circ D:\mathcal C\to\mathcal S, 
\end{equation*}
then $F_1=F_2$.
\end{lem}

\begin{proof}
 
Let $f':C_1\to C_2$ a morphism in $\mathcal C'$. 
Without loss of generality, we can assume that $f'=f\circ s^-{1}$ with
$f:C_0\to C_2$ and $s:C_0\to C_1$ morphisms in $\mathcal C$, such that
$s$ belongs to the class of morphisms of $\mathcal C$ we localize.
We have then,
\begin{eqnarray*}
F_1(f')=F_1(f'\circ s\circ s^{-1})&=&(F_1\circ D)(f'\circ s)\circ ((F_1\circ D)(s))^{-1} \\
&=&(F_2\circ D)(f'\circ s)\circ ((F_2\circ D)(s))^{-1}=F_2(f'\circ s\circ s^{-1})=F_2(f')
\end{eqnarray*}

\end{proof}

We deduce from theorem \ref{mainthm}, corollary \ref{cormain}(ii), 
proposition \ref{Noriprop}, end lemma \ref{Dfactor}, the following main result :
\begin{thm}\label{corNorithm} 
\begin{itemize}
\item[(i)] For $X\in\SmVar(\mathbb C)$, 
$\Bti^*\circ D(\mathbb A^1,et)(\mathbb Z(X))=o_N\circ\mathcal N(X)$

\item[(ii)] For $X,Y\in\SmVar(\mathbb C)$, the following diagram commutes
\begin{equation*} 
\xymatrix{
\Hom_{PC^-}(\mathbb Z(X),\mathbb Z(Y))\ar[r]^{\mathcal N}\ar[d]_{D(\mathbb A^1,et)} & 
\Hom_{D^b(\mathcal N)}(N(X),N(Y))\ar[d]^{o_N} \\
\Hom_{\DM^-(\mathbb C,\mathbb Z)}(M(X),M(Y))\ar[r]^{\Bti^*} & 
\Hom_{D^b(\mathbb Z)}(C_*(X,\mathbb Z), C_*(Y,\mathbb Z))}
\end{equation*}

\item[(iii)] The Betti realisation functor factor through  Nori motives. That is
$\Bti^*=o_N\circ\bar{\mathcal N}$
\end{itemize}
\end{thm}

\begin{proof}

\noindent(i): We have 
\begin{eqnarray*}
\Bti^*\circ D(\mathbb A^1,et)(\mathbb Z(X))&=&Re^{tr}_{cw*}D(\mathbb I^1,usu)(\mathbb Z_{tr}(X^{cw}))
\mbox{\; by \; theorem  \ref{mainthm}(i)} \\
&=&o_N\circ\mathcal N(X)
\mbox{\; by \; proposition  \ref{Noriprop}(i).}
\end{eqnarray*}

\noindent(ii): Let $\alpha\in\Hom_{PC^-}(\mathbb Z(X),\mathbb Z(Y))$,
we have
\begin{eqnarray*}
\Bti^*\circ D(\mathbb A^1,et)(\alpha)&=&K_0(X,Y)(H^0\hat{\mathcal T}_{X\times Y}(\alpha))
\mbox{\; by \;  corollary \ref{cormain}(ii)} \\
&=&o_N\circ\mathcal N(\alpha)\mbox{\; by \; proposition \ref{Noriprop}(ii).} 
\end{eqnarray*}

\noindent(iii): By (i) and (ii), we have the equality
\begin{equation}\label{norieqfunc}
\Bti^*\circ D(\mathbb A^1,et)=o_N\circ\mathcal N=o_N\circ\bar{\mathcal N}\circ D(\mathbb A^1,et)
\end{equation}
Lemma \ref{Dfactor} applied to this equality (\ref{norieqfunc}) say that $\Bti^*=o_N\circ\bar{\mathcal N}$.
This proves (iii).

\end{proof}

\subsection{The image of Ayoub's Betti realization functor on morphism and Bloch cycle class map}

Let $V\in\Var(\mathbb C)$ quasi-projective. 
Let $Y\in\PVar(\mathbb C)$ a compactification of $V$ and $E=Y\backslash V$
Denote by $j:V\hookrightarrow Y$ the open embedding.
For $r\in\mathbb N$, we denote by
$i_Y:Y\hookrightarrow\mathbb A^r\times Y$ and
$i_V=i_{Y|V}:V\hookrightarrow\mathbb A^r\times V$ the inclusions,
by $p_Y:\mathbb A^r\times Y\to Y$ and $p_V=p_{Y|V}:\mathbb A^r\times V\to V$ the projections,
by $a:\mathbb A^r\hookrightarrow\mathbb P^r$ the open embedding, 
and by $E'=Y\times\mathbb P^r\backslash(V\times\mathbb A^r)$.  
We consider 
\begin{itemize}
\item for $p\leq d_V$, the inclusion of complexes
$i_{V*}:\mathcal Z^p(V,*)\hookrightarrow\mathbb Z^{d_V}(\mathbb A^{d_V-p}\times V,*)$,
which is a quasi isomorphism.
\item for $p\geq d_V$ the inclusion of complexes
$p_{V}^*:\mathcal Z^p(\mathbb A^{p-d_V,*}\times V)\hookrightarrow\mathbb Z^{p}( V,*)$,  
which is a quasi-isomorphism.
\end{itemize}
Since $Y$ is projective, and $\mathbb A^r$ is smooth,
\begin{itemize}
\item $D(\mathbb A^1,et):\Hom_{PC^-}(\mathbb Z_{tr}(\mathbb A^r),\underline{C}_*\mathbb Z(Y,E)[n])
\to\Hom_{\DM^-(\mathbb C,\mathbb Z)}(M(\mathbb A^r),M(Y,E)[n])$ 
is an isomorphism by proposition \ref{DA1iso}
\item $D(\mathbb I^1,et):
\Hom_{PC^-(CW)}(\mathbb Z_{tr}(\mathbb A^{r,cw}),\underline{\sing}_{\mathbb I^*}\mathbb Z_{tr}(Y^{cw},E^{cw})[n])
\to\Hom_{\CwDM^-(\mathbb Z)}(M(\mathbb A^{r,cw}),M(Y^{cw},E^{cw})[n])$  
is an isomorphism by proposition \ref{DI1iso}.
\end{itemize}
On the other side, the inclusion of complexes of abelian groups
\begin{itemize}
\item for $p\leq d_V$, 
$(I_{\mathbb A}\times j)^{*}:
C_*\mathbb Z_{tr}(Y,E)(\mathbb A^{d_V-p})\hookrightarrow\mathcal Z^{d_V}(\mathbb A^{d_V-p}\times V,*)$
\item for $p\geq d_V$, 
$(a\times j)^{*}:C_*\mathbb Z_{tr}(\mathbb P^{p-d_V}\times Y,E')(\Spec(\mathbb C))
\hookrightarrow\mathcal Z^p(\mathbb A^{p-d_V}\times V,*)$
\end{itemize}
is a quasi-isomorphisms.

\begin{defi}
Let $V\in\Var(\mathbb C)$ quasi-projective. 
Let $Y\in\PVar(\mathbb C)$ a compactification of $V$ and $E=Y\backslash V$.
Recall that $j:V\hookrightarrow Y$ is the open embedding.
\begin{itemize}
\item[(i)] For $p\leq d_V$, we consider the following composition of isomorphisms of abelian groups
\begin{eqnarray*}
H^{p,n}(V):\CH^p(V,n)\xrightarrow{i_{V*}}\CH^{d_V}(\mathbb A^{d_V-p}\times V,n)
\xrightarrow{(H^n(I_{\mathbb A}\times j)^{*})^{-1}}
\Hom_{PC^-}(\mathbb Z_{tr}(\mathbb A^{d_V-p}),\underline{C}_*\mathbb Z_{tr}(Y,E)[n]) \\
\xrightarrow{D(\mathbb A^1,et)}\Hom_{\DM^-(\mathbb C,\mathbb Z)}(M(\mathbb A^r),M(Y,E)[n]) 
\end{eqnarray*}
\item[(ii)] For $p\geq d_V$, we consider the following composition of isomorphisms of abelian groups
\begin{eqnarray*}
H^{p,n}(V):\CH^p(V,n)\xrightarrow{p_V^*}\CH^p(\mathbb A^{p-d_V}\times V,n)
\xrightarrow{(H^n(a\times j)^{*})^{-1}}
\Hom_{PC^-}(\mathbb Z,\underline{C}_*\mathbb Z_{tr}(\mathbb P^{p-d_V}\times Y,E')[n]) \\
\xrightarrow{D(\mathbb A^1,et)}\Hom_{\DM^-(\mathbb C,\mathbb Z)}(\mathbb Z,M(Y\times\mathbb P^{p-d_V},E')[n]) 
\end{eqnarray*}
\end{itemize}
\end{defi}

We now prove that under these identifications, the image of the Betti realization
functor on morphism coincide with the Bloch cycle class map. 

\begin{thm}\label{corbloch}
Let $V\in\Var(\mathbb C)$. Let $Y\in\PVar(\mathbb C)$ be a compactification of $V$ and $E=Y\backslash V$.
Then, 
\begin{itemize}
\item[(i)] for $p\leq d_V$, the following diagram commutes :
\begin{equation*}
\xymatrix{
CH^p(V,n)\ar[d]^{\sim}_{H^{p,n}(V)}\ar[rr]^{H^n\mathcal T_V} & \, & 
H_{n+p}(Y,E,\mathbb Z)\ar[d]^{\overline{K_n(\pt,(Y,E))}}_{\sim} \\
\Hom_{\DM^-(\mathbb C,\mathbb Z)}(M(\mathbb A^{d_V-p}),M(Y,E)[n])\ar[rr]^{\widetilde\Bti^*} & \, &
\Hom_{D^-(\mathbb Z)}(\mathbb Z,C_*(Y,E)[n])}
\end{equation*}

\item[(ii)] for $p\geq d_V$, the following diagram commutes :
\begin{equation*}
\xymatrix{
CH^p(V,n)\ar[d]^{\sim}_{H^{p,n}(V)}\ar[rr]^{H^n\mathcal T_V} & \, &
 H_{n+p}(Y,E,\mathbb Z)\ar[d]_{\sim}^{\overline{K_n(\pt,(Y,E))}} \\
\Hom_{\DM^-(\mathbb C,\mathbb Z)}(\mathbb Z,M(\mathbb A^{p-d_V}\times Y,E)[n])\ar[rr]^{\widetilde\Bti^*} & \, &
\Hom_{D^-(\mathbb Z)}(\mathbb Z,C_*(Y,E)[n])}
\end{equation*}
\end{itemize}

\end{thm}

\begin{proof}

\noindent(i): 
By corollary \ref{cormain}(ii), the second square of following diagram commutes
\begin{equation*}
\xymatrix{
CH^{d_V}(\mathbb A^{d_V-p}\times V,n)
\ar[d]_{(H^n(I_{\mathbb A}\times j)^{*})^{-1}}\ar[rr]^{H^n\mathcal T^{d_V}_{\mathbb A^{d_V-p}\times V}} 
& \, & H_{d_V-p+n}(\mathbb A^{d_V-p}\times Y,E,\mathbb Z)  \\
\Hom_{PC^-}(\mathbb Z_{tr}(\mathbb A^{d_V-p}),\mathbb Z_{tr}(Y,E)[n])
\ar[d]_{D(\mathbb A^1,et)}\ar[rr]^{H^n\hat{\mathcal T}_{\mathbb A^{d_V-p}\times V}} & \, &
\Hom_{PC^-(CW)}(\mathbb Z_{tr}(\mathbb A^{d_V-p}),\mathbb Z_{tr}(Y,E)[n])
\ar[u]_{[p(\cdot)]}\ar[d]^{K_n(\mathbb A^{d_V-p},(Y,E))(p(\cdot))}  \\
\Hom_{\DM^-(\mathbb C,\mathbb Z)}(M(\mathbb A^{d_V-p}),M(Y,E)[n]) 
\ar[rr]^{\widetilde\Bti^*} & \, &  \Hom_{D^-(\mathbb Z)}(\mathbb Z,C_*(Y,E)[d_V-p+n])}
\end{equation*}
where $p:\mathbb A^{d_V-p}\times V\times\mathbb I^n\to\mathbb A^{d_V-p}\times V$ is the projection in $\CW$.
Moreover $K_n(\mathbb A^{d_V-p},(Y,E))\circ[\cdot]=\overline K_n(\mathbb A^{d_V-p},(Y,E))$
by definition.
On the other hand, we have the following commutative diagram.
\begin{equation*}
\xymatrix{
CH^p(V,n)\ar[d]_{i_{V*}}\ar[rr]^{H^n\mathcal T^p_{V}} 
&  \, & H_{d_V-p+n}(Y,E,\mathbb Z)\ar[d]^{i_{V*}} \\ 
CH^{d_V}(\mathbb A^{d_V-p}\times V,n)\ar[rr]^{H^n\mathcal T^{d_V-p}_{\mathbb A^{d_V-p}\times V}} 
& \, & H_{d_V-p+n}(\mathbb A^{d_V-p}\times Y,E,\mathbb Z)} 
\end{equation*}
This proves (i).

\noindent(ii):
By corollary \ref{cormain}(ii), the second square of following diagram commutes
\begin{equation*}
\xymatrix{
CH^p(\mathbb A^{p-d_V}\times V,n)
\ar[d]_{(H^n(a\times j)^{*})^{-1}}\ar[r]^{H^n\mathcal T^p_{\mathbb A^{d_V-p}\times V}} 
& H_{d_V-p+n}(\mathbb A^{p-d_V}\times Y,E',\mathbb Z)  \\
\Hom_{PC^-}(\mathbb Z,\mathbb Z_{tr}(\mathbb P^{p-d_V}\times Y,E')[n])
\ar[r]^{H^n\hat{\mathcal T}_{\mathbb A^{p-d_V}\times V}}\ar[d]_{D(\mathbb A^1,et)} &
\Hom_{PC^-(CW)}(\mathbb Z,\mathbb Z_{tr}(\mathbb P^{p-d_V}\times Y,E')[n])
\ar[u]_{[p(\cdot)]}\ar[d]^{K_n(\pt,(\mathbb P^{p-d_V}\times Y,E'))(p(\cdot))} &  \\
\Hom_{\DM^-(\mathbb C,\mathbb Z)}(\mathbb Z,M(\mathbb P^{p-d_V}\times Y,E')[n]) 
\ar[r]^{\widetilde\Bti^*} &  \Hom_{D^-(\mathbb Z)}(\mathbb Z,C_*(Y,E)[d_V-p+n])}
\end{equation*}
where $p:\mathbb P^{p-d_V}\times Y\times\mathbb I^n\to\mathbb P^{p-d_V}\times Y$ is the projection in $\CW$.
Moreover
$K_n(\pt,\mathbb P^{p-d_V}\times Y,E'))\circ[\cdot]=\overline K_n(\pt,\mathbb P^{p-d_V}\times Y,E'))$
by definition.
On the other hand, we have the following commutative diagram.
\begin{equation*}
\xymatrix{
CH^p(V,n)\ar[d]_{p_V^*}\ar[rr]^{H^n\mathcal T^p_{V}} 
& \, & H_{d_V-p+n}(Y,E,\mathbb Z)\ar[d]^{p_Y^*} \\ 
CH^p(\mathbb A^{d_V-p}\times V,n)\ar[rr]^{H^n\mathcal T^{d_V-p}_{\mathbb A^{d_V-p}\times V}} 
& \, & H_{d_V-p+n}(\mathbb A^{d_V-p}\times Y,E,\mathbb Z)} 
\end{equation*}
This proves (ii).

\end{proof}

\section{The relative case}

\subsection{The derived category of relative motives of algebraic varieties}

Let $S\in\Var(k)$. The category $\Var(k)/S$ is the category whose objects are 
$X/S=(X,h)$ with $X\in\Var(k)$ and $h:X\to S$ is a morphism,
and whose space of morphisms between
$X/S=(X,h_1)$ and $Y/S=(Y,h_2)\in\Var(k)/S$ are the morphism $g:X\to Y$ such that $h_2\circ g=h_1$. 
We denote by $\Var(k)^{sm}/S\subset\Var(k)/S$ the full subcategory consisting of the objects
$X/S=(X,h)$ with $X\in\Var(k)$, such that $h:X\to S$ is a smooth morphism.
For $X/S=(X,h)\in\Var(k)/S$, and $n\in\mathbb N$, we denote by 
\begin{itemize}
\item $(X\times\mathbb A^1/S):=(X\times_k\mathbb A^1,h\circ p_X)=(X\times_S(\mathbb A^1\times S)/S)$,
where $p_X:X\times_k\mathbb A^1\to X$ is the projection.
\item $(X\times\square^n/S):=(X\times_k\square^n,h\circ p_X)=(X\times_S(\square^n\times S)/S)$,
where $p_X:X\times_k\square^n\to X$ is the projection.
\end{itemize}

\begin{defi}\cite{C.Deglise}
Let $S\in\Var(k)$.
We define $\Cor^{fs}_{\Lambda}(Var(k)^{sm}/S)$ to be the category whose objects are the one of $\Var(k)^{sm}/S$
and whose space of morphisms between
$X/S$ and $Y/S\in\Var(k)^{sm}/S$ is the free $\Lambda$ module $\mathcal{Z}^{fs/X}(X\times_S Y,\Lambda)$. 
The composition of morphisms is defined similary then in the absolute case (see \cite{C.Deglise}) 
\end{defi}

We have 
\begin{itemize}
\item the additive embedding of categories 
$\Tr(S):\mathbb Z(\Var(k)^{sm}/S)\hookrightarrow\Cor^{fs}_{\mathbb Z}(\Var(k)^{sm}/S)$ 
which gives the corresponding morphism of sites
 $\Tr(S):\Cor^{fs}_{\mathbb Z}(\Var(k)^{sm}/S)\to \mathbb Z(\Var(k)^{sm}/S)$.
\item the inclusion functor
$e_{var}(S):\Ouv(S)\hookrightarrow\Var(k)^{sm}/S$,
which gives the corresponding morphism of sites
$e_{var}(S):\Var(k)^{sm}/S\to\Ouv(S)$,
\item the inclusion functor 
$e^{tr}_{var}(S):=\Tr\circ e_{var}\:\Ouv(S)\hookrightarrow\Cor^{fs}_{\mathbb Z}(\Var(k)^{sm}/S)$
which gives the corresponding morphism of sites
$e^{tr}_{var}(S):=\Tr\circ e_{var}:\Cor^{fs}_{\mathbb Z}(\Var(k)^{sm}/S)\to\Ouv(S)$.
\end{itemize}

For each morphism $f:T\to S$ in $\Var(\mathbb C)$, we have 
\begin{itemize}
\item the pullback functor
$P(f):\Var(\mathbb C)/S\to\Var(\mathbb C)/T$, $P(f)(X/S)=(X\times_S T/T)$, $P(f)(h)=h_T$,
which gives the morphism of sites $P(f):\Var(\mathbb C)/T\to\Var(\mathbb C)/S$
\item the pullback functor
$P(f):\Cor^{fs}_{\mathbb Z}(\Var(\mathbb C)^{sm}/S)\to\Cor^{fs}_{\mathbb Z}(\Var(\mathbb C)^{sm}/T)$, 
$P(f)(X/S)=(X\times_S T/T)$, $P(f)(h)=h_T$,
which gives the morphism of sites
$P(f):\Cor^{fs}_{\mathbb Z}(\Var(\mathbb C)^{sm}/T)\to\Cor^{fs}_{\mathbb Z}(\Var(\mathbb C)^{sm}/S)$. 
\end{itemize}

For $S\in\Var(k)$, we consider the following two big categories :
\begin{itemize}
\item $\PSh(\Var(k)^{sm}/S,C^-(\mathbb Z))=\PSh_{\mathbb Z}(\mathbb Z(\Var(k)^{sm}/S),C^-(\mathbb Z))$, 
the category of bounded above complexes of presheaves on $\Var(k)^{sm}/S$, 
or equivalently additive presheaves on $\mathbb Z(\Var(k)^{sm}/S)$, 
sometimes, we will write for short $P^-(S)=\PSh(\Var(k)^{sm}/S,C^-(\mathbb Z))$,
\item $\PSh_{\mathbb Z}(\Cor^{fs}_{\mathbb Z}(\Var(k)^{sm}/S),C^-(\mathbb Z))$, 
the category of bounded above complexes of additive presheaves on $\Cor^{fs}_{\mathbb Z}(\Var(k)^{sm}/S)$
sometimes, we will write for short $PC^-(S)=\PSh(\Cor^{fs}_{\mathbb Z}(\Var(k)^{sm}/S),C^-(\mathbb Z))$,
\end{itemize}
and the adjonctions :
\begin{itemize}
\item $(\Tr(S)^*,\Tr(S)_*):
\PSh(\Var(k)^{sm}/S,C^-(\mathbb Z))\leftrightarrows\PSh_{\mathbb Z}(\Cor^{fs}_{\mathbb Z}(\Var(k)^{sm}/S),C^-(\mathbb Z))$, 
\item $(e_{var}(S)^*,e_{var}(S)_*):\PSh(\SmVar(\mathbb C),C^-(\mathbb Z))\leftrightarrows C^-(\mathbb Z)$,
\item $(e_{var}^{tr}(S)^*,e_{var}^{tr}(S)):\PSh(\Cor^{fs}_{\mathbb Z}(\Var(k)^{sm}/S),C^-(\mathbb Z))\leftrightarrows C^-(\mathbb Z)$,
\end{itemize}
given by $\Tr(S):\Cor^{fs}_{\mathbb Z}(\Var(k)^{sm}/S)\to \mathbb Z(\Var(k)^{sm}/S)$,
$e_{var}(S):\Var(k)^{sm}/S\to\Ouv(S)$ and
$e^{tr}_{var}(S):\Cor^{fs}_{\mathbb Z}(\Var(k)^{sm}/S)\to\Ouv(S)$ respectively.
We denote by $a_{et}:\PSh_{\mathbb Z}(\Var(k)^{sm}/S,\Ab)\to\Sh_{\mathbb Z,et}(\Var(k)^{sm}/S,\Ab)$
the etale sheaftification functor.
For $X/S\in\Var(k)^{sm}/S$, we denote by
\begin{equation}
\mathbb Z(X/S)\in\PSh(\Var(k)^{sm}/S,C^-(\mathbb Z)) \mbox{\;\;}, \mbox{\;\;}  
\mathbb Z_{tr}(X/S)\in\PSh_{\mathbb Z}(\Cor^{fs}_{\mathbb Z}(\Var(k)^{sm}/S),C^-(\mathbb Z))
\end{equation}
the presheaves represented by $X$. They are etale sheaves.
For $X/S=(X,h)\in\Var(k)/S$ with $h:X\to S$ non smooth,
\begin{equation}
\mathbb Z_{tr}(X/S)\in\PSh_{\mathbb Z}(\Cor^{fs}_{\mathbb Z}(\Var(k)^{sm}/S),C^-(\mathbb Z)), \;
Y\in\Var(k)^{sm}/S\to\mathcal{Z}^{fs/Y}(Y\times_S X,\mathbb Z) 
\end{equation}
is also an etale sheaf ; of course if $h:X\to S$ is not dominant then $\mathbb Z_{tr}(X/S)=0$

For a morphism $f:T\to S$ in $\Var(k)$ we have the adjonctions
\begin{itemize}
\item $(f^*,f_*):=(P(f)^*,P(f)_*):
\PSh(\Var(k)^{sm}/S,C^-(\mathbb Z))\leftrightarrows\PSh(\Var(k)^{sm}/T,C^-(\mathbb Z))$, 
\item $(f^*,f_*):=(P(f)^*,P(f)_*):
\PSh_{\mathbb Z}(\Cor^{fs}_{\mathbb Z}(\Var(k)^{sm}/S),C^-(\mathbb Z))
\leftrightarrows\PSh_{\mathbb Z}(\Cor^{fs}_{\mathbb Z}(\Var(k)^{sm}/T),C^-(\mathbb Z))$, 
\end{itemize}
given by $P(f):\Var(k)^{sm}/T\to\Var(k)^{sm}/S$ and 
$P(f):\Cor^{fs}_{\mathbb Z}(\Var(k)^{sm}/T)\to\Cor^{fs}_{\mathbb Z}(\Var(k)^{sm}/S)$, 
respectively.

\begin{defi}\label{RAmodstr}
\begin{itemize}
\item[(i)] The projective etale toplogy model structure on $\PSh(\Var(k)^{sm}/S,C^-(\mathbb Z))$ 
is defined in the similar way of the absolute case (c.f. definition \ref{VaretTr}(i)). 

\item[(ii)]The projective $(\mathbb A_k^1,et)$ model structure on the category
 $\PSh(\Var(k)^{sm}/S,C^-(\mathbb Z))$ is the left Bousfield localization 
of the projective etale topology model structure with respect to the class of maps
$\left\{\mathbb Z(X\times\mathbb A_k^1/S)[n]\to\mathbb Z(X/S)[n], \; X/S\in\Var(k)^{sm}/S,n\in\mathbb Z\right\}$. 

\item[(iii)] The projective etale toplogy model structure on $\PSh(\Cor^{fs}(\Var(k)^{sm}/S),C^-(\mathbb Z))$ 
is defined in the similar way of the absolute case (c.f. definition \ref{VaretTr}(ii)). 

\item[(iv)]The projective $(\mathbb A_k^1,et)$ model structure on the category
$\PSh_{\mathbb Z}(\Cor^{fs}_{\mathbb Z}(\Var(k)^{sm}/S),C^-(\mathbb Z))$ 
is the left Bousfield localization of the projective etale topology model structure with respect to the class of maps 
$\left\{\mathbb Z_{tr}(X\times\mathbb A_k^1/S)[n]\to\mathbb Z(X/S)[n],\; X/S\in\Var(k)^{sm}/S,n\in\mathbb Z\right\}$. 
\end{itemize}
\end{defi}

\begin{defi}
Let $S\in\Var(k)$
\begin{itemize}
\item[(i)] We define 
$\DM^-(S,\mathbb Z)_{et}:=\Ho_{\mathbb A_k^1,et}(\PSh_{\mathbb Z}(\Cor^{fs}_{\mathbb Z}(\Var(k)^{sm}/S),C^-(\mathbb Z)))$, 
to be the derived category of (effective) motives, it is
the homotopy category of $\PSh(\Cor^{fs}_{\mathbb Z}(\Var(k)^{sm}/S),C^-(\mathbb Z))$ 
with respect to the projective $(\mathbb A^1,et)$ model structure (c.f. definition \ref{RAmodstr}(ii)).
We denote by
\begin{equation*}
D^{tr}(\mathbb A^1,et)(S):\PSh_{\mathbb Z}(\Cor^{fs}_{\mathbb Z}(\Var(k)^{sm}/S,C^-(\mathbb Z)))\to\DM^{-}(S,\mathbb Z)_{et} \, , \,
D^{tr}(\mathbb A^1,et)(S)(F^{\bullet})=F^{\bullet}
\end{equation*}
the canonical localization functor.

\item[(ii)] By the same way, we denote 
$\DA^-(S,\mathbb Z)_{et}:=\Ho_{\mathbb A_k^1,et}(\PSh(\Var(k)^{sm}/S,C^-(\mathbb Z)))$ (c.f.\ref{RAmodstr}(i)) and
\begin{equation*}
D(\mathbb A^1,et)(S):\PSh_{\mathbb Z}(\Cor^{fs}_{\mathbb Z}(\SmVar(\mathbb C)),C^-(\mathbb Z)))\to\DA^{-}(S,\mathbb Z)_{et} \, , \,
D(\mathbb A^1,et)(S)(F^{\bullet})=F^{\bullet}
\end{equation*}
the canonical localization functor.
\end{itemize}
\end{defi}

\begin{prop}
Let $S\in\Var(k)$
\begin{itemize}
\item[(i)] $(\Tr(S)^*,\Tr(S)_*):\PSh(\Var(k)^{sm}/S,C^-(\mathbb Z))\leftrightarrows
\PSh_{\mathbb Z}(\Cor^{fs}_{\mathbb Z}(\Var(k)^{sm}/S),C^-(\mathbb Z))$
is a Quillen adjonction for the etale topology model structures
and a Quillen adjonction for the $(\mathbb A^1,et)$ model structures 
(c.f. definition \ref{RAmodstr} (i) and (ii) respectively). 

\item[(ii)] $(e_{var}(S)^*,e_{var}(S)_*):\PSh(\Var(k)^{sm}/S,C^-(\mathbb Z))\leftrightarrows C^-(\mathbb Z)$
is a Quillen adjonction for the etale topology model structures 
and a Quillen adjonction for the $(\mathbb A^1,et)$ model structures 
(c.f. definition \ref{RAmodstr} (i)). 

\item[(iii)] $(e_{var}(S)^{tr*},e_{var}^{tr}(S)_*):
\PSh(\Cor^{fs}_{\mathbb Z}(\Var(k)^{sm}/S),C^-(\mathbb Z))\leftrightarrows C^-(\mathbb Z)$
is a Quillen adjonction for the etale topology model structures
and a Quillen adjonction for the $(\mathbb A^1,et)$ model structures 
(c.f. definition \ref{RAmodstr} (ii)). 
\end{itemize}
\end{prop}

\begin{proof}

\noindent(i): Follows from the fact that $\Tr(S)_*$ derive trivially hence is a right Quillen functor.

\noindent(ii): Follows from the fact that $e_{var}(S)^*$ derive trivially hence is a left Quillen functor.

\noindent(iii): Follows from the fact that $e^{tr}_{var}(S)^*$ derive trivially hence is a left Quillen functor.

\end{proof}

\begin{prop}\label{adfVar}
Let $f:T\to S$ a morphism in $\Var(k)$.
\begin{itemize}
\item[(i)] The functors 
\begin{itemize}
\item $f_*:\PSh(\Var(k)^{sm}/T,C^-(\mathbb Z))\to\PSh(\Var(k)^{sm}/S,C^-(\mathbb Z))$ and
\item $f_*:\PSh(\Cor^{fs}_{\mathbb Z}(\Var(k)^{sm}/T),C^-(\mathbb Z))
\to\PSh(\Cor^{fs}_{\mathbb Z}(\Var(k)^{sm}/S),C^-(\mathbb Z))$, 
\end{itemize}
preserve and detect $\mathbb A^1$ local object, preserve etale fibrant objects.
\item[(ii)] The functors 
\begin{itemize}
\item $f^*:\PSh(\Var(k)^{sm}/S,C^-(\mathbb Z))\to\PSh(\Var(k)^{sm}/T,C^-(\mathbb Z))$ and
\item $f^*:\PSh(\Cor^{fs}_{\mathbb Z}(\Var(k)^{sm}/S),C^-(\mathbb Z))
\to\PSh(\Cor^{fs}_{\mathbb Z}(\Var(k)^{sm}/T),C^-(\mathbb Z))$, 
\end{itemize}
preserve and detect $(\mathbb A^1,et)$ equivalence.
\end{itemize}
\end{prop}

\begin{proof}
Point (i) follows imediately from definition of $\mathbb A^1$ local objects.
Point (ii) follows from point (i).
\end{proof}

We immediately deduce the following

\begin{prop}\label{adfVarcor}
Let $f:T\to S$ a morphism in $\Var(k)$. 
\begin{itemize}
\item[(i)] The adjonction 
$(f^*,f_*):\PSh(\Var(k)^{sm}/S,C^-(\mathbb Z))\leftrightarrows\PSh(\Var(k)^{sm}/T,C^-(\mathbb Z))$ 
is a Quillen adjonction with respect to the etale model structures and the $(\mathbb A^1,et)$ model structures
\item[(ii)] The adjonction
$(f^*,f_*):\PSh(\Cor^{fs}_{\mathbb Z}(\Var(k)^{sm}/S),C^-(\mathbb Z))
\leftrightarrows\PSh(\Cor^{fs}_{\mathbb Z}(\Var(k)^{sm}/T),C^-(\mathbb Z))$, 
is a Quillen adjonction with respect to the etale model structures and the $(\mathbb A^1,et)$ model structures
\end{itemize}
\end{prop}

\begin{proof}
By proposition \ref{adfVar}(ii), $f^*$ is a left Quillen functor.
\end{proof}

\begin{thm}\cite{C.Deglise}
Let $S\in\Var(k)$. The adjonction
$(\Tr(S)^*,\Tr(S)_*):\PSh(\Var(k)^{sm}/S,C^-(\mathbb Z))\leftrightarrows
\PSh_{\mathbb Z}(\Cor^{fs}_{\mathbb Z}(\Var(k)^{sm}/S),C^-(\mathbb Z))$
is a Quillen equivalence, that is the derived functor
\begin{equation*}
L\Tr(S)^*:\DA^-(S,\mathbb Z)\xrightarrow{\sim}\DM^-(S,\mathbb Z)
\PSh_{\mathbb Z}(\Cor^{fs}_{\mathbb Z}(\Var(k)^{sm}/S),C^-(\mathbb Z))
\end{equation*} 
is an isomorphism and 
$\Tr(S)_*:\DM^-(S,\mathbb Z)\xrightarrow{\sim}\DA^-(S,\mathbb Z)
\PSh_{\mathbb Z}(\Cor^{fs}_{\mathbb Z}(\Var(k)^{sm}/S),C^-(\mathbb Z))$
is its inverse
\end{thm}

We have all the property of the six functor formalism on $\Var(k)$ :

\begin{thm}\label{sixfunctorsVar}\cite{C.Deglise}
We have the six functor formalism on $\DM^-:\Var(k)\to\TriCat$, $S\in\Var(k)\mapsto\DM^-(S,\mathbb Z)$.
Indeed, the 2-functor $\DM^-:\Var(k)\to\TriCat$ is an homotopic 2-functor is the sense of \cite{AyoubT}.
In particular, we have well a defined pair of adjoint functors  
$(f_!,f^!):\DM^-(T,\mathbb Z)\leftrightarrows\DM^-(S,\mathbb Z)$, such  that
\begin{itemize}
\item $f_!=Rf_*$ if $f:T\to S$ is proper, 
\item $j_!$ is the left adjoint of $j^*$ if $j:S^o\hookrightarrow S$ is an open embbeding
($j^*$ admits a left adjoint since it is a smooth morphism).
\end{itemize}
\end{thm}

\begin{proof}
See \cite{C.Deglise} or \cite{AyoubT}.
We just recall the definition of $f_!$.
Let $f:T\to S$ a morphism in $\Var(k)$. Let $\bar{T}\in\PVar(k)$ a compactification of $T$.
Take $\hat T\subset\bar T\times S$ the closure of the graph of $f$. 
Then, $f=\hat f\circ j$ where $j:T\hookrightarrow\hat T$ is the open embedding and
$\hat f:=p_{S|\bar{T}}$ is a proper morphism, with $p_S:\bar T\times S\to S$ the projection.
Then for $F^{\bullet}\in PC^-(T)$ a $\mathbb A^1$ local and etale fibrant object
$f_!F^{\bullet}:=D(\mathbb A^1,et)(S)(\hat{f}_*j_lF^{\bullet})$ 
does not depends of the compactification of $f$ by the support property.
\end{proof}

\begin{defi}
Let $S\in\Var(\mathbb C)$.
\begin{itemize}
\item Let $M\in\DM^-(S,\mathbb Z)$ and $p\in\mathbb Z$.
The Tate twist of $M$ is $M\otimes D(\mathbb A^1,et)(\mathbb Z_S(p))$ with
$\mathbb Z_S(p):=(\mathbb Z_{tr}(\mathbb P^1\times S,\infty))^{\otimes p}\in PC^-(S)$.
For $F^{\bullet}\in PC^-(S)$, we denote $F^{\bullet}(p):=F^{\bullet}\otimes\mathbb Z_S(p)\in PC^-(S)$.
\item A motive $M\in\DM^-(S,\mathbb Z)$ is called constructible if it belongs to the thick
subcategory generated by motives of the form $D(\mathbb A^1,et)(S)(\mathbb Z_{tr}(X/S)(p))$,
with $X/S=(X,h)$, where $h:X\to S$ is a smooth morphism, and $p\in\mathbb Z$.
\end{itemize}
\end{defi}

\begin{defiprop}\label{RBMMot}
Let $S,X\in\Var(k)$ and $h:X\to S$ a morphism.  
\begin{itemize}
\item[(i)] The motive of $X/S=(X,h)$ is $M(X/S):=h_!h^!\mathbb Z_{tr}(S/S)\in\DM^-(S,\mathbb Z)$.
It is a constructible object.

\item[(ii)] The Borel-Moore motive of $X/S=(X,h)$ 
is $M^{BM}(X/S):=h_!h^*\mathbb Z_{tr}(S/S)=h_!\mathbb Z_{tr}(X/X)\in\DM^-(S,\mathbb Z)$.
It is a constructible object.

\item[(iii)] The cohomological motive of $X/S=(X,h)$ is 
$M(X/S):=Rh_*h^*\mathbb Z_{tr}(S/S)=Rh_*\mathbb Z_{tr}(X/X)\in\DM^-(S,\mathbb Z)$.
It is a constructible object.

\item[(iv)] The motive with compact support of $X/S=(X,h)$ is 
$M(X/S):=Rh_*h^!\mathbb Z_{tr}(S/S)\in\DM^-(S,\mathbb Z)$. It is a constructible object.
\end{itemize}

\end{defiprop}

\begin{proof} 
The fact that these four object associated to $X/S=(X,h)$ 
are constructible follows from \cite{C.Deglise} section 4.
\end{proof}

\subsection{The derived category of relative motives of analytic spaces}

Let $S\in\AnSp(\mathbb C)$. The category $\AnSp(\mathbb C)/S$ is the category whose objects are 
$X/S=(X,h)$ with $X\in\AnSp(\mathbb C)$ and $h:X\to S$ is a morphism,
and whose space of morphisms between
$X/S=(X,h_1)$ and $Y/S=(Y,h_2)\in\SmVar(S)$ are the morphism $g:X\to Y$ such that $h_2\circ g=h_1$. 
We denote by $\AnSp(\mathbb C)^{sm}/S\subset\AnSp(\mathbb C)/S$ the full subcategory constisting of the objects
$X/S=(X,h)$ with $X\in\AnSp(\mathbb C)$, such that $h:X\to S$ is a smooth morphism.
For $X/S=(X,h)\in\AnSp(\mathbb C)/S$, and $n\in\mathbb N$, we denote by 
\begin{itemize}
\item $(X\times\mathbb D^1/S):=(X\times\mathbb D^1,h\circ p_X)=(X\times_S(\mathbb D^1\times S)/S)$,
where $p_X:X\times\mathbb D^1\to X$ is the projection.
\item $(X\times\bar{\mathbb D}^n/S):=(X\times\bar{\mathbb D}^n,h\circ p_X)=(X\times_S(\bar{\mathbb D}^n\times S)/S)$,
where $p_X:X\times\bar{\mathbb D}^n\to X$ is the projection.
\end{itemize}

\begin{defi}
Let $S\in\AnSp(\mathbb C)$.
We define $\Cor^{fs}_{\Lambda}(\AnSp(\mathbb C)^{sm}/S)$ 
to be the category whose objects are the one of $\AnSp(\mathbb C)^{sm}/S$
and whose space of morphisms between
$X/S$ and $Y/S\in\AnSp(\mathbb C)^{sm}/S$ is the free $\Lambda$ module $\mathcal{Z}^{fs/X}(X\times_S Y,\Lambda)$. 
The composition of morphisms is defined similary then in the absolute case. 
\end{defi}

We have 
\begin{itemize}
\item the additive embedding of categories 
$\Tr(S):\mathbb Z(\AnSp(\mathbb C)^{sm}/S)\hookrightarrow\Cor^{fs}_{\mathbb Z}(\AnSp(\mathbb C)^{sm}/S)$ 
which gives the corresponding morphism of sites
 $\Tr(S):\Cor^{fs}_{\mathbb Z}(\AnSp(\mathbb C)^{sm}/S)\to \mathbb Z(\AnSp(\mathbb C)^{sm}/S)$.
\item the inclusion functor
$e_{an}(S):\Ouv(S)\hookrightarrow\AnSp(\mathbb C)^{sm}/S$,
which gives the corresponding morphism of sites
$e_{an}(S):\AnSp(\mathbb C)^{sm}/S\to\Ouv(S)$,
\item the inclusion functor 
$e^{tr}_{an}(S):=\Tr\circ e_{an}\:\Ouv(S)\hookrightarrow\Cor^{fs}_{\mathbb Z}(\AnSp(\mathbb C)^{sm}/S)$
which gives the corresponding morphism of sites
$e^{tr}_{an}(S):=\Tr\circ e_{an}:\Cor^{fs}_{\mathbb Z}(\AnSp(\mathbb C)^{sm}/S)\to\Ouv(S)$.
\end{itemize}

For each morphism $f:T\to S$ in $\AnSp(\mathbb C)$, we have 
\begin{itemize}
\item the pullback functor
$P(f):\AnSp(\mathbb C)/S\to\AnSp(\mathbb C)/T$, $P(f)(X/S)=(X\times_S T/T)$, $P(f)(h)=h_T$,
which gives the morphism of sites $P(f):\Var(\mathbb C)/T\to\Var(\mathbb C)/S$
\item the pullback functor
$P(f):\Cor^{fs}_{\mathbb Z}(\AnSp(\mathbb C)^{sm}/S)\to\Cor^{fs}_{\mathbb Z}(\AnSp(\mathbb C)^{sm}/T)$, 
$P(f)(X/S)=(X\times_S T/T)$, $P(f)(h)=h_T$,
which gives the morphism of sites
$P(f):\Cor^{fs}_{\mathbb Z}(\AnSp(\mathbb C)^{sm}/T)\to\Cor^{fs}_{\mathbb Z}(\AnSp(\mathbb C)^{sm}/S)$. 
\end{itemize}

For $S\in\AnSp(\mathbb C)$, we consider the following two big categories :
\begin{itemize}
\item $\PSh(\AnSp(\mathbb C)^{sm}/S,C^-(\mathbb Z))=\PSh_{\mathbb Z}(\mathbb Z(\AnSp(\mathbb C)^{sm}/S),C^-(\mathbb Z))$, 
the category of bounded above complexes of presheaves on $\AnSp(\mathbb C)^{sm}/S$, 
or equivalently additive presheaves on $\mathbb Z(\AnSp(\mathbb C)^{sm}/S)$, 
sometimes, we will write for short $P^-(S)=\PSh(\AnSp(\mathbb C)^{sm}/S,C^-(\mathbb Z))$,
\item $\PSh_{\mathbb Z}(\Cor^{fs}_{\mathbb Z}(\AnSp(\mathbb C)^{sm}/S),C^-(\mathbb Z))$, 
the category of bounded above complexes of additive presheaves on $\Cor^{fs}_{\mathbb Z}(\AnSp(\mathbb C)^{sm}/S)$
sometimes, we will write for short $PC^-(S)=\PSh(\Cor^{fs}_{\mathbb Z}(\AnSp(\mathbb C)^{sm}/S),C^-(\mathbb Z))$,
\end{itemize}
and the adjonctions :
\begin{itemize}
\item $(\Tr(S)^*,\Tr(S)_*):
\PSh(\AnSp(\mathbb C)^{sm}/S,C^-(\mathbb Z))\leftrightarrows\PSh_{\mathbb Z}(\Cor^{fs}_{\mathbb Z}(\AnSp(\mathbb C)^{sm}/S),C^-(\mathbb Z))$, 
\item $(e_{an}(S)^*,e_{an}(S)_*):\PSh(\SmVar(\mathbb C),C^-(\mathbb Z))\leftrightarrows C^-(\mathbb Z)$,
\item $(e_{an}^{tr}(S)^*,e_{an}^{tr}(S)):\PSh(\Cor^{fs}_{\mathbb Z}(\AnSp(\mathbb C)^{sm}/S),C^-(\mathbb Z))\leftrightarrows C^-(\mathbb Z)$,
\end{itemize}
given by $\Tr(S):\Cor^{fs}_{\mathbb Z}(\AnSp(\mathbb C)^{sm}/S)\to \mathbb Z(\AnSp(\mathbb C)^{sm}/S)$,
$e_{an}(S):\AnSp(\mathbb C)^{sm}/S\to\Ouv(S)$ and
$e^{tr}_{an}(S):\Cor^{fs}_{\mathbb Z}(\AnSp(\mathbb C)^{sm}/S)\to\Ouv(S)$ respectively.
We denote by $a_{et}:\PSh_{\mathbb Z}(\AnSp(\mathbb C)^{sm}/S,\Ab)\to\Sh_{\mathbb Z,et}(\AnSp(\mathbb C)^{sm}/S,\Ab)$
the etale sheaftification functor.
For $X/S\in\AnSp(\mathbb C)^{sm}/S$, we denote by
\begin{equation}
\mathbb Z(X/S)\in\PSh(\AnSp(\mathbb C)^{sm}/S,C^-(\mathbb Z)) \mbox{\;\;}, \mbox{\;\;}  
\mathbb Z_{tr}(X/S)\in\PSh_{\mathbb Z}(\Cor^{fs}_{\mathbb Z}(\AnSp(\mathbb C)^{sm}/S),C^-(\mathbb Z))
\end{equation}
the presheaves represented by $X$. They are usu sheaves.

For a morphism $f:T\to S$ in $\AnSp(\mathbb C)$ we have the adjonction
\begin{itemize}
\item $(f^*,f_*):=(P(f)^*,P(f)_*):
\PSh(\AnSp(\mathbb C)^{sm}/S,C^-(\mathbb Z))\leftrightarrows\PSh(\AnSp(\mathbb C)^{sm}/T,C^-(\mathbb Z))$, 
\item $(f^*,f_*):=(P(f)^*,P(f)_*):
\PSh_{\mathbb Z}(\Cor^{fs}_{\mathbb Z}(\AnSp(\mathbb C)^{sm}/S),C^-(\mathbb Z))
\leftrightarrows\PSh_{\mathbb Z}(\Cor^{fs}_{\mathbb Z}(\AnSp(\mathbb C)^{sm}/T),C^-(\mathbb Z))$, 
\end{itemize}
given by $P(f):\AnSp(\mathbb C)^{sm}/T\to\AnSp(\mathbb C)^{sm}/S$ and 
$P(f):\Cor^{fs}_{\mathbb Z}(\AnSp(\mathbb C)^{sm}/T)\to\Cor^{fs}_{\mathbb Z}(\AnSp(\mathbb C)^{sm}/S)$, 
respectively.

\begin{defi}\label{RDmodstr}
Let $S\in\AnSp(\mathbb C)$.
\begin{itemize}
\item[(i)] The projective usual toplogy model structure on $\PSh(\AnSp(\mathbb C)^{sm}/S,C^-(\mathbb Z))$ 
is defined in the similar way of the absolute case (c.f. definition \ref{AnTrusu}(i)). 

\item[(ii)]The projective $(\mathbb D^1,et)$ model structure on the category
 $\PSh(\AnSp(\mathbb C)^{sm}/S,C^-(\mathbb Z))$ is the left Bousfield localization 
of the projective usual topology model structure with respect to the class of maps
$\left\{\mathbb Z(X\times\mathbb D^1/S)[n]\to\mathbb Z(X/S)[n], \; X/S\in\AnSp(\mathbb C)^{sm}/S,n\in\mathbb Z\right\}$. 

\item[(iii)] The projective usual toplogy model structure on $\PSh(\Cor^{fs}(\AnSp(\mathbb C)^{sm}/S),C^-(\mathbb Z))$ 
is defined in the similar way of the absolute case (c.f. definition \ref{AnTrusu}(ii)). 

\item[(iv)]The projective $(\mathbb D^1,et)$ model structure on the category
$\PSh_{\mathbb Z}(\Cor^{fs}_{\mathbb Z}(\AnSp(\mathbb C)^{sm}/S),C^-(\mathbb Z))$ 
is the left Bousfield localization of the projective usual topology model structure with respect to the class of maps 
$\left\{\mathbb Z_{tr}(X\times\mathbb D^1/S)[n]\to\mathbb Z(X/S)[n],\; X/S\in\AnSp(\mathbb C)^{sm}/S,n\in\mathbb Z\right\}$. 
\end{itemize}
\end{defi}

\begin{defi}
Let $S\in\AnSp(\mathbb C)$.
\begin{itemize}
\item[(i)] We define 
$\AnDM^-(S,\mathbb Z):=\Ho_{\mathbb D^1,et}(\PSh_{\mathbb Z}(\Cor^{fs}_{\mathbb Z}(\AnSp(\mathbb C)^{sm}/S),C^-(\mathbb Z)))$, 
to be the derived category of (effective) motives, it is
the homotopy category of $\PSh(\Cor^{fs}_{\mathbb Z}(\AnSp(\mathbb C)^{sm}/S),C^-(\mathbb Z))$ 
with respect to the projective $(\mathbb D^1,usu)$ model structure (c.f. definition \ref{RDmodstr}(ii)).
We denote by
\begin{equation*}
D^{tr}(\mathbb D^1,usu)(S):\PSh_{\mathbb Z}(\Cor^{fs}_{\mathbb Z}(\AnSp(\mathbb C)^{sm}/S,C^-(\mathbb Z)))\to\DM^{-}(S,\mathbb Z)_{et} \, , \,
D^{tr}(\mathbb D^1,usu)(S)(F^{\bullet})=F^{\bullet}
\end{equation*}
the canonical localization functor.

\item[(ii)] By the same way, we denote 
$\AnDA^-(S,\mathbb Z):=\Ho_{\mathbb D^1,usu}(\PSh(\AnSp(\mathbb C)^{sm}/S,C^-(\mathbb Z)))$ (c.f.\ref{RDmodstr}(i)) and
\begin{equation*}
D(\mathbb D^1,usu)(S):\PSh_{\mathbb Z}(\Cor^{fs}_{\mathbb Z}(\SmVar(\mathbb C)),C^-(\mathbb Z)))\to\AnDA^{-}(S,\mathbb Z) \, , \,
D(\mathbb D^1,usu)(S)(F^{\bullet})=F^{\bullet}
\end{equation*}
the canonical localization functor.
\end{itemize}
\end{defi}

We now look at an explicit localization functor.

For $F^{\bullet}\in\PSh(\AnSp(\mathbb C)^{sm}/S,\Ab)$ and $X/S\in\AnSp(\mathbb C)^{sm}/S$, 
we have the complex $F(X\times\bar{\mathbb D}^*/S)$ associated to the cubical object 
$F(X\times\bar{\mathbb D}^*/S)$ in the category of abelian groups.

\begin{itemize}
\item If $F^{\bullet}\in\PSh(\AnSp(\mathbb C)^{sm}/S,C^-(\mathbb Z))$ , 
\begin{equation}
\underline{\sing}_{\bar{\mathbb D}^*}F^{\bullet}:=\Tot(\underline\Hom(\mathbb Z(\bar{\mathbb D}^*\times S),F^{\bullet}))\in\PSh(\AnSp(\mathbb C)^{sm}/S,C^-(\mathbb Z)) 
\end{equation}
is the total complex of presheaves associated to the bicomplex of presheaves $X/S\mapsto F^{\bullet}(\bar{\mathbb D}^*\times X/S)$,
and $\sing_{\bar{\mathbb D}}^*F^{\bullet}:=e_{an}(S)_*\underline{\sing}_{\bar{\mathbb D}^*}F^{\bullet}\in C^-(\mathbb Z)$.
We denote by $S(F^{\bullet}):F^{\bullet}\to\underline{\sing}_{\bar{\mathbb D}^*}F^{\bullet}$, 
\begin{equation}\label{RSFAn}
S(F^{\bullet}):\xymatrix{
\cdots\ar[r] & 0\ar[r]\ar[d] & 0\ar[r]\ar[d] & F^{\bullet}\ar[d]^{I}\ar[r] & 0\ar[r]\ar[d] & \cdots  \\
\cdots\ar[r] & \underline{\sing}_{\bar{\mathbb D}^2}F^{\bullet}\ar[r] & 
\underline{\sing}_{\bar{\mathbb D}^1}F^{\bullet}\ar[r] & F^{\bullet}\ar[r] & 0\ar[r] & \cdots} 
\end{equation}
the inclusion morphism of $\PSh(\AnSp(\mathbb C)^{sm}/S,C^-(\mathbb Z))$ :
For $f:F_1^{\bullet}\to F_2^{\bullet}$ a morphism in $P^-(\An,S)$, we denote by 
$S(f):\underline{\sing}_{\bar{\mathbb D}^*}F_1^{\bullet}\to\underline{\sing}_{\bar{\mathbb D}^*}F_2^{\bullet}$,
the morphism of $\PSh(\AnSp(\mathbb C)^{sm}/S,C^-(\mathbb Z))$ given by for $X/S\in\AnSp(\mathbb C)^{sm}/S$,
\begin{equation}\label{RSfAn}
S(f)(X/S):\xymatrix{
\cdots\ar[r] & F_1(\bar{\mathbb D}^2\times X/S)\ar[r]\ar[d]^{f(\bar{\mathbb D}^2\times X/S)} &
F_1(\bar{\mathbb D}^1\times X/S)\ar[r]\ar[d]^{f(\bar{\mathbb D}^1\times X/S)} & F_1^{\bullet}(X)\ar[d]^{f(X/S)}\ar[r] 
& 0\ar[r]\ar[d] & \cdots \\
\cdots\ar[r] & F_2(\bar{\mathbb D}^2\times X/S)\ar[r] & F_2(\bar{\mathbb D}^1\times X/S)\ar[r] 
& F_2^{\bullet}(X/S)\ar[r] & 0\ar[r] & \cdots.}  
\end{equation}

\item If $F^{\bullet}\in\PSh_{\mathbb Z}(\Cor^{fs}_{\mathbb Z}(\AnSp(\mathbb C)^{sm}/S),C^-(\mathbb Z))$, 
\begin{equation}
\underline{\sing}_{\bar{\mathbb D}^*}F^{\bullet}:=\underline\Hom(\mathbb Z_{tr}(\bar{\mathbb D}^*\times S),F^{\bullet})
\in\PSh_{\mathbb Z}(\Cor^{fs}_{\mathbb Z}(\AnSp(\mathbb C)^{sm}/S),C^-(\mathbb Z)), 
\end{equation}
is the complex of presheaves 
associated to the bicomplex of presheaves $X/S\mapsto F^{\bullet}(\bar{\mathbb D}^*\times X/S)$,
and $\sing_{\bar{\mathbb D^*}F^{\bullet}}:=e^{tr}_{an}(S)_*\underline{\sing}_{\bar{\mathbb D}^*}F^{\bullet}\in C^-(\mathbb Z)$.
We have the inclusion morphism (\ref{RSFAn})
\begin{equation*}
S(\Tr_*F^{\bullet}):\Tr(S)_*F^{\bullet}\to
\underline{\sing}_{\bar{\mathbb D}^*}\Tr(S)_*F^{\bullet}=\Tr(S)_*\underline{\sing}_{\bar{\mathbb D}^*}F^{\bullet}
\end{equation*}
which is a morphism in $\PSh(\Cor^{fs}_{\mathbb Z}(\AnSp(\mathbb C)^{sm}/S),C^-(\mathbb Z))$ 
denoted the same way $S(F^{\bullet}):F^{\bullet}\to\underline{\sing}_{\bar{\mathbb D}^*}F^{\bullet}$.  
For $f:F_1^{\bullet}\to F_2^{\bullet}$ a morphism $PC^-(S)$, we 
have the morphism (\ref{RSfAn})
\begin{equation*} 
S(\Tr_*f):\Tr(S)_*\underline{\sing}_{\bar{\mathbb D}^*}F_1^{\bullet}=\underline{\sing}_{\bar{\mathbb D}^*}\Tr(S)_*F_1^{\bullet}\to
\underline{\sing}_{\bar{\mathbb D}^*}\Tr(S)_*F_2^{\bullet}=\Tr(S)_*\underline{\sing}_{\bar{\mathbb D}^*}F_2^{\bullet}
\end{equation*}
which is a morphism in $PC^-(\An,S)$ denoted the same way
$S(f):\underline{\sing}_{\bar{\mathbb D}^*}F_1^{\bullet}\to\underline{\sing}_{\bar{\mathbb D}^*}F_2^{\bullet}$.
\end{itemize}

For $F^{\bullet}\in\PSh(\Cor^{fs}_{\mathbb Z}(\AnSp(\mathbb C)^{sm}/S),C^-(\mathbb Z))$, we have by definition 
$\Tr(S)_*\underline{\sing}_{\bar{\mathbb D}^*}F^{\bullet}=\underline{\sing}_{\bar{\mathbb D}^*}\Tr(S)_*F^{\bullet}$ and 
$\Tr(S)_*S(F^{\bullet})=S(\Tr(S)_*F^{\bullet})$.

\begin{prop}
Let $S\in\AnSp(\mathbb C)$
\begin{itemize}
\item[(i)] $(\Tr(S)^*,\Tr(S)_*):\PSh(\AnSp(\mathbb C)^{sm}/S,C^-(\mathbb Z))\leftrightarrows
\PSh_{\mathbb Z}(\Cor^{fs}_{\mathbb Z}(\AnSp(\mathbb C)^{sm}/S),C^-(\mathbb Z))$
is a Quillen adjonction for the usual topology model structures
and a Quillen adjonction for the $(\mathbb D^1,et)$ model structures 
(c.f. definition \ref{RAmodstr} (i) and (ii) respectively). 

\item[(ii)] $(e_{an}(S)^*,e_{var}(S)_*):\PSh(\AnSp(\mathbb C)^{sm}/S,C^-(\mathbb Z))\leftrightarrows C^-(\mathbb Z)$
is a Quillen adjonction for the usual topology model structures 
and a Quillen adjonction for the $(\mathbb D^1,et)$ model structures 
(c.f. definition \ref{RAmodstr} (i)). 

\item[(iii)] $(e_{an}(S)^{tr*},e_{var}^{tr}(S)_*):
\PSh(\Cor^{fs}_{\mathbb Z}(\AnSp(\mathbb C)^{sm}/S),C^-(\mathbb Z))\leftrightarrows C^-(\mathbb Z)$
is a Quillen adjonction for the usual topology model structures
and a Quillen adjonction for the $(\mathbb D^1,et)$ model structures 
(c.f. definition \ref{RAmodstr} (ii)). 
\end{itemize}
\end{prop}

\begin{proof}

\noindent(i): Follows from the fact that $\Tr(S)_*$ derive trivially hence is a right Quillen functor.

\noindent(ii): Follows from the fact that $e_{var}(S)^*$ derive trivially hence is a left Quillen functor.

\noindent(iii): Follows from the fact that $e^{tr}_{var}(S)^*$ derive trivially hence is a left Quillen functor.

\end{proof}

\begin{prop}\label{adfAn}
Let $f:T\to S$ a morphism in $\AnSp(\mathbb C)$.
\begin{itemize}
\item[(i)] The functors 
\begin{itemize}
\item $f_*:\PSh(\AnSp(\mathbb C)^{sm}/T,C^-(\mathbb Z))\to\PSh(\AnSp(\mathbb C)^{sm}/S,C^-(\mathbb Z))$ and
\item $f_*:\PSh(\Cor^{fs}_{\mathbb Z}(\AnSp(\mathbb C)^{sm}/T),C^-(\mathbb Z))
\to\PSh(\Cor^{fs}_{\mathbb Z}(\AnSp(\mathbb C)^{sm}/S),C^-(\mathbb Z))$, 
\end{itemize}
preserve and detect $\mathbb D^1$ local object, usu fibrant objects.
\item[(ii)] The functors 
\begin{itemize}
\item $f^*:\PSh(\AnSp(\mathbb C)^{sm}/S,C^-(\mathbb Z))\to\PSh(\AnSp(\mathbb C)^{sm}/T,C^-(\mathbb Z))$ and
\item $f^*:\PSh(\Cor^{fs}_{\mathbb Z}(\AnSp(\mathbb C)^{sm}/S),C^-(\mathbb Z))
\to\PSh(\Cor^{fs}_{\mathbb Z}(\AnSp(\mathbb C)^{sm}/T),C^-(\mathbb Z))$, 
\end{itemize}
preserve and detect $(\mathbb D^1,usu)$ equivalence.
\end{itemize}
\end{prop}

\begin{proof}
Point (i) follows immediately from definition of $\mathbb D^1$ local objects.
Point (ii) follows from point (i).
\end{proof}

We immediately deduce the following

\begin{prop}\label{adfAncor}
Let $f:T\to S$ a morphism in $\AnSp(\mathbb C)$. 
\begin{itemize}
\item[(i)] The adjonction 
$(f^*,f_*):\PSh(\AnSp(\mathbb C)^{sm}/S,C^-(\mathbb Z))\leftrightarrows\PSh(\AnSp(\mathbb C)^{sm}/T,C^-(\mathbb Z))$ 
is a Quillen adjonction with respect to the usual topology model structures and the $(\mathbb D^1,et)$ model structures
\item[(ii)] The adjonction
$(f^*,f_*):\PSh(\Cor^{fs}_{\mathbb Z}(\AnSp(\mathbb C)^{sm}/S),C^-(\mathbb Z))
\leftrightarrows\PSh(\Cor^{fs}_{\mathbb Z}(\AnSp(\mathbb C)^{sm}/T),C^-(\mathbb Z))$, 
is a Quillen adjonction with respect to the usual topology model structures and the $(\mathbb D^1,et)$ model structures
\end{itemize}
\end{prop}

\begin{proof}

\noindent(i): By \cite{AyoubT}, $f^*$ derive trivially for the usual topology model structures.
Now, by proposition \ref{adfAn}(ii), $f^*$ derive trivially for the $(\mathbb D^1,usu)$ model structures.
In particular $f^*$ is a left Quillen functor.

\noindent(ii): Similar to point (i).

\end{proof}

\begin{thm}\label{RAnS}
Let $S\in\AnSp(\mathbb C)$
\begin{itemize}
\item[(i)] For $F^{\bullet}\in\PSh(\AnSp(\mathbb C)^{sm}/S,C^-(\mathbb Z))$,
$\underline{\sing}_{\bar{\mathbb D}^*}F^{\bullet}\in\PSh(\AnSp(\mathbb C)^{sm}/S,C^-(\mathbb Z))$ 
is $\mathbb D^1$ local and 
the inclusion morphism $S(F^{\bullet}):F^{\bullet}\to\underline{\sing}_{\bar{\mathbb D}^*}F^{\bullet}$
is an $(\mathbb D^1,usu)$ equivalence. 

\item[(ii)] For $F^{\bullet}\in\PSh_{\mathbb Z}(\Cor^{fs}_{\mathbb Z}(\AnSp(\mathbb C)^{sm}/S),C^-(\mathbb Z))$, 
$\underline{\sing}_{\bar{\mathbb D}^*}F^{\bullet}\in\PSh_{\mathbb Z}(\Cor^{fs}_{\mathbb Z}(\AnSp(\mathbb C)^{sm}),C^-(\mathbb Z))$ 
is $\mathbb D^1$ local and 
the inclusion morphism $S(F^{\bullet}):F^{\bullet}\to\underline{\sing}_{\bar{\mathbb D}^*}F^{\bullet}$ 
is an $(\mathbb D^1,usu)$ equivalence. 
\end{itemize}
\end{thm}

\begin{proof}

Similar to the proof of the absolute case.  
 
\end{proof}

\begin{thm}\label{RAnTr}
Let $S\in\AnSp(\mathbb C)$. 
\begin{itemize}
\item[(i)]The adjonction
$(\Tr(S)^*,\Tr(S)_*):\PSh(\AnSp(\mathbb C)^{sm}/S,C^-(\mathbb Z))\leftrightarrows
\PSh_{\mathbb Z}(\Cor^{fs}_{\mathbb Z}(\AnSp(\mathbb C)^{sm}/S),C^-(\mathbb Z))$
is a Quillen equivalence for the $(\mathbb D^1,usu)$ model structures. 
That is, the derived functor 
\begin{equation}
L\Tr^*:\AnDA^-(S,\mathbb Z)\xrightarrow{\sim}\AnDM^-(S,\mathbb Z) 
\end{equation}
is an isomorphism
and $\Tr_*:\AnDM^-(S,\mathbb Z)\xrightarrow{\sim}\AnDA^-(S,\mathbb Z)$ is it inverse. 

\item[(ii)] The adjonction
$(e_{an}(S)^*,e_{an}(S)_*):C^-(\mathbb Z)\leftrightarrows\PSh_{\mathbb Z}(\AnSm(\mathbb C),C^-(\mathbb Z))$
is a Quillen equivalence for the $(\mathbb I^1,usu)$ model structures. 
That is, the derived functor
\begin{equation}
e_{an}^*:D^-(S)\xrightarrow{\sim}\AnDA^-(S,\mathbb Z)
\end{equation}
is an isomorphism
and $Re_{an}(S)_*:\AnDA^-(S,\mathbb Z)\xrightarrow{\sim}D^-(S)$ is it inverse. 

\item[(iii)] The adjonction
$(e^{tr}_{an}(S)^*,e^{tr}_{an}(S)_*):C^-(S)\leftrightarrows
\PSh_{\mathbb Z}(\Cor^{fs}_{\mathbb Z}(\AnSp^{sm}(\mathbb C)/S),C^-(\mathbb Z))$
is a Quillen equivalence for the $(\mathbb D^1,usu)$ model structures. 
That is, the derived functor
\begin{equation}
e^{tr}_{an}(S)^*:D^-(S)\xrightarrow{\sim}\AnDM^-(S,\mathbb Z)
\end{equation}
is an isomorphism
and $Re^{tr}_{an}(S)_*:\AnDM^-(S,\mathbb Z)\xrightarrow{\sim}D^-(S)$ is it inverse. 
\end{itemize}
\end{thm}

\begin{proof}

\noindent(i): Similar to the proof of the absolute case.

\noindent(ii): It is proved in \cite{AyoubB}.

\noindent(iii): Follows from (i) and (ii).

\end{proof}

Let $f:T\to S$ a morphism in $\AnSp(\mathbb C)$. 
There is a canonical morphism of functor $\phi(f^*,S)$ which associate to 
$F^{\bullet}\in\PSh_{\mathbb Z}(\Cor^{fs}_{\mathbb Z}(\AnSp(\mathbb C)^{sm}/S),C^-(\mathbb Z))$
the morphism
\begin{equation*}
\phi(f^*,S)(F^{\bullet}):f^*\underline{\sing}_{\bar{\mathbb D}^*}F^{\bullet}
\to\underline{\sing}_{\bar{\mathbb D}^*}f^*F^{\bullet}
\end{equation*}
in $\PSh_{\mathbb Z}(\Cor^{fs}_{\mathbb Z}(\AnSp(\mathbb C)^{sm}/T),C^-(\mathbb Z))$
given by for $Y/T\in\AnSp(\mathbb C)^{sm}/T$, the morphism 
\begin{equation*}
\phi(f^*,S)(F^{\bullet})(Y/T):\lim_{Y/T\to X_T/T}F^{\bullet}(X\times\bar{\mathbb D}^*/S)
\to\lim_{Y\times\bar{\mathbb D}^*/T\to X_T}F^{\bullet}(X/S)
\end{equation*}
given by $(h:Y/T\to X_T/T)\mapsto(h\circ p_Y:Y\times\bar{\mathbb D}^*/T\to X_T/T)$
and $F^{\bullet}(p_X):(X\times\bar{\mathbb D}^*/S)\to F^{\bullet}(X/S)$.

Let $f:T\to S$ a morphism in $\AnSp(\mathbb C)$. 
There is also canonical morphism of functor $\phi(f,e_{an})$ which associate to 
$F^{\bullet}\in\PSh_{\mathbb Z}(\Cor^{fs}_{\mathbb Z}(\AnSp(\mathbb C)^{sm}/S),C^-(\mathbb Z))$
the following morphism in $C^-(T)$
\begin{equation*}
\phi(f^*,e)(F^{\bullet}):f^*e_{S*}F^{\bullet}\to
e_{T*}e_T^*f^*e(S)_*F^{\bullet}=e_{T*}f^*e_S^*e_{S*}F^{\bullet}\to e_{T*}f^*F^{\bullet}
\end{equation*}
given by the adjonction morphisms and
denoting for simplicity $e_S=e_{an}^{tr}(S)$ and $e_T=e_{an}^{tr}(T)$. 

Let $f:T\to S$ a morphism in $\AnSp(\mathbb C)$. 
We denote by $\phi(f^*,e_{an},S)$ the morphism of functor, which for
$F^{\bullet}\in\PSh_{\mathbb Z}(\Cor^{fs}_{\mathbb Z}(\AnSp(\mathbb C)^{sm}/S),C^-(\mathbb Z))$, 
associate the following composition in $C^-(T)$
\begin{equation*}
\phi(f^*,e_{an},S)(F^{\bullet}):f^*\sing_{\bar{\mathbb D}^*}F^{\bullet}\xrightarrow{\phi(f^*,e)(\sing_{\bar{\mathbb D}^*}F^{\bullet})}
e(T)_*f^*\underline{\sing}_{\bar{\mathbb D}^*}F^{\bullet}\xrightarrow{e(T)_*\phi(f^*,S)(F^{\bullet})}
\sing_{\bar{\mathbb D}^*}f^*F^{\bullet}.
\end{equation*}

By definition, we have:

\begin{prop}\label{adRSAn}
Let $f:T\to S$ a morphism in $\AnSp(\mathbb C)$. For $F^{\bullet}\in PC^-(\An,S)$, the morphism 
$\phi(f^*,e,S)(F^{\bullet}):f^*\sing_{\bar{\mathbb D}^*}F^{\bullet}\to\sing_{\bar{\mathbb D}^*}f^*F^{\bullet}$
in $C^-(T)$ is an isomorphism.
\end{prop}

\begin{proof}
By definition, the morphism $\phi(f^*,e,S)(F^{\bullet})$ 
is given by for $T^o\subset T$, 
\begin{equation*}
\phi(f^*,e,S)(F^{\bullet})(T^o/T):\lim_{T^o\to f^{-1}(S^o)}F^{\bullet}(S^o\times\bar{\mathbb D}^*/S)
\to\lim_{T^o\times\bar{\mathbb D}^*/T\to X_T}F^{\bullet}(X/S)
\end{equation*}
given by the isomorphism $(h:T^o/T\to X_T/T)\mapsto(h\circ p_T:T^o\times\bar{\mathbb D}^*/T\to X_T/T)$
and $F^{\bullet}(p_{S^o}):F^{\bullet}(S^o\times\bar{\mathbb D}^*/S)\to F^{\bullet}(S^o/S)$.
\end{proof}

We will use in the last subsection the following :

\begin{prop}\label{RDusuEq}
\begin{itemize}
\item[(i)] Let $G_1^{\bullet},G_2^{\bullet}\in\PSh_{\mathbb Z}(\Cor^{fs}_{\mathbb Z}(\AnSp(\mathbb C)^{sm}/S),C^-(\mathbb Z))$
and $f:G_1^{\bullet}\to G_2^{\bullet}$ a morphism. If 
\begin{equation*}
e^{tr}_{an}(S)_*S(f):\sing_{\bar{\mathbb D^*}}G_1^{\bullet}\to\sing_{\bar{\mathbb D^*}}G^{\bullet}_2 
\end{equation*}
is an equivalence usu local in $C^{-}(S)$,
then $f:G_1^{\bullet}\to G_2^{\bullet}$ is an  $(\mathbb D^1,usu)$ local equivalence. 

\item[(ii)] Let $G_1^{\bullet},G_2^{\bullet}\in\PSh_{\mathbb Z}(\Cor^{fs}_{\mathbb Z}(\AnSp(\mathbb C)^{sm}/S),C^-(\mathbb Z))$
and $f:G_1^{\bullet}\to G_2^{\bullet}$ a morphism. If 
\begin{itemize}
\item $G_1^{\bullet}$ and $G_2^{\bullet}$ are $\mathbb D^1$ local,
\item $e^{tr}_{an}(S)_*f:e^{tr}_{an*}G_1^{\bullet}\to e^{tr}_{an*}G_2^{\bullet}$ is an equivalence usu local in $C^{-}(S)$,
\end{itemize}
then $f:G_1^{\bullet}\to G_2^{\bullet}$ is an $(\mathbb D^1,usu)$ local equivalence. 
\end{itemize}
\end{prop}

\begin{proof}
Similar to the proof of proposition \ref{DusuEq}. Point (i) follows from theorem \ref{RAnTr}(iii),
and point (ii) follows from point (i).
\end{proof}

\subsection{Presheaves and transfers on relative CW complexes}

Let $S\in\Top$. The category $\Top/S$ is the category whose objects are 
$X/S=(X,h)$ with $X\in\Top$ and $h:X\to S$ is a morphism,
and whose space of morphisms between
$X/S=(X,h_1)$ and $Y/S=(Y,h_2)\in\Top/S$ are the morphism $g:X\to Y$ such that $h_2\circ g=h_1$. 

We now restrict to the full subcategory $\CW\subset\Top$ of CW complexes.

\begin{defi}
A morphism $f:X\to Y$, $X,Y\in\CW$ with $X$ connected is said to be smooth if 
for all $x\in X$, there exist a neighborhood $U_x\subset X$ of $x$ in $X$ and
an open embedding $k:U_x\hookrightarrow U_x\times\mathbb R^{d_X}$ such that
$f_{|U_x}=p_{U_x}\circ k$.
\end{defi}

Let $S\in\CW$. The category $\CW/S$ is the category whose objects are 
$X/S=(X,h)$ with $X\in\CW$ and $h:X\to S$ is a morphism,
and whose space of morphisms between
$X/S=(X,h_1)$ and $Y/S=(Y,h_2)\in\CW/S$ are the morphism $g:X\to Y$ such that $h_2\circ g=h_1$. 
We denote by $\CW^{sm}/S\subset\CW/S$ the full subcategory consisting of the objects
$X/S=(X,h)$ with $X\in\CW$, such that $h:X\to S$ is a smooth morphism.
For $X/S=(X,h)\in\CW/S$, and $n\in\mathbb N$, we denote by 
\begin{itemize}
\item $(X\times\mathbb I^1/S):=(X\times\mathbb I^1,h\circ p_X)=(X\times_S(\mathbb I^1\times S)/S)$,
where $p_X:X\times_k\mathbb I^1\to X$ is the projection.
\item $(X\times\mathbb I^n/S):=(X\times\mathbb I^n,h\circ p_X)=(X\times_S(\mathbb I^n\times S)/S)$,
where $p_X:X\times\mathbb I^n\to X$ is the projection.
\end{itemize}

\begin{defi}
Let $S\in\CW$.
We define $\Cor^{fs}_{\Lambda}(\CW^{sm}/S)$ to be the category whose objects are the one of $\CW^{sm}/S$
and whose space of morphisms between
$X/S$ and $Y/S\in\CW^{sm}/S$ is the free $\Lambda$ module $\mathcal{Z}^{fs/X}(X\times_S Y,\Lambda)$. 
The composition of morphisms is defined similarly then in the absolute case. 
\end{defi}

We have 
\begin{itemize}
\item the additive embedding of categories 
$\Tr(S):\mathbb Z(\CW^{sm}/S)\hookrightarrow\Cor^{fs}_{\mathbb Z}(\CW^{sm}/S)$ 
which gives the corresponding morphism of sites
 $\Tr(S):\Cor^{fs}_{\mathbb Z}(\CW^{sm}/S)\to \mathbb Z(\CW^{sm}/S)$.
\item the inclusion functor
$e_{cw}(S):\Ouv(S)\hookrightarrow\CW^{sm}/S$,
which gives the corresponding morphism of sites
$e_{cw}(S):\CW^{sm}/S\to\Ouv(S)$,
\item the inclusion functor 
$e^{tr}_{cw}(S):=\Tr\circ e_{cw}\:\Ouv(S)\hookrightarrow\Cor^{fs}_{\mathbb Z}(\CW^{sm}/S)$
which gives the corresponding morphism of sites
$e^{tr}_{cw}(S):=\Tr\circ e_{cw}:\Cor^{fs}_{\mathbb Z}(\CW^{sm}/S)\to\Ouv(S)$.
\end{itemize}

For each morphism $f:T\to S$ in $\CW$, we have 
\begin{itemize}
\item the pullback functor
$P(f):\CW/S\to\CW/T$, $P(f)(X/S)=(X\times_S T/T)$, $P(f)(h)=h_T$,
which gives the morphism of sites $P(f):\CW/T\to\CW/S$
\item the pullback functor
$P(f):\Cor^{fs}_{\mathbb Z}(\CW^{sm}/S)\to\Cor^{fs}_{\mathbb Z}(\CW^{sm}/T)$, 
$P(f)(X/S)=(X\times_S T/T)$, $P(f)(h)=h_T$,
which gives the morphism of sites
$P(f):\Cor^{fs}_{\mathbb Z}(\CW^{sm}/T)\to\Cor^{fs}_{\mathbb Z}(\CW^{sm}/S)$. 
\end{itemize}

For $S\in\CW$, we consider the following two big categories :
\begin{itemize}
\item $\PSh(\CW^{sm}/S,C^-(\mathbb Z))=\PSh_{\mathbb Z}(\mathbb Z(\CW^{sm}/S),C^-(\mathbb Z))$, 
the category of bounded above complexes of presheaves on $\CW^{sm}/S$, 
or equivalently additive presheaves on $\mathbb Z(\CW^{sm}/S)$, 
sometimes, we will write for short $P^-(CW,S)=\PSh(\CW^{sm}/S,C^-(\mathbb Z))$,
\item $\PSh_{\mathbb Z}(\Cor^{fs}_{\mathbb Z}(\CW^{sm}/S),C^-(\mathbb Z))$, 
the category of bounded above complexes of additive presheaves on $\Cor^{fs}_{\mathbb Z}(\CW^{sm}/S)$
sometimes, we will write for short $PC^-(CW,S)=\PSh(\Cor^{fs}_{\mathbb Z}(\CW^{sm}/S),C^-(\mathbb Z))$,
\end{itemize}
and the adjonctions :
\begin{itemize}
\item $(\Tr(S)^*,\Tr(S)_*):
\PSh(\CW^{sm}/S,C^-(\mathbb Z))\leftrightarrows\PSh_{\mathbb Z}(\Cor^{fs}_{\mathbb Z}(\CW^{sm}/S),C^-(\mathbb Z))$, 
\item $(e_{cw}(S)^*,e_{cw}(S)_*):\PSh(\SmVar(\mathbb C),C^-(\mathbb Z))\leftrightarrows C^-(\mathbb Z)$,
\item $(e_{cw}^{tr}(S)^*,e_{cw}^{tr}(S)):\PSh(\Cor^{fs}_{\mathbb Z}(\CW^{sm}/S),C^-(\mathbb Z))\leftrightarrows C^-(\mathbb Z)$,
\end{itemize}
given by $\Tr(S):\Cor^{fs}_{\mathbb Z}(\CW^{sm}/S)\to \mathbb Z(\CW^{sm}/S)$,
$e_{cw}(S):\CW^{sm}/S\to\Ouv(S)$ and
$e^{tr}_{cw}(S):\Cor^{fs}_{\mathbb Z}(\CW^{sm}/S)\to\Ouv(S)$ respectively.
We denote by $a_{et}:\PSh_{\mathbb Z}(\CW^{sm}/S,\Ab)\to\Sh_{\mathbb Z,et}(\CW^{sm}/S,\Ab)$
the etale sheaftification functor.
For $X/S\in\CW^{sm}/S$, we denote by
\begin{equation}
\mathbb Z(X/S)\in\PSh(\CW^{sm}/S,C^-(\mathbb Z)) \mbox{\;\;}, \mbox{\;\;}  
\mathbb Z_{tr}(X/S)\in\PSh_{\mathbb Z}(\Cor^{fs}_{\mathbb Z}(\CW^{sm}/S),C^-(\mathbb Z))
\end{equation}
the presheaves represented by $X$. They are usu sheaves.
For $X/S=(X,h)\in\CW/S$ with $h:X\to S$ non smooth,
\begin{equation}
\mathbb Z_{tr}(X/S)\in\PSh_{\mathbb Z}(\Cor^{fs}_{\mathbb Z}(\CW^{sm}/S),C^-(\mathbb Z)), \;
Y\in\CW^{sm}/S\to\mathcal{Z}^{fs/Y}(Y\times_S X,\mathbb Z) 
\end{equation}
is also an usu sheaf ; of course if $h:X\to S$ is not dominant then $\mathbb Z_{tr}(X/S)=0$

For a morphism $f:T\to S$ in $\CW$ we have the adjonctions
\begin{itemize}
\item $(f^*,f_*):=(P(f)^*,P(f)_*):
\PSh(\CW^{sm}/S,C^-(\mathbb Z))\leftrightarrows\PSh(\CW^{sm}/T,C^-(\mathbb Z))$, 
\item $(f^*,f_*):=(P(f)^*,P(f)_*):
\PSh_{\mathbb Z}(\Cor^{fs}_{\mathbb Z}(\CW^{sm}/S),C^-(\mathbb Z))
\leftrightarrows\PSh_{\mathbb Z}(\Cor^{fs}_{\mathbb Z}(\CW^{sm}/T),C^-(\mathbb Z))$, 
\end{itemize}
given by $P(f):\CW^{sm}/T\to\CW^{sm}/S$ and 
$P(f):\Cor^{fs}_{\mathbb Z}(\CW^{sm}/T)\to\Cor^{fs}_{\mathbb Z}(\CW^{sm}/S)$, 
respectively.

\begin{defi}\label{RImodstr}
\begin{itemize}
\item[(i)] The projective usual topology model structure on $\PSh(\CW^{sm}/S,C^-(\mathbb Z))$ 
is defined in the similar way of the absolute case (c.f. definition \ref{CWTrusu}(i)). 

\item[(ii)]The projective $(\mathbb I^1,usu)$ model structure on the category
 $\PSh(\CW^{sm}/S,C^-(\mathbb Z))$ is the left Bousfield localization 
of the projective usual topology model structure with respect to the class of maps
$\left\{\mathbb Z(X\times\mathbb I^1/S)[n]\to\mathbb Z(X/S)[n], \; X/S\in\CW^{sm}/S,n\in\mathbb Z\right\}$. 

\item[(iii)] The projective usual topology model structure on $\PSh(\Cor^{fs}(\CW^{sm}/S),C^-(\mathbb Z))$ 
is defined in the similar way of the absolute case (c.f. definition \ref{CWTrusu}(ii)). 

\item[(iv)]The projective $(\mathbb I^1,usu)$ model structure on the category
$\PSh_{\mathbb Z}(\Cor^{fs}_{\mathbb Z}(\CW^{sm}/S),C^-(\mathbb Z))$ 
is the left Bousfield localization of the projective usual topology model structure with respect to the class of maps 
$\left\{\mathbb Z_{tr}(X\times\mathbb I^1/S)[n]\to\mathbb Z(X/S)[n],\; X/S\in\CW^{sm}/S,n\in\mathbb Z\right\}$. 
\end{itemize}
\end{defi}

\begin{defi}
Let $S\in\CW$.
\begin{itemize}
\item[(i)] We define 
$\CwDM^-(S,\mathbb Z):=\Ho_{\mathbb I^1,usu}(\PSh_{\mathbb Z}(\Cor^{fs}_{\mathbb Z}(\CW^{sm}/S),C^-(\mathbb Z)))$, 
to be the derived category of (effective) motives, it is
the homotopy category of $\PSh(\Cor^{fs}_{\mathbb Z}(\CW^{sm}/S),C^-(\mathbb Z))$ 
with respect to the projective $(\mathbb I^1,usu)$ model structure (c.f. definition \ref{RImodstr}(ii)).
We denote by
\begin{equation*}
D^{tr}(\mathbb I^1,usu)(S):\PSh_{\mathbb Z}(\Cor^{fs}_{\mathbb Z}(\CW^{sm}/S,C^-(\mathbb Z)))\to\CwDM^{-}(S,\mathbb Z) \, , \,
D^{tr}(\mathbb I^1,usu)(S)(F^{\bullet})=F^{\bullet}
\end{equation*}
the canonical localization functor.

\item[(ii)] By the same way, we denote 
$\CwDA^-(S,\mathbb Z):=\Ho_{\mathbb I^1,usu}(\PSh(\CW^{sm}/S,C^-(\mathbb Z)))$ (c.f.\ref{RImodstr}(i)) and
\begin{equation*}
D(\mathbb I^1,usu)(S):\PSh_{\mathbb Z}(\Cor^{fs}_{\mathbb Z}(\CW^{sm}/S),C^-(\mathbb Z)))\to\CwDA^{-}(S,\mathbb Z) \, , \,
D(\mathbb I^1,usu)(S)(F^{\bullet})=F^{\bullet}
\end{equation*}
the canonical localization functor.
\end{itemize}
\end{defi}

We now look at an explicit localization functor.

For $F^{\bullet}\in\PSh(\CW^{sm}/S,\Ab)$ and $X/S\in\CW^{sm}/S$, 
we have the complex $F(X\times\mathbb I^*/S)$ associated to the cubical object 
$F(X\times\mathbb I^*/S)$ in the category of abelian groups.

\begin{itemize}
\item If $F^{\bullet}\in\PSh(\CW^{sm}/S,C^-(\mathbb Z))$ , 
\begin{equation}
\underline{\sing}_{\mathbb I^*}F^{\bullet}:=\Tot(\underline\Hom(\mathbb Z(\mathbb I^*\times S),F^{\bullet}))\in\PSh(\CW^{sm}/S,C^-(\mathbb Z)) 
\end{equation}
is the total complex of presheaves associated to the bicomplex of presheaves $X/S\mapsto F^{\bullet}(\mathbb I^*\times X/S)$,
and $\sing_{\mathbb I^*}F^{\bullet}:=e_{cw}(S)_*\underline{\sing}_{\mathbb I^*}F^{\bullet}\in C^-(\mathbb Z)$.
We denote by $S(F^{\bullet}):F^{\bullet}\to\underline{\sing}_{\mathbb I^*}F^{\bullet}$, 
\begin{equation}\label{RSFCW}
S(F^{\bullet}):\xymatrix{
\cdots\ar[r] & 0\ar[r]\ar[d] & 0\ar[r]\ar[d] & F^{\bullet}\ar[d]^{I}\ar[r] & 0\ar[r]\ar[d] & \cdots  \\
\cdots\ar[r] & \underline{\sing}_{\mathbb I^2}F^{\bullet}\ar[r] & \underline{\sing}_{\mathbb I^1}F^{\bullet}\ar[r] 
& F^{\bullet}\ar[r] & 0\ar[r] & \cdots} 
\end{equation}
the inclusion morphism of $\PSh(\CW^{sm}/S,C^-(\mathbb Z))$ :
For $f:F_1^{\bullet}\to F_2^{\bullet}$ a morphism $\PSh(\CW^{sm}/S,C^-(\mathbb Z))$, we denote by 
$S(f):\underline{\sing}_{\mathbb I^*}F_1^{\bullet}\to\underline{\sing}_{\mathbb I^*}F_2^{\bullet}$,
the morphism of $\PSh(\CW^{sm}/S,C^-(\mathbb Z))$ given by for $X/S\in\CW^{sm}/S$,
\begin{equation}\label{RSfCW}
S(f)(X/S):\xymatrix{
\cdots\ar[r] & F_1(\mathbb I^2\times X/S)\ar[r]\ar[d]^{f(\mathbb I^2\times X/S)} &
F_1(\mathbb I^1\times X/S)\ar[r]\ar[d]^{f(\mathbb I^1\times X/S)} & F_1^{\bullet}(X/S)\ar[d]^{f(X/S)}\ar[r] & 0\ar[r]\ar[d] & \cdots \\
\cdots\ar[r] & F_2(\mathbb I^2\times X/S)\ar[r] & F_2(\mathbb I^1\times X/S)\ar[r] & F_2^{\bullet}(X/S)\ar[r] & 0\ar[r] & \cdots.}  
\end{equation}

\item If $F^{\bullet}\in\PSh_{\mathbb Z}(\Cor^{fs}_{\mathbb Z}(\CW^{sm}/S),C^-(\mathbb Z))$, 
\begin{equation}
\underline{\sing}_{\mathbb I^*}F^{\bullet}:=\underline\Hom(\mathbb Z_{tr}(\mathbb I^*\times S),F^{\bullet})
\in\PSh_{\mathbb Z}(\Cor^{fs}_{\mathbb Z}(\CW^{sm}/S),C^-(\mathbb Z)), 
\end{equation}
is the complex of presheaves 
associated to the bicomplex of presheaves $X/S\mapsto F^{\bullet}(\mathbb I^*\times X/S)$,
and $\sing_{\mathbb I^*}F^{\bullet}:=e^{tr}_{cw}(S)_*\underline{\sing}_{\mathbb I^*}F^{\bullet}\in C^-(\mathbb Z)$.
We have the inclusion morphism (\ref{RSFCW})
\begin{equation*}
S(\Tr_*F^{\bullet}):\Tr(S)_*F^{\bullet}\to\underline{\sing}_{\mathbb I^*}\Tr(S)_*F^{\bullet}
=\Tr(S)_*\underline{\sing}_{\mathbb I^*}F^{\bullet}
\end{equation*}
which is a morphism in $\PSh(\Cor^{fs}_{\mathbb Z}(\CW^{sm}/S),C^-(\mathbb Z))$ 
denoted the same way $S(F^{\bullet}):F^{\bullet}\to\underline{\sing}_{\mathbb I^*}F^{\bullet}$.  
For $f:F_1^{\bullet}\to F_2^{\bullet}$ a morphism $PC^-(S)$, we 
have the morphism (\ref{RSfCW})
\begin{equation*} 
S(\Tr_*f):\Tr(S)_*\underline{\sing}_{\mathbb I^*}F_1^{\bullet}=\underline{\sing}_{\mathbb I^*}\Tr(S)_*F_1^{\bullet}\to
\underline{\sing}_{\mathbb I^*}\Tr(S)_*F_2^{\bullet}=\Tr(S)_*\underline{\sing}_{\mathbb I^*}F_2^{\bullet}
\end{equation*}
which is a morphism in $PC^-(S)$ denoted the same way
$S(f):\underline{\sing}_{\mathbb I^*}F_1^{\bullet}\to\underline{\sing}_{\mathbb I^*}F_2^{\bullet}$.
\end{itemize}

For $F^{\bullet}\in\PSh(\Cor^{fs}_{\mathbb Z}(\CW^{sm}/S),C^-(\mathbb Z))$, we have by definition 
$\Tr(S)_*\underline{\sing}_{\mathbb I^*}F^{\bullet}=\underline{\sing}_{\mathbb I^*}\Tr(S)_*F^{\bullet}$ and 
$\Tr(S)_*S(F^{\bullet})=S(\Tr(S)_*F^{\bullet})$.

We recall the definition of homotopy in the relative setting :
\begin{itemize}
\item Let $S,X,Y\in\Top$ and $h_X:X\to S$, $h_Y:Y\to S$ two maps.
Let $f_0:X\to Y$, $f_1:X\to Y$ be two maps such that $h_Y\circ f_0=h_X$ and $h_Y\circ f_0=h_X$,
that is such that they define morphisms from $X/S=(X,h_X)$ to $Y/S=(Y,h_Y)$ in $\Top/S$.
We say that $f_0$, $f_1$ are $\mathbb I^1$ homotopic,
if there exist $h:X\times\mathbb I^1\to Y$ such that 
\begin{itemize}
\item $h_Y\circ h=h_X\circ p_X$ that is $h:X\times\mathbb I^1/S\to Y/S$ is a morphism in $\Top/S$, 
\item $f_0=h\circ(I_X\times i_0)$ and $f_1=h\circ(I_X\times i_1)$, 
\end{itemize}
with $(I_X\times i_0):X\times\left\{0\right\}\hookrightarrow X\times\mathbb I^1$ 
and $(I_X\times i_1):X\times\left\{1\right\}\hookrightarrow X\times\mathbb I^1$ 
the inclusions and $p_X:X\times\mathbb I^1\to X$ the projection.

\item Let $S\in\CW$. Let $F^{\bullet},G^{\bullet}\in\PSh(\CW/S,C^-(\mathbb Z))$. 
We say that two maps $\phi_0:F^{\bullet}\to G^{\bullet}$ and $\phi_1:F^{\bullet}\to G^{\bullet}$ 
are $\mathbb I^1$ homotopic if there exist
$\tilde{\phi}:F^{\bullet}\to\underline{\Hom}(\mathbb Z(\mathbb I^1\times S),G^{\bullet})$ such that
$\phi_0=\underline{G}^{\bullet}(i_0)\circ\tilde{\phi}$ and $\phi_1=\underline{G}^{\bullet}(i_1)\circ\tilde{\phi}$, where,
\begin{itemize}
\item $\underline{G}^{\bullet}(i_0):\underline{\Hom}(\mathbb Z(\mathbb I^1),G^{\bullet})\to G^{\bullet}$ is
induced by $i_0:\left\{\pt\right\}\to\mathbb I^1$, that is, for $X/S\in\CW/S$, 
$\underline{G}^{\bullet}(i_0)(X/S)=G^{\bullet}(I_X\times i_0):G^{\bullet}(X\times\mathbb I^1/S)\to G^{\bullet}(X/S)$,  
\item $\underline{G}^{\bullet}(i_1):\underline{\Hom}(\mathbb Z(\mathbb I^1),G^{\bullet})\to G^{\bullet}$ is
induced by $i_1:\left\{\pt\right\}\to\mathbb I^1$, that is, for $X/S\in\CW/S$, 
$\underline{G}^{\bullet}(i_1)(X/S)=G^{\bullet}(I_X\times i_1):G^{\bullet}(X\times\mathbb I^1)\to G^{\bullet}(X/S).$  
\end{itemize}
\end{itemize} 

Similary to the absolute case, we have the following lemmas :

\begin{lem}\label{RI1hptlem}
Let $X/S,Y/S\in\CW/S$ and $f_0:X/S\to Y/S$, $f_1:X/S\to Y/S$ two morphisms.
If $f_0$ and $f_1$ are $\mathbb I^1$ homotopic, then
\begin{itemize}
\item $\mathbb Z(f_0):\mathbb Z(X/S)\to\mathbb Z(Y/S)$ and $\mathbb Z(f_1):\mathbb Z(X/S)\to\mathbb Z(Y)$ 
are $\mathbb I^1$ homotopic in $\PSh(\CW/S,C^-(\mathbb Z))$,
\item $\Tr_*\mathbb Z_{tr}(f_0):\Tr_*\mathbb Z_{tr}(X/S)\to\Tr_*\mathbb Z_{tr}(Y/S)$ 
and $\Tr_*\mathbb Z_{tr}(f_1):\Tr_*\mathbb Z_{tr}(X/S)\to\Tr_*\mathbb Z_{tr}(Y/S)$ 
are $\mathbb I^1$ homotopic in $\PSh(\CW/S,C^-(\mathbb Z))$.
\end{itemize}
\end{lem}

\begin{proof}

Similar to the absolute case.

\end{proof}

\begin{lem}\label{RI1hpt}
Let $S\in\CW$. Let $F^{\bullet}\in\PSh(\CW/S,C^-(\mathbb Z))$.
\begin{itemize}
\item[(i)] Let $X/S,Y/S\in\CW/S$ and $f_0:X/S\to Y/S$, $f_1:X/S\to Y/S$ be two morphism. 
If $f_0$ and $f_1$ are $\mathbb I^1$ homotopic, then the maps of complexes 
\begin{itemize}
\item $\underline{\sing}_{\mathbb I^*}F^{\bullet}(f_0): \Tot F^{\bullet}(Y\times\mathbb I^*/S)
\to\Tot F^{\bullet}(X\times\mathbb I^*/S)$  and
\item $\underline{\sing}_{\mathbb I^*}F^{\bullet}(f_1): \Tot F^{\bullet}(Y\times\mathbb I^*/S)
\to\Tot F^{\bullet}(X\times\mathbb I^*/S)$ 
\end{itemize}
induces the same map on homology.
\item[(ii)] Let $X/S,Y/S\in\CW/S$, if $f:X\to Y$ is a $\mathbb I^1$ homotopy equivalence then
\begin{equation}
\underline{\sing}_{\mathbb I^*}F^{\bullet}(f): \Tot F^{\bullet}(Y\times\mathbb I^*/S)\to\Tot F^{\bullet}(X\times\mathbb I^*/S)
\end{equation}
is a quasi-isomorphism of complexes of abelian groups.
\end{itemize}
\begin{itemize}
\item[(iii)] Let $F^{\bullet},G^{\bullet}\in\PSh(\CW/S,C^-(\mathbb Z))$ and 
$\phi_0:F^{\bullet}\to G^{\bullet}$, $\phi_1:F^{\bullet}\to G^{\bullet}$ be two maps.
If $\phi_0$ and $\phi_1$ are $\mathbb I^1$ homotopic, then $\phi_0=\phi_1\in\CwDA^{-}(S)$. 
\item[(iv)] Let $F^{\bullet},G^{\bullet}\in\PSh(\CW/S,C^-(\mathbb Z))$, if  
$\phi:F^{\bullet}\to G^{\bullet}$ is a $\mathbb I^1$ homotopy equivalence then $\phi$ is a $(\mathbb I^1,usu)$ local equivalence.
\end{itemize}
\end{lem}

\begin{proof}

Similar to the absolute case.

\end{proof}

\begin{lem}\label{RCWTr}
\begin{itemize}
\item[(i)] A complex of presheaves $F^{\bullet}\in\PSh(\Cor^{fs}_{\mathbb Z}(\CW^{sm}/S),C^-(\mathbb Z))$ is 
$\mathbb I^1$ local if and only if 
$\Tr(S)_*F^{\bullet}\in\PSh(\CW^{sm}/S,C^-(\mathbb Z))$ is 
$\mathbb I^1$ local.

\item[(ii)] A morphism $\phi:F^{\bullet}\to G^{\bullet}$ in $\PSh(\Cor^{fs}_{\mathbb Z}(\CW^{sm}/S),C^-(\mathbb Z))$ is 
an $(\mathbb I^1,usu)$ local equivalence if and only if
$\Tr(S)_*\phi:\Tr(S)_*F^{\bullet}\to\Tr_*G^{\bullet}$ is an $(\mathbb I^1,usu)$ local equivalence.
\end{itemize}
\end{lem}

\begin{proof}

\noindent(i): Similar to the absolute case.

\noindent(ii): As in the absolute case the only if part follows from lemma \ref{RI1hptlem}(ii) and the if part follows from (i).

\end{proof}

\begin{prop}
Let $S\in\CW$
\begin{itemize}
\item[(i)] $(\Tr(S)^*,\Tr(S)_*):\PSh(\CW^{sm}/S,C^-(\mathbb Z))\leftrightarrows
\PSh_{\mathbb Z}(\Cor^{fs}_{\mathbb Z}(\CW^{sm}/S),C^-(\mathbb Z))$
is a Quillen adjonction for the etale topology model structures
and a Quillen adjonction for the $(\mathbb I^1,et)$ model structures 
(c.f. definition \ref{RImodstr} (i) and (ii) respectively). 

\item[(ii)] $(e_{cw}(S)^*,e_{cw}(S)_*):\PSh(\CW^{sm}/S,C^-(\mathbb Z))\leftrightarrows C^-(\mathbb Z)$
is a Quillen adjonction for the etale topology model structures 
and a Quillen adjonction for the $(\mathbb I^1,et)$ model structures 
(c.f. definition \ref{RImodstr} (i)). 

\item[(iii)] $(e_{cw}(S)^{tr*},e_{cw}^{tr}(S)_*):
\PSh(\Cor^{fs}_{\mathbb Z}(\CW^{sm}/S),C^-(\mathbb Z))\leftrightarrows C^-(\mathbb Z)$
is a Quillen adjonction for the etale topology model structures
and a Quillen adjonction for the $(\mathbb I^1,et)$ model structures 
(c.f. definition \ref{RImodstr} (ii)). 
\end{itemize}
\end{prop}

\begin{proof}

\noindent(i): Follows from the fact that $\Tr(S)_*$ by lemme \ref{RCWTr} derive trivially hence is a right Quillen functor.

\noindent(ii): Follows from the fact that $e_{cw}(S)^*$ derive trivially hence is a left Quillen functor.

\noindent(iii): Follows from the fact that $e^{tr}_{cw}(S)^*$ derive trivially hence is a left Quillen functor.

\end{proof}

\begin{prop}\label{adfCW}
Let $f:T\to S$ a morphism in $\CW$.
\begin{itemize}
\item[(i)] The functors 
\begin{itemize}
\item $f_*:\PSh(\CW^{sm}/T,C^-(\mathbb Z))\to\PSh(\CW^{sm}/S,C^-(\mathbb Z))$ and
\item $f_*:\PSh(\Cor^{fs}_{\mathbb Z}(\CW^{sm}/T),C^-(\mathbb Z))
\to\PSh(\Cor^{fs}_{\mathbb Z}(\CW^{sm}/S),C^-(\mathbb Z))$, 
\end{itemize}
preserve and detect $\mathbb I^1$ local object, usu fibrant objects.
\item[(ii)] The functors 
\begin{itemize}
\item $f^*:\PSh(\CW^{sm}/S,C^-(\mathbb Z))\to\PSh(\CW^{sm}/T,C^-(\mathbb Z))$ and
\item $f^*:\PSh(\Cor^{fs}_{\mathbb Z}(\CW^{sm}/S),C^-(\mathbb Z))
\to\PSh(\Cor^{fs}_{\mathbb Z}(\CW^{sm}/T),C^-(\mathbb Z))$, 
\end{itemize}
preserve and detect $(\mathbb I^1,usu)$ equivalence.
\end{itemize}
\end{prop}

\begin{proof}
Point (i) follows immediately from definition of $\mathbb I^1$ local objects.
Point (ii) follows from point (i).
\end{proof}

We immediately deduce the following

\begin{prop}\label{adfCWcor}
Let $f:T\to S$ a morphism in $\CW$. 
\begin{itemize}
\item[(i)] The adjonction 
$(f^*,f_*):\PSh(\CW^{sm}/S,C^-(\mathbb Z))\leftrightarrows\PSh(\CW^{sm}/T,C^-(\mathbb Z))$ 
is a Quillen adjonction with respect to the usu model structures and the $(\mathbb I^1,usu)$ model structures
\item[(ii)] The adjonction
$(f^*,f_*):\PSh(\Cor^{fs}_{\mathbb Z}(\CW^{sm}/S),C^-(\mathbb Z))
\leftrightarrows\PSh(\Cor^{fs}_{\mathbb Z}(\CW^{sm}/T),C^-(\mathbb Z))$, 
is a Quillen adjonction with respect to the usu model structures and the $(\mathbb I^1,usu)$ model structures
\end{itemize}
\end{prop}

\begin{proof}

\noindent(i): By \cite{AyoubT}, $f^*$ derive trivially for the usual topology model structure.
Now, by proposition \ref{adfCW}(ii), $f^*$ derive trivially for the $(\mathbb I^1,usu)$ model structure.
In particular, $f^*$ is a left Quillen functor.

\noindent(i): Similar to point (i).

\end{proof}

We deduce from lemma \ref{RI1hpt}(ii), the point (i) of the following proposition :

\begin{prop}\label{RactadjCW}
Let $S\in\CW$
\begin{itemize}
\item[(i)] For $F^{\bullet}\in\PSh(\CW^{sm}/S,C^-(\mathbb Z))$, the adjonction morphism 
\begin{equation}
\ad(e_{cw}(S)^*,e_{cw}(S)_*)(\underline{\sing}_{\mathbb I^*}F^{\bullet}):
e_{cw}(S)^*e_{cw}(S)_*\underline{\sing}_{\mathbb I^*}F^{\bullet}\to\underline{\sing}_{\mathbb I^*}F^{\bullet}
\end{equation}
is an equivalence usu local.

\item[(ii)] For $F^{\bullet}\in\PSh_{\mathbb Z}(\Cor^{fs}_{\mathbb Z}(\CW^{sm}/S),C^-(\mathbb Z))$, the adjonction morphism 
\begin{equation}
\ad(e^{tr}_{cw}(S)^*,e^{tr}_{cw}(S)_*)(\underline{\sing}_{\mathbb I^*}F^{\bullet}):
e^{tr}_{cw}(S)^*e^{tr}_{cw}(S)_*\underline{\sing}_{\mathbb I^*}F^{\bullet}\to\underline{\sing}_{\mathbb I^*}F^{\bullet}
\end{equation}
is an equivalence usu local.

\item[(iii)] For $K^{\bullet}\in C^-(\mathbb Z)$, the adjonction morphisms
\begin{itemize}
\item $\ad(e_{cw}(S)^*,e_{cw}(S)_*)(K^{\bullet}):K^{\bullet}\to e_{cw}(S)_*e_{cw}(S)^*K^{\bullet}$ 
\item $\ad(e^{tr}_{cw}(S)^*,e^{tr}_{cw}(S)_*)(K^{\bullet}):
K^{\bullet}\to e^{tr}_{cw}(S)_*e^{tr}_{cw}(S)_*K^{\bullet}$
\end{itemize}
are isomorphisms.
\end{itemize}
\end{prop}

\begin{proof}

\noindent(i): Let $X/S\in\CW^{sm}/S$. Since the question is local, 
we can assume, after shrinking $X/S$, 
that $X/S=(V\times S^o,p_S)$, with $S^o\subset S$ an open subset
and $V\subset\mathbb R^{d_X}$ a contractible open subset and 
$p_S:\mathbb R^{d_X}\to S$ the projection. 
It then follows from lemma \ref{RI1hpt}(ii).

\noindent(ii): It is a particular point of (i), since $\Tr(S)_*$ preserve usu local equivalences.

\noindent(iii): Obvious

\end{proof}

\begin{prop}\label{Rcle}
Let $S\in\CW$
\begin{itemize}
\item[(i)] The functor 
$\Tr(S)_*:\PSh_{\mathbb Z}(\Cor^{fs}_{\mathbb Z}(\CW),C^-(\mathbb Z))\to\PSh(\CW,C^-(\mathbb Z))$ 
derive trivially.
\item[(ii)] For $K^{\bullet}\in C^-(S)$, $e_{cw}(S)^*K^{\bullet}$ is $\mathbb I^1$ local.
\item[(iii)] For $K^{\bullet}\in C^-(S)$, $e^{tr}_{cw}(S)^*K^{\bullet}$ is $\mathbb I^1$ local.
\end{itemize}
\end{prop}

\begin{proof}

\noindent(i): Follows from lemma \ref{RCWTr}(ii).

\noindent (ii): Let $e_{cw}(S)^*K^{\bullet}\to L^{\bullet}$ an usu local equivalence, with $L^{\bullet}$ usu fibrant.
Since $e_{cw}(S)^*K^{\bullet}$ is usu equivalent to $ L^{\bullet}$ it suffices to prove that $L^{\bullet}$ is $\mathbb I^1$ local.
Since $L^{\bullet}$ is usu fibrant, we have to prove that 
\begin{equation*}
\underline{L}(p):L^{\bullet}\to\underline\Hom(\mathbb Z(S\times\mathbb I^1),L^{\bullet}) 
\end{equation*}
is an equivalence usu local
The proof is now similar to \cite{AyoubB} proposition 1.6 etape B.

\noindent(iii): Follows from (ii) and lemma \ref{RCWTr}(i).

\end{proof}

\begin{thm}\label{RCWsingI}
Let $S\in\CW$
\begin{itemize}
\item[(i)] For $F^{\bullet}\in\PSh(\CW^{sm}/S,C^-(\mathbb Z))$,
$\underline{\sing}_{\mathbb I^*}F^{\bullet}\in\CW^{sm}/S,C^-(\mathbb Z))$ 
is $\mathbb I^1$ local and 
the inclusion morphism $S(F^{\bullet}):F^{\bullet}\to\underline{\sing}_{\mathbb I^*}F^{\bullet}$
is an $(\mathbb I^1,usu)$ equivalence. 

\item[(ii)] For $F^{\bullet}\in\PSh_{\mathbb Z}(\Cor^{fs}_{\mathbb Z}(\CW^{sm}/S),C^-(\mathbb Z))$, 
$\underline{\sing}_{\bar{\mathbb I}^*}F^{\bullet}\in\PSh_{\mathbb Z}(\Cor^{fs}_{\mathbb Z}(\CW^{sm}/S),C^-(\mathbb Z))$ 
is $\mathbb I^1$ local and 
the inclusion morphism $S(F^{\bullet}):F^{\bullet}\to\underline{\sing}_{\mathbb I^*}F^{\bullet}$ 
is an $(\mathbb I^1,usu)$ equivalence. 
\end{itemize}
\end{thm}

\begin{proof}

\noindent(i): As in the absolute case, the fact that $\underline{\sing}_{\mathbb I^*}F^{\bullet}$
is a $\mathbb I^1$ local object follows from proposition \ref{RactadjCW}(i) and proposition \ref{Rcle}(ii).
We now prove that $S(F^{\bullet})$ is an $(\mathbb I^1,usu)$ local equivalence.
As in the absolute case, it suffice to show that
that for all $n\in\mathbb Z$, the morphism 
\begin{equation*}
\underline{F}^{\bullet}(p_n):F^{\bullet}\to\underline{\Hom}(\mathbb Z(\mathbb I^n),F^{\bullet}), \;
X/S\in\CW/S\mapsto\underline{F}^{\bullet}(p_n)(X/S)=F^{\bullet}(p_X):
F^{\bullet}(X/S)\to F^{\bullet}(\mathbb I^n\times X/S) 
\end{equation*}
is an equivalence $(\mathbb I^1,usu)$ local.
For $X/S\in\CW/S$, the morphism
\begin{equation*}
\theta_{1,n}(X/S):(\mathbb I^n\times\mathbb I^1\times X)/S\to\mathbb I^n\times X/S \; ; \; 
(t_1,\cdots,t_n,t_{n+1},x)\mapsto(t_1-t_{n+1}t_1,\cdots, t_n-t_{n+1}t_n,x)
\end{equation*}
define an $\mathbb I^1$ homotopy from $I_{\mathbb I^n\times X}$ to $0\times I_X$ ; 
on the other side
$p_X\circ(0\times I_X)=I_X$, with $p_X:\mathbb I^n\times X\to X$ the projection.
Thus, $\underline{F}^{\bullet}(\theta_{1,n})$
define an $\mathbb I^1$ homotopy from $\underline{F}^{\bullet}(I_{\mathbb I^n})$ to
$\underline{F}^{\bullet}(0)$ ; on the other side,
$\underline{F}^{\bullet}(0)\circ\underline{F}^{\bullet}(p_n)=I$.
This proves (i).

\noindent(ii): As in the absolute case, follows from (i).

\end{proof}

\begin{thm}\label{RCWMot}
Let $S\in\CW$. Then,
\begin{itemize}  
\item[(i)] The adjonction
$(e_{cw}(S)^*,e_{cw}(S)_*):C^-(\mathbb Z)\leftrightarrows\PSh_{\mathbb Z}(\CW^{sm}/S,C^-(\mathbb Z))$
is a Quillen equivalence for the $(\mathbb I^1,usu)$ model structures. 
That is, the derived functor
\begin{equation}
e_{cw}(S)^*:D^-(\mathbb Z)\xrightarrow{\sim}\CwDA^-(\mathbb Z)
\end{equation}
is an isomorphism
and $Re_{cw}(S)_*:\CwDA^-(\mathbb Z)\xrightarrow{\sim}D^-(S)$ is it inverse. 

\item[(ii)] The adjonction
$(e^{tr}_{cw}(S)^*,e^{tr}_{cw}(S)_*):C^-(\mathbb Z)\leftrightarrows\PSh_{\mathbb Z}(\Cor^{fs}_{\mathbb Z}(\CW^{sm}/S),C^-(\mathbb Z))$
is a Quillen equivalence for the $(\mathbb I^1,usu)$ model structures. 
That is, the derived functor
\begin{equation}
e^{tr}_{cw}(S)^*:D^-(\mathbb Z)\xrightarrow{\sim}\CwDM^-(S,\mathbb Z)
\end{equation}
is an isomorphism
and $Re^{tr}_{cw}(S)_*:\CwDM^-(\mathbb Z)\xrightarrow{\sim}D^-(S)$ is it inverse.  

\item[(iii)] The functor 
$\Tr(S)_*:\PSh(\CW^{sm}/S,C^-(\mathbb Z))\to\PSh_{\mathbb Z}(\Cor^{fs}_{\mathbb Z}(\CW^{sm}/S),C^-(\mathbb Z))$
induces an isomorphism
and $\Tr(S)_*:\CwDM^-(S,\mathbb Z)\xrightarrow{\sim}\CwDA^-(S,\mathbb Z)$. 
\end{itemize}
\end{thm}

\begin{proof}

\noindent(i): It follows from proposition \ref{RactadjCW}(i) and theorem \ref{RCWsingI}(i). 

\noindent(ii): It follows from proposition \ref{RactadjCW}(ii) and theorem \ref{RCWsingI}(ii). 
It also follows from (i) and (ii).

\noindent(iii): It follows from (i) and (ii).

\end{proof}

\begin{rem}\label{RTrCWrem}
As in the absolute case, for $F^{\bullet}\in\PSh(\Cor_{\mathbb Z}(\CW^{sm}/S),C^-(\mathbb Z))$, 
$\ad(\Tr^*,\Tr_*)(F^{\bullet}):F^{\bullet}\to\Tr^*\Tr_*F^{\bullet}$
is an isomorphism in $\PSh(\Cor_{\mathbb Z}(\CW^{sm}),C^-(\mathbb Z))$ and we can prove that
for $X/S\in\CW^{sm}/S$, the embedding
\begin{equation*}
\ad(\Tr^*,\Tr_*)(\underline{\sing}_{\mathbb I^*}\mathbb Z(X)):
\underline{\sing}_{\mathbb I^*}\mathbb Z(X)\to\Tr_*\underline{\sing}_{\mathbb I^*}\mathbb Z_{tr}(X)
\end{equation*}
in $\PSh(\CW^{sm}/S,C^-(\mathbb Z))$ is an equivalence usu local. 
We would deduce from this that $L\Tr^*$ is the inverse of $\Tr_*$.
We will not do it since we don't use it.
\end{rem}

Let $f:T\to S$ a morphism in $\CW$. 
There is a canonical morphism of functor $\phi(f^*,S)$ which associate to 
$F^{\bullet}\in\PSh_{\mathbb Z}(\Cor^{fs}_{\mathbb Z}(\CW^{sm}/S),C^-(\mathbb Z))$
the morphism
\begin{equation*}
\phi(f^*,S)(F^{\bullet}):f^*\underline{\sing}_{\mathbb I^*}F^{\bullet}
\to\underline{\sing}_{\mathbb I^*}f^*F^{\bullet}
\end{equation*}
in $\PSh_{\mathbb Z}(\Cor^{fs}_{\mathbb Z}(\CW^{sm}/T),C^-(\mathbb Z))$
given by for $Y/T\in\CW^{sm}/T$, the morphism 
\begin{equation*}
\phi(f^*,S)(F^{\bullet})(Y/T):\lim_{Y/T\to X_T/T}F^{\bullet}(X\times\mathbb I^*/S)
\to\lim_{Y\times\mathbb I^*/T\to X_T/T}F^{\bullet}(X/S)
\end{equation*}
given by $(h:Y/T\to X_T)\mapsto(h\circ p_Y:Y\times\mathbb I^*/T\to X_T)$
and $F^{\bullet}(p_X):(X\times\mathbb I^*/S)\to F^{\bullet}(X/S)$.

Let $f:T\to S$ a morphism in $\CW$. 
There is also canonical morphism of functor $\phi(f,e_{cw})$ which associate to
$F^{\bullet}\in\PSh_{\mathbb Z}(\Cor^{fs}_{\mathbb Z}(\CW^{sm}/S),C^-(\mathbb Z))$
the morphism in $C^-(T)$
\begin{equation*}
\phi(f^*,e)(F^{\bullet}):f^*e_{S*}F^{\bullet}\to
e_{T*}e_T^*f^*e(S)_*F^{\bullet}=e_{T*}f^*e_S^*e_{S*}F^{\bullet}\to e_{T*}f^*F^{\bullet},
\end{equation*} 
given by the adjonction morphisms and denoting for simplicity 
$e_S=e_{cw}^{tr}(S)$, $e_T=e_{cw}^{tr}(T)$. 

Let $f:T\to S$ a morphism in $\CW$. 
We denote by $\phi(f^*,e_{cw},S)$ the morphism of functor, which for
$F^{\bullet}\in\PSh_{\mathbb Z}(\Cor^{fs}_{\mathbb Z}(\CW^{sm}/S),C^-(\mathbb Z))$, 
associate the following composition in $C^-(T)$
\begin{equation*}
\phi(f^*,e_{cw},S)(F^{\bullet}):
f^*\sing_{\mathbb I^*}F^{\bullet}\xrightarrow{\phi(f^*,e)(\sing_{\mathbb I^*}F^{\bullet})}
e(T)_*f^*\underline{\sing}_{\mathbb I^*}F^{\bullet}\xrightarrow{e(T)_*\phi(f^*,S)(F^{\bullet})}
\sing_{\mathbb I^*}f^*F^{\bullet}.
\end{equation*}

By definition, we have :

\begin{prop}\label{adRSCW}
Let $f:T\to S$ a morphism in $\CW$. For $F^{\bullet}\in PC^-(CW,S)$, the morphism in $C^-(T)$
$\phi(f^*,e,S)(F^{\bullet}):f^*\sing_{\mathbb I^*}F^{\bullet}\to\sing_{\mathbb I^*}f^*F^{\bullet}$
is an isomorphism in $C^-(T)$.
\end{prop}

\begin{proof}
By definition, the morphism $\phi(f^*,e,S)(F^{\bullet})$ 
is given by for $T^o\subset T$, 
\begin{equation*}
\phi(f^*,S,e)(F^{\bullet})(T^o/T):\lim_{T^o\to f^{-1}(S^o)}F^{\bullet}(S^o\times\mathbb I^*/S)
\to\lim_{T^o\times\mathbb I^*/T\to X_T}F^{\bullet}(X/S)
\end{equation*}
given by the isomorphism $(h:T^o/T\to X_T/T)\mapsto(h\circ p_T:T^o\times\mathbb I^*/T\to X_T/T)$
and $F^{\bullet}(p_{S^o}):F^{\bullet}(S^o\times\mathbb I^*/S)\to F^{\bullet}(S^o/S)$.
\end{proof}

\subsection{The relative Betti realisation functor}

\begin{itemize}
\item For each $S\in\Var(\mathbb C)$  we have  
the analytical functor $\An(S):\Var(\mathbb C))/S\to\AnSp(\mathbb C)/S^{an}$ 
given by $(V/S)\mapsto V^{an}/S^{an}$ on objects and $g\mapsto g^{an}$ on morphisms,
the analytical functor on transfers
$\An(S):\Cor^{fs}_{\Lambda}(\Var(\mathbb C)^{sm}/S)\to\Cor^{fs}_{\mathbb Z}(\AnSp(\mathbb C)^{sm}/S^{an}))$ 
given by $V/S\mapsto V^{an}/S^{an}$ on objects and $\Gamma\mapsto\Gamma^{an}$ on morphisms,

\item For each $S\in\AnSp(\mathbb C)$, we have  
the forgetful functor $\Cw(S):\AnSp(\mathbb C)^{sm}/S\to\CW/S^{cw}$ 
given by $W/S\mapsto W^{cw}/S^{cw}$ on objects and $g\mapsto g^{cw}$ on morphisms,
and the forgetful functor on transfers 
$\Cw(S):\Cor^{fs}_{\mathbb Z}(\AnSm(\mathbb C)^{sm}/S)\to\Cor^{fs}_{\mathbb Z}(\CW^{sm}/S^{cw})$ 
given by $W/S\mapsto W^{cw}/S^{cw}$ on objects and $\Gamma\mapsto\Gamma^{cw}$ on morphisms,

\item For each $S\in\Var(\mathbb C)$, we have  
the composites $\widetilde\Cw(S)=\Cw(S^{an})\circ\An(S):\Var(\mathbb C)/S\to\CW/S$, 
given by $V/S\mapsto V^{cw}/S^{cw}$ on objects and $g\mapsto g^{cw}$ on morphisms, and 
$\widetilde\Cw(S)=\Cw(S^{an})\circ\An(S):\Cor^{fs}_{\mathbb Z}(\Var(\mathbb C)^{sm}/S)\to\Cor^{fs}_{\mathbb Z}(\CW/S)$, 
given by $V/S\mapsto V^{cw}/S^{cw}$ and $\Gamma\mapsto\Gamma^{cw}$.

\item For each $S\in\Var(\mathbb C)$ and $S'\in\AnSp(\mathbb C)$, we have  
the embeddings of categories $\iota_{an}(S'):\AnSp(\mathbb C)^{sm}/S'\to\AnSp(\mathbb C)/S'$ and 
$\iota_{var}(S):\Var(\mathbb C)^{sm}/S\to\Var(\mathbb C)/S$.
\end{itemize}

By definition, for each $S\in\Var(\mathbb C)$, we have the following commutative diagram of sites $DCat(S)$
\begin{equation}\label{DCat(S)}
DCat(S):=\xymatrix{
\widetilde\Cw(S): \Cor^{fs}_{\mathbb Z}(\CW^{sm}/S^{cw})\ar[r]^{\Cw(S)}\ar[d]^{\Tr(S)} & 
\Cor^{fs}_{\mathbb Z}(\AnSp(\mathbb C)^{sm}/S^{an})\ar[r]^{\An(S)}\ar[d]^{\Tr(S)} 
& \Cor^{fs}_{\mathbb Z}(\Var(\mathbb C)^{sm}/S)\ar[d]^{\Tr(S)} \\
\widetilde\Cw(S): \mathbb Z(\CW^{sm}/S^{cw})\ar[r]^{\Cw(S)}\ar[rd]^{\Cw(S)} & 
\mathbb Z(\AnSp(\mathbb C)^{sm}/S^{an})\ar[r]^{\An(S)} & \mathbb Z(\Var(\mathbb C)^{sm}/S) \\
 \, & \mathbb Z(\AnSp(\mathbb C)/S^{an})\ar[r]^{\An(S)}\ar[u]^{\iota_{an}(S)} & \mathbb Z(\Var(\mathbb C)/S)\ar[u]^{\iota_{var}(S)} } 
\end{equation}
For $T,S\in\Var(\mathbb C)$ and $f:T\to S$ a morphism, the morphism of sites
\begin{itemize}
\item $P(f):\Var(\mathbb C)/T\to\Var(\mathbb C)/S$, $P(f^{an}):\AnSp(\mathbb C)/T^{an}\to\AnSp(\mathbb C)/S^{an}$,
and $P(f^{cw}):\CW/T^{cw}\to\CW/S^{cw}$ 
given by the pullback functor,
\item $P(f):\Cor^{fs}_{\mathbb Z}(\Var(\mathbb C)^{sm}/T)\to\Cor^{fs}_{\mathbb Z}(\Var(\mathbb C)^{sm}/S)$, 
$P(f^{an}):\Cor^{fs}_{\mathbb Z}(\AnSp(\mathbb C)/T^{an})\to\Cor^{fs}_{\mathbb Z}(\AnSp(\mathbb C)/S^{an})$, and
$P(f^{cw}):\Cor^{fs}_{\mathbb Z}(\CW/T^{cw})\to\Cor^{fs}_{\mathbb Z}(\CW^{sm}/S^{cw})$ 
given by the pullback functor,
\end{itemize}
gives a morphism of diagram of sites 
\begin{equation}
P(f):DCat(T)\to DCat(S).
\end{equation}  

We have the relative analogue of proposition \ref{AnCwtCw} :

\begin{prop}\label{RAnCwtCw}
\begin{itemize}
\item[(i)] For $S\in\Var(\mathbb C)$, the functors 
\begin{itemize}
\item $\An(S)^*:\PSh(\Var(\mathbb C)^{sm}/S,C^-(\mathbb Z))\to\PSh(\AnSp(\mathbb C)^{sm}/S^{an},C^-(\mathbb Z))$ and
\item $\An(S)^*:\PSh(\Cor^{fs}_{\mathbb Z}(\Var(\mathbb C)^{sm}/S),C^-(\mathbb Z))
\to\PSh(\Cor^{fs}_{\mathbb Z}(\AnSp(\mathbb C)^{sm}/S^{an}),C^-(\mathbb Z))$, 
\end{itemize}
derive trivially for the $(\mathbb A^1,et)$ and $(\mathbb D^1,usu)$ model structures.

\item[(ii)] Let $S\in\AnSp(\mathbb C)$.
\begin{itemize}
\item Let  $K^{\bullet}\in\PSh(\Cor^{fs}_{\mathbb Z}(\CW^{sm}/S^{cw}),C^-(\mathbb Z))$. If $K^{\bullet}$ is $\mathbb I^1$ local,
then $\Cw(S)_*K^{\bullet}$ is $\mathbb D^1$ local.
\item Let $K^{\bullet}\in\PSh(\CW^{sm}/S^{cw},C^-(\mathbb Z))$. 
If $K^{\bullet}$ is $\mathbb I^1$ local, then $\Cw(S)_*K^{\bullet}$ is $\mathbb D^1$ local.
\end{itemize}
Moreover, the functors 
\begin{itemize}
\item $\Cw(S)^*:\PSh(\AnSp(\mathbb C)^{sm}/S,C^-(\mathbb Z))\to\PSh(\CW^{sm}/S^{cw},C^-(\mathbb Z))$ and
\item $\Cw(S)^*:\PSh(\Cor^{fs}_{\mathbb Z}(\AnSp(\mathbb C)^{sm}/S),C^-(\mathbb Z))
\to\PSh(\Cor^{fs}_{\mathbb Z}(\CW^{sm}/S^{cw}),C^-(\mathbb Z))$, 
\end{itemize}
derive trivially for the $(\mathbb D^1,usu)$ and $(\mathbb I^1,usu)$ model structures.

\item[(iii)] For $S\in\Var(\mathbb C)$, the functors 
\begin{itemize}
\item $\widetilde\Cw(S)^*:\PSh(\Var(\mathbb C)^{sm}/S,C^-(\mathbb Z))\to\PSh(\CW^{sm}/S^{cw},C^-(\mathbb Z))$ and
\item $\widetilde\Cw(S)^*:\PSh(\Cor^{fs}_{\mathbb Z}(\Var(\mathbb C)^{sm}/S),C^-(\mathbb Z))
\to\PSh(\Cor^{fs}_{\mathbb Z}(\CW^{sm}/S^{cw}),C^-(\mathbb Z))$, 
\end{itemize}
derive trivially for the $(\mathbb A^1,et)$ and $(\mathbb I^1,usu)$ model structures.
\end{itemize}
\end{prop}

\begin{proof}
The proof is completely similar to the proof of proposition \ref{AnCwtCw} using
lemma \ref{RI1hpt} (ii), since for $X/S\in\AnSp(\mathbb C)/S$, 
$p_X:X^{cw}\times\mathbb D^1\to X^{cw}$ is a homotopy equivalence in $\CW/S^{cw}$.
\end{proof}

Now we have :
\begin{itemize}
\item The 2-functor $\DM^-:\Var(k)\to\TriCat$ is an homotopic 2-functor is the sense of \cite{AyoubT}
(theorem\ref{sixfunctorsVar}).
\item The 2-functor $D^-:\Var(\mathbb C)\to\TriCat$, $S\in\Var(\mathbb C)\mapsto D^-(S^{cw})$.
is an homotopic 2-functor (see \cite{KS})
\end{itemize}

\begin{defi}\cite{AyoubB}\cite{AyoubG2}
Let $S\in\Var(\mathbb C)$,
\begin{itemize}
\item[(i)] The Betti realisation functor (without transfers) is the composite :
\begin{equation*}
\Bti_0(S)^*:\DA^{-}(S,\mathbb Z)\xrightarrow{\An(S)^*}\AnDA^{-}(S^{an},\mathbb Z)\xrightarrow{Re_{an*}(S)}D^{-}(S^{an})
\end{equation*}

\item[(ii)] The Betti realization functor with transfers is the composite :
\begin{equation*}
\Bti(S)^*:\DM^{-}(S,\mathbb Z)\xrightarrow{\An(S)^*}\AnDM^{-}(S^{an},\mathbb Z)\xrightarrow{Re^{tr}_{an}(S)_*}D^{-}(S^{an})
\end{equation*}
\end{itemize}
Since $\An(S)^*$ derive trivially by proposition \ref{RAnCwtCw}(i)
and and $L\Tr(S^{an})^*:\AnDA^-(S^{an},\mathbb Z)\to\CwDM^-(S^{an},\mathbb Z)$ is the inverse of $\Tr(S^{an})_*$ 
(c.f.theorem \ref{RAnTr}(i)), we have $\widetilde\Bti_0(S)^*=\widetilde\Bti(S)^*\circ L\Tr(S)^*$.

\end{defi}

We have the following :

\begin{thm}\label{RAnBetti}\cite{AyoubB}
The Betti realization functor define morphisms of homotopic 2-functors :
\begin{itemize}
\item $S\in\Var(\mathbb C)\mapsto(\Bti_0(S):DA^-(S,\mathbb Z)\to D^-(S))$
\item $S\in\Var(\mathbb C)\mapsto(\Bti(S):DM^-(S,\mathbb Z)\to D^-(S))$.
\end{itemize}
\end{thm}

As in the absolute case, we define :

\begin{defi}
Let $S\in\Var(\mathbb C)$.
\begin{itemize}
\item[(i)] The CW-Betti realization functor (without transfers) is the composite :
\begin{equation*}
\widetilde{\Bti}_0(S)^*:\DA^{-}(S,\mathbb Z)\xrightarrow{\widetilde\Cw(S)^*}\CwDA^{-}(S^{cw},\mathbb Z)
\xrightarrow{Re_{cw}(S)_*}D^{-}(S^{cw})
\end{equation*}

\item[(ii)] The CW-Betti realisation functor with transfers is the composite :
\begin{equation*}
\widetilde\Bti(S)^*:\DM^{-}(S,\mathbb Z)\xrightarrow{\widetilde\Cw(S)^*}\CwDM^{-}(S^{cw},\mathbb Z)
\xrightarrow{Re^{tr}_{cw}(S)_*}D^{-}(S^{cw})
\end{equation*}
\end{itemize}
Similarly, in the relative case, since $\widetilde\Cw(S)^*$ derive trivially by proposition \ref{RAnCwtCw}(ii)
and $L\Tr(S^{cw})^*:\CwDA^-(S^{cw},\mathbb Z)\to\CwDM^-(S^{cw},\mathbb Z)$ is the inverse of $\Tr_(S^{cw}*$ 
(c.f.remark \ref{RTrCWrem}), we have $\widetilde\Bti_0(S)^*=\widetilde\Bti(S)^*\circ L\Tr(S)^*$.
\end{defi}

We also have the following
\begin{thm}\label{RCWBetti}
The Betti realization functor define morphisms of homotopic 2-functors :
\begin{itemize}
\item $S\in\Var(\mathbb C)\mapsto(\widetilde\Bti_0(S):DA^-(S,\mathbb Z)\to D^-(S))$
\item $S\in\Var(\mathbb C)\mapsto(\widetilde\Bti(S):DM^-(S,\mathbb Z)\to D^-(S))$.
\end{itemize}
\end{thm}

\begin{proof}
The proof is similar to the proof of theorem \ref{RAnBetti}. 
Indeed, let $T,S\in\Var(\mathbb C)$ and $f:T\to S$ a morphism.
Then for $F^{\bullet}\in PC^-(CW,S)$, the morphism
\begin{equation*}
f^*\sing_{\mathbb I^*}\widetilde\Cw(S)^*F^{\bullet}\xrightarrow{\phi(f^*,e_{cw},S)(F^{\bullet})}
\sing_{\mathbb I^*}f^*\widetilde\Cw(S)^*F^{\bullet}=\sing_{\mathbb I^*}\widetilde\Cw(T)^*f^*F^{\bullet}
\end{equation*}
in $C^-(T)$ is an isomorphism by proposition \ref{adRSCW} (i).
\end{proof}

As in the absolute case, we the following two canonical morphisms of functors :
\begin{itemize}
\item for $S\in\AnSp(\mathbb C)$, the morphism $\psi^{\Cw}(S)$, which,
for $G^{\bullet}\in\PSh(\Cor^{fs}_{\mathbb Z}(\AnSp(\mathbb C)^{sm}/S),C^-(\mathbb Z))$, associate the morphism 
\begin{equation*}
\psi^{\Cw}(S)(G^{\bullet}):\Cw(S)^*(\underline{\sing}_{\mathbb I_{an}^*}G^{\bullet})
\to\underline{\sing}_{\mathbb I^*}\Cw(S)^*G^{\bullet}
\end{equation*}
in $\PSh(\Cor^{fs}_{\mathbb Z}(\CW/S^{cw}),C^-(\mathbb Z))$ ;
the morphism $\psi^{\Cw}(S)(G^{\bullet})$ is given by, for $Z/S^{cw}\in\CW/S^{cw}$,
\begin{equation}
\psi^{\Cw}(G^{\bullet})(Z/S^{cw}):
\lim_{X^{cw}/S^{cw}\to Z/S^{cw}}G^{\bullet}(X\times\mathbb I_{an}^*/S)
\to\lim_{Y^{cw}/S^{cw}\to Z/S^{cw}\times\mathbb I^*}G^{\bullet}(Y/S)
\end{equation}
given by 
$(f:X^{cw}/S^{cw}\to Z/S^{cw})\mapsto(f\times I_{\mathbb I^*}:
(X\times\mathbb I_{an}^*)^{cw}/S^{cw}\to Z\times\mathbb I^*/S^{cw})$
and the identity of $G^{\bullet}(X\times\mathbb I_{an}^*/S)$ ;

\item for $S\in\Var(\mathbb C)$, the morphism $\psi^{\widetilde{\Cw}}(S)$, which,
for $F^{\bullet}\in\PSh(\Cor^{fs}_{\mathbb Z}(\Var(\mathbb C)/S),C^-(\mathbb Z))$, associate the morphism 
\begin{equation*}
\psi^{\widetilde{\Cw(S)}}(F^{\bullet}):\widetilde{\Cw(S)}^*(\underline{\sing}_{\mathbb I_{et}^*}F^{\bullet})
\to\underline{\sing}_{\mathbb I^*}\widetilde{\Cw(S)}^*F^{\bullet}
\end{equation*}
in $\PSh(\Cor^{fs}_{\mathbb Z}(\CW/S^{cw}),C^-(\mathbb Z))$ ;
the morphism $\psi^{\widetilde{\Cw(S)}}(F^{\bullet})$ is given by, for $Z/S^{cw}\in\CW/S^{cw}$,
\begin{equation}
\psi^{\widetilde{\Cw}}(F^{\bullet})(Z):
\lim_{X^{cw}/S^{cw}\to Z/S^{cw}}F^{\bullet}(X\times\mathbb I_{et}^*/S)\to\lim_{Y^{cw}/S^{cw}\to Z\times\mathbb I^*/S^{cw}}F^{\bullet}(Y/S)
\end{equation}
given by 
$(f:X^{cw}/S^{cw}\to Z/S^{cw})\mapsto(f\times I_{\mathbb I^*}:(X\times\mathbb I_{et}^*)^{cw}\to Z\times\mathbb I^*/S^{cw})$
and the identity of $F^{\bullet}(X\times\mathbb I_{et}^*/S)$.
\end{itemize}

\begin{defi}
As in the absolute case, we define the following two morphism of functors :
\begin{itemize}
\item[(i)] for $S\in\AnSp(\mathbb C)$, the morphism $W(S)$, which,
for $G^{\bullet}\in\PSh(\Cor^{fs}_{\mathbb Z}(\AnSp(\mathbb C)^{sm}/S),C^-(\mathbb Z))$, associate the composition
\begin{equation*}
W(S)(G^{\bullet}):
\Cw(S)^*(\underline{\sing}_{\bar{\mathbb D}^*}G^{\bullet})\xrightarrow{\Cw(S)^*(\underline{G}^{\bullet}(i'_1))}
\Cw(S)^*(\underline{\sing}_{\mathbb I_{an}^*}G^{\bullet})\xrightarrow{\psi^{\Cw}(S)(G^{\bullet})}
\underline{\sing}_{\mathbb I^*}\Cw(S)^*G^{\bullet}
\end{equation*}
in $\PSh(\Cor^{fs}_{\mathbb Z}(\CW/S^{cw}),C^-(\mathbb Z))$,

\item[(ii)] for $S\in\Var(\mathbb C)$ the morphism $\widetilde W(S)$, which,
for $F^{\bullet}\in\PSh(\Cor^{fs}_{\mathbb Z}(\SmVar(\mathbb C)),C^-(\mathbb Z))$, associate the composition
\begin{equation*}
\widetilde W(S)(F^{\bullet}):
\widetilde\Cw(S)^*(\underline{\sing}_{\mathbb I^*}F^{\bullet})\xrightarrow{\widetilde\Cw(S)^*(\underline{F}^{\bullet}(i))}
\widetilde\Cw(S)^*(\underline{\sing}_{\mathbb I_{et}^*}F^{\bullet})\xrightarrow{\psi^{\widetilde\Cw}(S)(F^{\bullet})}
\underline{\sing}_{\mathbb I^*}\widetilde\Cw(S)^*F^{\bullet}
\end{equation*}
in $\PSh(\Cor^{fs}_{\mathbb Z}(\CW/S^{cw}),C^-(\mathbb Z))$.
\end{itemize}
\end{defi}

We have the relative analogue to proposition \ref{AWtW} : 

\begin{prop}\label{RAWtW}
\begin{itemize}
\item[(i)] Let $S\in\AnSp(\mathbb C)$.
For $G^{\bullet}\in PC^-(\An,S)$, 
$W(G^{\bullet})(S):\Cw(S)^*(\underline{\sing}_{\bar{\mathbb D}^*}G^{\bullet})
\to\underline{\sing}_{\mathbb I^*}\Cw(S)^*G^{\bullet}$
is an equivalence $(\mathbb I^1,usu)$ local in $PC^-(CW,S^{cw})$.

\item[(ii)] Let $S\in\Var(\mathbb C)$.
For $F^{\bullet}\in PC^-(S)$,
$\widetilde W(S)(F^{\bullet}):\widetilde\Cw(S)^*(\underline{\sing}_{\mathbb I^*}F^{\bullet})
\to\underline{\sing}_{\mathbb I^*}\widetilde\Cw(S)^*F^{\bullet}$
is an $(\mathbb I^1,usu)$ local equivalence in $PC^-(CW,S^{cw})$.
\end{itemize}
\end{prop}

\begin{proof}
Similar to proposition \ref{AWtW}.
\end{proof}

\begin{defi}
For $S\in\AnSp(\mathbb C)$, we define the morphism of functor $B(S)$, by associating
to $G^{\bullet}\in\PSh(\Cor^{fs}_{\mathbb Z}(\AnSp^{sm}(\mathbb C)/S),C^-(\mathbb Z))$, 
the morphism $B(S)(G^{\bullet})$ which is the composite 
\begin{equation*}
\xymatrix{
\underline{\sing}_{\bar{\mathbb D}^*}G^{\bullet}
\ar[d]_{\ad(\Cw(S)^*,\Cw(S)_*)(\underline{\sing}_{\bar{\mathbb D}^*}G^{\bullet})}\ar[rrrd]^{B(S)(G^{\bullet})} & \, & \, & \, \\
\Cw(S)_*\Cw(S)^*\underline{\sing}_{\bar{\mathbb D}^*}G^{\bullet}\ar[rrr]^{\Cw(S)_*(W(S)(G^{\bullet}))} & \, & \, &
\Cw(S)_*\underline{\sing}_{\mathbb I^*}\Cw(S)^*G^{\bullet}}
\end{equation*}
in $\PSh(\Cor^{fs}_{\mathbb Z}(\AnSp(\mathbb C)^{sm}/S),C^-(\mathbb Z))$
\end{defi}

We have the following key proposition.

\begin{prop}\label{RBettiprop}
Let $S\in\Var(\mathbb C)$ and $F^{\bullet}\in PC^-(S)$ 
such that $D(\mathbb A^1,et)(S)(F^{\bullet})\in\DM^-(S,\mathbb Z)$ is a constructible motive. Then,
\begin{itemize}
\item[(i)] 
$e^{tr}_{an}(S)_*B(S)(F^{\bullet}):
\sing_{\bar{\mathbb D}^*}\An(S)^*F^{\bullet}\to\sing_{\mathbb I^*}\widetilde\Cw(S)^*F^{\bullet}$
is an equivalence usu local in $C^{-}(S^{an})$.

\item[(ii)] 
$B(S)(F^{\bullet}):\underline{\sing}_{\bar{\mathbb D}^*}\An(S)^*F^{\bullet}\to
\Cw(S^{an})_*\underline{\sing}_{\mathbb I^*}\widetilde\Cw(S)^*F^{\bullet}$
is an equivalence $(\mathbb D^1,usu)$ local.
\end{itemize}
\end{prop}

\begin{proof}

\noindent(i): As $M=D(\mathbb A^1,et)(S)(F^{\bullet})$ is a constructible motive,
we may assume using induction that there exist $X/S,Y/S\in\Var(\mathbb C)^{sm}/S$,
$X/S=(X,h_1),Y/S=(X,h_2)$ such that 
\begin{equation*}
M=D(\mathbb A_1,et)(S)(\Cone(\alpha)) 
\end{equation*}
with $\alpha\in\Hom_{PC^-(S)}(F(X/S),F(Y/S,p,n))$, where we have chosen  
$k_1:\mathbb Z_{tr}(X/S)\to F(X/S)$ and 
$k_2:\mathbb Z_{tr}(Y/S)(p)[n]\to F(Y/S,p,n)$ 
equivalences $(\mathbb A^1,et)$ local with
$F(X/S)$ and $F(X/S,p,n)$ are $\mathbb A^1$ local and etale fibrant objects. 
Consider the following commutative diagrams in $PC^-(\An,S^{an})$ :
\begin{itemize}
\item[(1)]
$\xymatrix{
\sing_{\mathbb D^*}\mathbb Z_{tr}(X^{an}/S^{an})\ar[rrr]^{e_{an}(S)_*B(\mathbb Z_{tr}(X^{an}/S^{an}))}
\ar[d]^{S(\An(S)^*k_1)} & \, & \, &
\sing_{\mathbb I^*}\mathbb Z_{tr}(X^{cw}/S^{cw})\ar[d]^{S(\widetilde\Cw(S)^*k_2)} \\
\sing_{\mathbb D^*}\An(S)^*F(X/S)\ar[rrr]^{e_{an}(S)_*B(\An(S)^*F(X/S))} & \, & \, &
\sing_{\mathbb I^*}\widetilde\Cw(S)^*F(X/S)}$,
\item[(2)] 
$\xymatrix{
\sing_{\mathbb D^*}\mathbb Z_{tr}(Y^{an}/S^{an})(p)\ar[rrr]^{e_{an}(S)_*B(\mathbb Z_{tr}(Y^{an}/S^{an})(p))}
\ar[d]^{S(\An(S)^*k_2)} & \, & \, &
\sing_{\mathbb I^*}\mathbb Z_{tr}(Y^{cw}/S^{cw})(p)\ar[d]^{S(\widetilde\Cw(S)^*k_2)} \\
\sing_{\mathbb D^*}\An(S)^*F(Y/S)\ar[rrr]^{e_{an}(S)_*B(\An(S)^*F(Y/S))} & \, & \, &
\sing_{\mathbb I^*}\widetilde\Cw(S)^*F(Y/S)}$.
\end{itemize}
Then, 
\begin{itemize}
\item $S(\An(S)^*k_1):\sing_{\mathbb D^*}\mathbb Z_{tr}(X^{an}/S^{an})\to \sing_{\mathbb D^*}\An(S)^*F(X/S)$ 
\item $S(\An(S)^*k_2):\sing_{\mathbb D^*}\mathbb Z_{tr}(Y^{an}/S^{an})(p)[n]\to\sing_{\mathbb D^*}\An(S)^*F(Y/S,p,n)$ 
\end{itemize}
are equivalences $(\mathbb D^1,usu)$ local by proposition \ref{RAnCwtCw}(i) and theorem \ref{RAnS}(ii).
Similarly, 
\begin{itemize}
\item $S(\widetilde\Cw(S)^*k_1):\sing_{\mathbb I^*}\mathbb Z_{tr}(X^{cw}/S^{cw})\to
\sing_{\mathbb I^*} \widetilde\Cw(S)^*F(X/S)$
\item $S(\widetilde\Cw(S)^*k_2):\sing_{\mathbb I^*}\mathbb Z_{tr}(Y^{cw}/S^{cw})(p)[n]
\to\sing_{\mathbb I^*}\widetilde\Cw(S)^*F(Y/S,p,n)$ 
\end{itemize}
are equivalences $(\mathbb I^1,usu)$ local by proposition \ref{RAnCwtCw} (iii) and theorem \ref{RCWsingI}(ii).
On the other side,
\begin{itemize}
\item $e_{an}(S)_*B(\mathbb Z_{tr}(X^{an}/S^{an})):
\sing_{\mathbb D^*}\mathbb Z_{tr}(X^{an}/S^{an})\to\sing_{\mathbb I^*}\mathbb Z_{tr}(X^{cw}/S^{cw})$ 
is an equivalence usu local by proposition \ref{Bettiprop}(i) applied to 
\begin{equation*}
e_{an*}B(\mathbb Z_{tr}(X^{an}_s)):\sing_{\mathbb D^*}\mathbb Z_{tr}(X^{an}_s)
\to\sing_{\mathbb I^*}\mathbb Z_{tr}(X^{cw}_s)  
\end{equation*}
for each $s\in S$ : since $h_1:X\to S$ is smooth, 
$i_s^*\mathbb Z_{tr}(X^{an}/S^{an})=\mathbb Z_{tr}(X^{an}_s)$ 
and $i_s^*\mathbb Z_{tr}(X^{cw}/S^{cw})=\mathbb Z_{tr}(X^{cw}_s)$,
where $i_s:\left\{s\right\}\hookrightarrow S$ is the closed embedding.
\item $e_{an}(S)_*B(\mathbb Z_{tr}(Y/S)(p)[n]):
\sing_{\mathbb D^*}\mathbb Z_{tr}(Y^{an}/Y^{an})(p)[n]\to\sing_{\mathbb I^*}\mathbb Z_{tr}(Y^{cw}/S^{cw})[n]$ 
is an equivalence usu local by proposition \ref{Bettiprop}(i) applied to
\begin{equation*}
e_{an*}B(\mathbb Z_{tr}(Y^{an}_s)(p)):\sing_{\mathbb D^*}\mathbb Z_{tr}(Y^{an}_s)(p)
\to\sing_{\mathbb I^*}\mathbb Z_{tr}(Y^{cw}_s)(p)  
\end{equation*}
for each $s\in S$ : since $h_2:Y\to S$ is smooth
$i_s^*\mathbb Z_{tr}(Y^{an}/S^{an})=\mathbb Z_{tr}(Y^{an}_s)$ 
and $i_s^*\mathbb Z_{tr}(Y^{cw}/S^{cw})=\mathbb Z_{tr}(Y^{cw}_s)$,
where $i_s:\left\{s\right\}\hookrightarrow S$ is the closed embedding.
\end{itemize} 
The diagram (1) then shows that
\begin{equation*}
e_{an}(S)_*B(\An(S)^*F(X/S)):\sing_{\mathbb D^*}F(X/S)\to\sing_{\mathbb I^*}\widetilde\Cw(S)^*F(X/S)  
\end{equation*}
is an equivalence usu local in $C^-(S^{an})$,
and the diagram (2) that
\begin{equation*}
e_{an}(S)_*B(\An(S)^*F(Y/S)):\sing_{\mathbb D^*}F(Y/S,p,n)\to\sing_{\mathbb I^*}\widetilde\Cw(S)^*F(Y/S,p,n)  
\end{equation*}
is equivalence usu local in $C^-(S^{an})$.

\noindent(ii): Follows from (i) 
Let us explain.
\begin{itemize}
\item On the one hand,
\begin{itemize}
\item By theorem \ref{RAnS} (ii),
$\underline{\sing}_{\bar{\mathbb D}^*}\An(S)^*F^{\bullet}$
is $\mathbb D^1$ local.
\item By theorem \ref{CWsingI}(ii), 
$\underline{\sing}_{\mathbb I^*}\widetilde\Cw(S)^*F^{\bullet}$
is $\mathbb I^1$ local.
Hence, $\Cw(S^{an})_*\underline{\sing}_{\mathbb I^*}\widetilde\Cw(S)^*F^{\bullet}$
is $\mathbb D^1$ local, by proposition \ref{AnCwtCw} (ii).
\end{itemize}
\item On the other hand by (i)
$e^{tr}_{an}(S)_*(B(S)(F^{\bullet})):
\sing_{\bar{\mathbb D}^*}\An(S)^*F^{\bullet}\to\sing_{\mathbb I^*}\widetilde\Cw(S)^*F^{\bullet}$
is a quasi isomorphism in $C^{-}(S)$.
\end{itemize}
Hence, by proposition \ref{RDusuEq} (ii), 
\begin{equation*}
B(S)(F^{\bullet}):\underline{\sing}_{\bar{\mathbb D}^*}\An(S)^*F^{\bullet}\to
\Cw(S^{an})_*\underline{\sing}_{\mathbb I^*}\widetilde\Cw(S)^*F^{\bullet}
\end{equation*}
is an equivalence $(\mathbb D^1,usu)$ local.

\end{proof}

The proposition \ref{RBettiprop} gives the following relative version of theorem \ref{mainthm}.

\begin{thm}\label{Rmainthm}
Let $S\in\Var(\mathbb C)$. 
Let $M\in\DM^-(S,\mathbb Z)$ is a constructible motive.
\begin{itemize}
\item[(i)] We have $\Bti^*M=\widetilde\Bti^*M$

\item[(ii)] Let $M_1,M_2\in\DM^-(S,\mathbb Z)$ contructible motives.
Let $F_1^{\bullet},F_2^{\bullet}\in PC^-(S)$ 
such that $M_i=D(\mathbb A^1,et)(S)(F_i^{\bullet})\in\DM^-(S,\mathbb Z)$
for $i=1,2$. The following diagram is commutative
\begin{equation*}
\xymatrix{\Hom_{\DM^{-}(S)}(M_1,M_2)\ar^{\widetilde\Cw(S)^*}[r]\ar[d]_{\An(S)^*} &  
\Hom_{\CwDM^-(S)}(\widetilde\Cw(S)^*M_1,\widetilde\Cw(S)^*M_2)\ar[d]^{Re^{tr}_{cw}(S)_*} \\  
\Hom_{\AnDM^-(S)}(\An(S)^*M_1,\An(S)^*M_2)\ar[r]^{Re^{tr}_{an}(S)_*} &  
\Hom_{D^-(S)}(\sing_{\mathbb I^*}\widetilde\Cw(S)^*F^{\bullet}_1,
\sing_{\mathbb I^*}\widetilde\Cw(S)^*F^{\bullet}).} 
\end{equation*}
\end{itemize}
\end{thm}

\begin{proof}
\noindent(i): By definition, $\Bti(S)^*M=Re^{tr}_{an}(S)_*\An(S)^*M$.
Since, by theorem \ref{RAnS}(ii), 
\begin{itemize}
\item $S(F^{\bullet}):\An(S)^*F^{\bullet}\to\underline{\sing}_{\mathbb D^*}\An(S)^*F^{\bullet}$
is an equivalence $(\mathbb D^1,usu)$ local in $PC^-(\An,S^{an})$ and
\item $\underline{\sing}_{\mathbb D^*}\An(S)^*F^{\bullet}$ is a $\mathbb D^1$ local object,
\end{itemize}
we have, since $\An(S)^*$ derive trivially by proposition \ref{RAnCwtCw}(i),
$\Bti^*M=e^{tr}_{an}(S)_*(\underline{\sing}_{\mathbb D^*}\An(S)^*F^{\bullet})
=\sing_{\mathbb D^*}\An(S)^*F^{\bullet}$.
Since
\begin{itemize}
\item $B(\An(S)^*F^{\bullet}):\underline{\sing}_{\mathbb D^*}\An(S)^*F^{\bullet}
\to\Cw_*\underline{\sing}_{\mathbb I^*}\widetilde\Cw(S)^*F^{\bullet}$
is an equivalence $(\mathbb D^1,usu)$ local in $PC^-(\An,S^{an})$ 
by proposition \ref{RBettiprop} (ii), and
\item $\Cw(S)_*\underline{\sing}_{\mathbb I^*}\widetilde\Cw(S)^*F^{\bullet}$
is a $\mathbb D^1$ local object by theorem \ref{RCWsingI}(ii) and proposition \ref{RAnCwtCw}(ii),
\end{itemize}
we have, since $\widetilde\Cw(S)^*$ derive trivially by proposition \ref{RAnCwtCw}(iii),
\begin{equation}\label{Rmaineq}
\Bti^*M=e^{tr}_{an}(S)_*(\Cw(S)_*\underline{\sing}_{\mathbb I^*}\widetilde\Cw(S)^*F^{\bullet})
=\sing_{\mathbb I^*}\widetilde\Cw(S)^*F^{\bullet}
\end{equation}

By definition, $\widetilde\Bti^*M=Re^{tr}_{cw*}\widetilde\Cw(S)^*M$.
Since, by theorem \ref{RCWsingI}(ii),
\begin{itemize}
\item $S(\widetilde\Cw(S)^*F^{\bullet}):\widetilde\Cw(S)^*F^{\bullet}\to
\underline{\sing}_{\mathbb I^*}\widetilde\Cw(S)^*F^{\bullet}$
is an equivalence $(\mathbb I^1,usu)$ local in $PC^-(\CW,S^{cw})$ and 
\item $\underline{\sing}_{\mathbb I^*}\widetilde\Cw(S)^*F^{\bullet}$ is an $\mathbb I^1$ local object,
\end{itemize}
we have
\begin{eqnarray*}
\widetilde\Bti^*M=e^{tr}_{cw}(S)(\underline{\sing}_{\mathbb I^*}\widetilde\Cw(S)^*F^{\bullet})
&=&\sing_{\mathbb I^*}\widetilde\Cw(S)^*F^{\bullet}) \\
&=&\Bti(S)^*M \mbox{\; by (\ref{Rmaineq}) }
\end{eqnarray*}
This proves (i).

\noindent(ii): 
Let $\alpha\in\Hom_{\DM^-(\mathbb C,\mathbb Z)}(M_1,M_2)$.
Consider the commutative diagram in $\AnDM(S,\mathbb Z)$ 
\begin{equation*}
\xymatrix{
\underline{\sing}_{\mathbb D^*}\An(S)^*F_1^{\bullet}\ar[rr]^{\An(S)^*\alpha}\ar[d]_{B(\An(S)^*F_1^{\bullet}} & \, &
\underline{\sing}_{\mathbb D^*}\An(S)^*F_2^{\bullet}\ar[d]^{B(\An(S)^*F_2^{\bullet})} \\
\Cw(S)_*\underline{\sing}_{\mathbb I^*}\widetilde\Cw(S)^*F_1^{\bullet}\ar[rr]^{\Cw(S)_*\widetilde\Cw(S)^*\alpha} & \, &
\Cw(S)_*\underline{\sing}_{\mathbb I^*}\widetilde\Cw(S)^*F_2^{\bullet}}
\end{equation*}
Since
$\underline{\sing}_{\mathbb D^*}\An^*F_1^{\bullet}$ and $\underline{\sing}_{\mathbb D^*}\An^*F_2^{\bullet}$
are $\mathbb D^1$ local objects by theorem \ref{RAnS}(ii)
and since $\An(S)$ derive trivially by proposition \ref{RAnCwtCw}(i),
\begin{equation*}
\An(S)^*\alpha\in\Hom_{\Ho_{usu}(PC^-(\An(S),S^{an}))} 
(\underline{\sing}_{\mathbb D^*}\An(S)^*F_1^{\bullet},\underline{\sing}_{\mathbb D^*}\An(S)^*F_2^{\bullet})
\end{equation*}
Thus,
\begin{equation}\label{RAnaplha}
\Bti(S)^*(\alpha):=Re^{tr}_{an}(S)_*(\An(S)^*\alpha)=e^{tr}_{an}(S)_*\An(S)^*\alpha
\end{equation}
Since
$\underline{\sing}_{\mathbb I^*}\widetilde\Cw(S)^*F_1^{\bullet}$ and 
$\underline{\sing}_{\mathbb I^*}\widetilde\Cw(S)^*F_2^{\bullet}$
are $\mathbb I^1$ local objects by theorem \ref{RCWsingI}(ii), 
\begin{equation*}
\widetilde\Cw(S)^*\alpha\in\Hom_{\Ho_{usu}(PC^-(CW,S^{cw}))} 
(\underline{\sing}_{\mathbb I^*}\widetilde\Cw(S)^*F_1^{\bullet},
\underline{\sing}_{\mathbb I^*}\widetilde\Cw(S)^*F_2^{\bullet})
\end{equation*}
Thus, since $\widetilde\Cw(S)^*$ derive trivially by proposition \ref{RAnCwtCw}(iii),
\begin{equation}\label{RCwalpha1}
\widetilde\Bti(S)^*(\alpha):=Re^{tr}_{cw}(S)_*(\widetilde\Cw(S)^*\alpha)=e^{tr}_{cw}(S)_*\widetilde\Cw(S)^*\alpha
\end{equation}
Since
$\Cw(S)_*\underline{\sing}_{\mathbb I^*}\widetilde\Cw(S)^*F_1^{\bullet}$ and 
$\Cw(S)_*\underline{\sing}_{\mathbb I^*}\widetilde\Cw(S)^*F_2^{\bullet}$
are $\mathbb D^1$ local objects by theorem \ref{RCWsingI}(ii) and proposition \ref{RAnCwtCw} (ii), 
\begin{equation}\label{Radcwusu}
\Cw(S)_*\widetilde\Cw(S)^*\alpha\in\Hom_{\Ho_{usu}(PC^-(\An,S^{an}))} 
(\Cw(S)_*\underline{\sing}_{\mathbb I^*}\widetilde\Cw(S)^*F_1^{\bullet},
\Cw(S)_*\underline{\sing}_{\mathbb I^*}\widetilde\Cw(S)^*F_2^{\bullet})
\end{equation}
Thus
\begin{eqnarray*}
\Bti^*(\alpha)&=&Re^{tr}_{an*}(\Cw(S)_*\widetilde\Cw(S)^*\alpha) 
\mbox{\; since \;} B(\An(S)^*F_1^{\bullet}) \mbox{\; and \; } B(\An(S)^*F_2^{\bullet}) \\
&\, &\mbox{\; are \;} (\mathbb D^1,usu) \mbox{\; local \;  equivalence \; by \; proposition \ref{RBettiprop}(ii)} \\
&=&e^{tr}_{an}(S)_*(\Cw(S)_*\widetilde\Cw(S)^*\alpha) \mbox{\; by \; (\ref{Radcwusu})} \\
&=&e^{tr}_{cw}(S)_*(\widetilde\Cw(S)^*\alpha) \\
&=&\widetilde\Bti(S)^*(\alpha) \mbox{\; by \; (\ref{RCwalpha1}).}
\end{eqnarray*}

\end{proof}


\end{document}